\documentclass{article}
\title{Functionary Models of Real Analysis}
\author{Matouš Vladimír Schnabel}
\usepackage{times}
\usepackage{etoolbox}
\usepackage{amsmath}
\usepackage{amsfonts}
\usepackage{amssymb}
\usepackage{amsthm}
\usepackage{fancyhdr}
\usepackage{cases}
\usepackage[a4paper, margin=1.25in]{geometry}
\usepackage{setspace}
\usepackage{nicefrac}
\usepackage[english]{babel}
\usepackage{thmtools}
\usepackage{enumitem}
\usepackage{tikz}
\usetikzlibrary{decorations.pathreplacing,arrows.meta,calc,positioning,fit,backgrounds}
\tikzset{
  settled/.style={draw=blue!58!black, fill=blue!12, line width=0.6pt},
  active/.style ={draw=orange!78!black, fill=orange!28, line width=0.6pt},
  baseaxis/.style={draw=black!55, line width=0.7pt},
  poslab/.style ={font=\footnotesize, text=black!70},
  wlab/.style   ={font=\scriptsize, text=black!50},
  read/.style   ={font=\small, text=black!85},
}
\newenvironment{subtheorems}
 {\enumerate[
     label=\Roman*.\protect\thissubtheorem,
     ref=\Roman*.,
     font=\normalfont\bfseries,
     leftmargin=0pt,itemindent=!,
     align=left,
  ]}
 {\endenumerate}
\newcommand{\thissubtheorem}{}
\newcommand{\subtheorem}[1][]{%
  \if\relax\detokenize{#1}\relax
    \def\thissubtheorem{}%
  \else
    \def\thissubtheorem{ (#1)}%
  \fi
  \item
}
\makeatletter
\newenvironment{proofidea}[1][]{%
  \par\addvspace{0.18cm}%
  \begin{list}{}{%
    \setlength{\leftmargin}{1.6em}\setlength{\rightmargin}{0.6em}%
    \setlength{\labelwidth}{0pt}\setlength{\labelsep}{0pt}%
    \setlength{\itemindent}{0pt}\setlength{\listparindent}{1em}%
    \setlength{\topsep}{0pt}\setlength{\parsep}{0.4ex}}%
  \item[]\small\itshape Idea of the proof\ifstrempty{#1}{}{ (#1)}.\upshape
  \hspace{0.35em}\ignorespaces
}{%
  \end{list}\addvspace{0.18cm}%
}
\makeatother

\declaretheoremstyle[
  spaceabove=0.8\baselineskip, spacebelow=0.6\baselineskip,
  headfont=\normalfont\bfseries,
  bodyfont={\normalfont\itshape\setlength{\parskip}{\stmtsep}},
  headpunct={.}, postheadspace={\newline},
  postheadhook={\hspace*{\parindent}}
]{fmtheorem}
 
\declaretheoremstyle[
  spaceabove=0.8\baselineskip, spacebelow=0.6\baselineskip,
  headfont=\normalfont\bfseries\scshape,
  bodyfont={\normalfont\itshape\setlength{\parskip}{\stmtsep}},
  headpunct={.}, postheadspace={\newline},
  postheadhook={\hspace*{\parindent}}
]{fmaxiom}
 
\declaretheoremstyle[
  spaceabove=0.7\baselineskip, spacebelow=0.55\baselineskip,
  headindent=1em,
  headfont=\normalfont\bfseries,
  bodyfont={\normalfont\itshape\setlength{\parskip}{\stmtsep}},
  headpunct={.}, postheadspace={\newline},
  postheadhook={\hspace*{\parindent}}
]{fmlemma}
 
\declaretheoremstyle[
  spaceabove=0.6\baselineskip, spacebelow=0.5\baselineskip,
  headindent=2em,
  headfont=\normalfont\bfseries\itshape,
  bodyfont={\normalfont\itshape\setlength{\parskip}{\stmtsep}},
  headpunct={.}, postheadspace={\newline},
  postheadhook={\hspace*{\parindent}}
]{fmcorollary}
 
\declaretheoremstyle[
  spaceabove=0.8\baselineskip, spacebelow=0.6\baselineskip,
  headfont=\normalfont\bfseries,
  bodyfont={\normalfont\upshape\setlength{\parskip}{\stmtsep}},
  headpunct={.}, postheadspace={\newline},
  postheadhook={\hspace*{\parindent}}
]{fmdefinition}
 
\declaretheoremstyle[
  spaceabove=0.6\baselineskip, spacebelow=0.5\baselineskip,
  headfont=\normalfont\itshape,
  bodyfont={\normalfont\upshape\setlength{\parskip}{\stmtsep}},
  headpunct={.}, postheadspace=0.5em
]{fmintuition}
 
\declaretheoremstyle[
  spaceabove=0.5\baselineskip, spacebelow=0.45\baselineskip,
  headindent=1em,
  headfont=\normalfont\scshape,
  bodyfont={\normalfont\upshape\setlength{\parskip}{\stmtsep}},
  headpunct={.}, postheadspace=0.5em
]{fmexample}
 
\declaretheoremstyle[
  spaceabove=0.5\baselineskip, spacebelow=0.45\baselineskip,
  headfont=\normalfont\itshape,
  bodyfont={\normalfont\upshape\small\setlength{\parskip}{\stmtsep}},
  headpunct={.}, postheadspace=0.5em
]{fmremark}

\makeatletter
\newcommand{\HelperDeclareTheorem}[5][plain]{%
  \declaretheorem[
    style=#1,
    numberlike=subsubsection,
    name=#4,
    refname={#2, #3},
    Refname={#4, #5},
    postheadhook={%
      \addcontentsline{toc}{subsubsection}{\protect\numberline{\thesubsubsection} #4 \ifdefempty{\thmt@optarg}{}{(\thmt@optarg)}}%
    }
  ]{#2}%
}
\makeatother
 
\HelperDeclareTheorem[fmtheorem]{theorem}{theorems}{Theorem}{Theorems}
\HelperDeclareTheorem[fmdefinition]{definition}{definitions}{Definition}{Definitions}
\HelperDeclareTheorem[fmaxiom]{axiom}{axioms}{Axiom}{Axioms}
 
\theoremstyle{fmlemma}     \newtheorem{lemma}{Lemma}[theorem]
\theoremstyle{fmcorollary} \newtheorem{corollary}{Corollary}[lemma]



\theoremstyle{fmintuition} \newtheorem*{intuition}{Intuition}
\theoremstyle{fmexample}   \newtheorem*{example}{Example}
\theoremstyle{fmremark}    \newtheorem*{remark}{Remark}

\makeatletter
\renewenvironment{proof}[1][\proofname]{\par
  \pushQED{\qed}%
  \normalfont \topsep6\p@\@plus6\p@\relax
  \setlength{\parskip}{\stmtsep}%
  \trivlist
  \item[\hskip\labelsep\normalfont\bfseries #1\@addpunct{.}]\ignorespaces
}{%
  \popQED\endtrivlist\@endpefalse
}
\makeatother

\tikzset{
  fnbox/.style={draw=black!55, fill=white, rounded corners=2pt, inner sep=3.5pt,
                font=\footnotesize, minimum height=6mm},
  prim/.style ={fnbox, draw=blue!58!black, fill=blue!9},
  sec/.style  ={fnbox, draw=orange!78!black, fill=orange!14},
  wbox/.style ={draw=black!45, dashed, rounded corners=4pt, inner sep=5pt},
  rbox/.style ={draw=black!70, rounded corners=6pt, inner sep=9pt, line width=0.7pt},
  lbl/.style  ={font=\scriptsize, text=black!65},
  hd/.style   ={font=\scriptsize\itshape, text=black!70},
}
\usepackage{longtable}
\usepackage{array}
\usepackage{booktabs}
\numberwithin{equation}{section}
\usepackage[hidelinks]{hyperref}
\usepackage{indentfirst}
\usepackage{indentfirst}
\newlength{\stmtsep}
\makeatletter
\let\FMoldcr\\                     

\DeclareRobustCommand{\FMparcr}{\@ifstar{\FMparcr@i}{\FMparcr@i}}
\newcommand{\FMparcr@i}[1][\z@]{%
  \par\ifdim\dimexpr#1\relax>\z@\vskip#1\relax\fi\ignorespaces}
\AtBeginDocument{\let\\\FMparcr}
\makeatother
\begin{document}
\onehalfspacing
\pagenumbering{roman}
\maketitle
\tableofcontents
\begin{abstract}
    A real number can have more than one name: $0.5$ and $0.4999\ldots$ denote the same thing, and which numbers enjoy such a doubling depends on the base one writes in. We make that dependence the object of study. Fixing a \emph{system function} $\vartheta$, which assigns a base $\vartheta(n)\geq2$ to every position independently, we build the real numbers as equivalence classes of digit functions $f:\omega\rightarrow\omega$, where the equivalence is generated by two local carrying moves, \emph{contraction} and \emph{broadening}, and tested by agreement on finite initial segments. Each class turns out to contain at most two canonical representatives (up to a sign representation), its \emph{primary} and \emph{secondary} auxiliary functions, so the doubling above is a theorem of the theory rather than a convention imposed on it. We define order, addition and multiplication and verify the axioms of a Dedekind-complete ordered field, so that by categoricity every choice of $\vartheta$ delivers $\mathbb{R}$ itself. The models are therefore indistinguishable as ordered fields and differ only in how their elements are named, and since $\vartheta$ ranges over an uncountable parameter space, that naming can be chosen to suit a problem. We show what this buys: in any base whose partial products absorb every denominator, a non-zero real number is rational precisely when it has a second name, and Cantor's 1869 irrationality criterion for Cantor series follows from the representation theory rather than from number theory.
\end{abstract}
\section{Introduction}
Every mathematician knows that $0.999\ldots$ and $1$ are the same real number, and most regard the fact as an awkwardness of decimal notation rather than as information. It is worth asking what the awkwardness depends on. In base ten the numbers with two names are exactly the rationals whose denominators divide a power of ten; in base two they are a different set; and one may ask what happens if the base is allowed to change from one position to the next. This paper is an attempt to answer that question properly, which requires first building a setting in which it can be asked.
 
There are two classical routes from the rationals to the reals: Dedekind's cuts and Cantor's Cauchy sequences \cite{Dedekind1901,Cantor1872}. Both are constructions \emph{about} the reals rather than constructions one computes \emph{in}, and neither has anything to say about naming: a cut is not an object one writes down digit by digit. A third route, from digit expansions, is the one every student first meets and the one that working mathematics silently uses, yet it is the least travelled in the foundational literature, for a well-known reason. Addition of infinite digit strings cannot be performed digit by digit from the left, the identification $0.4999\ldots=0.5$ has to be imposed by hand, and checking that arithmetic is well defined on the identified strings is disagreeable. The obstacle, in a word, is carrying.
 
The difficulty is real but not fatal, and it has been faced before. Faltin, Metropolis, Ross and Rota constructed the reals from digit strings with carrying, organising the carry structure algebraically as a wreath product \cite{FALTIN1976271}; de Bruijn gave a digit-based construction not routed through the rationals \cite{de1976defining}; and the Eudoxus reals build $\mathbb{R}$ from almost-homomorphisms of $\mathbb{Z}$, avoiding expansions altogether \cite{arthan2004eudoxus}. On the representation side, expansions of a continuum already in hand, rather than constructions of one, writing numbers over an arbitrary sequence of bases goes back to Cantor's second publication \cite{Cantor1869} and has grown into the literature on Cantor series \cite{galambos1976representations,oppenheim1954criteria,hancl2004irrationality}.
 
Our resolution of the carrying problem is to stop treating carrying as an obstacle inside an operation and to treat it instead as the generating move of an equivalence. Fix a system function $\vartheta:\omega_+\rightarrow\left[2,\infty\right]_\omega$, assigning to each position its own base. A candidate real number is any function $f:\omega\rightarrow\omega$, the parity of $f(0)$ recording the sign and $f(n)$ the digit at position $n$; crucially, digits are \emph{not} required to be reduced below their base. Two local moves act on such functions: a contraction $c_n$ performs one carry at position $n$, and a broadening $b_n$ undoes one. Two functions name the same number when each can be carried into agreement with a common form on every finite initial segment. Three facts make this workable where naive digit arithmetic is not. The moves cancel in pairs, so the identification is generated rather than imposed. Carrying transports along domination, which is the engine behind every proof that an operation is well defined. And every class has a normal form consisting of at most two auxiliary functions (up to a sign representation), the primary and the secondary, which are the terminating and $(\vartheta-1)$-tailed expansions of the same number. The identification $0.5=0.4999\ldots$ is thereby not an axiom of the construction but a theorem about it.
 
Because nothing in the definitions mentions a particular $\vartheta$, the construction is uniform in the base, and the system functions form a set of cardinality $2^{\aleph_0}$. What varies across that family is precisely the thing we set out to study. At a constant base $b$ the numbers with two names are the $b$-adically terminating rationals (\ref{b-adic}); when every integer divides some partial product they are exactly the rationals (\ref{Irrational Models Theorem}); intermediate choices of $\vartheta$ realise intermediate sets. Prior digit constructions fix the base \cite{FALTIN1976271,de1976defining} and so cannot exhibit this; constructions parametrised by an infinite sequence do exist \cite{shiu1974,pintilie1988}, but their parameter is a summation scheme and they have no notion of a second representation for the variation to act on.
 
That this freedom costs nothing mathematically is the content of a classical fact. Each of our models satisfies the axioms of a Dedekind-complete ordered field, and any two such fields are isomorphic by a unique isomorphism; so every model is canonically $\mathbb{R}$, and a statement about real numbers proved in one holds in all. Moving between models therefore changes nothing about the numbers and everything about their names, which is exactly the licence we want: one may choose the naming scheme to suit the problem without incurring any obligation to check that the answer survives.
 
Section \ref{sec:cantor} puts the licence to work. In any model whose system function has Cantor's divisibility property, every positive integer divides some partial product, we show that a non-zero real number is rational \emph{if and only if} it possesses a secondary auxiliary function. Rationality, in such a base, simply is the property of having two names (or being $0$). Cantor's irrationality criterion for Cantor series \cite{Cantor1869} follows as a corollary, with the analytic content of the classical argument absorbed into the representation theory. The criterion is Cantor's, and its first strengthening, the removal of the divisibility hypothesis, is Oppenheim's \cite{oppenheim1954criteria}; what is ours is the route, and the route is the point. Whether the same instrument reaches into the post-Oppenheim regime, where the rationality of general Cantor series remains only partially understood \cite{galambos1976representations,hancl2004irrationality}, is a question we regard as the natural next one and do not answer here.
 
The paper proceeds as follows. Section \ref{sec:prelim} fixes notation and the two arithmetic facts we need. Section \ref{sec:construction} builds the classes and their normal forms and proves the combinatorial core; Section \ref{sec:structure} establishes the structure theory that makes the normal forms canonical. Sections \ref{sec:order} to \ref{sec:completeness} develop the order, the analytic tools, and completeness, and Sections \ref{sec:addition} and \ref{sec:mult} the two operations, together verifying the field axioms. Section \ref{sec:cantor} contains the rationality theorem and its Cantor corollary, Section \ref{sec:related} the relation to earlier work, and Section \ref{sec:conclusion} the directions we intend to pursue.
\section{Preliminaries}\label{sec:prelim}
This section fixes the notation we use and records the two pieces of arithmetic the later chapters lean on. Neither is deep: the first is a convention for intervals of ordinals, the second the elementary theory of rational floors, which we shall need only in Section \ref{sec:mult}, where multiplication has to decide how much of a partial product belongs at a given position. We collect it here rather than there so that the multiplication chapter can proceed without interruption, and a reader willing to take the floors on trust may pass over Subsection \ref{Methods} on a first reading.
\subsection{Set-Theoretic Notation}\label{Notation}
We use standard set-theoretic notation, which we review below so that readers unfamiliar with the subject can still follow the paper. We also introduce some notation of our own. A summary of the notation introduced here and throughout the paper is collected in Appendix~\ref{app:notation}.
\subparagraph{$\omega$:} The set of all finite ordinals. Readers unfamiliar with set theory might think about $\omega$ as the set of natural numbers (including zero).
\subparagraph{$\omega_+$:} This will be a set which we are going to use very often and it will denote $\omega-\left\{0\right\}$ (or $\omega-1$) where the ``$-$" is a relative complement and not an arithmetic minus.
\subparagraph{$\left[a,b\right]_{\left(A,\leq\right)}$:} Let $A$ be a set where $a,b\in A$ with a total ordering $\leq$. Then $\left[a,b\right]_{\left(A,\leq\right)}:=\left\{c\in A\middle|a\leq c\wedge c\leq b\right\}$. For particular sets we are going to drop the ``$\leq$" part, for example for $\omega$ we have a reflexive total ordering $\leq$ defined as $a\leq b\iff a\in b\text{ or }a=b$, we shall write only $\left[a,b\right]_\omega$.\\
Another special case for this notation which we will use is $\left[a,\infty\right]_\omega:=\omega-a$. For those unfamiliar with ordinals, that is the same set as $\left\{b\in\omega\middle|a\leq b\right\}$, where the $\leq$ is the same as the one described above.
\subparagraph{${}^AB$:=}$\left\{f\middle|f:A\rightarrow B\right\}$ is the set of all functions from a set $A$ to a set $B$.
\subparagraph{$-$} This shall denote two distinct operations in two settings. First it should denote the relative complement in the setting of general sets. However, in the setting of finite ordinals this denotes the subtraction operation. We define such operation as $\forall\alpha,\beta,\gamma\in\omega,\left(\alpha\geq\beta\right)\implies\left(\alpha-\beta=\gamma\iff\alpha=\gamma+\beta\right)$. Notice that this is a well-defined operation as addition on finite ordinals is commutative and associative. As one of the conditions is $\alpha\geq\beta$ we shall always make sure that that is the case when we use subtraction.
\subparagraph{Quantifiers:} In this work we will be using many quantifiers. Even though it will go against the convention in logic we opt to use commas after quantifiers to make the script more readable. Anyone dissatisfied with this may erase as many commas as they wish.
\subparagraph{$\mathcal{N}$:=}${}^\omega\omega$. The convention is that this set is denoted as $N$. However, as we will be using that symbol to represent many different finite ordinals throughout the work, we introduce this new notation to distinguish the two.\\\\
We will also use simplified notation for functions $f:\omega\rightarrow\omega$ in examples which will be $f=\left(a;b,c,d,\dots\right)$ to mean
\begin{align*}
    f\left(n\right)=\begin{cases}
        a&n=0\\
        b&n=1\\
        c&n=2\\
        d&otherwise
    \end{cases}
\end{align*}
The position $0$ is singled out due to it acting in a different manner. In our definitions, statements, proofs and calculations we shall use the formal way of writing the function to be explicit. This convention is adopted for the simplicity of the examples.
\subsection{Rational Floors}\label{Methods}
We shall be using a few methods in the final part of this paper without citing them. This is the list of them.
\paragraph{Permutation:} of two bound-dependent sums
\begin{align*}
\sum\limits_{M=1}^K\left(\sum\limits_{n=1}^M\left(\sum\limits_{k=1}^M\left(\dots\right)\right)\right)=\sum\limits_{n=1}^K\left(\sum\limits_{k=1}^K\left(\sum\limits_{M=\max\left(\left\{n,k\right\}\right)}^K\left(\dots\right)\right)\right)
\end{align*}
\paragraph{Rational Floor Function:} is the natural number reached by
\begin{align*}
\max\left(\left\{t\in\omega\middle|ta\leq b\right\}\right)\indent\text{where $a,b\in\omega$.}
\end{align*}
We have two methods that we will be using in regards to this set.
\subparagraph{Inequalities:}
\begin{align*}
\max\left(\left\{t\in\omega\middle|ta\leq bc\right\}\right)-c\leq c\max\left(\left\{t\in\omega\middle|ta\leq b\right\}\right)\leq\max\left(\left\{t\in\omega\middle|ta\leq bc\right\}\right)
\end{align*}
and paired with sums
\begin{align*}
&\max\left(\left\{t\in\omega\middle|ta\leq\sum\limits_{n=1}^K\left(b_nc_n\right)\right\}\right)-\sum\limits_{n=1}^K\left(c_n\right)\\
\leq&\sum\limits_{n=1}^K\left(c_n\max\left(\left\{t\in\omega\middle|ta\leq b_n\right\}\right)\right)\\
\leq&\max\left(\left\{t\in\omega\middle|ta\leq\sum\limits_{n=1}^K\left(b_nc_n\right)\right\}\right)
\end{align*}
Notice that in both inequalities, we need to ensure the existence of the first part, as we have not defined negative integers.
\subparagraph{Composition:}
\begin{align*}
\max\left(\left\{t\in\omega\middle|ta\leq\max\left(\left\{s\in\omega\middle|sb\leq c\right\}\right)\right\}\right)=\max\left(\left\{t\in\omega\middle|tab\leq c\right\}\right)
\end{align*}
\newpage
\pagenumbering{arabic}
\section{The Construction}\label{sec:construction}
We wish to construct a system of models of the real analysis. In this section we shall construct the set of this model. In the next section we shall construct the relations and functions on this set.
\subsection{System Functions, Contractions and Broadenings}
Firstly, we will define a few operations and results on the space $\mathcal{N}:={}^\omega\omega=\left\{f\middle|f:\omega\rightarrow\omega\right\}$. These definitions and theorems will be the foundation on which we will construct the set of real numbers.
\begin{definition}[System Function]\label{System Function}\noindent\\
Consider any $\vartheta:\omega_+\rightarrow\left[2,\infty\right]_\omega$. For any such function chosen we shall define a construction (model) of the reals using this function. For any system function $\vartheta$ chosen we wish to define any real number $x$ as the set of all functions $f:\omega\rightarrow\omega$ such that
\begin{align*}
x=\pm\sum\limits_{n=1}^\infty\frac{f\left(n\right)}{\prod\limits_{k=1}^{n-1}\left(\vartheta\left(k\right)\right)}
\end{align*}
\end{definition}

\begin{definition}[Contraction]\label{Contraction}\noindent\\
If $\exists n\in\omega_+\text{ such that }f\left(n+1\right)\geq\vartheta\left(n\right)$ we call the function $c_n\left(f\right):\omega\rightarrow\omega$ defined as
\begin{align*}
c_n\left(f\right)\left(k\right):=\begin{cases}f\left(n\right)+1&\text{if $k=n$}\\f\left(n+1\right)-\vartheta\left(n\right)&\text{if $k=n+1$}\\f\left(k\right)&\text{otherwise}\end{cases}
\end{align*}
 for any $k\in\omega$ the contraction of $f$ about $n$.
\end{definition}

\begin{definition}[Broadening]\label{Broadening}\noindent\\
If $\exists n\in\omega_+\text{ such that }f\left(n\right)\geq1$ we call the function $b_n\left(f\right):\omega\rightarrow\omega$ defined as
\begin{align*}
b_n\left(f\right)\left(k\right):=\begin{cases}f\left(n\right)-1&\text{if $k=n$}\\f\left(n+1\right)+\vartheta\left(n\right)&\text{if $k=n+1$}\\f\left(k\right)&\text{otherwise}\end{cases}
\end{align*} 
 for any $k\in\omega$ the broadening of $f$ about $n$.
\end{definition}

\begin{intuition}[Contractions and Broadenings] We want to express any real number by many different interpretations within a chosen system. For example $1$ can be described as one integer or as two halves, i.e. $1+\nicefrac{0}{2}$ or $0+\nicefrac{2}{2}$. These definitions describe doing ``one step forward" and ``one step backward". Notice that neither contractions nor broadenings affect the value at $0$. This is because the value at $0$ should give us the sign, whereas contractions and broadenings tell us how to manipulate the magnitude.
\end{intuition}

\begin{example} Take the standard binary base, so that $\vartheta$ is the constant function $\vartheta(n)=2$, and in this system take the constant function $f=(1;1,1,1,\dots)$.\\
There is no $m\in\omega_+$ with $f(m+1)\geq\vartheta(m)=2$, since $f$ is constantly $1$, so $f$ is not contractable about any $m$. Broadening, however, only requires $f(m)\geq1$, which holds everywhere, so $f$ may be broadened about any $m\in\omega_+$. Given such an $m$ we have $b_m(f)=g$ where $g(m)=0$, $g(m+1)=f(m+1)+\vartheta(m)=3$, and $g(n)=1$ otherwise; broadening about $2$, for instance, gives
\begin{align*}
b_2(f)=(1;1,0,3,1,1,\dots).
\end{align*}
The two operations are therefore not symmetric in their availability: a function may always be spread out further, but it can only be gathered up when some position has accumulated enough to give.
\end{example}

\begin{example} Take a base where the $n$th place is $\nicefrac{1}{n}$ of the previous place, that is $\vartheta(n)=n+1$; the weights are then $1,\nicefrac{1}{2},\nicefrac{1}{6},\nicefrac{1}{24},\dots$, and position $n$ carries weight $\nicefrac{1}{n!}$. In this system take $f=(0;1,2,3,4,\dots)$, so that $f(n)=n$.\\
Here every position is exactly full: for any $m\in\omega_+$ we have $f(m+1)=m+1=\vartheta(m)$, so $f$ is contractable about every $m$, and $c_m(f)=g$ where $g(m)=m+1$, $g(m+1)=0$ and $g(n)=n$ otherwise. Contracting about $2$ gives
\begin{align*}
c_2(f)=(0;1,3,0,4,5,\dots).
\end{align*}
Equally, $f(m)=m\geq1$ for every $m\in\omega_+$, so $f$ may also be broadened anywhere, and $b_m(f)=h$ where $h(m)=m-1$, $h(m+1)=(m+1)+\vartheta(m)=2m+2$ and $h(n)=n$ otherwise; about $2$ this gives
\begin{align*}
b_2(f)=(0;1,1,6,4,5,\dots).
\end{align*}
Figure~\ref{fig:contraction} shows the same move as a redistribution of units: the $\vartheta(2)=3$ units at position~3 bundle into one unit at position~2, and the value is unchanged.
\end{example}

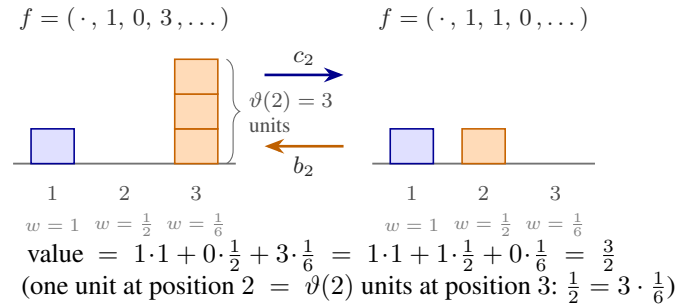
\begin{figure}[htbp]
  \centering
\begin{tikzpicture}[x=0.95cm, y=0.46cm]
 
  \begin{scope}
    \draw[baseaxis] (0.45,0) -- (3.55,0);
    \filldraw[settled] (0.7,0) rectangle (1.3,1);
    \foreach \k in {0,1,2}{\filldraw[active] (2.7,\k) rectangle (3.3,\k+1);}
    \draw[decorate,decoration={brace,amplitude=5pt,mirror},black!60]
        (3.42,0) -- (3.42,3)
        node[midway,right=5pt,poslab,align=left]{$\vartheta(2)=3$\\units};
    \foreach \p/\w in {1/{1},2/{\frac12},3/{\frac16}}{
        \node[poslab] at (\p,-0.85){$\p$};
        \node[wlab]   at (\p,-1.75){$w=\w$};}
    \node[read] at (2.0,4.15){$f=(\,\cdot\,,\,1,\,0,\,3\,,\dots)$};
  \end{scope}
 
  \draw[-{Stealth[length=2.6mm]},line width=1pt,blue!55!black]
      (3.95,2.55) -- (5.05,2.55) node[midway,above,read]{$c_2$};
  \draw[-{Stealth[length=2.6mm]},line width=1pt,orange!75!black]
      (5.05,0.5) -- (3.95,0.5) node[midway,below,read]{$b_2$};
 
  \begin{scope}[shift={(5,0)}]
    \draw[baseaxis] (0.45,0) -- (3.55,0);
    \filldraw[settled] (0.7,0) rectangle (1.3,1);      
    \filldraw[active]  (1.7,0) rectangle (2.3,1);      
    \foreach \p/\w in {1/{1},2/{\frac12},3/{\frac16}}{
        \node[poslab] at (\p,-0.85){$\p$};
        \node[wlab]   at (\p,-1.75){$w=\w$};}
    \node[read] at (2.0,4.15){$f=(\,\cdot\,,\,1,\,1,\,0\,,\dots)$};
  \end{scope}
 
  \node[align=center] at (4.75,-3.1)
    {value $\;=\;1\!\cdot\!1+0\!\cdot\!\tfrac12+3\!\cdot\!\tfrac16
      \;=\;1\!\cdot\!1+1\!\cdot\!\tfrac12+0\!\cdot\!\tfrac16\;=\;\tfrac32$\\
     \qquad(one unit at position $2\;=\;\vartheta(2)$ units at position $3$:\ $\tfrac12=3\cdot\tfrac16$)};
\end{tikzpicture}
  \caption{A contraction $c_2$---a single carry---in the flowing base
  $\vartheta(1)=2,\ \vartheta(2)=3$. Three units at position~3 (each of weight
  $\tfrac16$) bundle into one unit at position~2 (weight $\tfrac12$); the
  represented value $\tfrac32$ is unchanged. The broadening $b_2$ is the
  same move run backwards, so $c_2$ and $b_2$ are mutually inverse
  (\nameref{Cancellation Theorems} \ref{Cancellation Theorems}). The leading ``$\cdot$'' is the sign slot $f(0)$.}
  \label{fig:contraction}
\end{figure}
\subsection{The Calculus of Moves}
\begin{theorem}[Cancellation Theorems]\label{Cancellation Theorems}\noindent\\
\textbf{Part 1.} Suppose that for a function $f:\omega\rightarrow\omega$ a contraction about some $n$ in $\omega_+$ is applicable to $f$. Then a broadening about the same $n$ is applicable to $c_n\left(f\right)$ and furthermore $b_n\left(c_n\left(f\right)\right)=f$.\\
\textbf{Part 2.} Suppose that for a function $f:\omega\rightarrow\omega$ a broadening about some $n$ in $\omega_+$ is applicable to $f$. Then a contraction about the same $n$ is applicable to $b_n\left(f\right)$ and furthermore $c_n\left(b_n\left(f\right)\right)=f$.
\end{theorem}
\begin{proof}[Proof of part 1.] Suppose a contraction about $n\in\omega_+$ is applicable to a function $f:\omega\rightarrow\omega$, then $c_n\left(f\right)\left(n\right)=f\left(n\right)+1\geq1$, thus the condition to apply broadening to $c_n\left(f\right)$ is met. Then for any $k\in\omega$:
\begin{alignat*}{2}
\hspace{-0.25cm}
b_n\left(c_n\left(f\right)\right)\left(k\right)=
\begin{cases}
c_n\left(f\right)\left(n\right)-1=f\left(n\right)+1-1&=f\left(n\right)\\
c_n\left(f\right)\left(n+1\right)+\vartheta\left(n\right)=f\left(n+1\right)-\vartheta\left(n\right)+\vartheta\left(n\right)\hspace{-0.25cm}&=f\left(n+1\right)\\
c_n\left(f\right)\left(k\right)&=f\left(k\right)
\end{cases}
\def\arraystretch{1.2}\begin{array}{@{}l}
\text{if $k=n$}\\
\text{if $k=n+1$}\\
\text{otherwise}
\end{array}
\end{alignat*}
Therefore, for all $k\in\omega$ we have $f\left(k\right)=b_n\left(c_n\left(f\right)\right)\left(k\right)$ and thus $b_n\left(c_n\left(f\right)\right)=f$.
\end{proof}
\begin{proof}[Proof of part 2.] Suppose a broadening about $n\in\omega_+$ is applicable to our function $f:\omega\rightarrow\omega$, then $b_n\left(f\right)\left(n+1\right)=f\left(n+1\right)+\vartheta\left(n\right)\geq\vartheta\left(n\right)$, thus the condition to apply contraction to $b_n\left(f\right)$ is met. Then for any $k\in\omega$:
\begin{alignat*}{2}
\hspace{-0.25cm}
c_n\left(b_n\left(f\right)\right)\left(k\right)=
\begin{cases}
b_n\left(f\right)\left(n\right)+1=f\left(n\right)-1+1&=f\left(n\right)\\
b_n\left(f\right)\left(n+1\right)-\vartheta\left(n\right)=f\left(n+1\right)+\vartheta\left(n\right)-\vartheta\left(n\right)\hspace{-0.25cm}&=f\left(n+1\right)\\
b_n\left(f\right)\left(k\right)&=f\left(k\right)
\end{cases}
\def\arraystretch{1.2}\begin{array}{@{}l}
\text{if $k=n$}\\
\text{if $k=n+1$}\\
\text{otherwise}
\end{array}
\end{alignat*}
Therefore, for all $k\in\omega$ we have $f\left(k\right)=c_n\left(b_n\left(f\right)\right)\left(k\right)$ and thus $c_n\left(b_n\left(f\right)\right)=f$.
\end{proof}

\begin{example} Take again the standard binary base, $\vartheta(n)=2$, and the function $f=(0;0,2,0,\dots)$, whose value is $2\cdot\nicefrac{1}{2}=1$.\\
Since $f(2)=2\geq\vartheta(1)$, the function $f$ is contractable about $1$, and $c_1(f)=g$ where $g=(0;1,0,0,\dots)$. This $g$ may be broadened about $1$, and doing so returns $b_1(g)=(0;0,2,0,\dots)=f$, as the first Cancellation Theorem requires.\\
Running the pair the other way, $f$ may be broadened about $2$, giving $b_2(f)=(0;0,1,2,0,\dots)$, and that function is contractable about $2$, with $c_2(b_2(f))=(0;0,2,0,\dots)=f$ again, as the second Cancellation Theorem requires. All three functions
\begin{align*}
(0;1,0,0,\dots),\qquad(0;0,2,0,\dots),\qquad(0;0,1,2,0,\dots)
\end{align*}
have value $1$, which is the first hint of what is to come: a single real number will be described by many functions, and the moves that pass between them are exactly the contractions and broadenings.
\end{example}

\begin{definition}[Sequences of Contractions and Broadenings]\label{Sequences of Contractions and Broadenings}\noindent\\
We define the set of ``CB-functions" as 
\begin{align*}
\mathcal{CB}:=\left\{\mathcal{D}\middle|\mathcal{D}:\omega\rightarrow\left\{c_n\middle|n\in\omega_+\right\}\cup\left\{b_n\middle|n\in\omega_+\right\}\right\}
\end{align*}
in words this is a set of all functions which assign to a value in $\omega$ a particular contraction or broadening about a particular value. For any $\mathcal{D}\in\mathcal{CB}$ and any $r\in\omega,q\in\omega_+$ we call $\mathcal{D}_{k=q}^r$ a sequence of contractions and broadenings where $\mathcal{D}_{k=q}^r:=\mathcal{D}\left(r\right)\circ\mathcal{D}\left(r-1\right)\circ\dots\circ\mathcal{D}\left(q\right)$. We define $\mathcal{D}_{k=q}^r\left(f\right)=f$ for any $f:\omega\rightarrow\omega$ and any $r\in\omega,q\in\omega_+$ where $r<q$ and we say that this sequence of contractions is applicable to any $f:\omega\rightarrow\omega$.\\
We say that $\mathcal{D}_{k=q}^r$ is applicable to some $f:\omega\rightarrow\omega$ if for any $k\in\left[q,r\right]_\omega$ we have $\mathcal{D}\left(k\right)$ being applicable to $\mathcal{D}_{j=q}^{k-1}\left(f\right)$. Let $\mathcal{FS}$ be the set of all sequences of contractions and broadenings.\\
Furthermore, define the ``conditional of $D_{k=q}^r$" as
\begin{align*}
\text{con}\left({D_{k=q}^r}\right):=\left\{n+1\in\omega_+\middle|\exists k\in\left[q,r\right]_\omega,\mathcal{D}\left(k\right)=c_n\right\}\cup\left\{n\in\omega_+\middle|\exists k\in\left[q,r\right]_\omega,\mathcal{D}\left(k\right)=b_n\right\}
\end{align*}
A special case of sequences of contractions and broadenings are $\mathcal{C}_{k=q}^r$ and $\mathcal{B}_{k=q}^r$ which are sequences on CB-functions of only contractions and only broadenings respectively.
\end{definition}
\begin{remark} Notice that all sequences of contractions and broadenings (by our construction) are finite.
\end{remark}

\begin{intuition}[CB-functions] These are functions which assign to each value $n\in\omega$ a contraction or broadening about some $m\in\omega$. We do not allow infinitely many contractions or broadenings to be applied to a function, since allowing infinitely many moves would collapse the construction, identifying every real number with every other.
\end{intuition}

\begin{intuition}[Sequence of contractions and broadenings] These are transformations on functions which are equivalent to a composition of contractions and broadenings. We intuitively take a contraction and broadening as a shifting the distribution of the modulus of a real number. Therefore, a sequence of contractions and broadenings does not change the modulus of the number described by the function, just changes the distribution of it. The set of all sequences of contractions and broadenings is denoted as $\mathcal{FS}$ which stands for ``finite sequences".
\end{intuition}
\begin{lemma}[Conditional Lemma]\label{Conditional Lemma}\noindent\\
For any function $f:\omega\rightarrow\omega$ and any sequence of contractions $\mathcal{C}_{k=q}^r\in\mathcal{FS}$ applicable to $f$ we have that for all $n\in\omega$ such that $n>\max\text{con}\left(\mathcal{C}_{k=q}^r\right)$ we have $\mathcal{C}_{k=q}^r\left(f\right)\left(n\right)=f\left(n\right)$.
\end{lemma}
\begin{proof} We shall prove this by induction. Let such a function and sequence of contractions be given. We see that for all $n>\max\text{con}\left(\mathcal{C}_{k=q}^r\right)$ and thus not equal to $m$ and $m+1$ where $\mathcal{C}\left(q\right)=c_m$ we have 
\begin{align*}
\mathcal{C}\left(q\right)\left(f\right)\left(n\right)=f\left(n\right)
\end{align*}
Now suppose that for some $K\in\left[q,r-1\right]_\omega$ we have for all $n>\max\text{con}\left(\mathcal{C}_{k=q}^r\right)$ the result $\mathcal{C}_{k=q}^K\left(f\right)\left(n\right)=f\left(n\right)$. Then for all $n>\max\text{con}\left(\mathcal{C}_{k=q}^r\right)$ and thus not equal to $m$ and $m+1$ where $\mathcal{C}\left(K+1\right)=c_m$ we have
\begin{align*}
\mathcal{C}_{k=q}^{K+1}\left(f\right)\left(n\right)=\mathcal{C}\left(K+1\right)\left(\mathcal{C}_{k=q}^{K}\left(f\right)\right)\left(n\right)=\mathcal{C}_{k=q}^{K}\left(f\right)\left(n\right)=f\left(n\right)
\end{align*}
Therefore, by induction we have the desired result.
\end{proof}
\begin{theorem}[Submission Theorem]\label{Submission Theorem}\noindent\\
For any two functions $f,g:\omega\rightarrow\omega$ and any sequence of contractions and broadenings $\mathcal{D}_{k=q}^r\in\mathcal{FS}$ applicable to $f$, if there exist $S\subset\omega_+$ such that $\text{con}\left(\mathcal{D}_{k=q}^r\right)\subset S$ and $\forall n\in S,f\left(n\right)\leq g\left(n\right)$ then $\mathcal{D}_{k=q}^r$ is applicable to $g$ and furthermore for all $n\in\omega$ we have $\mathcal{D}_{k=q}^r\left(f\right)\left(n\right)+g\left(n\right)=f\left(n\right)+\mathcal{D}_{k=q}^r\left(g\right)\left(n\right)$.\footnote{Remark: All this is saying that the difference at each value pre- and post-application of $\mathcal{D}_{k=q}^r$ stays the same. Better visualised as $\mathcal{D}_{k=q}^r\left(f\right)\left(n\right)-\mathcal{D}_{k=q}^r\left(g\right)\left(n\right)=f\left(n\right)-g\left(n\right)$. However, we choose not to write it as such, as we do not know that for all values $n$ we have $f\left(n\right)\geq g\left(n\right).$}
\end{theorem}

\begin{intuition}[Submission Theorem] If we have two functions, where one is larger than the other at every point on the interval on which we apply a sequence of contractions and broadenings to the other, then this sequence is applicable to the larger function. Therefore, we say that if the conditions are satisfied by a smaller function, then they will also be satisfied by the larger function as this function ``has more material to play with".\\
The second part of the theorem tells us that the differences at every point of the two functions after applying the sequence to both of them is the same as the difference before.\\
From now on we shall say that
\begin{align*}
\begin{array}{ll@{}}
\text{$f$ dominates $g$ on $S$}&\text{if for all $n\in S$ we have $f\left(n\right)\geq g\left(n\right)$}\\
\text{$f$ dominates $g$}&\text{if for all $n\in\omega_+$ we have $f\left(n\right)\geq g\left(n\right)$}\\ 
\text{$f$ strictly dominates $g$}&\text{if $f$ dominates $g$ and there exists an $M\in\omega_+$}\\&\text{such that we have $f\left(M\right)>g\left(M\right)$}
\end{array}
\end{align*}

\end{intuition}
\begin{proof} We shall do a proof by induction. Let any $q\in\omega_+$ be given. For any $r\in\left[0,q-1\right]_\omega$ we have $\text{con}\left(\mathcal{D}_{k=q}^r\right)=\emptyset\subset S$. By definition $\mathcal{D}_{k=q}^r$ is applicable to $g$ and furthermore $\mathcal{D}_{k=q}^r\left(f\right)=f$ and $\mathcal{D}_{k=q}^r\left(g\right)=g$ and thus $\forall n\in\omega,\mathcal{D}_{k=q}^r\left(f\right)\left(n\right)+g\left(n\right)=f\left(n\right)+g\left(n\right)=f\left(n\right)+\mathcal{D}_{k=q}^r\left(g\right)\left(n\right)$.
\\
Now suppose that for some $r\in\left[q-1,\infty\right]_\omega$ all the sequences of contractions and broadenings $\mathcal{D}_{k=q}^r\in\mathcal{FS}$ which are applicable to $f$ such that $\text{con}\left(\mathcal{D}_{k=q}^r\right)\subset S$ are also applicable to $g$ and for all $n\in\omega$ we have $\mathcal{D}_{k=q}^r\left(f\right)\left(n\right)+g\left(n\right)=f\left(n\right)+\mathcal{D}_{k=q}^r\left(g\right)\left(n\right)$.
\subparagraph{Case 1. $\mathcal{D}\left(r+1\right)=c_\alpha$:} We suppose that $\alpha+1\in\text{con}\left(\mathcal{D}_{k=q}^{r+1}\right)\subset S$ and $c_{\alpha}$ is applicable to $\mathcal{D}_{k=q}^r\left(f\right)$ which means that $\mathcal{D}_{k=q}^r\left(f\right)\left(\alpha+1\right)\geq\vartheta\left(\alpha\right)$. As $\alpha+1\in S$ we know that $\mathcal{D}_{k=q}^r\left(g\right)\left(\alpha+1\right)\geq\mathcal{D}_{k=q}^r\left(f\right)\left(\alpha+1\right)\geq\vartheta\left(\alpha\right)$\footnote{We know this, as $\mathcal{D}_{k=q}^r\left(g\right)\left(\alpha+1\right)-\mathcal{D}_{k=q}^r\left(f\right)\left(\alpha+1\right)=g\left(\alpha+1\right)-f\left(\alpha+1\right)\geq0$ which implies that $\mathcal{D}_{k=q}^r\left(g\right)\left(\alpha+1\right)\geq\mathcal{D}_{k=q}^r\left(f\right)\left(\alpha+1\right)$.} and thus $c_{\alpha}$ is applicable to $\mathcal{D}_{k=q}^r\left(g\right)$.\\
We see that:
\begin{align*}
\begin{cases}
\mathcal{D}_{k=q}^{r+1}\left(f\right)\left(\alpha\right)+g\left(\alpha\right)&=\mathcal{D}_{k=q}^r\left(f\right)\left(\alpha\right)+1+g\left(\alpha\right)\\
&=f\left(\alpha\right)+\mathcal{D}_{k=q}^r\left(g\right)\left(\alpha\right)+1\\
&=f\left(\alpha\right)+\mathcal{D}_{k=q}^{r+1}\left(g\right)\left(\alpha\right)\\
\mathcal{D}_{k=q}^{r+1}\left(f\right)\left(\alpha+1\right)+g\left(\alpha+1\right)&=\mathcal{D}_{k=q}^r\left(f\right)\left(\alpha+1\right)-\vartheta\left(\alpha\right)+g\left(\alpha+1\right)\\
&=f\left(\alpha+1\right)+\mathcal{D}_{k=q}^r\left(g\right)\left(\alpha+1\right)-\vartheta\left(\alpha\right)\\
&=f\left(\alpha+1\right)+\mathcal{D}_{k=q}^{r+1}\left(g\right)\left(\alpha+1\right)\\
\mathcal{D}_{k=q}^{r+1}\left(f\right)\left(n\right)+g\left(n\right)&=\mathcal{D}_{k=q}^r\left(f\right)\left(n\right)+g\left(n\right)\\
&=f\left(n\right)+\mathcal{D}_{k=q}^r\left(g\right)\left(n\right)\\&=f\left(n\right)+\mathcal{D}_{k=q}^{r+1}\left(g\right)\left(n\right)
\end{cases}
\def\arraystretch{1.2}\begin{array}{@{}l}
\hspace{1cm}\text{if $n=\alpha$}\\\\\\
\hspace{1cm}\text{if $n=\alpha+1$}\\\\\\
\hspace{1cm}\text{otherwise}
\end{array}
\end{align*}
Therefore for all $n\in\omega$ we have $\mathcal{D}_{k=q}^{r+1}\left(f\right)\left(n\right)+g\left(n\right)=f\left(n\right)+\mathcal{D}_{k=q}^{r+1}\left(g\right)\left(n\right)$.
\subparagraph{Case 2. $\mathcal{D}\left(r+1\right)=b_\alpha$:} We suppose that $\alpha\in\text{con}\left(\mathcal{D}_{k=q}^{r+1}\right)\subset S$ and $b_{\alpha}$ is applicable to $\mathcal{D}_{k=q}^r\left(f\right)$ which means that $\mathcal{D}_{k=q}^r\left(f\right)\left(\alpha\right)\geq1$. As $\alpha\in S$ we know that $\mathcal{D}_{k=q}^r\left(g\right)\left(\alpha\right)\geq\mathcal{D}_{k=q}^r\left(f\right)\left(\alpha\right)\geq1$ and thus $b_{\alpha}$ is applicable to $\mathcal{D}_{k=q}^r\left(g\right)$.\\
We see that:
\begin{align*}
\begin{cases}
\mathcal{D}_{k=q}^{r+1}\left(f\right)\left(\alpha\right)+g\left(\alpha\right)&=\mathcal{D}_{k=q}^r\left(f\right)\left(\alpha\right)-1+g\left(\alpha\right)\\
&=f\left(\alpha\right)+\mathcal{D}_{k=q}^r\left(g\right)\left(\alpha\right)-1\\
&=f\left(\alpha\right)+\mathcal{D}_{k=q}^{r+1}\left(g\right)\left(\alpha\right)\\
\mathcal{D}_{k=q}^{r+1}\left(f\right)\left(\alpha+1\right)+g\left(\alpha+1\right)&=\mathcal{D}_{k=q}^r\left(f\right)\left(\alpha+1\right)+\vartheta\left(\alpha\right)+g\left(\alpha+1\right)\\
&=f\left(\alpha+1\right)+\mathcal{D}_{k=q}^r\left(g\right)\left(\alpha+1\right)+\vartheta\left(\alpha\right)\\
&=f\left(\alpha+1\right)+\mathcal{D}_{k=q}^{r+1}\left(g\right)\left(\alpha+1\right)\\
\mathcal{D}_{k=q}^{r+1}\left(f\right)\left(n\right)+g\left(n\right)&=\mathcal{D}_{k=q}^r\left(f\right)\left(n\right)+g\left(n\right)\\
&=f\left(n\right)+\mathcal{D}_{k=q}^r\left(g\right)\left(n\right)\\
&=f\left(n\right)+\mathcal{D}_{k=q}^{r+1}\left(g\right)\left(n\right)
\end{cases}
\def\arraystretch{1.2}\begin{array}{@{}l}
\hspace{1cm}\text{if $n=\alpha$}\\\\\\
\hspace{1cm}\text{if $n=\alpha+1$}\\\\\\
\hspace{1cm}\text{otherwise}
\end{array}
\end{align*}
Therefore for all $n\in\omega$ we have $\mathcal{D}_{k=q}^{r+1}\left(f\right)\left(n\right)+g\left(n\right)=f\left(n\right)+\mathcal{D}_{k=q}^{r+1}\left(g\right)\left(n\right)$. Hence we have proven the theorem by cases and induction.
\end{proof}

\begin{theorem}[Consequences Theorem]\label{Consequences Theorem}\noindent\\
Let $f:\omega\rightarrow\omega$ and a sequence of contractions $\mathcal{C}_{k=q}^r$ where $r\geq q$ applicable to $f$ be given. Then for $t:=\min\left(\left\{n\in\omega_+\middle|\exists k\in\left[q,r\right]_\omega,\mathcal{C}\left(k\right)=c_n\right\}\right)$ we have $\mathcal{C}_{k=q}^r\left(f\right)\left(t\right)>f\left(t\right)$ and $\forall n\in\left[0,t-1\right]_\omega,\mathcal{C}_{k=q}^r\left(f\right)\left(n\right)=f\left(n\right)$.
\end{theorem}

\begin{intuition}[Consequences Theorem] If we contract a function multiple times, then the resulting function will have a higher output for the smallest value about which we contracted. Allegorically, if you push material to be concentrated more at the start, there will be more material where you stop pushing.
\end{intuition}

\begin{proof} We shall prove this by induction. Let $q\in\omega_+$ be given. Take $r=q$. Then $\mathcal{C}\left(q\right)=c_\alpha$ for some $\alpha\in\omega_+$ and thus $t=\alpha$. By definition, we know that $c_t\left(f\right)\left(t\right)=f\left(t\right)+1$ and therefore $\mathcal{C}_{k=q}^r\left(f\right)\left(t\right)=\mathcal{C}_{k=q}^q\left(f\right)\left(t\right)=c_t\left(f\right)\left(t\right)>f\left(t\right)$. Moreover, as for all $n\neq t\text{ or }t+1$ we have $\mathcal{C}_{k=q}^r\left(f\right)\left(n\right)=f\left(n\right)$ we must have $\forall n\in\left[0,t-1\right]_\omega,\mathcal{C}_{k=q}^r\left(f\right)\left(n\right)=f\left(n\right)$.
\\
Now suppose that for some $r\geq q$ and any sequence of contractions $\mathcal{C}_{k=q}^r$ applicable to $f$ we have for $t_1:=\min\left(\left\{n\in\omega_+\middle|\exists k\in\left[q,r\right]_\omega,\mathcal{C}\left(k\right)=c_n\right\}\right)$ the result $\mathcal{C}_{k=q}^r\left(f\right)\left(t_1\right)>f\left(t_1\right)$ and $\forall n\in\left[0,t_1-1\right]_\omega,\mathcal{C}_{k=q}^r\left(f\right)\left(n\right)=f\left(n\right)$. Denote $\mathcal{C}\left(r+1\right)=c_\alpha$ and we suppose that it is applicable to $\mathcal{C}_{k=q}^r\left(f\right)$. There are two cases:
\subparagraph{Case 1. $\alpha>t_1$:} We have $t_2:=\min\left(\left\{n\in\omega_+\middle|\exists k\in\left[q,r+1\right]_\omega,\mathcal{C}\left(k\right)=c_n\right\}\right)=t_1<\alpha$.\\
As for all $n\neq\alpha\text{ or }\alpha+1$ we have $\mathcal{C}_{k=q}^{r+1}\left(f\right)\left(n\right)=\mathcal{C}_{k=q}^r\left(f\right)\left(n\right)$ we must have
\begin{align*}
\forall n\in\left[0,t_2-1\right]_\omega=\left[0,t_1-1\right]_\omega\subset\left[0,\alpha-1\right]_\omega,\mathcal{C}_{k=q}^{r+1}\left(f\right)\left(n\right)=\mathcal{C}_{k=q}^r\left(f\right)\left(n\right)=f\left(n\right)
\end{align*}
and $\mathcal{C}_{k=q}^{r+1}\left(f\right)\left(t_2\right)=\mathcal{C}_{k=q}^{r+1}\left(f\right)\left(t_1\right)=\mathcal{C}_{k=q}^r\left(f\right)\left(t_1\right)>f\left(t_1\right)=f\left(t_2\right)$.
\subparagraph{Case 2. $\alpha\leq t_1$:} We see that $t_2:=\min\left(\left\{n\in\omega_+\middle|\exists k\in\left[q,r+1\right]_\omega,\mathcal{C}\left(k\right)=c_n\right\}\right)=\alpha$. By definition for all $n\neq\alpha\text{ or }\alpha+1$ we have $\mathcal{C}_{k=q}^{r+1}\left(f\right)\left(n\right)=\mathcal{C}_{k=q}^r\left(f\right)\left(n\right)$ and thus 
\begin{align*}
\forall n\in\left[0,t_2-1\right]_\omega=\left[0,\alpha-1\right]_\omega,\mathcal{C}_{k=q}^{r+1}\left(f\right)\left(n\right)=\mathcal{C}_{k=q}^r\left(f\right)\left(n\right)=f\left(n\right)
\end{align*}
 and $\mathcal{C}_{k=q}^{r+1}\left(f\right)\left(t_2\right)=\mathcal{C}_{k=q}^r\left(f\right)\left(t_2\right)+1>\mathcal{C}_{k=q}^r\left(f\right)\left(t_2\right)\geq f\left(t_2\right)$.
\end{proof}

\subsection{Signs and Auxiliary Functions}
As we have gathered information about the functions in $\mathcal{N}$ we will now start shaping those functions into the set of real numbers. We will see more and more familiar ideas (at least intuitively) in this section than in the previous one.

\begin{definition}[Sign of a Function]\label{Sign of a Function}\noindent\\
We call the function $\sigma:\mathcal{N}\rightarrow\left\{0,1\right\}$ the sign of a function $f\in\mathcal{N}$ (remember $\mathcal{N}={}^\omega\omega$ and thus $f:\omega\rightarrow\omega$) defined as $\sigma\left(f\right)=0$ if and only if $f\left(0\right)\in\left\{2n+1\middle|n\in\omega\right\}$ and $\sigma\left(f\right)=1$ if and only if $f\left(0\right)\in\left\{2n\middle|n\in\omega\right\}$.
\end{definition}

\begin{intuition}[Sign of a function] We say that all the positive real numbers will be described by functions with an even output for $0$ and all the negative real numbers will be described by functions with an odd output for $0$.
\end{intuition}

\newpage
\begin{definition}[Primary Auxiliary Function]\label{Primary Auxiliary Function}\noindent\\
We call a function $h:\omega\rightarrow\omega$ a primary auxiliary function if and only if it satisfies the two conditions
\begin{align*}
\begin{array}{ll@{}}
\forall k\in\left[2,\infty\right]_\omega,h\left(k\right)<\vartheta\left(k-1\right)&\text{\indent and \indent}\\
\forall N\in\omega_+,\exists M\in\left[N+1,\infty\right]_\omega,h\left(M\right)\neq\vartheta\left(M-1\right)-1
\end{array}
\end{align*}
Denote the set of all primary auxiliary functions as $A_p$.
\end{definition}

\begin{definition}[Secondary Auxiliary Function]\label{Secondary Auxiliary Function}\noindent\\
We call a function $h:\omega\rightarrow\omega$ a secondary auxiliary function if and only if it satisfies the two conditions
\begin{align*}
\begin{array}{ll@{}}
\forall k\in\left[2,\infty\right]_\omega,h\left(k\right)<\vartheta\left(k-1\right)&\text{\indent and \indent}\\
\exists N\in\omega_+,\forall M\in\left[N+1,\infty\right]_\omega,h\left(M\right)=\vartheta\left(M-1\right)-1
\end{array}
\end{align*}
Denote the set of all secondary auxiliary functions as $A_s$. For any secondary auxiliary function $h$ denote 
\begin{align*}
N_{\min}^h:=\min\left(\left\{N\in\omega_+\middle|\forall M\in\left[N+1,\infty\right]_\omega,h\left(M\right)=\vartheta\left(M-1\right)-1\right\}\right)
\end{align*}
\end{definition}

\begin{remark}Notice that all auxiliary functions are not contractable about any $n\in\omega_+$ and any such function is an auxiliary function. This is by design, as it tells us that those functions are in their most compact form.
\end{remark}

\begin{intuition}[Auxiliary Functions] Not all functions would define a number in the way we want to construct the reals, but we know that auxiliary functions will.
\\
We have said that we wish to think about the number described by the function $f:\omega\rightarrow\omega$, in the system governed by some $\vartheta:\omega\rightarrow\omega$, as $\pm\sum^{\infty}_{n=1}\frac{f\left(n\right)}{\prod_{k=1}^{n-1}\vartheta\left(k\right)}$. It is easy to see that in any system not all functions would describe a real number. For example $f\left(n\right)=\prod_{k=1}^{n-1}\vartheta\left(k\right)$ clearly does not.\footnote{This shows that there does not exist a $\vartheta$ such that all the functions $\omega\rightarrow\omega$ would describe a real number.}
\\
We see that a function which has an output $f\left(n\right)\in\left[0,\vartheta\left(n-1\right)-1\right]_\omega$ for all values of $n$ greater or equal $2$ clearly describes a real number in our system (this could be seen by comparison theorem to a series of $2^{2-n}$). And naïvely we also see that any real number should be describable by such a function. Therefore we want to sort through functions describing a real number and those not describing a real number.
\end{intuition}

\begin{intuition}[Secondary Auxiliary Function and $N_{\min}^h$] The notion of a secondary auxiliary function is based on the thought that (in decimal base) $1$ and $0.9999...$ or $0.7$ and $0.6999...$ are the same numbers. In our system the first one would be the representation by a primary auxiliary function and the second one the representation by a secondary auxiliary function. $N_{\min}^h$ describes the minimal value after which the secondary auxiliary function is always equal to $\vartheta\left(M-1\right)-1$. Therefore, this is the value where the primary auxiliary function and the secondary auxiliary function of the same number start disagreeing.
\end{intuition}
\begin{example} Work in base ten, so $\vartheta(n)=10$ for every $n$, and recall that position $n$ carries weight $\nicefrac{1}{\prod_{j=1}^{n-1}(\vartheta(j))}$; position $1$ therefore carries weight $1$, position $2$ weight $\nicefrac{1}{10}$, and so on.\\
Consider $f:\omega\rightarrow\omega$ with $f(0)=0$, $f(2)=5$ and $f(n)=0$ otherwise, which we write as $(0;0,5,0,0,\dots)$. Its value is $5\cdot\nicefrac{1}{10}=\nicefrac{1}{2}$, and since $f(0)=0$ is even the number is positive. No digit reaches its base, so $f$ is a primary auxiliary function.\\
Now take $g=(0;0,4,9,9,9,\dots)$. Its value is $\nicefrac{4}{10}+\nicefrac{9}{100}+\nicefrac{9}{1000}+\cdots=\nicefrac{1}{2}$ as well, and from position $3$ onwards it is constantly $\vartheta(n-1)-1=9$, so it is a secondary auxiliary function with $N^{g}_{\min}=\Psi(f,g)=2$. These two functions are the two names of $\nicefrac{1}{2}$, and the familiar identity $0.5=0.4999\ldots$ is exactly the statement that they lie in the same real number.\\
Figure~\ref{fig:normalforms} shows the carry-wave that takes one name to the other, and why the two agree on every finite initial segment.
\end{example}
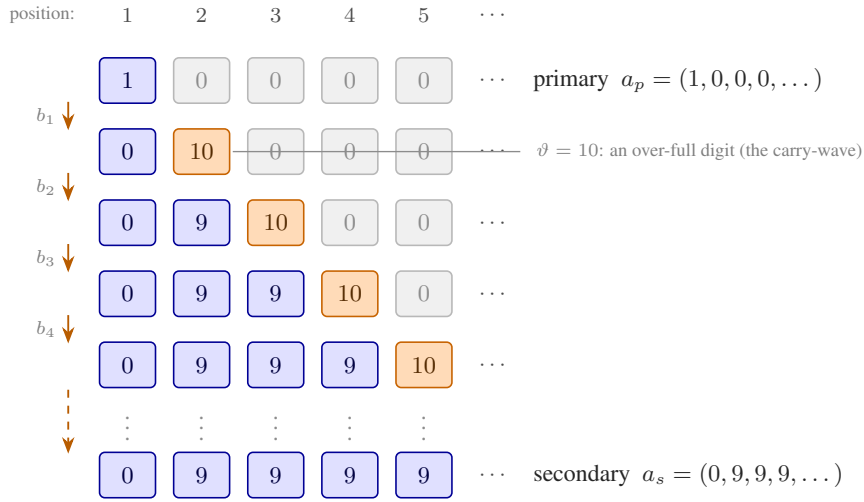
\begin{figure}[htbp]
  \centering
\begin{tikzpicture}[x=0.98cm, y=0.94cm,
  dcell/.style={rounded corners=2pt, minimum width=0.74cm, minimum height=0.6cm,
                inner sep=1pt, font=\small, line width=0.6pt},
  dset/.style ={dcell, draw=blue!58!black,  fill=blue!12,   text=blue!25!black},
  dact/.style ={dcell, draw=orange!78!black,fill=orange!28, text=orange!32!black},
  dnew/.style ={dcell, draw=black!28,       fill=black!6,   text=black!45},
  barr/.style ={-{Stealth[length=2.2mm]}, line width=0.7pt, orange!72!black},
]
 
  \node[wlab] at (-0.15,0.92){position:};
  \foreach \c in {1,2,3,4,5}{\node[poslab] at (\c,0.92){$\c$};}
  \node[poslab] at (5.95,0.92){$\cdots$};
 
  \foreach \c/\r/\v/\s in {%
    1/0/1/dset,  2/0/0/dnew, 3/0/0/dnew, 4/0/0/dnew, 5/0/0/dnew,
    1/1/0/dset,  2/1/10/dact,3/1/0/dnew, 4/1/0/dnew, 5/1/0/dnew,
    1/2/0/dset,  2/2/9/dset, 3/2/10/dact,4/2/0/dnew, 5/2/0/dnew,
    1/3/0/dset,  2/3/9/dset, 3/3/9/dset, 4/3/10/dact,5/3/0/dnew,
    1/4/0/dset,  2/4/9/dset, 3/4/9/dset, 4/4/9/dset, 5/4/10/dact}
    {\node[\s] at (\c,-\r){$\v$};}
 
  \foreach \r in {0,1,2,3,4}{\node[poslab] at (5.95,-\r){$\cdots$};}
 
  \foreach \c in {1,2,3,4,5}{\node[text=black!45,font=\small] at (\c,-4.78){$\vdots$};}
  \foreach \c/\v in {1/0,2/9,3/9,4/9,5/9}{\node[dset] at (\c,-5.55){$\v$};}
  \node[poslab] at (5.95,-5.55){$\cdots$};
 
  \foreach \r/\lab in {0/1,1/2,2/3,3/4}
    {\draw[barr] (0.2,-\r-0.30) -- (0.2,-\r-0.70)
        node[midway,left=1pt,wlab]{$b_{\lab}$};}
  \draw[barr,dashed] (0.2,-4.35) -- (0.2,-5.25)
        node[midway,left=1pt,wlab]{};
 
  \node[read,anchor=west] at (6.35,0)     {primary\ \ $a_p=(1,0,0,0,\dots)$};
  \node[read,anchor=west] at (6.35,-5.55) {secondary\ \ $a_s=(0,9,9,9,\dots)$};
 
  \node[wlab,anchor=west] at (6.35,-1)
     {\,$\vartheta=10$: an over-full digit (the carry-wave)};
  \draw[black!45,line width=0.5pt] (6.3,-1) -- (2.42,-1);
\end{tikzpicture}
 
\par\vspace{4pt}
\caption{The two normal forms of the real number $1$, in base $\vartheta\equiv 10$. Broadenings $b_1,b_2,\dots$ carry the primary $(1,0,0,\dots)$ toward the secondary $(0,9,9,\dots)$: an over-full digit $\vartheta=10$ (amber) travels rightward, locking the prefix behind it (blue) while positions ahead (grey) are still untouched. Once the wave passes position~$m$, the first $m$ digits never change again, so the two forms agree on every finite initial segment, which is exactly why they name the same real number. This is $1=0.999\ldots$ inside the model. The sign slot $f(0)$ is suppressed. This is only an illustration for intuition, as we do not allow infinitely many operations to be done (the reader might think what that would mean in the context where we broaden about every position as much as we can).}
  \label{fig:normalforms}
\end{figure}
\subsection{Base Set and a Real Number}
\begin{definition}[Base Set]\label{Base Set}\noindent\\
Define the set
\begin{align*}
W_1:=\left\{\left(h_p,h_s\right)\in A_p\times A_s\middle|
\begin{array}{ll@{}}
&\forall k\in\left[0,N_{\min}^{h_s}-1\right]_\omega,h_p\left(k\right)=h_s\left(k\right)\\
\wedge&h_p\left(N_{\min}^{h_s}\right)=h_s\left(N_{\min}^{h_s}\right)+1\\
\wedge&\forall k\in\left[N_{\min}^{h_s}+1,\infty\right]_\omega,h_p\left(k\right)=0
\end{array}
\right\}
\end{align*}
Define the equivalence relation for $\left(w_{1p},w_{1s}\right),\left(w_{2p},w_{2s}\right)\in W_1$ which groups together all pairs with the same sign as $\left(w_{1p},w_{1s}\right)R_1\left(w_{2p},w_{2s}\right)$ if and only if 
\begin{align*}
&\left(\forall k\in\omega_+,w_{1p}\left(k\right)=w_{2p}\left(k\right)\right)&\text{and}\\
&\left(\sigma\left(w_{1p}\right)=\sigma\left(w_{2p}\right)\right)
\end{align*} Let the set $G_1$ be the set of equivalence classes of $W_1$ for the relation $R_1$.\\
Now define $W_2:=\left\{h_p\in A_p\middle|h_p\not\in\bigcup\bigcup\left(W_1\right)\right\}$. Define the equivalence relation for $w_1,w_2\in W_2$ as $w_1R_2w_2$ if and only if 
\begin{align*}
&\left(\forall k\in\omega_+,w_1\left(k\right)=w_2\left(k\right)\right)&\text{and}\\
&\left(\sigma\left(w_{1}\right)=\sigma\left(w_{2}\right)\right)
\end{align*}
Let the set $G_2$ be the set of equivalence classes of $W_2$ for the relation $R_2$.\\
Define $F:=\left\{\bigcup\bigcup\left(b\right)\middle|b\in G_1\right\}$. We call the set $B:=F\cup G_2$ the base set.
\end{definition}

\begin{intuition}[Base Set] We wish to group together all the auxiliary functions describing the same real number in one set. We need to group together those with the same sign, but different output at $0$, and also the pairs of primary and secondary auxiliary functions. By this we mean pairing $1$ and $0.99999...$ and similar cases.
\end{intuition}

\begin{intuition}[Construction of the Base Set] $W_1$ corresponds to the set of pairs of auxiliary primary and auxiliary functions describing the same real number and having the same output at $0$. $R_1$ then takes all the pairs which describe the same real number, by having the same magnitude and sign but different output at $0$ and groups them together. This is how we get $G_1$.
\\
$W_2$ is the set of all the primary auxiliary functions which do not have a corresponding secondary auxiliary function. $R_2$ then does the equivalent to $R_1$. Therefore, $G_2$ is a set of sets of primary auxiliary functions describing the same real number.
\\
However, as $G_1$ is a set of pairs we have to define $F$ in such a way, as we want a set of sets containing the functions contained in equivalence classes of $W_1$, not the pairs. Thus we need to go down two layers of sets as $G_1$ is a set of sets containing ordered pairs of those functions. This way $B$ describes what we wanted it to describe. Figure~\ref{fig:baseset} carries out this assembly for $\nicefrac{1}{2}$ in base ten.
\end{intuition}

\begin{figure}[tbp]
  \centering
\begin{tikzpicture}[x=1cm,y=1cm]
 
 \foreach \i/\z in {0/0, 1/2, 2/4}{
   \node[prim] (p\i) at (\i*4.3, 0)   {$(\,\z;0,5,0,0,\dots)$};
   \node[sec]  (s\i) at (\i*4.3, -0.95){$(\,\z;0,4,9,9,\dots)$};
   \node[wbox, fit=(p\i)(s\i)] (w\i) {};
   \node[lbl, above=2.5pt of p\i] {$f(0)=\z$};
 }
 
 \node[hd, right=8pt of w2.east, align=left, text width=2.4cm]
      {paired by $W_1$: primary with its secondary};
 
 \begin{scope}[on background layer]
   \node[rbox, fit=(w0)(w1)(w2), fill=black!3] (R) {};
 \end{scope}
 \node[hd, below=3pt of R.south, align=center]
   {merged by $R_1$: same magnitude and sign, different value at $0$};
 
 \node[font=\small, above=17pt of R.north] (num)
   {$\left[\tfrac12\right]\;=\;$ this whole box $\;\in\mathcal{S}_\mathbb{R}$};
 \draw[-{Stealth[length=2.4mm]}, black!55] (R.north) -- (num.south);
 
 \node[lbl, below=20pt of R.south, align=left] (leg)
   {\textcolor{blue!58!black}{$\blacksquare$} primary auxiliary function
    \qquad
    \textcolor{orange!78!black}{$\blacksquare$} secondary auxiliary function};
\end{tikzpicture}
  \caption{Assembling the base set of a single real number,
shown for $\tfrac12$ in base ten. Each sign slot $f(0)\in\{0,2,4,\dots\}$
carries its own primary/secondary pair; $W_1$ ties each primary to its
secondary, and $R_1$ then merges the columns, which differ only in the value
recorded at position $0$. The real number is the union: one class holding
every legitimate way of writing $\tfrac12$. Negative numbers occupy the odd
sign slots, and the two zero classes are merged separately.}
  \label{fig:baseset}
\end{figure}
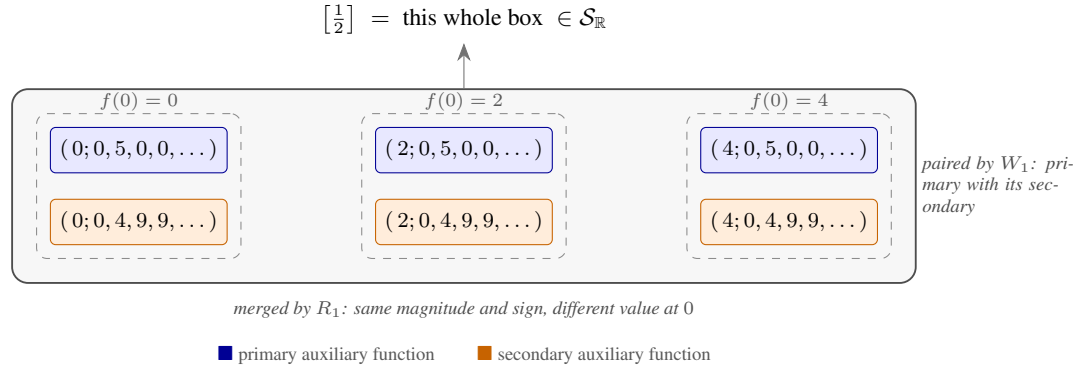

\begin{definition}[Real Number]\label{Real Number}\noindent\\
For any set $b$ in the base set $B$ we call the set
\begin{align*}
r_b:=\left\{f\in\mathcal{N}\middle|\exists h\in b,\forall N\in\omega_+,\exists\mathcal{C}_{k=q}^r\in\mathcal{FS},\forall n\in\left[0,N\right]_\omega,\mathcal{C}_{k=q}^r\left(f\right)\left(n\right)=h\left(n\right)\right\}
\end{align*}
the real number of $b$. We shall call $N$ and $\left[0,N\right]_\omega$ the number of inspection and the range of inspection respectively.
\end{definition}
\begin{lemma}[Equivalence Lemma]\label{Equivalence Lemma}\noindent\\
For all functions $f,g:\omega\rightarrow\omega$ such that for all values $n\in\omega_+$ we have $f\left(n\right)=g\left(n\right)$ and if there exists a value $m\in\omega$ such that $f(0)+g(0)=2m$ then if $f\in r_b$ then $g\in r_b$. Explicitly:
\begin{align*}
\hspace{-0.5cm}\begin{array}{c@{}}
\forall f,g\in\mathcal{N}\\
\left(\left(\forall n\in\omega_+,f\left(n\right)=g\left(n\right)\right)\wedge\left(\exists m\in\omega,f\left(0\right)+g\left(0\right)=2m\right)\right)\implies\left(\forall b\in B,\left(f\in r_b\implies g\in r_b\right)\right)
\end{array}
\end{align*}
\end{lemma}
\begin{proof} Suppose $f\in r_b$ for some $b\in B$. By the definition of $r_b$ we have that there is an auxiliary function $h\in b$ such that for any $N\in\omega$ there exists a sequence of contractions $\mathcal{C}_{k=q}^r$  such that for all values $n$ on the interval $\left[0,N\right]_\omega$ we have $\mathcal{C}_{k=q}^r\left(f\right)\left(n\right)=h\left(n\right)$. By the definition of contractions for any sequence of contractions $\mathcal{C}_{k=q}^r\left(f\right)\left(0\right)=f\left(0\right)$. Therefore, $f\left(0\right)=h\left(0\right)$. Thus $\exists m\in\omega,h\left(0\right)+g\left(0\right)=2m$, which means that $g\left(0\right)=2m-h\left(0\right)$.
\\
Take $h':\omega\rightarrow\omega$ such that $h'\left(0\right)=2m-h\left(0\right)$ and $h'\left(n\right)=h\left(n\right)$ otherwise. We see that $h'$ is an auxiliary function (it is primary if and only if $h$ is primary and secondary if and only if $h$ is secondary). Furthermore, $\sigma\left(h'\right)=\sigma\left(h\right)$ as if $h\left(0\right)$ is even then $2m-h\left(0\right)=h'\left(0\right)$ is even and if $h\left(0\right)$ is odd then $2m-h\left(0\right)=h'\left(0\right)$ is odd. Therefore, by the construction of the base set $h'\in b$.
\\
Now, for any $N\in\omega_+$ given, take the same sequence of contractions $\mathcal{C}_{k=q}^r$ as for $f$, which is applicable to $g$, by the \nameref{Submission Theorem} \ref{Submission Theorem}, as $\text{con}\left(\mathcal{C}_{k=q}^r\right)\subset\omega_+$ and $g$ dominates $f$. Therefore, by the \nameref{Submission Theorem} \ref{Submission Theorem} $\forall n\in\left[0,N\right]_\omega,\mathcal{C}_{k=q}^r\left(g\right)\left(n\right)=h'(n)$. This shows that $g\in r_b$.
\end{proof}
\subsection{Unicity and Shifting Theorems}
\begin{theorem}[Shifting Theorem]\label{Shifting Theorem}\noindent\\ 
For any function $f:\omega\rightarrow\omega$ such that $f\in r_b$ and any finite sequence of contractions and broadenings $\mathcal{D}_{k=q}^r\in\mathcal{FS}$ applicable to $f$ we have $\mathcal{D}_{k=q}^r\left(f\right)\in r_b$.
\end{theorem}
\begin{proofidea} It is enough to treat a single move and induct, so the question becomes: if $f$ lies in $r_b$, why do $c_m\left(f\right)$ and $b_m\left(f\right)$ lie there too? For a broadening the Cancellation Theorems answer it immediately, contracting about the same position returns $f$, and a broadening always satisfies the condition to be contracted. A contraction needs work, because the sequence that matched $f$ against the auxiliary function $h$ need not match the moved function. We therefore build a new sequence $\mathcal{C}'$ from the old one and use the Submission Theorem to certify that it is applicable and still reaches agreement with the same $h$ on the interval of inspection.
\end{proofidea}
\begin{proof} As a proof, we show that if $f\in r_b$ then $c_m\left(f\right)\in r_b$ and $b_m\left(f\right)\in r_b$ for any $m\in\omega_+$. The Shifting Theorem follows inductively.
\\
We know that if $f\in r_b$ then $b_m\left(f\right)\in r_b$ as we simply apply $c_m$ to $b_m\left(f\right)$.\footnote{We know that we may do this by the \nameref{Cancellation Theorems} \ref{Cancellation Theorems}. In particular we use the fact used in the proof that a broadening fulfils the condition to be contractable.} This results in $f$. Therefore, we know that 
\begin{align*}
\exists h\in b,\forall N\in\omega_+,\exists\mathcal{C}_{k=q}^r\in\mathcal{FS},\forall n\in\left[0,N\right]_\omega,\mathcal{C}_{k=q}^r\left(c_m\left(b_m\left(f\right)\right)\right)\left(n\right)=h\left(n\right)
\end{align*}
Thus, take $\mathcal{C'}\in\mathcal{CB}$ defined as 
\begin{align*}
\mathcal{C'}\left(k\right)=\begin{cases}c_m&k=q\\
\mathcal{C}\left(k-1\right)&k\in\left[q+1,r+1\right]_\omega\\
c_1&otherwise
\end{cases}
\end{align*}
By the construction of $\mathcal{C'}$ and the \nameref{Submission Theorem} \ref{Submission Theorem} we have that for all $n\in\left[0,N\right]_\omega$ we have $\mathcal{C'}_{k=q}^{r+1}\left(b_m\left(f\right)\right)\left(n\right)=h\left(n\right)$. Therefore we conclude that 
\begin{align*}
\exists h\in b,\forall N\in\omega_+,\exists\mathcal{C'}_{k=q}^{r+1}\in\mathcal{FS},\forall n\in\left[0,N\right]_\omega,\mathcal{C'}_{k=q}^{r+1}\left(b_m\left(f\right)\right)\left(n\right)=h\left(n\right)
\end{align*}
Hence $b_m\left(f\right)\in r_b$.
\\
We shall show that $c_m\left(f\right)\in r_b$. Let $N\in\omega_+$ be given. Take $R:=\max\left(\left\{N,m+1\right\}\right)$. Let $\mathcal{C}_{k=q}^r$ be a sequence of contractions, such that for all $n\in\left[0,R\right]_\omega$ we have $\mathcal{C}_{k=q}^r\left(f\right)\left(n\right)=h\left(n\right)$. Thus as $\left[0,N\right]_\omega\subset\left[0,R\right]_\omega$ we see that for all $n\in\left[0,N\right]_\omega$ we have $\mathcal{C}_{k=q}^r\left(f\right)\left(n\right)=h\left(n\right)$.
\\
Define $T:=\min\left(\left\{k\in\left[q,r\right]_\omega\middle|\mathcal{C}\left(k\right)=c_m\right\}\right)$. We will show that for all $n$ in $\omega$ the equality $\mathcal{C}_{k=q}^T\left(f\right)\left(n\right)=\mathcal{C}_{k=q}^{T-1}\left(c_m\left(f\right)\right)\left(n\right)$ holds. In order to prove this we need to prove the result:\\
\textit{For all $f:\omega\rightarrow\omega$, for all $m,n\in\omega_+$, $n\neq m$ and $f\left(m+1\right)\geq\vartheta\left(m\right)$ and $f\left(n+1\right)\geq\vartheta\left(n\right)$ implies} \begin{align}\label{eq:commute}
c_m\left(c_n\left(f\right)\right)=c_n\left(c_m\left(f\right)\right)
\end{align}
We see that both $c_m\left(c_n\left(f\right)\right)$ and $c_n\left(c_m\left(f\right)\right)$ exist. We see the first case as $f\left(n+1\right)\geq\vartheta\left(n\right)$ and $c_n\left(f\right)\left(m+1\right)\geq f\left(m+1\right)\geq\vartheta\left(m\right)$\footnote{This is the case as $n\neq m$ and the only values of $f$ affected by $c_n$ are $f\left(n\right)$ and $f\left(n+1\right)$. So if $n\neq m+1$ then $c_n\left(f\right)\left(m+1\right)=f\left(m+1\right)\geq\vartheta\left(m\right)$ and if $n=m+1$ then $c_n\left(f\right)\left(m+1\right)=f\left(m+1\right)+1>\vartheta\left(m\right)$.} and thus $c_m\left(c_n\left(f\right)\right)$. $c_n\left(c_m\left(f\right)\right)$ exists by the same argument. The only thing left to check is that the output of the functions is the same at all the values for all the cases.\\
\newpage
\textbf{Case 1: $m\neq n+1\wedge n\neq m+1$} We want to show that for all $k\in\omega$ we have $c_m\left(c_n\left(f\right)\right)\left(k\right)=c_n\left(c_m\left(f\right)\right)\left(k\right)$. We verify:
\begin{align*}
\begin{cases}
c_m\left(c_n\left(f\right)\right)\left(n\right)&=c_n\left(f\right)\left(n\right)=f\left(n\right)+1\\
&=c_m\left(f\right)\left(n\right)+1=c_n\left(c_m\left(f\right)\right)\left(n\right)
\\
c_m\left(c_n\left(f\right)\right)\left(n+1\right)&=c_n\left(f\right)\left(n+1\right)=f\left(n+1\right)-\vartheta\left(n\right)\\
&=c_m\left(f\right)\left(n+1\right)-\vartheta\left(n\right)=c_n\left(c_m\left(f\right)\right)\left(n+1\right)
\\
c_m\left(c_n\left(f\right)\right)\left(m\right)&=c_n\left(f\right)\left(m\right)+1=f\left(m\right)+1\\
&=c_m\left(f\right)\left(m\right)=c_n\left(c_m\left(f\right)\right)\left(m\right)
\\
c_m\left(c_n\left(f\right)\right)\left(m+1\right)&=c_n\left(f\right)\left(m+1\right)-\vartheta\left(m\right)=f\left(m+1\right)-\vartheta\left(m\right)\\
&=c_m\left(f\right)\left(m+1\right)=c_n\left(c_m\left(f\right)\right)\left(m+1\right)
\\
c_m\left(c_n\left(f\right)\right)\left(k\right)&=c_n\left(f\right)\left(k\right)=f\left(k\right)\\
&=c_m\left(f\right)\left(k\right)=c_n\left(c_m\left(f\right)\right)\left(k\right)\\
\end{cases}
\def\arraystretch{1.2}\begin{array}{@{}l}
\text{\indent if $k=n$}\\\\
\text{\indent if $k=n+1$}\\\\
\text{\indent if $k=m$}\\\\
\text{\indent if $k=m+1$}\\\\
\text{\indent otherwise}
\end{array}
\end{align*}
\textbf{Case 2: $m=n+1$} We want to show that $\forall k\in\omega,c_m\left(c_n\left(f\right)\right)\left(k\right)=c_n\left(c_m\left(f\right)\right)\left(k\right)$. We verify:
\begin{align*}
\begin{cases}
c_m\left(c_n\left(f\right)\right)\left(n\right)&=c_n\left(f\right)\left(n\right)=f\left(n\right)+1\\
&=c_m\left(f\right)\left(n\right)+1=c_n\left(c_m\left(f\right)\right)\left(n\right)
\\
c_m\left(c_n\left(f\right)\right)\left(n+1\right)&=c_n\left(f\right)\left(n+1\right)+1=f\left(n+1\right)-\vartheta\left(n\right)+1\\
&=c_m\left(f\right)\left(n+1\right)-\vartheta\left(n\right)=c_n\left(c_m\left(f\right)\right)\left(n+1\right)
\\
c_m\left(c_n\left(f\right)\right)\left(n+2\right)&=c_n\left(f\right)\left(n+2\right)-\vartheta\left(n+1\right)=f\left(n+2\right)-\vartheta\left(n+1\right)\\
&=c_m\left(f\right)\left(n+2\right)=c_n\left(c_m\left(f\right)\right)\left(n+2\right)
\\
c_m\left(c_n\left(f\right)\right)\left(k\right)&=c_n\left(f\right)\left(k\right)=f\left(k\right)\\
&=c_m\left(f\right)\left(k\right)=c_n\left(c_m\left(f\right)\right)\left(k\right)\\
\end{cases}
\def\arraystretch{1.2}\begin{array}{@{}l}
\text{\indent if $k=n$}\\\\
\text{\indent if $k=n+1$}\\\\
\text{\indent if $k=n+2$}\\\\
\text{\indent otherwise}
\end{array}
\end{align*}
\newpage
\textbf{Case 3: $n=m+1$} We want to show that $\forall k\in\omega,c_m\left(c_n\left(f\right)\right)\left(k\right)=c_n\left(c_m\left(f\right)\right)\left(k\right)$. We verify:
\begin{align*}
\begin{cases}
c_m\left(c_n\left(f\right)\right)\left(m\right)&=c_n\left(f\right)\left(m\right)+1=f\left(m\right)+1\\
&=c_m\left(f\right)\left(m\right)=c_n\left(c_m\left(f\right)\right)\left(m\right)
\\
c_m\left(c_n\left(f\right)\right)\left(m+1\right)&=c_n\left(f\right)\left(m+1\right)-\vartheta\left(m\right)=f\left(m+1\right)-\vartheta\left(m\right)+1\\
&=c_m\left(f\right)\left(m+1\right)+1=c_n\left(c_m\left(f\right)\right)\left(m+1\right)
\\
c_m\left(c_n\left(f\right)\right)\left(m+2\right)&=c_n\left(f\right)\left(m+2\right)=f\left(m+2\right)-\vartheta\left(m+1\right)\\
&=c_m\left(f\right)\left(m+2\right)-\vartheta\left(m+1\right)=c_n\left(c_m\left(f\right)\right)\left(m+2\right)
\\
c_m\left(c_n\left(f\right)\right)\left(k\right)&=c_n\left(f\right)\left(k\right)=f\left(k\right)\\
&=c_m\left(f\right)\left(k\right)=c_n\left(c_m\left(f\right)\right)\left(k\right)
\end{cases}
\def\arraystretch{1.2}\begin{array}{@{}l}
\text{\indent if $k=n$}\\\\
\text{\indent if $k=n+1$}\\\\
\text{\indent if $k=n+2$}\\\\
\text{\indent otherwise}
\end{array}
\end{align*}
This proves the identity.\\
By continual repetition of the identity we get that for all $n\in\omega$ we have $\mathcal{C}_{k=q}^T\left(f\right)\left(n\right)=\mathcal{C}_{k=q}^{T-1}\left(c_m\left(f\right)\right)\left(n\right)$. Here we use the fact that as for all $K\in\left[q,T-1\right]_\omega$ we have 
\begin{align*}
\mathcal{C}_{k=q}^{K-1}\left(f\right)\left(m+1\right)\geq f\left(m+1\right)\geq\vartheta\left(m\right)
\end{align*}
(see remark) This tells us that $\left(c_m\circ\mathcal{C}\left(K\right)\right)\left(\mathcal{C}_{k=q}^{K-1}\left(f\right)\right)=\left(\mathcal{C}\left(K\right)\circ c_m\right)\left(\mathcal{C}_{k=q}^{K-1}\left(f\right)\right)$.
\\
We know that for all $n\in\left[0,N\right]_\omega$ we have 
\begin{align*}
h\left(n\right)=\mathcal{C}_{k=q}^r\left(f\right)\left(n\right)=\mathcal{C}_{k=T+1}^r\left(\mathcal{C}_{k=q}^T\left(f\right)\right)\left(n\right)=\mathcal{C}_{k=T+1}^r\left(\mathcal{C}_{k=q}^{T-1}\left(c_m\left(f\right)\right)\right)\left(n\right)
\end{align*}
Hence, by taking $\mathcal{C'}\in\mathcal{CB}$ defined as 
\begin{align*}
\mathcal{C'}\left(k\right)=\begin{cases}
\mathcal{C}\left(k\right)&k\in\left[q,T-1\right]_\omega\\
\mathcal{C}\left(k+1\right)&k\in\left[T,r-1\right]_\omega\\
c_1&otherwise
\end{cases}
\end{align*} 
we get that for all $n\in\left[0,N\right]_\omega$ we have $\mathcal{C'}_{k=q}^{r-1}\left(c_m\left(f\right)\right)\left(n\right)=h\left(n\right)$. This shows that $c_m\left(f\right)\in r_b$.
\\
By induction the theorem follows. Notice that $f$ and $g$ share the same auxiliary function.
\end{proof}

\begin{remark}[$T$ is not empty] Firstly, as $c_m$ is applicable to $f$ we know that $f\left(m+1\right)\geq\vartheta\left(m\right)$. Suppose then that for all $k\in\left[q,r\right]_\omega$ we have $\mathcal{C}\left(k\right)\neq c_m$. We prove by induction that $\mathcal{C}_{k=q}^r\left(f\right)\left(m+1\right)\geq f\left(m+1\right)$.
\\
As $\mathcal{C}\left(q\right)\neq c_m$ we have that $\mathcal{C}_{k=q}^q\left(f\right)\left(m+1\right)\geq f\left(m+1\right)$. Suppose that for some $K\in\left[q,r-1\right]_\omega$ we have $\mathcal{C}_{k=q}^K\left(f\right)\left(m+1\right)\geq f\left(m+1\right)$. As $\mathcal{C}\left(K+1\right)\neq c_m$ we must have
\begin{align*}
\mathcal{C}(K+1)\left(\mathcal{C}_{k=q}^K\left(f\right)\right)\left(m+1\right)\geq\mathcal{C}_{k=q}^K\left(f\right)\left(m+1\right)\geq f\left(m+1\right)
\end{align*}
Hence, $\mathcal{C}_{k=q}^r\left(f\right)\left(m+1\right)\geq f\left(m+1\right)$. As $R\geq m+1$ then $\left\{k\in\left[q,r\right]_\omega\middle|\mathcal{C}\left(k\right)=c_m\right\}$ cannot be empty, as if it was then $\mathcal{C}_{k=q}^r\left(f\right)\left(m+1\right)\geq f\left(m+1\right)\geq\vartheta\left(m\right)>h\left(m+1\right)$, a contradiction to what we supposed.
\end{remark}

\begin{theorem}[Unicity of Real Numbers]\label{Unicity of Real Numbers}\noindent\\
For all $b_1,b_2\in B$, $b_1\neq b_2$ implies $r_{b_1}\cap r_{b_2}=\emptyset$.
\end{theorem}
\begin{proof} See that as the base set was constructed from equivalence classes $b_1$ and $b_2$ do not share any elements. Now suppose for the sake of contradiction that there exists an $f:\omega\rightarrow\omega$ such that $\left(f\in r_{b_1}\wedge f\in r_{b_2}\right)$. Call $h_1\in b_1$ and $h_2\in b_2$ the respective auxiliary functions from the definition of a real number for $f$. Notice that by the presumption $h_1\left(0\right)=h_2\left(0\right)=f\left(0\right)$.
\\
Define a new function which will be used throughout the entire work, $\Psi\in\mathcal{N}\times\mathcal{N}\rightarrow\omega$ defined as 
\begin{align*}
\Psi\left(f,g\right):=\min\left(\left\{n\in\omega_+\middle|f\left(n\right)\neq g\left(n\right)\right\}\right)
\end{align*}
Suppose without loss of generality that $h_1\left(\Psi\left(h_1,h_2\right)\right)>h_2\left(\Psi\left(h_1,h_2\right)\right)$. Now, set the number of inspection $N=\Psi\left(h_1,h_2\right)$. We know that 
\begin{align*}
\exists\mathcal{C}_{k=q}^{r}\in\mathcal{FS},\forall n\in\left[0,N\right]_\omega,\mathcal{C}_{k=q}^{r}\left(f\right)\left(n\right)=h_1\left(n\right)
\end{align*}
By the \nameref{Shifting Theorem} \ref{Shifting Theorem} $\mathcal{C}_{k=q}^{r}\left(f\right)\in r_{b_2}$ as $f\in r_{b_2}$.
\\
However, as only contractions may be applied to $\mathcal{C}_{k=q}^{r}\left(f\right)$ we see that we may not find any sequence of contractions such that it would match with $h_2$ on $\left[0,N\right]_\omega$ as it already matches on $\left[0,N-1\right]_\omega$ and at the same time $\mathcal{C}_{k=q}^{r}\left(f\right)\left(N\right)=h_1\left(N\right)>h_2\left(N\right)$.
\\
This is shown by taking any sequence of contractions $\mathcal{C'}_{k=a}^{b}$ applicable to $\mathcal{C}_{k=q}^{r}\left(f\right)$. Take $t:=\min\left(\left\{n\in\omega_+\middle|\exists k\in\left[a,b\right]_\omega,\mathcal{C'}\left(k\right)=c_n\right\}\right)$.
\\
If $t\in\left[1,N\right]_\omega$ then by the \nameref{Consequences Theorem} \ref{Consequences Theorem} 
\begin{align*}
\mathcal{C'}_{k=a}^{b}\left(\mathcal{C}_{k=q}^{r}\left(f\right)\right)\left(t\right)>\mathcal{C}_{k=q}^{r}\left(f\right)\left(t\right)\geq h_2\left(t\right)
\end{align*}
If $t\in\left[N+1,\infty\right]_\omega$ then by the \nameref{Consequences Theorem} \ref{Consequences Theorem} as $N\in\left[1,t-1\right]_\omega$ we have 
\begin{align*}
\mathcal{C'}_{k=a}^{b}\left(\mathcal{C}_{k=q}^{r}\left(f\right)\right)\left(N\right)=\mathcal{C}_{k=q}^{r}\left(f\right)\left(N\right)>h_2\left(N\right)
\end{align*}
Therefore, for all sequences of contractions $\mathcal{C'}_{k=a}^{b}\in\mathcal{FS}$ applicable to $\mathcal{C}_{k=q}^{r}\left(f\right)$ there exists a value $n$ in $\left[0,N\right]_\omega$ such that $\mathcal{C'}_{k=a}^{b}\left(\mathcal{C}_{k=q}^{r}\left(f\right)\right)\left(n\right)>h_2\left(n\right)$. However, this is a contradiction to the \nameref{Shifting Theorem} \ref{Shifting Theorem} and the presumption that $h_2$ is an auxiliary function of $f$ in $r_{b_2}$.
\end{proof}

\begin{remark} In the definition of $\Psi$ notice that even if the functions agree on $\omega_+$ then we have a minimum of the empty set, which is the empty set, which is $0$. And as we are taking a set of non-zero finite ordinals, this is in fact the only way to obtain $0$. Notice also that $\Psi$ is commutative, a property we will be using in this work without specifying it.
\end{remark}

\begin{definition}[Set of Real Numbers - $\mathcal{S}_\mathbb{R}$]\label{Set of Reals}\noindent\\
Define the set $\mathcal{S}_\mathbb{R}:=\left\{r_b\middle|b\in B\right\}-\left\{r_{b_1}\right\}-\left\{r_{b_2}\right\}\cup\left\{r_{b_1}\cup r_{b_2}\right\}$, where $f\in b_1,g\in b_2$  where $f,g:\omega\rightarrow\omega$ such that for all $k\in\omega$ we have
\begin{align*}
\begin{array}{ll@{}}
f\left(k\right)=0&\hspace{2cm}g\left(k\right)=\begin{cases}1&k=0\\0&otherwise\end{cases}
\end{array}
\end{align*}
Formally we should call this set the functionary space of real numbers described by $\vartheta$ (as this set depends on the choice of the \nameref{System Function} \ref{System Function}). However, we will call this set the set of real numbers for brevity.
\\
We shall denote $\left[0\right]:=r_{b_1}\cup r_{b_2}$ and call this set zero. Define also $\mathcal{S}_\mathbb{R}^*:=\mathcal{S}_\mathbb{R}-{\left[0\right]}$ and call this set the non-zero real numbers.
\\
We also define $\mathcal{N}_{\mathcal{F}in}:=\bigcup\mathcal{S}_\mathbb{R}$ and $\mathcal{N}_{\mathcal{F}in}^*=\mathcal{N}_{\mathcal{F}in}-\bigcup\left(\left[0\right]\right)$. We call the functions $f\in\mathcal{N}_{\mathcal{F}in}$ finite.
\\
We shall define a relation $E$ on $\mathcal{N}_{\mathcal{F}in}$ where $fEg$ if and only if $\exists x\in\mathcal{S}_{\mathbb{R}}, f,g\in x$. By the \nameref{Unicity of Real Numbers} \ref{Unicity of Real Numbers} we know that $E$ is an equivalence relation. We choose the notation of $\forall f\in \mathcal{N}_{\mathcal{F}in},\left[f\right]:=x\in\mathcal{S}_{\mathbb{R}}$ where $f\in x$, again, by the \nameref{Unicity of Real Numbers} \ref{Unicity of Real Numbers} this set is unique. 
\end{definition}

\begin{intuition}[Set of Real Numbers] Before we had ``two zeros" a positive zero and a negative zero. This would cause problem later, thus we unite them into a single real number and call the number zero. As we wish to work with real numbers represented by functions and we have already established that in any system there are functions not defining any real number. We do not need these functions and therefore we restrict our original set $\mathcal{N}$ to $\mathcal{N}_{\mathcal{F}in}$ so that we work only with functions which define some real number.
\end{intuition}

\begin{lemma}[Null Lemma]\label{Null Lemma}\noindent\\
For all finite functions $w$ we have that $w\in\left[0\right]$ if and only if $\forall n\in\omega_+,w\left(n\right)=0$.
\end{lemma}
\begin{proof}\noindent\\
$\left(\impliedby\right):$\noindent\\
Either $w\left(0\right)+f\left(0\right)=2m$ or $w\left(0\right)+g\left(0\right)=2m$ for some $m\in\omega$ where $f$ and $g$ are the same functions as described in the definition of \nameref{Set of Reals} \ref{Set of Reals}. Then by the \nameref{Equivalence Lemma} \ref{Equivalence Lemma} we have $w\in\left[0\right]$.\\
$\left(\implies\right):$\noindent\\
Suppose for the sake of contradiction that there exists $n\in\omega_+,w\left(n\right)\neq0$. Then there exists a minimal such $n$. Denote that value as $m$. Thus by taking $m$ as the value of inspection we arrive at a contradiction to $w\in\left[0\right]$ by the \nameref{Consequences Theorem} \ref{Consequences Theorem} and the same technique used in the proof of the \nameref{Unicity of Real Numbers} \ref{Unicity of Real Numbers}.
\end{proof}

\begin{definition}[Sign of a Real Number]\label{Sign of a Real Number}\noindent\\
By the definition of our non-zero equivalence classes $\forall a,b\in\left[a\right]\in\mathcal{S}_\mathbb{R}^*,\sigma\left(a\right)=\sigma\left(b\right)$. Therefore, we introduce a sign of the whole equivalence class (the real numbers). Define $\sigma:{\mathcal{S}_\mathbb{R}^*}\rightarrow\left\{0,1\right\}$ as 
\begin{align*}
\begin{array}{llll@{}}
\sigma\left(\left[a\right]\right)=0&\iff&\forall a\in\left[a\right],\sigma\left(a\right)=0&\text{and}\\
\sigma\left(\left[a\right]\right)=1&\iff&\forall a\in\left[a\right],\sigma\left(a\right)=1
\end{array}
\end{align*}
\end{definition}
\section{Structure of the Models}\label{sec:structure}
Now that we have constructed the set of real numbers we shall show crucial theorems and lemmas which will help us prove the axioms of real analysis on our model.
\subsection{Paradise City Lemma}
\begin{lemma}[Paradise City Lemma]\label{Paradise City Lemma}\noindent\\
For any finite function $f$ we have that for all $m$ in $\omega_+$ the inequality 
\begin{align*}
\sum\limits_{n=1}^{m}\left(f\left(n\right)\prod\limits_{j=n}^{m-1}\left(\vartheta\left(j\right)\right)\right)\leq\sum\limits_{n=1}^{m}\left(h_f\left(n\right)\prod\limits_{j=n}^{m-1}\left(\vartheta\left(j\right)\right)\right)
\end{align*}
where $h_f$ is the auxiliary function of $f$ holds.\footnote{Even though it is simple to see that there is only one such function, we have not proved it and thus we take ``any of the auxiliary functions of f" even though we know that there is exactly one.}
\end{lemma}

\begin{intuition}[Paradise City Lemma] As primary auxiliary functions are the most compact way to describe a real number no function in that real number can have ``more material" in any interval $\left[1,m\right]_\omega$.
\end{intuition}

\begin{proofidea} The observation that makes this work is that the weighted truncation $$\sum_{n=1}^{m}f\left(n\right)\prod_{j=n}^{m-1}\vartheta\left(j\right)$$ is not merely a number computed from $f$: it is the value \emph{at position $m$} of the function obtained by broadening $f$ completely on $\left[1,m-1\right]_\omega$, all of the material pushed down into a single position. That is verified by induction, and by the \nameref{Shifting Theorem} \ref{Shifting Theorem} the broadened function stays in the same real number. Comparing truncations is therefore comparing single digits of completely broadened functions. Supposing for contradiction that $f$ overshoots at some $m$, we broaden $h_f$ completely as well, invert that sequence of broadenings into contractions, and apply it to the broadened $f$, applicable by the \nameref{Submission Theorem} \ref{Submission Theorem}, since one dominates the other throughout $\left[0,m\right]_\omega$, and the \nameref{Shifting Theorem} \ref{Shifting Theorem} places the result in the same real number, where the digits contradict.
\end{proofidea}

\begin{proof} We shall prove this by contradiction. Suppose that there is a finite $f$ such that $\exists m\in\omega_+$ such that 
\begin{align*}
\sum\limits_{n=1}^{m}\left(f\left(n\right)\prod\limits_{j=n}^{m-1}\left(\vartheta\left(j\right)\right)\right)\geq\sum\limits_{n=1}^{m}\left(h_f\left(n\right)\prod\limits_{j=n}^{m-1}\left(\vartheta\left(j\right)\right)\right)+1
\end{align*}
Now we see that
\begin{equation}\label{eq:complete-broadening}
\sum\limits_{n=1}^{m}\left(f\left(n\right)\prod\limits_{j=n}^{m-1}\left(\vartheta\left(j\right)\right)\right)=\mathcal{B}_{k=q}^{r}\left(f\right)\left(m\right)
\end{equation}
where $\mathcal{B}_{k=q}^{r}$ broadens $f$ completely on $\left[1,m-1\right]_\omega$
\\
We may show this inductively. For $m=2$ we have the original function broadened $f\left(1\right)$ times about $1$ and we get precisely $\vartheta\left(1\right)f\left(1\right)+f\left(2\right)$ for the resulting output for $2$. Now suppose that this holds for some $m$. Then we have a completely broadened function on $\left[1,m-1\right]_\omega$ and for $m$ we have 
\begin{align*}
\sum_{n=1}^{m}\left(f\left(n\right)\prod_{j=n}^{m-1}\left(\vartheta\left(j\right)\right)\right)
\end{align*}
Thus, by broadening this all about $m$ we get 
\begin{align*}
\vartheta\left(m\right)\sum_{n=1}^{m}\left(f\left(n\right)\prod_{j=n}^{m-1}\left(\vartheta\left(j\right)\right)\right)
\end{align*} taken over to the value of $m+1$. Therefore, the resulting value at $m+1$ is 
\begin{align*}
\vartheta\left(m\right)\sum_{n=1}^{m}\left(f\left(n\right)\prod_{j=n}^{m-1}\left(\vartheta\left(j\right)\right)\right)+f\left(m+1\right)=\sum_{n=1}^{m+1}\left(f\left(n\right)\prod_{j=n}^{m}\left(\vartheta\left(j\right)\right)\right)
\end{align*} Hence, we have proven this inductively.
\\
Thus, by the \nameref{Shifting Theorem} \ref{Shifting Theorem}, the function resulting from completely broadening $f$ on $\left[1,m-1\right]_\omega$ is in the same real number and has the same auxiliary function as $f$, call this function $g$.
\\
Now, take the function resulting from broadening $h_f$ completely on $\left[1,m-1\right]_\omega$, denote this sequence $\mathcal{B}_{k=q}^r$. That function has for the value at $m$
\begin{align*}
\sum\limits_{n=1}^{m}\left(h_f\left(n\right)\prod\limits_{j=n}^{m-1}\left(\vartheta\left(j\right)\right)\right)
\end{align*}
Now take the inverse of the sequence of the broadening which we define for any $\mathcal{B}_{k=q}^r$ as $\mathcal{C}_{k=q}^r\in\mathcal{CB}$ where
\begin{align*}
\mathcal{C}\left(k\right)=\begin{cases}
\begin{array}{llll@{}}
c_n&\text{where}&\mathcal{B}\left(r+q-k\right)=b_n&k\in\left[q,r\right]_\omega\\
c_1&&&otherwise
\end{array}
\end{cases}
\end{align*}
Apply it to the function $\mathcal{B}_{k=q}^r\left(h_f\right)$ (we may apply it by induction on the \nameref{Cancellation Theorems} \ref{Cancellation Theorems}). By the \nameref{Submission Theorem} \ref{Submission Theorem}, as for all $n$ in the interval $\left[0,m\right]_\omega$ we have $\mathcal{B}_{k=q}^r\left(h_f\right)\left(n\right)\leq g\left(n\right)$ this sequence of contractions is applicable to $g$. Furthermore, by the \nameref{Shifting Theorem} \ref{Shifting Theorem} the resulting function $\mathcal{C}_{k=q}^r\left(g\right)$ has the same auxiliary function as $g$, $h_f$.\\
However, we have 
\begin{align*}
\begin{array}{ll@{}}
\forall n\in\left[0,m-1\right]_\omega,h_f\left(n\right)=\mathcal{C}_{k=q}^r\left(\mathcal{B}_{k=q}^r\left(h_f\right)\right)\left(n\right)=\mathcal{C}_{k=q}^r\left(g\right)\left(n\right)&\text{and}\\
h_f\left(m\right)+1\leq\mathcal{C}_{k=q}^r\left(g\right)\left(m\right)
\end{array}
\end{align*} Thus taking as the interval of inspection $\left[0,m\right]_\omega$ we obtain that $\mathcal{C}_{k=q}^r\left(g\right)$ may not be contracted to $h_f$, by the \nameref{Consequences Theorem} \ref{Consequences Theorem}, which is a contradiction to $h_f$ being an auxiliary function of $\mathcal{C}_{k=q}^r\left(g\right)$ and thus $g$ and thus $f$ as, by the \nameref{Shifting Theorem} \ref{Shifting Theorem} all of these share the same auxiliary function.\\
Therefore, 
\begin{align*}
\forall f\in \mathcal{N}_{\mathcal{F}in},\forall m\in\omega_+,\sum\limits_{n=1}^{m}\left(f\left(n\right)\prod\limits_{j=n}^{m-1}\left(\vartheta\left(j\right)\right)\right)\leq\sum\limits_{n=1}^{m}\left(h_f\left(n\right)\prod\limits_{j=n}^{m-1}\left(\vartheta\left(j\right)\right)\right)
\end{align*}
\end{proof}
\begin{corollary}[Paradise City Corollary I]\label{PC Cor I} We see that for all primary auxiliary functions $f_p$ and secondary auxiliary functions $f_s$ (if there are any) of a real number $\left[f\right]$ we have that for all $m\in\left[1,\Psi\left(f_p,f_s\right)-1\right]_\omega$
\begin{align*}
\sum\limits_{n=1}^{m}\left(f_s\left(n\right)\prod\limits_{j=n}^{m-1}\left(\vartheta\left(j\right)\right)\right)=\sum\limits_{n=1}^{m}\left(f_p\left(n\right)\prod\limits_{j=n}^{m-1}\left(\vartheta\left(j\right)\right)\right)
\end{align*}
and for all $m\in\left[\Psi\left(f_p,f_s\right),\infty\right]_\omega$
\begin{align*}
&\sum\limits_{n=1}^{m}\left(f_s\left(n\right)\prod\limits_{j=n}^{m-1}\left(\vartheta\left(j\right)\right)\right)\\
=&\sum\limits_{n=1}^{\Psi\left(f_p,f_s\right)}\left(f_s\left(n\right)\prod\limits_{j=n}^{m-1}\left(\vartheta\left(j\right)\right)\right)+\hspace{-0.5cm}\sum\limits_{n=\Psi\left(f_p,f_s\right)+1}^{m}\left(f_s\left(n\right)\prod\limits_{j=n}^{m-1}\left(\vartheta\left(j\right)\right)\right)\\
=&\sum\limits_{n=1}^{\Psi\left(f_p,f_s\right)}\left(f_p\left(n\right)\prod\limits_{j=n}^{m-1}\left(\vartheta\left(j\right)\right)\right)-\hspace{-0.45cm}\prod\limits_{j=\Psi\left(f_p,f_s\right)}^{m-1}\left(\vartheta\left(j\right)\right)+\hspace{-0.5cm}\sum\limits_{n=\Psi\left(f_p,f_s\right)+1}^{m}\left(\left(\vartheta\left(n-1\right)-1\right)\prod\limits_{j=n}^{m-1}\left(\vartheta\left(j\right)\right)\right)\\
=&\sum\limits_{n=1}^{\Psi\left(f_p,f_s\right)}\left(f_p\left(n\right)\prod\limits_{j=n}^{m-1}\left(\vartheta\left(j\right)\right)\right)-\hspace{-0.45cm}\prod\limits_{j=\Psi\left(f_p,f_s\right)}^{m-1}\left(\vartheta\left(j\right)\right)+\hspace{-0.5cm}\prod\limits_{j=\Psi\left(f_p,f_s\right)}^{m-1}\left(\vartheta\left(j\right)\right)-1\\
=&\sum\limits_{n=1}^{m}\left(f_p\left(n\right)\prod\limits_{j=n}^{m-1}\left(\vartheta\left(j\right)\right)\right)-1
\end{align*}
\end{corollary}
\begin{corollary}[Paradise City Corollary II]\label{PC Cor II} For all primary auxiliary functions $\pi_1$ and $\pi_2$ of $\left[f\right]$ we have that $\pi_1\left(n\right)=\pi_2\left(n\right)$ for all $n\in\omega_+$ . Thus for all $m$ in $\omega_+$ we have 
\begin{align*}
\sum\limits_{n=1}^{m}\left(\pi_1\left(n\right)\prod\limits_{j=n}^{m-1}\left(\vartheta\left(j\right)\right)\right)=\sum\limits_{n=1}^{m}\left(\pi_2\left(n\right)\prod\limits_{j=n}^{m-1}\left(\vartheta\left(j\right)\right)\right)
\end{align*}
 For all secondary auxiliary functions $\delta_1$ and $\delta_2$ of $\left[f\right]$ we have that $\delta_1\left(n\right)=\delta_2\left(n\right)$ for all $n\in\omega_+$ . Thus for all $m$ in $\omega_+$ we have 
\begin{align*}
\sum\limits_{n=1}^{m}\left(\delta_1\left(n\right)\prod\limits_{j=n}^{m-1}\left(\vartheta\left(j\right)\right)\right)=\sum\limits_{n=1}^{m}\left(\delta_2\left(n\right)\prod\limits_{j=n}^{m-1}\left(\vartheta\left(j\right)\right)\right)
\end{align*}
\end{corollary}
\begin{corollary}[Paradise City Corollary III]\label{Maximum Corollary} Combined all together we have that for all $m\in\omega_+$
\begin{align*}
\max\left(\left\{f\left(m\right)\middle|f\in\left[x\right]\right\}\right)=\sum\limits_{n=1}^{m}\left(\pi\left(n\right)\prod\limits_{j=n}^{m-1}\left(\vartheta\left(j\right)\right)\right)
\end{align*}
where $\pi$ is any primary auxiliary function of $\left[x\right]$.
\end{corollary}
\begin{example} Return to the two names of $\nicefrac{1}{2}$ in base ten, where
$\vartheta\left(n\right)=10$ for every $n$ and position $n$ carries weight
$\nicefrac{1}{10^{n-1}}$. Write them again as
\begin{align*}
f_p=\left(0;0,5,0,0,0,\dots\right)\qquad\text{and}\qquad f_s=\left(0;0,4,9,9,9,\dots\right),
\end{align*}
so that $f_p$ is primary, $f_s$ is secondary, and $\Psi\left(f_p,f_s\right)=2$.\\
The Paradise City Lemma and its corollaries are statements about truncations rather than about digits, so it is worth seeing the truncations themselves. The first few levels of inspection read
\begin{align*}
\begin{array}{r|ccccc}
m&1&2&3&4&5\\\hline
\prod_{j=1}^{m-1}\left(\vartheta\left(j\right)\right)&1&10&100&1000&10000\\
\sum_{n=1}^{m}\left(f_p\left(n\right)\prod_{j=n}^{m-1}\left(\vartheta\left(j\right)\right)\right)&0&5&50&500&5000\\
\sum_{n=1}^{m}\left(f_s\left(n\right)\prod_{j=n}^{m-1}\left(\vartheta\left(j\right)\right)\right)&0&4&49&499&4999
\end{array}
\end{align*}
Below $\Psi\left(f_p,f_s\right)=2$ the two truncations agree; from $\Psi\left(f_p,f_s\right)$ onwards the secondary sits exactly one unit beneath the primary, at every level and forever. That is precisely what \nameref{PC Cor I} \ref{PC Cor I} asserts, and it is worth noting how little room the statement leaves: not that the gap is small, or that it shrinks, but that it is the constant $1$.\\
Dividing the level-$m$ truncation by $\prod_{j=1}^{m-1}\left(\vartheta\left(j\right)\right)$ returns the value being approximated, and recovers the familiar picture: the primary reads $0,\;0.5,\;0.5,\;0.5,\dots$ and the secondary reads $0,\;0.4,\;0.49,\;0.499,\dots$, the second climbing towards the first by one unit of the current scale at each step without ever arriving. The two functions name the same real number, and a single unit of scale is the whole of the difference between the two names.
\end{example}
\subsection{Domination Theorem}
\begin{theorem}[Domination Theorem]\label{Domination Theorem}\noindent\\
For any real number $[x]$ there do not exist two functions $f$ and $g$ in the real number $[x]$ such that $g$ strictly dominates $f$.
\end{theorem}

\begin{intuition}[Domination Theorem] There may not be any two functions representing the same real number where one of them is larger than the other.
\end{intuition}

\begin{proof} Suppose for the sake of contradiction, that for some $[x]\in\mathcal{S}_\mathbb{R}$ there exist $f\in\left[x\right]$ and $g\in\left[x\right]$ such that $g$ strictly dominates $f$. We know by the definition of $\left[x\right]$ that 
\begin{align*}
\forall N\in\omega_+,\exists\mathcal{C}_{k=q}^r\in\mathcal{FS},\forall n\in\left[0,N\right]_\omega,\mathcal{C}_{k=q}^r\left(f\right)\left(n\right)=h_f\left(n\right)
\end{align*} where $h_f\in b\in B$.
\\
Now as we suppose that for all $n\in\omega_+$ we have $f\left(n\right)\leq g\left(n\right)$ any sequence of contractions applicable to $f$ is also applicable to $g$, by the \nameref{Submission Theorem} \ref{Submission Theorem}. Take $N=M+1$ where $M:=\Psi\left(f,g\right)$. We know that there exists a sequence of contractions mapping $f$ to $h_f$ on the interval $\left[0,N\right]_\omega$, call the sequence $\mathcal{C}_{k=q}^r$.
\\
By the \nameref{Shifting Theorem} \ref{Shifting Theorem} $\mathcal{C}_{k=q}^r\left(g\right)\in\left[g\right]=\left[x\right]=\left[f\right]$. By the \nameref{Submission Theorem} \ref{Submission Theorem} we have
\begin{align*}
\begin{array}{lll@{}}
\forall n\in\left[1,N\right]_\omega&\mathcal{C}_{k=q}^r\left(g\right)\left(n\right)\geq\mathcal{C}_{k=q}^r\left(f\right)\left(n\right)=h_f\left(n\right)&\text{and}\\
&\mathcal{C}_{k=q}^r\left(g\right)\left(M\right)\geq\mathcal{C}_{k=q}^r\left(f\right)\left(M\right)+1=h_f\left(M\right)+1
\end{array}
\end{align*}
\footnote{This is due to the result that $\mathcal{C}_{k=q}^r\left(f\right)\left(M\right)+g\left(M\right)=f\left(M\right)+\mathcal{C}_{k=q}^r\left(g\right)\left(M\right)$. Which means that $\mathcal{C}_{k=q}^r\left(g\right)\left(M\right)-\mathcal{C}_{k=q}^r\left(f\right)\left(M\right)=g\left(M\right)-f\left(M\right)>0$ and thus $\mathcal{C}_{k=q}^r\left(g\right)\left(M\right)>\mathcal{C}_{k=q}^r\left(f\right)\left(M\right)$.}Therefore, by the \nameref{Paradise City Lemma} \ref{Paradise City Lemma} we have that
\begin{align*}
\sum\limits_{n=1}^N\left(h_g\left(n\right)\prod\limits_{j=n}^{N-1}\left(\vartheta\left(j\right)\right)\right)&\geq\sum\limits_{n=1}^N\left(\mathcal{C}_{k=q}^r\left(g\right)\left(n\right)\prod\limits_{j=n}^{N-1}\left(\vartheta\left(j\right)\right)\right)\\
&\geq\sum\limits_{n=1}^N\left(h_f\left(n\right)\prod\limits_{j=n}^{N-1}\left(\vartheta\left(j\right)\right)\right)+\vartheta\left(M\right)\\
&\geq\sum\limits_{n=1}^N\left(h_f\left(n\right)\prod\limits_{j=n}^{N-1}\left(\vartheta\left(j\right)\right)\right)+2
\end{align*}
where $h_g$ is an auxiliary function of $g$ and $\mathcal{C}_{k=q}^r\left(g\right)$. Hence, by the \nameref{PC Cor I} \ref{PC Cor I} and \nameref{PC Cor II} \ref{PC Cor II} we have that $h_f$ and $h_g$ are not in the same real number, and thus neither are $f$ and $g$.
\end{proof}

\begin{corollary}[Domination Theorem Corollary]\label{Domination Theorem Corollary} By expanding the \nameref{Domination Theorem} \ref{Domination Theorem} out and applying it twice we obtain that for all $[x]\in\mathcal{S}_{\mathbb{R}}$, for all $f,g\in\left[x\right]$ we have that either
\begin{align*}
\begin{array}{ll@{}}
\left(\exists n_1\in\omega_+,f\left(n_1\right)>g\left(n_1\right)\wedge\exists n_2\in\omega_+,g\left(n_2\right)>f\left(n_2\right)\right)&\text{or}\\
\left(\forall M\in\omega_+,g\left(M\right)=f\left(M\right)\right)
\end{array}
\end{align*}
\end{corollary}
\subsection{Tying Theorem}
\begin{theorem}[Tying Theorem]\label{Tying Theorem}\noindent\\
A function $f$ is finite if and only if there exists a constant $C\in\omega_+$ such that 
\begin{align*}
\sum\limits_{n=1}^K\left(f\left(n\right)\prod\limits_{p=n}^{K-1}\left(\vartheta\left(p\right)\right)\right)<C\prod\limits_{p=1}^{K-1}\left(\vartheta\left(p\right)\right)
\end{align*} is true for any $K$ in $\omega_+$.
\end{theorem}

\begin{intuition}[Tying Theorem] Any function representing a real number must have a finite magnitude and therefore its magnitude has to be less than some natural number. On the other hand, if we know that a magnitude of a function is smaller than some natural number then this function can define a real number.
\end{intuition}

\begin{proofidea} Forwards, the \nameref{Paradise City Lemma} \ref{Paradise City Lemma} bounds the truncations of $f$ by those of its primary auxiliary function, and telescoping the maximal-digit tail bounds those in turn by $\left(h_f\left(1\right)+1\right)\prod_{p=1}^{K-1}\vartheta\left(p\right)$, which is the required constant. Backwards, the bound is first converted into the statement that no sequence of contractions can push the value at position $1$ up to $C$; the auxiliary function is then obtained position by position, and the identity that contractions preserve weighted truncations above their condition set is what guarantees the construction is consistent and does not stall.
\end{proofidea}

\begin{proof}\noindent\\
$\left(\implies\right):$ Suppose $f\in \mathcal{N}_{\mathcal{F}in}$. By the \nameref{Paradise City Lemma} \ref{Paradise City Lemma} we have that for all $K$ in $\omega_+$
\begin{align*}
\sum\limits_{n=1}^{K}\left(f\left(n\right)\prod\limits_{j=n}^{K-1}\left(\vartheta\left(j\right)\right)\right)&\leq\sum\limits_{n=1}^K\left(h_f\left(n\right)\prod\limits_{p=n}^{K-1}\left(\vartheta\left(p\right)\right)\right)\\
&=h_f\left(1\right)\prod\limits_{p=1}^{K-1}\left(\vartheta\left(p\right)\right)+\sum\limits_{n=2}^K\left(h_f\left(n\right)\prod\limits_{p=n}^{K-1}\left(\vartheta\left(p\right)\right)\right)\\
&\leq h_f\left(1\right)\prod\limits_{p=1}^{K-1}\left(\vartheta\left(p\right)\right)+\sum\limits_{n=2}^K\left(\left(\vartheta\left(n-1\right)-1\right)\prod\limits_{p=n}^{K-1}\left(\vartheta\left(p\right)\right)\right)\\
&=h_f\left(1\right)\prod\limits_{p=1}^{K-1}\left(\vartheta\left(p\right)\right)+\sum\limits_{n=1}^{K-1}\left(\prod\limits_{p=n}^{K-1}\left(\vartheta\left(p\right)\right)\right)-\sum\limits_{n=2}^{K}\left(\prod\limits_{p=n}^{K-1}\left(\vartheta\left(p\right)\right)\right)\\
&=h_f\left(1\right)\prod\limits_{p=1}^{K-1}\left(\vartheta\left(p\right)\right)+\prod\limits_{p=1}^{K-1}\left(\vartheta\left(p\right)\right)-1\\
&<\left(h_f\left(1\right)+1\right)\prod\limits_{p=1}^{K-1}\left(\vartheta\left(p\right)\right)
\end{align*}
By taking $C=h_f\left(1\right)+1$ we have proven that
\begin{align*}
f\in \mathcal{N}_{\mathcal{F}in}\implies\exists C\in\omega_+,\forall K\in\omega_+,\sum\limits_{n=1}^K\left(f\left(n\right)\prod\limits_{p=n}^{K-1}\left(\vartheta\left(p\right)\right)\right)<C\prod\limits_{p=1}^{K-1}\left(\vartheta\left(p\right)\right)
\end{align*}
$\left(\impliedby\right):$ Suppose that there exists a constant $C$ in $\omega_+$ such that for any $K$ in $\omega_+$ we have $\sum\limits_{n=1}^K\left(f\left(n\right)\prod\limits_{p=n}^{K-1}\left(\vartheta\left(p\right)\right)\right)<C\prod\limits_{p=1}^{K-1}\left(\vartheta\left(p\right)\right)$. This means that $\forall\mathcal{C}_{k=q}^r\in\mathcal{FS},\mathcal{C}_{k=q}^r\left(f\right)\left(1\right)<C$. We shall prove this by first showing that for all $f:\omega\rightarrow\omega$ and for all sequences of contractions $\mathcal{C}_{k=q}^r$ applicable to $f$ and for all $K$ in $\omega_+$ such that $K\geq\max\left(\text{con}\left(\mathcal{C}_{k=q}^r\right)\right)$ we have 
\begin{equation}\label{eq:contraction-invariance}
\sum\limits_{n=1}^{K}\left(f\left(n\right)\prod\limits_{j=n}^{K-1}\left(\vartheta\left(j\right)\right)\right)=\sum\limits_{n=1}^{K}\left(\mathcal{C}_{k=q}^r\left(f\right)\left(n\right)\prod\limits_{j=n}^{K-1}\left(\vartheta\left(j\right)\right)\right)
\end{equation}
To prove this we have to prove one result first. For any function $f:\omega\rightarrow\omega$ and for any sequence of contractions $\mathcal{C}_{k=q}^r$ applicable to $f$ we have that $\mathcal{C}_{k=q}^r\left(f\right)={}^o\mathcal{C}_{k=q}^r\left(f\right)$ where ${}^o\mathcal{C}_{k=q}^r$ is a special sequence of contractions. This is the sequence where we first apply the contractions about the largest values and then continuing in a decreasing order. Formally, let 
\begin{align*}
\begin{array}{lll@{}}
\kappa&:=\left|\left\{\mathcal{C}\left(k\right)\middle|k\in\left[q,r\right]_\omega\right\}\right|\\
t_i&:=\max\left(\left\{m\in\omega_+\middle|c_m\in\left\{\mathcal{C}\left(k\right)\middle|k\in\left[q,r\right]_\omega\right\}-\left\{c_{t_n}\middle|n\in\left[1,i-1\right]_\omega\right\}\right\}\right)&\text{for all $i$ in $\left[1,\kappa\right]_\omega$}\\
v_i&:=\left|\left\{k\in\left[q,r\right]_\omega\middle|\mathcal{C}\left(k\right)=c_{t_i}\right\}\right|&\text{for all $i$ in $\left[1,\kappa\right]_\omega$}
\end{array}
\end{align*}
We define ${}^o\mathcal{C}\in\mathcal{CB}$ as
\begin{align*}
{}^o\mathcal{C}\left(k\right)=\begin{cases}
c_{t_i}&\text{if $k\in\left[q+\sum\limits_{j=1}^{i-1}\left(v_j\right),q+\sum\limits_{j=1}^{i}\left(v_j\right)-1\right]_\omega$ for some $i\in\left[1,\kappa\right]_\omega$}\\
c_1&\text{otherwise}
\end{cases}
\end{align*}
We show this by adjusting the identity proven for the \nameref{Shifting Theorem} \ref{Shifting Theorem} we obtain that for all $f:\omega\rightarrow\omega$ and for all $n$ and $m$ in $\omega_+$ such that $n<m$ we have that $c_m\left(c_n\left(f\right)\right)$ exists implies that $c_n\left(c_m\left(f\right)\right)$ exists and moreover $c_m\left(c_n\left(f\right)\right)=c_n\left(c_m\left(f\right)\right)$.
\\
As $n<m$ we have that $f\left(m+1\right)=c_n\left(f\right)\left(m+1\right)\geq\vartheta\left(m\right)$. At the same time as $n<m$ we have that $c_m\left(f\right)\left(n+1\right)\geq f\left(n+1\right)\geq\vartheta\left(n\right)$. Therefore, the conditions for the existence of $c_n\left(c_m\left(f\right)\right)$ are satisfied. By \eqref{eq:commute} we have that $c_m\left(c_n\left(f\right)\right)=c_n\left(c_m\left(f\right)\right)$. By continual repetition of this identity, we obtain that $\mathcal{C}_{k=q}^r\left(f\right)={}^o\mathcal{C}_{k=q}^r\left(f\right)$.
\\
To prove $\left(*\right)$ we suppose that it is false for the sake of contradiction. We take a case with the minimal $\max\left(\text{con}\left(\mathcal{C}_{k=q}^r\right)\right)$ (as this value is in $\omega_+$ we are assured that a minimal case exists, however, it may not be unique). For simplicity we introduce the notation $\mathcal{M}\left(\mathcal{C}_{k=q}^r\right):=\max\left(\text{con}\left(\mathcal{C}_{k=q}^r\right)\right)$. We observe that
\begin{align*}
&&\sum\limits_{n=1}^{K}\left(\mathcal{C}_{k=q}^r\left(f\right)\left(n\right)\prod\limits_{j=n}^{K-1}\left(\vartheta\left(j\right)\right)\right)&\neq&&\sum\limits_{n=1}^{K}\left(f\left(n\right)\prod\limits_{j=n}^{K-1}\left(\vartheta\left(j\right)\right)\right)\\
\iff&&\sum\limits_{n=1}^{\mathcal{M}\left(\mathcal{C}_{k=q}^r\right)}\left(\mathcal{C}_{k=q}^r\left(f\right)\left(n\right)\hspace{-0.35cm}\prod\limits_{j=n}^{\mathcal{M}\left(\mathcal{C}_{k=q}^r\right)-1}\hspace{-0.35cm}\left(\vartheta\left(j\right)\right)\right)&\neq&&\hspace{-0.35cm}\sum\limits_{n=1}^{\mathcal{M}\left(\mathcal{C}_{k=q}^r\right)}\left(f\left(n\right)\hspace{-0.35cm}\prod\limits_{j=n}^{\mathcal{M}\left(\mathcal{C}_{k=q}^r\right)-1}\hspace{-0.35cm}\left(\vartheta\left(j\right)\right)\right)
\end{align*}
As, by the \nameref{Conditional Lemma} \ref{Conditional Lemma} for all values $n$ greater than $\mathcal{M}\left(\mathcal{C}_{k=q}^r\right)$ we have $\mathcal{C}_{k=q}^r\left(f\right)\left(n\right)=f\left(n\right)$. Now take ${}^o\mathcal{C}_{k=q}^r$ as defined above for the particular $\mathcal{C}_{k=q}^r$ we are dealing with. This tells us that
{\small
\begin{align*}
&\sum\limits_{n=1}^{\mathcal{M}\left(\mathcal{C}_{k=q}^r\right)}\left(\mathcal{C}_{k=q}^r\left(f\right)\left(n\right)\hspace{-0.35cm}\prod\limits_{j=n}^{\mathcal{M}\left(\mathcal{C}_{k=q}^r\right)-1}\hspace{-0.35cm}\left(\vartheta\left(j\right)\right)\right)=\hspace{-0.35cm}\sum\limits_{n=1}^{\mathcal{M}\left(\mathcal{C}_{k=q}^r\right)}\left({}^o\mathcal{C}_{k=q}^r\left(f\right)\left(n\right)\hspace{-0.35cm}\prod\limits_{j=n}^{\mathcal{M}\left(\mathcal{C}_{k=q}^r\right)-1}\hspace{-0.35cm}\left(\vartheta\left(j\right)\right)\right)\\
=&\hspace{-0.35cm}\sum\limits_{n=1}^{\mathcal{M}\left(\mathcal{C}_{k=q}^r\right)-2}\left({}^o\mathcal{C}_{k=q}^r\left(f\right)\left(n\right)\hspace{-0.35cm}\prod\limits_{j=n}^{\mathcal{M}\left(\mathcal{C}_{k=q}^r\right)-1}\hspace{-0.35cm}\left(\vartheta\left(j\right)\right)\right)\\
&+\vartheta\left(\mathcal{M}\left(\mathcal{C}_{k=q}^{r}\right)-1\right){}^o\mathcal{C}_{k=q}^r\left(f\right)\left(\mathcal{M}\left(\mathcal{C}_{k=q}^r\right)-1\right)+{}^o\mathcal{C}_{k=q}^r\left(f\right)\left(\mathcal{M}\left(\mathcal{C}_{k=q}^r\right)\right)\\
=&\hspace{-0.35cm}\sum\limits_{n=1}^{\mathcal{M}\left(\mathcal{C}_{k=q}^r\right)-2}\left({}^o\mathcal{C}_{k=q+v_1}^{r}\left(f\right)\left(n\right)\hspace{-0.35cm}\prod\limits_{j=n}^{\mathcal{M}\left(\mathcal{C}_{k=q}^r\right)-1}\hspace{-0.35cm}\left(\vartheta\left(j\right)\right)\right)\\
&+\vartheta\left(\mathcal{M}\left(\mathcal{C}_{k=q}^{r}\right)-1\right)\left({}^o\mathcal{C}_{k=q+v_1}^{r}\left(f\right)\left(\mathcal{M}\left(\mathcal{C}_{k=q}^r\right)-1\right)+v_1\right)+{}^o\mathcal{C}_{k=q+v_1}^{r}\left(f\right)\left(\mathcal{M}\left(\mathcal{C}_{k=q}^r\right)\right)\\
&-v_1\vartheta\left(\mathcal{M}\left(\mathcal{C}_{k=q}^{r}\right)-1\right)\\
=&\hspace{-0.35cm}\sum\limits_{n=1}^{\mathcal{M}\left(\mathcal{C}_{k=q}^r\right)-2}\left({}^o\mathcal{C}_{k=q+v_1}^{r}\left(f\right)\left(n\right)\hspace{-0.35cm}\prod\limits_{j=n}^{\mathcal{M}\left(\mathcal{C}_{k=q}^r\right)-2}\hspace{-0.35cm}\left(\vartheta\left(j\right)\right)\right)\\
&+\vartheta\left(\mathcal{M}\left(\mathcal{C}_{k=q}^{r}\right)-1\right){}^o\mathcal{C}_{k=q+v_1}^{r}\left(f\right)\left(\mathcal{M}\left(\mathcal{C}_{k=q}^r\right)-1\right)+{}^o\mathcal{C}_{k=q+v_1}^{r}\left(f\right)\left(\mathcal{M}\left(\mathcal{C}_{k=q}^r\right)\right)\\
=&\hspace{-0.35cm}\sum\limits_{n=1}^{\mathcal{M}\left(\mathcal{C}_{k=q}^r\right)}\left({}^o\mathcal{C}_{k=q+v_1}^{r}\left(f\right)\left(n\right)\hspace{-0.35cm}\prod\limits_{j=n}^{\mathcal{M}\left(\mathcal{C}_{k=q}^r\right)-1}\hspace{-0.35cm}\left(\vartheta\left(j\right)\right)\right)=\hspace{-0.35cm}\sum\limits_{n=1}^{\mathcal{M}\left(\mathcal{C}_{k=q}^r\right)}\left(f\left(n\right)\hspace{-0.35cm}\prod\limits_{j=n}^{\mathcal{M}\left(\mathcal{C}_{k=q}^{r}\right)-1}\hspace{-0.35cm}\left(\vartheta\left(j\right)\right)\right)
\end{align*}
}%
Here we simply applied the first $v_1$-many contractions, which are by construction all about $t_1=\mathcal{M}\left(\mathcal{C}_{k=q}^r\right)-1$ and the fact that as $\mathcal{M}\left({}^o\mathcal{C}_{k=q+v_1}^{r}\right)<\mathcal{M}\left(\mathcal{C}_{k=q}^r\right)$ (and therefore by our presumption the last equality holds). This shows us that there is no case contradicting \eqref{eq:contraction-invariance} with the minimal value of $\mathcal{M}\left(\mathcal{C}_{k=q}^r\right)$. Therefore, there is no such case at all.
\\
Hence $\forall\mathcal{C}_{k=q}^r\in\mathcal{FS},\mathcal{C}_{k=q}^r\left(f\right)\left(1\right)<C$, as if there was any sequence of contractions $\mathcal{C}_{k=q}^r$ contradicting this, then we see that
\begin{align*}
\sum\limits_{n=1}^{\mathcal{M}\left(\mathcal{C}_{k=q}^r\right)}\left(f\left(n\right)\hspace{-0.35cm}\prod\limits_{j=n}^{\mathcal{M}\left(\mathcal{C}_{k=q}^{r}\right)-1}\hspace{-0.35cm}\left(\vartheta\left(j\right)\right)\right)&=\hspace{-0.35cm}\sum\limits_{n=1}^{\mathcal{M}\left(\mathcal{C}_{k=q}^r\right)}\left(\mathcal{C}_{k=q}^{r}\left(f\right)\left(n\right)\hspace{-0.35cm}\prod\limits_{j=n}^{\mathcal{M}\left(\mathcal{C}_{k=q}^{r}\right)-1}\hspace{-0.35cm}\left(\vartheta\left(j\right)\right)\right)\\
&\geq\mathcal{C}_{k=q}^{r}\left(f\right)\left(1\right)\hspace{-0.35cm}\prod\limits_{j=n}^{\mathcal{M}\left(\mathcal{C}_{k=q}^{r}\right)-1}\hspace{-0.35cm}\left(\vartheta\left(j\right)\right)\\
&>C\hspace{-0.35cm}\prod\limits_{j=n}^{\mathcal{M}\left(\mathcal{C}_{k=q}^{r}\right)-1}\hspace{-0.35cm}\left(\vartheta\left(j\right)\right)
\end{align*}
Which directly contradicts our original presumption. Thus there exists maximal such value, call it $D_1:=\max\left(\left\{\mathcal{C}_{k=q}^r\left(f\right)\left(1\right)\middle|\mathcal{C}_{k=q}^r\in\mathcal{FS}\right\}\right)$.
\\
Now, for all $\mathcal{C}_{k=q}^r\in\mathcal{FS}$ we have
\begin{align*}
    \mathcal{C}_{k=q}^r\left(f\right)\left(1\right)=D_1\implies\mathcal{C}_{k=q}^r\left(f\right)\left(2\right)\leq\vartheta\left(1\right)-1
\end{align*} as otherwise, we could contract to obtain a value greater than $D_1$. Pick the greatest such value and call it $D_2$.
\\
Following inductively, 
\begin{align*}
D_n:=\max\left(\left\{t\in\omega\middle|\exists\mathcal{C}_{k=q}^r\in\mathcal{FS},\begin{array}{ll@{}}&\forall k\in\left[1,n-1\right]_\omega,\mathcal{C}_{k=q}^r\left(f\right)\left(k\right)=D_k\\\wedge&\mathcal{C}_{k=q}^r\left(f\right)\left(n\right)=t\end{array}\right\}\right)
\end{align*}
$D_n$ is well-defined by induction. The case for $D_1$ has been shown. If for $n\geq1$, $D_n$ is well-defined, then $\exists\mathcal{C}_{k=q}^r\in\mathcal{FS},\forall k\in\left[1,n\right]_\omega,\mathcal{C}_{k=q}^r\left(f\right)\left(k\right)=D_k$. This shows that the set we define $D_{n+1}$ from is not empty. We also have that for all $\mathcal{C}_{k=q}^r\in\mathcal{FS}$ we have 
\begin{align*}
    \forall k\in\left[1,n\right]_\omega,\mathcal{C}_{k=q}^r\left(f\right)\left(k\right)=D_k\implies\mathcal{C}_{k=q}^r\left(f\right)\left(n+1\right)\leq\vartheta\left(n\right)-1
\end{align*} as otherwise, we could contract to obtain a value greater than $D_n$. Therefore, the defining set is non-empty, bounded set of finite ordinals. Therefore, it has a maximum in $\omega$ and thus the value $D_{n+1}$ is well-defined.\\
As a consequence we have $\forall n\in\left[2,\infty\right]_\omega,D_n\leq\vartheta\left(n-1\right)-1$. Therefore, the function $h$ where $h\left(0\right)=f\left(0\right)$ and $\forall n\in\omega_+,h\left(n\right)=D_n$ is an auxiliary function and by its definition $\forall N\in\omega_+,\exists\mathcal{C}_{k=q}^r\in\mathcal{FS},\forall n\in\left[0,N\right]_\omega,\mathcal{C}_{k=q}^r\left(f\right)\left(n\right)=h\left(n\right)$ and therefore $f\in \mathcal{N}_{\mathcal{F}in}$.
\end{proof}

We have constructed the set which we are going to show is a model of the real numbers. But in order to do that we have to define order, addition and multiplication on this set. That will be precisely the intent of the next chapter. However, before we do that, we would like to point out what is the particular strength of this model of analysis. One set in our reals, one of our real numbers, consists of many representations. In the upcoming chapter we will see that there are results which suffice to be shown for one representation and will globally apply to all.
\\
This adjustability gives a lot of power to the working mathematician who wishes to use this meta-system of creating models. Precisely because they may create a base and thus a model fitting their particular problem and then they may choose to work with the most convenient set of representations while knowing that there is nothing to be worried about in regards of rigour or well-definedness.
\section{Order}\label{sec:order}
In order to show that we have actually constructed the real numbers we show that we have a complete ordered field. Hence, we are to define order, addition and multiplication on our set of the numbers. Afterwards we are to prove all the standard axioms. We might be going in an un-traditional order as to build everything properly.
\begin{definition}[Order]\noindent\\ 
We shall define a total order on our set. We construct a binary relation $\leq_{\mathcal{S}_\mathbb{R}}$. There are $3$ types of evaluation of this relation, both variables being non-zero, one being zero, and both being zero. Firstly we define for all $\left[a\right],\left[b\right]\in\mathcal{S}_{\mathbb{R}}^*$ the relation $\left[a\right]\leq_{\mathcal{S}_\mathbb{R}}\left[b\right]$ if and only if:
\begin{align*}
\begin{array}{ll@{}}
\sigma\left(\left[a\right]\right)<\sigma\left(\left[b\right]\right)&\vee\vspace{0.35cm}\\
\left(\sigma\left(\left[a\right]\right)=\sigma\left(\left[b\right]\right)=1\wedge\forall a\in\left[a\right],\exists b\in\left[b\right],\forall n\in\omega_+,a\left(n\right)\leq b\left(n\right)\right)&\vee\vspace{0.35cm}\\
\left(\sigma\left(\left[a\right]\right)=\sigma\left(\left[b\right]\right)=0\wedge\forall b\in\left[b\right],\exists a\in\left[a\right],\forall n\in\omega_+,a\left(n\right)\geq b\left(n\right)\right)
\end{array}
\end{align*}
We furthermore define for all $\left[a\right]\in\mathcal{S}_{\mathbb{R}}^*$
\begin{align*}
\begin{array}{lll@{}}
\left[a\right]\leq_{\mathcal{S}_\mathbb{R}}\left[0\right]&\text{if and only if}&\sigma\left(\left[a\right]\right)=0\vspace{0.35cm}\\
\left[0\right]\leq_{\mathcal{S}_\mathbb{R}}\left[a\right]&\text{if and only if}&\sigma\left(\left[a\right]\right)=1
\end{array}
\end{align*}
And finally, we define $\left[0\right]\leq_{\mathcal{S}_\mathbb{R}}\left[0\right]$ as true.
\end{definition}

\begin{remark} For the sake of brevity, we shall stop using $\leq_{\mathcal{S}_\mathbb{R}}$ and instead we will use simply $\leq$. As we are using $\left[\cdot\right]$ to denote the real numbers, no confusion should arise about when we use the order for the naturals and when for the reals. As of now, we do not claim that this definition is an order, or total, this is all to be shown.
\end{remark}

\begin{axiom}[Reflexivity of Order]\label{Reflexivity of Order}\noindent\\
For all real numbers $\left[a\right]$ we have that $\left[a\right]\leq\left[a\right]$. Formally
\begin{align*}
\forall\left[a\right]\in\mathcal{S}_\mathbb{R},\left[a\right]\leq\left[a\right]
\end{align*}
\end{axiom}
\begin{proof} By definition, if $\left[a\right]=\left[0\right]$ then by the definition of order $\left[a\right]\leq\left[a\right]$.
\\
If $\left[a\right]\in\mathcal{S}_\mathbb{R}^*$ then we see that $\sigma\left(\left[a\right]\right)=\sigma\left(\left[a\right]\right)$. Let $a\in\left[a\right]$ be given. Then we see that there exists a $b\in\left[a\right]$ (where $b=a$) such that
\begin{align*}
\forall n\in\omega_+,a\left(n\right)\leq b\left(n\right)
\end{align*}
thus $\left[a\right]\leq\left[a\right]$.
\end{proof}

\begin{axiom}[Anti-Symmetry of Order]\label{Anti-Symmetry of Order}\noindent\\
For all real numbers $\left[a\right]$ and $\left[b\right]$ we have that $\left[a\right]\leq\left[b\right]$  and $\left[b\right]\leq\left[a\right]$ implies that $\left[a\right]=\left[b\right]$.
\end{axiom}
\begin{proof} Suppose both $\left[a\right]\leq\left[b\right]$ and $\left[b\right]\leq\left[a\right]$ are true, then $\left[a\right]=\left[b\right]=\left[0\right]$ or $\sigma\left(\left[a\right]\right)=\sigma\left(\left[b\right]\right)$.
\\
This is due to the fact that if we have without loss of generality $\left[a\right]=\left[0\right]$ then we cannot have both $\sigma\left(\left[b\right]\right)=0$ and $\sigma\left(\left[b\right]\right)=1$. Therefore, the only way for both $\left[a\right]\leq\left[b\right]$ and $\left[b\right]\leq\left[a\right]$ to be true is if $\left[b\right]=\left[0\right]$.
\\
On the other hand, if both $\left[a\right]$ and $\left[b\right]$ are not $\left[0\right]$ and they differ in sign we have a similar problem. Take without loss of generality $\sigma\left(\left[a\right]\right)=1$ and $\sigma\left(\left[b\right]\right)=0$. Thus, we cannot have $\left[a\right]\leq\left[b\right]$.
\\
Suppose $\sigma\left(\left[a\right]\right)=\sigma\left(\left[b\right]\right)=1$. From what we have supposed we have 
\begin{align*}
\begin{array}{lll@{}}
\forall a\in\left[a\right],\exists b\in\left[b\right],\forall n\in\omega_+,a\left(n\right)\leq b\left(n\right)&\text{and}\\
\forall b\in\left[b\right],\exists a\in\left[a\right],\forall n\in\omega_+,a\left(n\right)\leq b\left(n\right)
\end{array}
\end{align*}
This implies that 
\begin{align*}
\forall a_1\in\left[a\right],\exists b\in\left[b\right],\exists a_2\in\left[a\right],\forall n\in\omega_+,a_1\left(n\right)\leq b\left(n\right)\leq a_2\left(n\right)
\end{align*}
Now, as $a_1,a_2\in\left[a\right]$ and $\forall n\in\omega_+,a_1\left(n\right)\leq a_2\left(n\right)$ we must have $\forall n\in\omega_+,a_1\left(n\right)=a_2\left(n\right)$, by the \nameref{Domination Theorem Corollary} \ref{Domination Theorem Corollary}.\\
Therefore, $\forall n\in\omega_+, a_1\left(n\right)=b\left(n\right)=a_2\left(n\right)$. By the \nameref{Equivalence Lemma} \ref{Equivalence Lemma} as $\sigma\left(a_1\right)=\sigma\left(b\right)$ we have $a_1\left(0\right)+b\left(0\right)=2m$ for some $m\in\omega$ and therefore $b\in\left[a_1\right]=\left[a\right]$. Thus, by the \nameref{Unicity of Real Numbers} \ref{Unicity of Real Numbers}, $\left[a\right]=\left[b\right]$.
\\
The case $\sigma\left(\left[a\right]\right)=\sigma\left(\left[b\right]\right)=0$ is proven in the exact same manner.
\\
Ergo $\left[a\right]\leq\left[b\right]\wedge\left[b\right]\leq\left[a\right]\implies\left[a\right]=\left[b\right]$.
\end{proof}
\begin{corollary}[ Corollary of the Anti-Symmetry of Order]\label{Corollary of the Anti-Symmetry of Order} By the contrapositive of the property above
$$\neg\left(\left[a\right]=\left[b\right]\right)\implies\neg\left(\left[a\right]\leq\left[b\right]\right)\vee\neg\left(\left[b\right]\leq\left[a\right]\right)$$
Thus, supposing $\left[a\right]$ and $\left[b\right]$ distinct real numbers and $\left[a\right]\leq\left[b\right]$ we must have $\neg\left(\left[b\right]\leq\left[a\right]\right)$.
\end{corollary}
\begin{axiom}[Transitivity of Order]\noindent\\
For all real numbers $\left[a\right],\left[b\right]$ and $\left[c\right]$ if $\left[a\right]\leq\left[b\right]$ and $\left[b\right]\leq\left[c\right]$ then $\left[a\right]\leq\left[c\right]$.
\end{axiom}
\begin{proofidea} Nine sign configurations are enumerated. Wherever a sign gap appears anywhere in the chain the comparison is decided by the sign clause alone, and those cases are dispatched in a line each. The substantive work is confined to the configurations in which all three numbers carry the same sign, and there it is the per-position form of the order that does the chaining: it allows the witness for the first comparison and the witness for the second to be different functions, which is what makes the two dominations composable at all.
\end{proofidea}
\begin{proof} Let $\left[a\right],\left[b\right],\left[c\right]\in\mathcal{S}_\mathbb{R}^*$ such that $\left[a\right]\leq\left[b\right]$ and $\left[b\right]\leq\left[c\right]$ are given. We have 9 cases to go through, which are as follows:
\begin{align}
\sigma\left(\left[a\right]\right)<\sigma\left(\left[b\right]\right)&\wedge\sigma\left(\left[b\right]\right)<\sigma\left(\left[c\right]\right)\\
\sigma\left(\left[a\right]\right)<\sigma\left(\left[b\right]\right)&\wedge\sigma\left(\left[b\right]\right)=\sigma\left(\left[c\right]\right)=1\\
\sigma\left(\left[a\right]\right)<\sigma\left(\left[b\right]\right)&\wedge\sigma\left(\left[b\right]\right)=\sigma\left(\left[c\right]\right)=0\\
\sigma\left(\left[a\right]\right)=\sigma\left(\left[b\right]\right)=1&\wedge\sigma\left(\left[b\right]\right)<\sigma\left(\left[c\right]\right)\\
\sigma\left(\left[a\right]\right)=\sigma\left(\left[b\right]\right)=1&\wedge\sigma\left(\left[b\right]\right)=\sigma\left(\left[c\right]\right)=1\\
\sigma\left(\left[a\right]\right)=\sigma\left(\left[b\right]\right)=1&\wedge\sigma\left(\left[b\right]\right)=\sigma\left(\left[c\right]\right)=0\\
\sigma\left(\left[a\right]\right)=\sigma\left(\left[b\right]\right)=0&\wedge\sigma\left(\left[b\right]\right)<\sigma\left(\left[c\right]\right)\\
\sigma\left(\left[a\right]\right)=\sigma\left(\left[b\right]\right)=0&\wedge\sigma\left(\left[b\right]\right)=\sigma\left(\left[c\right]\right)=1\\
\sigma\left(\left[a\right]\right)=\sigma\left(\left[b\right]\right)=0&\wedge\sigma\left(\left[b\right]\right)=\sigma\left(\left[c\right]\right)=0
\end{align}
Observe, that by the definition of $\sigma$, the cases (1),(3),(4),(6),(8) are impossible. Hence, we shall prove the remaining ones.
\paragraph{(2)} We have:
\begin{align*}
\sigma\left(\left[a\right]\right)<\sigma\left(\left[b\right]\right)\wedge\sigma\left(\left[b\right]\right)=\sigma\left(\left[c\right]\right)=1\implies\sigma\left(\left[a\right]\right)<\sigma\left(\left[c\right]\right)\implies\left[a\right]\leq\left[c\right]
\end{align*}
\paragraph{(5)} We have:
\begin{align*}
\sigma\left(\left[a\right]\right)=\sigma\left(\left[b\right]\right)=1\wedge\sigma\left(\left[b\right]\right)=\sigma\left(\left[c\right]\right)=1\implies\sigma\left(\left[a\right]\right)=\sigma\left(\left[c\right]\right)=1
\end{align*}
Additionally:
\begin{align*}
\begin{array}{lcl@{}}
\forall a\in\left[a\right],\exists b\in\left[b\right],\forall n\in\omega_+,a\left(n\right)\leq b\left(n\right)\text{ and }\forall b\in\left[b\right],\exists c\in\left[c\right],\forall n\in\omega_+,b\left(n\right)\leq c\left(n\right)
\vspace{0.35cm}\\
\implies\forall a\in\left[a\right],\exists c\in\left[c\right],\forall n\in\omega_+,a\left(n\right)\leq c\left(n\right)\implies\left[a\right]\leq\left[c\right]
\end{array}
\end{align*}
\paragraph{(7)} We have:
\begin{align*}
\sigma\left(\left[a\right]\right)=\sigma\left(\left[b\right]\right)=0\wedge\sigma\left(\left[b\right]\right)<\sigma\left(\left[c\right]\right)\implies\sigma\left(\left[a\right]\right)<\sigma\left(\left[c\right]\right)\implies\left[a\right]\leq\left[c\right]
\end{align*}
\paragraph{(9)} We have:
\begin{align*}
\sigma\left(\left[a\right]\right)=\sigma\left(\left[b\right]\right)=0\wedge\sigma\left(\left[b\right]\right)=\sigma\left(\left[c\right]\right)=0\implies\sigma\left(\left[a\right]\right)=\sigma\left(\left[c\right]\right)=0
\end{align*}
Additionally:
\begin{align*}
\begin{array}{lcl@{}}
\forall b\in\left[b\right],\exists a\in\left[a\right],\forall n\in\omega_+,a\left(n\right)\geq b\left(n\right)\text{ and }\forall c\in\left[c\right],\exists b\in\left[b\right],\forall n\in\omega_+,b\left(n\right)\geq c\left(n\right)
\vspace{0.35cm}\\
\implies\forall c\in\left[c\right],\exists a\in\left[a\right],\forall n\in\omega_+,a\left(n\right)\geq c\left(n\right)\implies\left[a\right]\leq\left[c\right]
\end{array}
\end{align*}
We also must consider the three additional cases brought to us by comparing two non-zero real numbers with $\left[0\right]$.
\paragraph{Case 1:} $\left[0\right]\leq\left[a\right]$ and $\left[a\right]\leq\left[b\right]$. Therefore:
\begin{align*}
\sigma\left(\left[a\right]\right)=1\implies\sigma\left(\left[b\right]\right)=1\implies\left[0\right]\leq\left[b\right]
\end{align*}
\paragraph{Case 2:} $\left[a\right]\leq\left[0\right]$ and $\left[0\right]\leq\left[b\right]$. Therefore:
\begin{align*}
\sigma\left(\left[a\right]\right)=0\wedge\sigma\left(\left[b\right]\right)=1\implies\left[a\right]\leq\left[b\right]
\end{align*}
\paragraph{Case 3:} $\left[a\right]<\left[b\right]$ and $\left[b\right]<\left[0\right]$. Therefore:
\begin{align*}
\sigma\left(\left[b\right]\right)=0\implies\sigma\left(\left[a\right]\right)=0\implies\left[a\right]\leq\left[0\right]
\end{align*}
All the possible cases show that for any three distinct real numbers $\left[a\right],\left[b\right],\left[c\right]$ we have $\left[a\right]\leq\left[b\right]\wedge\left[b\right]\leq\left[c\right]\implies\left[a\right]\leq\left[c\right]$.\footnote{Note that for the non-distinct cases we just invoke substitutivity of equality.}
\end{proof}
\begin{theorem}[Equivalent to the Definition of Order]\label{Equivalent to the Definition of Order}\noindent\\
For all $\left[a\right],\left[b\right]\in\mathcal{S}_\mathbb{R}^*$ we have $\left[a\right]\leq\left[b\right]$ if and only if
\begin{align*}
\begin{array}{ll@{}}
\sigma\left(\left[a\right]\right)<\sigma\left(\left[b\right]\right)&\vee\vspace{0.35cm}\\
\left(\sigma\left(\left[a\right]\right)=\sigma\left(\left[b\right]\right)=1\wedge\forall a\in\left[a\right],\forall n\in\omega_+,\exists b\in\left[b\right],a\left(n\right)\leq b\left(n\right)\right)&\vee\vspace{0.35cm}\\
\left(\sigma\left(\left[a\right]\right)=\sigma\left(\left[b\right]\right)=0\wedge\forall b\in\left[b\right],\forall n\in\omega_+,\exists a\in\left[a\right],a\left(n\right)\geq b\left(n\right)\right)
\end{array}
\end{align*}
Also, the same way as before, we say that for all $\left[a\right]\in\mathcal{S}_{\mathbb{R}}^*$
\begin{align*}
\begin{array}{lll@{}}
\left[a\right]\leq_{\mathcal{S}_\mathbb{R}}\left[0\right]&\text{if and only if}&\sigma\left(\left[a\right]\right)=0\vspace{0.35cm}\\
\left[0\right]\leq_{\mathcal{S}_\mathbb{R}}\left[a\right]&\text{if and only if}&\sigma\left(\left[a\right]\right)=1
\end{array}
\end{align*}
And finally, we once again take $\left[0\right]\leq_{\mathcal{S}_\mathbb{R}}\left[0\right]$ as true.
\end{theorem}

\begin{intuition}[Equivalent to the Definition of Order] The definition of order asks, for a fixed writing of $\left[a\right]$, for a \emph{single} writing of $\left[b\right]$ that beats it at every position at once. That is awkward to verify. This theorem says we may instead check the same thing one position at a time: at each $n$ it is enough to find \emph{some} representative of $\left[b\right]$ that is at least $\left[a\right]$ there, and the witness is allowed to change from position to position. The two formulations agree, and the per-position version is the one we can actually work with.
\end{intuition}
\begin{proofidea} One direction is immediate. The other must turn infinitely many position-by-position witnesses into a single representative dominating everywhere at once, and the whole difficulty is manufacturing that representative. We build it explicitly: truncate the primary auxiliary function of $\left[b\right]$ at a suitable position and continue with a maximal, $\left(\vartheta-1\right)$-valued tail. That tail is what supplies the surplus needed at every later position, and the Submission Theorem then transports the required sequence of moves onto it. The two cases correspond to whether the witness for $\left[a\right]$ is its primary or its secondary auxiliary function.
\end{proofidea}
\begin{proof}\noindent\\
\textbf{First}, we wish to show that for all non-zero real numbers $\left[a\right]$ and $\left[b\right]$ if $\sigma\left(\left[a\right]\right)=\sigma\left(\left[b\right]\right)$ then
$$\left(\forall a\in\left[a\right],\forall n\in\omega_+,\exists b\in\left[b\right],a\left(n\right)\leq b\left(n\right)\right)\implies\left(\forall a\in\left[a\right],\exists b\in\left[b\right],\forall n\in\omega_+,a\left(n\right)\leq b\left(n\right)\right)$$
We see that if $\left[a\right]=\left[b\right]$ then the implication above is true. Going forward we suppose that $\left[a\right]\neq\left[b\right]$.
\\
For definiteness we shall establish a convention, for any real number $\left[x\right]$ we have the functions $x_p\in\left[x\right]\cap A_p$ and $x_s\in\left[x\right]\cap A_s$ (if there are secondary auxiliary functions in the real number $\left[x\right]$) such that 
\begin{align*}
\begin{array}{llllll@{}}
x_p\left(0\right)=x_s\left(0\right)=0&\text{if}&\sigma\left(\left[x\right]\right)=1&\text{or}&\left[x\right]=\left[0\right]&\text{and}\\
x_p\left(0\right)=x_s\left(0\right)=1&\text{if}&\sigma\left(\left[x\right]\right)=0
\end{array}
\end{align*} 
As $\left[a\right]\neq\left[b\right]$, by the \nameref{Unicity of Real Numbers} \ref{Unicity of Real Numbers} and \nameref{Equivalence Lemma} \ref{Equivalence Lemma} as $\sigma\left(\left[a\right]\right)=\sigma\left(\left[b\right]\right)$ we have $\Psi\left(a_p,b_p\right)\neq0$. Furthermore, by our presumption for all $n$ in $\omega_+$ we have $\max\left(\left\{f\left(n\right)\middle|f\in\left[a\right]\right\}\right)\leq\max\left(\left\{g\left(n\right)\middle|g\in\left[b\right]\right\}\right)$. Thus, by the \nameref{Maximum Corollary} \ref{Maximum Corollary} we have that for all $n\in\omega_+$
\begin{align*}
\sum\limits_{k=1}^{n}\left(a_p\left(k\right)\prod\limits_{j=k}^{n-1}\left(\vartheta\left(j\right)\right)\right)&=\max\left(\left\{f\left(n\right)\middle|f\in\left[a\right]\right\}\right)\\
&\leq\max\left(\left\{g\left(n\right)\middle|g\in\left[b\right]\right\}\right)=\sum\limits_{k=1}^{n}\left(b_p\left(k\right)\prod\limits_{j=k}^{n-1}\left(\vartheta\left(j\right)\right)\right)
\end{align*}
Thus we have that
\begin{align*}
\begin{array}{ll@{}}
\forall n\in\left[1,\Psi\left(a_p,b_p\right)-1\right]_\omega,\sum\limits_{k=1}^{n}\left(a_p\left(k\right)\prod\limits_{j=k}^{n-1}\left(\vartheta\left(j\right)\right)\right)=\sum\limits_{k=1}^{n}\left(b_p\left(k\right)\prod\limits_{j=k}^{n-1}\left(\vartheta\left(j\right)\right)\right)\indent&\text{and}\\
\sum\limits_{k=1}^{\Psi\left(a_p,b_p\right)}\left(a_p\left(k\right)\hspace{-0.35cm}\prod\limits_{j=k}^{\Psi\left(a_p,b_p\right)-1}\hspace{-0.35cm}\left(\vartheta\left(j\right)\right)\right)\leq\hspace{-0.35cm}\sum\limits_{k=1}^{\Psi\left(a_p,b_p\right)}\left(b_p\left(k\right)\hspace{-0.35cm}\prod\limits_{j=k}^{\Psi\left(a_p,b_p\right)-1}\hspace{-0.35cm}\left(\vartheta\left(j\right)\right)\right)&\text{and}\\
a_p\left(\Psi\left(a_p,b_p\right)\right)\neq b_p\left(\Psi\left(a_p,b_p\right)\right)
\end{array}
\end{align*}
Therefore, $a_p\left(\Psi\left(a_p,b_p\right)\right)<b_p\left(\Psi\left(a_p,b_p\right)\right)$. Define $\xi:\omega\rightarrow\omega$ such that 
\begin{align*}
\xi\left(k\right):=
\begin{cases}
b_p\left(0\right)&k=0\\
b_p\left(k\right)&k\in\left[1,\Psi\left(a_p,b_p\right)-1\right]_\omega\\
b_p\left(\Psi\left(a_p,b_p\right)\right)-1&k=\Psi\left(a_p,b_p\right)\\
b_p\left(k\right)+\vartheta\left(k-1\right)-1&k\in\left[\Psi\left(a_p,b_p\right)+1,\infty\right]_\omega 
\end{cases}
\end{align*}
We shall show that $\xi\in\left[b\right]$. We show this by cases.
\\
\textbf{Case 1. $\exists H\in\omega_+,\forall m\in\left[H+1,\infty\right]_\omega,b_p\left(m\right)=0$}
\\
Denote the minimal such $H$ as $K$. Then $K\geq\Psi\left(a_p,b_p\right)$.\footnote{As we may not have $a_p\left(\Psi\left(a_p,b_p\right)\right)<b_p\left(\Psi\left(a_p,b_p\right)\right)=0$} Furthermore, we see that if $K=\Psi\left(a_p,b_p\right)$ then $\xi=b_s$ and if $\Psi\left(a_p,b_p\right)<K$ then $\xi$ is contractable about $K-1$, as $\xi\left(K\right)\geq\vartheta\left(K-1\right)$. Then if $\Psi\left(a_p,b_p\right)=K-1$ we have $c_{K-1}\left(\xi\right)=b_s$. If $\Psi\left(a_p,b_p\right)<K-1$ then $c_{K-1}\left(\xi\right)$ is contractable about $K-2$. We continue this way and inductively show that by taking $\mathcal{C}\in\mathcal{CB}$ such that 
\begin{align*}
\mathcal{C}\left(k\right)=
\begin{cases}
c_{K-k}&\forall k\in\left[1,K-\Psi\left(a_p,b_p\right)\right]\\
c_1&otherwise
\end{cases}
\end{align*}
we have $\mathcal{C}_{k=1}^{K-\Psi\left(a_p,b_p\right)}\left(\xi\right)=b_s$. Therefore, by the \nameref{Shifting Theorem} \ref{Shifting Theorem} $\xi\in\left[b\right]$.
\\
\textbf{Case 2. $\forall H\in\omega_+,\exists m\in\left[H+1,\infty\right]_\omega,b_p\left(m\right)\geq1$}
\\ 
We let a number of inspection $N\in\omega_+$ be given. If $N<\Psi\left(a_p,b_p\right)$ then for all $n\in\left[0,N\right]_\omega$ we have $\xi\left(n\right)=b_p\left(n\right)$. If $N\geq\Psi\left(a_p,b_p\right)$ then there exists $m\in\left[N+1,\infty\right]_\omega,b_p\left(m\right)\geq1$ and thus by using the same construction of the sequence of contractions as above we obtain that for all $n$ in $\left[0,N\right]_\omega$ we have $\mathcal{C}_{k=1}^{m-\Psi\left(a_p,b_p\right)}\left(\xi\right)\left(n\right)=b_p\left(n\right)$. Therefore, by the \nameref{Shifting Theorem} \ref{Shifting Theorem} $\xi\in\left[b\right]$.
\\
We see that 
\begin{align*}
\begin{cases}
a_p\left(k\right)=b_p\left(k\right)=\xi\left(k\right)&k\in\left[0,\Psi\left(a_p,b_p\right)-1\right]_\omega\\
a_p\left(\Psi\left(a_p,b_p\right)\right)\leq b_p\left(\Psi\left(a_p,b_p\right)\right)-1=\xi\left(\Psi\left(a_p,b_p\right)\right)&k=\Psi\left(a_p,b_p\right)\\
a_p\left(k\right)\leq\vartheta\left(k-1\right)-1\leq\xi\left(k\right)&k\in\left[\Psi\left(a_p,b_p\right)+1,\infty\right]_\omega
\end{cases}
\end{align*}
Hence, for all $n$ in $\omega$ we have $a_p\left(n\right)\leq\xi\left(n\right)$. Thus $\xi$ dominates $a_p$. Define for all $a\in\left[a\right]$ the function $\zeta_a:\omega\rightarrow\omega$
\begin{align*}
\zeta_a\left(k\right):=
\begin{cases}
b_p\left(0\right)&k=0\\
a\left(k\right)+\xi\left(k\right)-a_p\left(k\right)&otherwise
\end{cases}
\end{align*}
Notice that $\zeta_a$ dominates $a$ as for all $k$ in $\omega_+$ we have $a\left(k\right)\leq\zeta_a\left(k\right)$. We shall show that $\zeta_a\in\left[b\right]$.
\\
As $\xi\in\left[b\right]$ by definition there exists an auxiliary function $h_\xi$ in $\left[b\right]$ such that for all numbers of inspection $N$ in $\omega_+$ there exists a sequence of contractions ${}^\xi\mathcal{C}_{k=q'}^{r'}\in\mathcal{FS}$ such that for all $n$ in the interval of inspection $\left[0,N\right]_\omega$ we have ${}^\xi\mathcal{C}_{k=q'}^{r'}\left(\xi\right)\left(n\right)=h_\xi\left(n\right)$.
\\
Denote $\Xi:=\min\left(\left\{t\in\left[\max\left(\text{con}\left({}^\xi\mathcal{C}_{k=q'}^{r'}\right)\cup\left\{N\right\}\right),\infty\right]_\omega\middle|\xi\left(t\right)\geq a_p\left(t\right)+1\right\}\right)$\footnote{Note that we have an infinity of cases where $\xi\left(t\right)\geq a_p\left(t\right)+1$ as $a_p$ is a primary auxiliary function and for $n\in\left[\Psi\left(a_p,b_p\right)+1,\infty\right]_\omega$ we have $\xi\left(n\right)\geq\vartheta\left(n-1\right)-1$}. There are two cases we need to investigate. For all $n$ in $\omega_+$ we have $h_a\left(n\right)=a_p\left(n\right)$. And for all $n$ in $\omega_+$ we have $h_a\left(n\right)=a_s\left(n\right)$.\footnote{In some cases, there are no secondary auxiliary functions. Also notice that here we only talk about $\omega_+$, as we do not know if $a\left(0\right)=a_p\left(0\right)$, however, by the construction of the base set we know that the auxiliary function of $a$ has to agree with either of those functions on $\omega_+$.}
\\
\textbf{Case 1.} We know that for all numbers of inspection $K$ in $\omega_+$ there exists a sequence of contractions $\mathcal{C}_{k=q}^r\in\mathcal{FS}$ such that for all $n$ in the interval of inspection $\left[0,K\right]_\omega$ we have $\mathcal{C}_{k=q}^r\left(a\right)\left(n\right)=h_a\left(n\right)$.\\
Let a number of inspection $N\in\omega_+$ be given, we take $K=\Xi$. Then as $\zeta_a$ dominates $a$, by the \nameref{Submission Theorem} \ref{Submission Theorem}, $\mathcal{C}_{k=q}^r$ is applicable to $\zeta_a$ and furthermore for all $n$ in $\left[1,\Xi\right]_\omega$ we have
\begin{align*}
\mathcal{C}_{k=q}^r\left(\zeta_a\right)\left(n\right)+a\left(n\right)&=\zeta_a\left(n\right)+\mathcal{C}_{k=q}^r\left(a\right)\left(n\right)=a\left(n\right)+\xi\left(n\right)-a_p\left(n\right)+h_a\left(n\right)\\
&=a\left(n\right)+\xi\left(n\right)-a_p\left(n\right)+a_p\left(n\right)=\xi\left(n\right)+a\left(n\right)
\end{align*}
which implies that for all $n$ in $\left[1,\Xi\right]_\omega$ we have $\mathcal{C}_{k=q}^r\left(\zeta_a\right)\left(n\right)=\xi\left(n\right)$. Denote $\delta:=\mathcal{C}_{k=q}^r\left(\zeta_a\right)$.
\\
By the \nameref{Submission Theorem} \ref{Submission Theorem} ${}^\xi\mathcal{C}_{k=q'}^{r'}$ is applicable to $\delta$ and furthermore for all $n$ in $\left[1,N\right]_\omega$ we have 
\begin{align*}
{}^\xi\mathcal{C}_{k=q'}^{r'}\left(\delta\right)\left(n\right)+\xi\left(n\right)=\delta\left(n\right)+{}^\xi\mathcal{C}_{k=q'}^{r'}\left(\xi\right)\left(n\right)=\xi\left(n\right)+h_\xi\left(n\right)
\end{align*}
This implies that for all $n$ in $\left[1,N\right]_\omega$ we have ${}^\xi\mathcal{C}_{k=q'}^{r'}\left(\delta\right)\left(n\right)=h_\xi\left(n\right)$. Thus, as $\zeta_a\left(0\right)=\xi\left(0\right)=h_\xi\left(0\right)$ we have that for all $n$ in $\left[0,N\right]_\omega$ we have ${}^\xi\mathcal{C}_{k=q'}^{r'}\left(\mathcal{C}_{k=q}^r\left(\zeta_a\right)\right)\left(n\right)=h_\xi\left(n\right)$.
\\
Therefore, by taking ${}^{\zeta_a}\mathcal{C}\in\mathcal{CB}$ such that
\begin{align*}
\begin{array}{ll@{}}
{}^{\zeta_a}\mathcal{C}\left(k\right)=
\begin{cases}\mathcal{C}\left(k\right)&\text{if $k\in\left[q,r\right]_\omega$}\\
{}^\xi\mathcal{C}\left(k+q'-r-1\right)&\text{if $k\in\left[r+1,r'-q'+r+1\right]_\omega$}\\
c_1&otherwise
\end{cases}
\end{array}
\end{align*}
Thus by construction we know that for all $n$ in $\left[0,N\right]_\omega$ we have ${}^{\zeta_a}\mathcal{C}_{k=q}^{r'-q'+r+1}\left(\zeta_a\right)\left(n\right)=h_\xi\left(n\right)$. Therefore $\zeta_a\in\left[\xi\right]=\left[b\right]$.
\\
\textbf{Case 2.} We know that for all numbers of inspection $K$ in $\omega_+$ there exists a sequence of contractions $\mathcal{C}_{k=q}^r\in\mathcal{FS}$ such that for all $n$ in the interval of inspection $\left[0,K\right]_\omega$ we have $\mathcal{C}_{k=q}^r\left(a\right)\left(n\right)=h_a\left(n\right)$.
\\
Let $N\in\omega_+$ be given. Take $K=\Xi$. Then as $\zeta_a$ dominates $a$, by the \nameref{Submission Theorem} \ref{Submission Theorem}, $\mathcal{C}_{k=q}^r$ is applicable to $\zeta_a$ and furthermore for all $n$ in $\left[1,\Xi\right]_\omega$ we have 
\begin{align*}
\mathcal{C}_{k=q}^r\left(\zeta_a\right)\left(n\right)+a\left(n\right)&=\zeta_a\left(n\right)+\mathcal{C}_{k=q}^r\left(a\right)\left(n\right)=a\left(n\right)+\xi\left(n\right)-a_p\left(n\right)+h_a\left(n\right)\\
&=a\left(n\right)+\xi\left(n\right)-a_p\left(n\right)+a_s\left(n\right)=\xi\left(n\right)+a\left(n\right)-a_p\left(n\right)+a_s\left(n\right)
\end{align*}
This implies that for all $n$ in $\left[1,\Xi\right]_\omega$ we have $\mathcal{C}_{k=q}^r\left(\zeta_a\right)\left(n\right)=\xi\left(n\right)+a_s\left(n\right)-a_p\left(n\right)$. Denote $\lambda=\mathcal{C}_{k=q}^r\left(\zeta_a\right)$.
\\
We know that 
\begin{align*}
\begin{array}{llll@{}}
a_p\left(n\right)=a_s\left(n\right)&\text{if $n\in\left[1,N_{\min}^{a_s}-1\right]_\omega$}\\
a_p\left(N_{\min}^{a_s}\right)=a_s\left(N_{\min}^{a_s}\right)+1&\text{if $n=N_{\min}^{a_s}$}\\
a_p\left(n\right)=0\text{ and }a_s\left(n\right)=\vartheta\left(n-1\right)-1&\text{if $n\in\left[N_{\min}^{a_s}+1,\infty\right]_\omega$}
\end{array}
\end{align*}
If $\Xi<N_{\min}^{a_s}$ we have the same case as \textbf{Case 1}. If $\Xi\geq N_{\min}^{a_s}$ then we do the following.\\
Define $w:\omega\rightarrow\omega$ where 
\begin{align*}
w\left(k\right):=
\begin{cases}
a_s\left(\Xi\right)+1&k=\Xi\\
a_s\left(k\right)&otherwise
\end{cases}
\end{align*}
Take $\mathcal{C'}\in\mathcal{CB}$ defined as
\begin{align*}
\mathcal{C'}\left(k\right)=\begin{cases}
c_{\Xi+r-k}&k\in\left[r+1,\Xi-N_{\min}^{a_s}+r\right]_\omega\\
c_1&otherwise
\end{cases}
\end{align*}
We then see that $\mathcal{C'}_{k=r+1}^{\Xi-N_{\min}^{a_s}+r}$ is applicable to $w$ and furthermore for all $n$ in $\left[1,\Xi\right]_\omega$ we have $\mathcal{C'}_{k=r+1}^{\Xi-N_{\min}^{a_s}+r}\left(w\right)\left(n\right)=a_p\left(n\right)$.\\
By the \nameref{Submission Theorem} \ref{Submission Theorem}, as $\lambda$ dominates $w$ on $\left[1,\Xi\right]_\omega$ because by construction we have $\xi\left(n\right)\geq a_p\left(n\right)$ and thus
\begin{align*}
\begin{array}{ll@{}}
     \lambda\left(n\right)=\xi\left(n\right)-a_p\left(n\right)+a_s\left(n\right)\geq a_s\left(n\right)+1=w\left(n\right)&n=\Xi\\
     \lambda\left(n\right)=\xi\left(n\right)-a_p\left(n\right)+a_s\left(n\right)\geq a_s\left(n\right)=w\left(n\right)&n\neq \Xi
\end{array}
\end{align*}, 
$\mathcal{C'}_{k=r+1}^{\Xi-N_{\min}^{a_s}+r}$ is applicable to $\lambda$ and furthermore for all $n$ in $\left[1,\Xi\right]_\omega$ we have 
\begin{align*}
\mathcal{C'}_{k=r+1}^{\Xi-N_{\min}^{a_s}+r}\left(\lambda\right)\left(n\right)+w\left(n\right)&=\mathcal{C'}_{k=r+1}^{\Xi-N_{\min}^{a_s}+r}\left(\lambda\right)\left(n\right)+a_s\left(n\right)=\lambda\left(n\right)+\mathcal{C'}_{k=r+1}^{\Xi-N_{\min}^{a_s}+r}\left(w\right)\left(n\right)\\
&=\xi\left(n\right)+a_s\left(n\right)-a_p\left(n\right)+a_p\left(n\right)=\xi\left(n\right)+a_s\left(n\right)
\end{align*}
This implies that for all $n$ in $\left[1,\Xi\right]_\omega$ we have $\mathcal{C'}_{k=r+1}^{\Xi-N_{\min}^{a_s}+r}\left(\lambda\right)\left(n\right)=\xi\left(n\right)$. Denote $\delta:=\mathcal{C'}_{k=r+1}^{\Xi-N_{\min}^{a_s}+r}\left(\lambda\right)$.
\\
By the \nameref{Submission Theorem} \ref{Submission Theorem} ${}^\xi\mathcal{C}_{k=q'}^{r'}$ is applicable to $\delta$ and furthermore for all $n$ in $\left[1,N\right]_\omega$ we have 
\begin{align*}
{}^\xi\mathcal{C}_{k=q'}^{r'}\left(\delta\right)\left(n\right)+\xi\left(n\right)=\delta\left(n\right)+{}^\xi\mathcal{C}_{k=q'}^{r'}\left(\xi\right)\left(n\right)=\xi\left(n\right)+h_\xi\left(n\right)
\end{align*}
This implies that for all $n$ in $\left[1,N\right]_\omega$ we have ${}^\xi\mathcal{C}_{k=q'}^{r'}\left(\delta\right)\left(n\right)=h_\xi\left(n\right)$. Thus, as $\zeta_a\left(0\right)=\xi\left(0\right)=h_\xi\left(0\right)$ we have that for all $n$ in $\left[0,N\right]_\omega$ we have ${}^\xi\mathcal{C}_{k=q'}^{r'}\left(\mathcal{C'}_{k=r+1}^{\Xi-N_{\min}^{a_s}+r}\left(\mathcal{C}_{k=q}^r\left(\zeta_a\right)\right)\right)\left(n\right)=h_\xi\left(n\right)$. Therefore, by taking ${}^{\zeta_a}\mathcal{C}\in\mathcal{CB}$ such that
\begin{align*}
\hspace{-0.5cm}{}^{\zeta_a}\mathcal{C}\left(k\right)=
\begin{cases}
\mathcal{C}\left(k\right)&k\in\left[q,r\right]_\omega\\
c_{\Xi+r-k}&k\in\left[r+1,\Xi-N_{\min}^{a_s}+r\right]_\omega\\
{}^\xi\mathcal{C}\left(k+q'-\Xi+N_{\min}^{a_s}+r-1\right)&k\in\left[\Xi-N_{\min}^{a_s}+r+1,r'-q'+\Xi-N_{\min}^{a_s}+r+1\right]_\omega\\
c_1&otherwise
\end{cases}
\end{align*}
Thus by construction we know that for all $n$ in $\left[0,N\right]_\omega$ we have ${}^{\zeta_a}\mathcal{C}_{k=q}^{r'-q'+\Xi-N_{\min}^{a_s}+r+1}\left(\zeta_a\right)\left(n\right)=h_\xi\left(n\right)$. Therefore $\zeta_a\in\left[\xi\right]=\left[b\right]$.\\\\
Also, it follows from logic that 
$$\left(\forall a\in\left[a\right],\exists b\in\left[b\right],\forall n\in\omega_+,a\left(n\right)\leq b\left(n\right)\right)\implies\left(\forall a\in\left[a\right],\forall n\in\omega_+,\exists b\in\left[b\right],a\left(n\right)\leq b\left(n\right)\right)$$
Ergo:
$$\left(\forall a\in\left[a\right],\forall n\in\omega_+,\exists b\in\left[b\right],a\left(n\right)\leq b\left(n\right)\right)\iff\left(\forall a\in\left[a\right],\exists b\in\left[b\right],\forall n\in\omega_+,a\left(n\right)\leq b\left(n\right)\right)$$
Second, by relabeling in the result above we obtain
\begin{align*}
\left(\forall b\in\left[b\right],\forall n\in\omega_+,\exists a\in\left[a\right],b\left(n\right)\leq a\left(n\right)\right)\iff\left(\forall b\in\left[b\right],\exists a\in\left[a\right],\forall n\in\omega_+,b\left(n\right)\leq a\left(n\right)\right)
\end{align*}
Which is the same as 
\begin{align*}
\left(\forall b\in\left[b\right],\forall n\in\omega_+,\exists a\in\left[a\right],a\left(n\right)\geq b\left(n\right)\right)\iff\left(\forall b\in\left[b\right],\exists a\in\left[a\right],\forall n\in\omega_+,a\left(n\right)\geq b\left(n\right)\right)
\end{align*}
Thus this finishes the proof of the desired result.
\end{proof}

\begin{axiom}[Totality of Order]\label{Totality of Order}\noindent\\
For any real numbers $\left[a\right]$ and $\left[b\right]$ either $\left[a\right]\leq\left[b\right]$ or $\left[b\right]\leq\left[a\right]$.
\end{axiom}
\begin{proofidea} Identical numbers are immediate, so the content is that two distinct numbers cannot be incomparable. Rather than compare them directly we assume $\neg\left(\left[a\right]\leq\left[b\right]\vee\left[b\right]\leq\left[a\right]\right)$ and negate the per-position form of the order supplied by the previous theorem. That turns the assumption into a conjunction of concrete existential claims, a position at which some representative of $\left[a\right]$ beats every representative of $\left[b\right]$, and simultaneously one where the reverse happens, and these are what we play against each other to reach a contradiction.
\end{proofidea}
\begin{proof} We see that for $\left[a\right]=\left[b\right]$ we always have $\left[a\right]\leq\left[b\right]$. This shows the totality for non-distinct elements
\\
Let two distinct non-zero real numbers $\left[a\right]$ and $\left[b\right]$ be given. We shall prove totality by showing that $\neg\left(\left[a\right]\leq\left[b\right]\vee\left[b\right]\leq\left[a\right]\right)$ is a contradiction for any two non-zero distinct real numbers. We use the \nameref{Equivalent to the Definition of Order} \ref{Equivalent to the Definition of Order} and negate it:
\begin{align*}
\begin{array}{ll@{}}
\sigma\left(\left[a\right]\right)\geq\sigma\left(\left[b\right]\right)&\wedge\vspace{0.35cm}\\
\left(\sigma\left(\left[a\right]\right)=0\vee\sigma\left(\left[b\right]\right)=0\vee\exists a\in\left[a\right],\exists n\in\omega_+,\forall b\in\left[b\right],a\left(n\right)>b\left(n\right)\right)&\wedge\vspace{0.35cm}\\
\left(\sigma\left(\left[a\right]\right)=1\vee\sigma\left(\left[b\right]\right)=1\vee\exists b\in\left[b\right],\exists n\in\omega_+,\forall a\in\left[a\right],a\left(n\right)<b\left(n\right)\right)
\end{array}
\end{align*}
We shall show that $\neg\left(\left[a\right]\leq\left[b\right]\right)\wedge\neg\left(\left[b\right]\leq\left[a\right]\right)$ is a contradiction. We see that the presumption that such $\left[a\right]$ and $\left[b\right]$ exist requires $\sigma\left(\left[a\right]\right)\geq\sigma\left(\left[b\right]\right)$ and $\sigma\left(\left[b\right]\right)\geq\sigma\left(\left[a\right]\right)$ which means that $\sigma\left(\left[a\right]\right)=\sigma\left(\left[b\right]\right)$.
\\
Both cases now boil down to showing that:
\begin{align*}
\begin{array}{ll@{}}
\left(\exists b_1\in\left[b\right],\exists n_1\in\omega_+,\forall a\in\left[a\right],a\left(n_1\right)<b_1\left(n_1\right)\right)&\wedge\\
\left(\exists a_2\in\left[a\right],\exists n_2\in\omega_+,\forall b\in\left[b\right],b\left(n_2\right)<a_2\left(n_2\right)\right)
\end{array}
\end{align*}
is a contradiction.\\
We know that there exists $b_1\in\left[b\right]$ and $t_1:=\min\left(\left\{n_1\in\omega_+\middle|\forall a\in\left[a\right],a\left(n_1\right)<b_1\left(n_1\right)\right\}\right)$ and $a_2\in\left[a\right]$ and $t_2:=\min\left(\left\{n_2\in\omega_+\middle|\forall b\in\left[b\right],b\left(n_2\right)<a_2\left(n_2\right)\right\}\right)$. It would be contradictory to have $t_1=t_2$. Take without loss of generality $t_1<t_2$.
\\
By the \nameref{Maximum Corollary} \ref{Maximum Corollary}, we have that
\begin{align*}
\sum\limits_{n=1}^{t_1}b_p\left(n\right)\prod\limits_{k=n}^{t_1-1}\left(\vartheta\left(k\right)\right)&=\max\left(\left\{f\left(t_1\right)\middle|f\in\left[b\right]\right\}\right)\geq b_1\left(t_1\right)\\
&>\max\left(\left\{g\left(t_1\right)\middle|g\in\left[a\right]\right\}\right)=\sum\limits_{n=1}^{t_1}a_p\left(n\right)\prod\limits_{k=n}^{t_1-1}\left(\vartheta\left(k\right)\right)
\end{align*}
and thus
\begin{align*}
\sum\limits_{n=1}^{t_1}b_p\left(n\right)\prod\limits_{k=n}^{t_1-1}\left(\vartheta\left(k\right)\right)-\sum\limits_{n=1}^{t_1}a_p\left(n\right)\prod\limits_{k=n}^{t_1-1}\left(\vartheta\left(k\right)\right)\geq1
\end{align*}
Therefore as $\forall n\in\left[t_1+1,\infty\right]_\omega$ we have $a_p\left(n\right)\leq\vartheta\left(n-1\right)-1$ we must have $\forall m\in\left[t_1+1,\infty\right]_\omega$
\begin{align*}
\sum\limits_{n=t_1+1}^{m}\left(a_p\left(n\right)\prod\limits_{j=n}^{m-1}\left(\vartheta\left(j\right)\right)\right)\leq\sum\limits_{n=t_1+1}^{m}\left(\left(\vartheta\left(n-1\right)-1\right)\prod\limits_{j=n}^{m-1}\left(\vartheta\left(j\right)\right)\right)=\prod\limits_{j=t_1}^{m-1}\left(\vartheta\left(j\right)\right)-1
\end{align*} This means that $\forall m\in\left[t_1+1,\infty\right]_\omega$
\begin{align*}
\sum\limits_{n=1}^{m}b_p\left(n\right)\prod\limits_{k=n}^{m-1}\left(\vartheta\left(k\right)\right)-\sum\limits_{n=1}^{m}a_p\left(n\right)\prod\limits_{k=n}^{m-1}\left(\vartheta\left(k\right)\right)&\geq\sum\limits_{n=1}^{t_1-1}b_p\left(n\right)\prod\limits_{k=n}^{m-1}\left(\vartheta\left(k\right)\right)-\sum\limits_{n=1}^{m}a_p\left(n\right)\prod\limits_{k=n}^{m-1}\left(\vartheta\left(k\right)\right)\\
&\geq\prod\limits_{k=t_1}^{m-1}\left(\vartheta\left(k\right)\right)-\sum\limits_{n=t_1+1}^{m}\left(a_p\left(n\right)\prod\limits_{j=n}^{m-1}\left(\vartheta\left(j\right)\right)\right)\geq1
\end{align*}
and thus
\begin{align*}
\sum\limits_{n=1}^{m}b_p\left(n\right)\prod\limits_{k=n}^{m-1}\left(\vartheta\left(k\right)\right)>\sum\limits_{n=1}^{m}a_p\left(n\right)\prod\limits_{k=n}^{m-1}\left(\vartheta\left(k\right)\right)
\end{align*}
However, by the usage of the \nameref{Maximum Corollary} \ref{Maximum Corollary} we have a contradiction to our presumption that
\begin{align*}
\sum\limits_{n=1}^{t_2}b_p\left(n\right)\prod\limits_{k=n}^{t_2-1}\left(\vartheta\left(k\right)\right)&=\max\left(\left\{f\left(t_2\right)\middle|f\in\left[b\right]\right\}\right)<a_2\left(t_2\right)\\
&\leq\max\left(\left\{g\left(t_2\right)\middle|g\in\left[a\right]\right\}\right)=\sum\limits_{n=1}^{t_2}a_p\left(n\right)\prod\limits_{k=n}^{t_2-1}\left(\vartheta\left(k\right)\right)
\end{align*}
Therefore, $\neg\left(\left[a\right]\leq\left[b\right]\right)\wedge\neg\left(\left[b\right]\leq\left[a\right]\right)$ is a contradiction. Hence $\left[a\right]\leq\left[b\right]\vee\left[b\right]\leq\left[a\right]$ is a tautology. For the zero case $\neg\left(\left[a\right]\leq\left[0\right]\vee\left[0\right]\leq\left[a\right]\right)\iff\neg\left(\sigma\left(\left[a\right]\right)=0\vee\sigma\left(\left[a\right]\right)=1\right)$ which is a contradiction, thus $\left[a\right]\leq\left[0\right]\vee\left[0\right]\leq\left[a\right]$ is a tautology. This shows totality for distinct elements.
\end{proof}

\begin{lemma}[One-Case Lemma]\label{One-Case Lemma}\noindent\\
For all real numbers $\left[f\right]$ and $\left[g\right]$ such that there exist functions $f\in\left[f\right]$ and $g\in\left[g\right]$ such that for all $n\in\omega_+$ we have $f\left(n\right)\geq g\left(n\right)$ and $\sigma\left(f\right)=\sigma\left(g\right)$ we then have that
\begin{align*}
\begin{array}{ll@{}}
\left[f\right]\geq\left[g\right]&\text{if $\sigma\left(f\right)=1$}\\
\left[f\right]\leq\left[g\right]&\text{if $\sigma\left(f\right)=0$}
\end{array}
\end{align*}
\end{lemma}

\begin{intuition}[One-Case Lemma] The definition of order quantifies over all ways of writing the two numbers, but in practice we almost never want to inspect all of them. This lemma says a single lucky pair is enough: if we can catch $\left[f\right]$ and $\left[g\right]$ in \emph{one} pair of same-sign representatives with $f$ dominating $g$ at every position, the order between the classes is already decided. It is the tool that turns a concrete pointwise inequality into a statement about the real numbers.
\end{intuition}

\begin{proof} \textbf{Case 1. $\sigma\left(f\right)=1$:} Suppose that $\sigma\left(f\right)=\sigma\left(g\right)$. If for all $n\in\omega_+$ we have that $f\left(n\right)=g\left(n\right)$ then as $\sigma\left(f\right)=\sigma\left(g\right)$ we have that $f\left(0\right)+g\left(0\right)=2m$ for some $m\in\omega$ and thus by the \nameref{Equivalence Lemma} \ref{Equivalence Lemma} we have that $\left[f\right]=\left[g\right]$ and hence $\left[f\right]\geq\left[g\right]$.
\\
If $\exists k\in\omega_+,f\left(k\right)>g\left(k\right)$ then by the \nameref{Domination Theorem Corollary} \ref{Domination Theorem Corollary} $\forall w\in\left[g\right],\exists n\in\omega_+,w\left(n\right)<f\left(n\right)$ and thus $\neg\left(\left[f\right]\leq\left[g\right]\right)$, which entails by totality $\left[f\right]\geq\left[g\right]$.\\
The proof for the \textbf{Case 2. $\sigma\left(f\right)=0$} is the same.
\end{proof}

\begin{definition}[Strict Order]\noindent\\
We shall now define an exclusive order ``$<$" as $\left[a\right]<\left[b\right]\iff\neg\left(\left[b\right]\leq\left[a\right]\right)$.
\end{definition}

\begin{theorem}[Equivalent to the Definition of Strict Order]\label{Equivalent to the Definition of Strict Order}\noindent\\
We have that $\left[a\right]<\left[b\right]$ is equivalent to
\begin{align*}
\begin{array}{ll@{}}
\sigma\left(\left[b\right]\right)>\sigma\left(\left[a\right]\right)&\vee\vspace{0.35cm}\\
\left(\sigma\left(\left[b\right]\right)=1\wedge\sigma\left(\left[a\right]\right)=1\wedge a_p\left(\Psi\left(a_p,b_p\right)\right)<b_p\left(\Psi\left(a_p,b_p\right)\right)\right)&\vee\vspace{0.35cm}\\
\left(\sigma\left(\left[b\right]\right)=0\wedge\sigma\left(\left[a\right]\right)=0\wedge a_p\left(\Psi\left(a_p,b_p\right)\right)>b_p\left(\Psi\left(a_p,b_p\right)\right)\right)
\end{array}
\end{align*}
We furthermore have $\forall\left[a\right]\in\mathcal{S}_{\mathbb{R}}^*$
\begin{align*}
\begin{array}{ll@{}}
\left[a\right]<\left[0\right]&\text{if and only if $\sigma\left(\left[a\right]\right)=0$}\\
\left[0\right]<\left[a\right]&\text{if and only if $\sigma\left(\left[a\right]\right)=1$}
\end{array}
\end{align*}
and finally $\left[0\right]<\left[0\right]$ as false.
\end{theorem}

\begin{intuition}[Equivalent to the Definition of Strict Order] This is the rule everyone already uses for decimals: to compare two numbers, look at the first position where their most compact forms disagree, and whoever is larger there is larger overall. Here the ``most compact forms" are the primary auxiliary functions and the first disagreement is at $\Psi\left(a_p,b_p\right)$. For positives the bigger digit wins; for negatives it loses. That the primary auxiliary functions suffice, rather than every representative, is what makes strict comparison finite and mechanical.
\end{intuition}
\begin{proofidea} Totality is the way in. It guarantees that at least one of the two comparisons holds, so denying one of them pins down the other, and unwinding which comparison survives. through the per-position form of the order, converts a negative statement into a positive one about digits. What emerges is the familiar rule: pass to the primary auxiliary functions and read off the first position at which they disagree.
\end{proofidea}
\begin{proof} We will first focus on the case where $\left[a\right]$ and $\left[b\right]$ are both non-zero real numbers.\\
($\implies$):\indent By totality $\left[a\right]\leq\left[b\right]\vee\left[b\right]\leq\left[a\right]$ and thus 
\begin{align*}
\begin{array}{ll@{}}
\left[a\right]<\left[b\right]\iff\neg\left(\left[b\right]\leq\left[a\right]\right)\implies\left[a\right]\leq\left[b\right]&\text{and}\\
\left[a\right]<\left[b\right]\iff\neg\left(\left[b\right]\leq\left[a\right]\right)\implies\neg\left(\left[b\right]\leq\left[a\right]\wedge\left[a\right]\leq\left[b\right]\right)\implies\neg\left(\left[a\right]=\left[b\right]\right)
\end{array}
\end{align*} thus:
\begin{align*}
\begin{array}{ll@{}}
\sigma\left(\left[a\right]\right)<\sigma\left(\left[b\right]\right)&\vee\vspace{0.35cm}\\
\left(\sigma\left(\left[a\right]\right)=\sigma\left(\left[b\right]\right)=1\wedge\forall a\in\left[a\right],\exists b\in\left[b\right],\forall n\in\omega_+,a\left(n\right)\leq b\left(n\right)\right)&\vee\vspace{0.35cm}\\
\left(\sigma\left(\left[a\right]\right)=\sigma\left(\left[b\right]\right)=0\wedge\forall b\in\left[b\right],\exists a\in\left[a\right],\forall n\in\omega_+,a\left(n\right)\geq b\left(n\right)\right)
\end{array}
\end{align*}

If $\sigma\left(\left[a\right]\right)<\sigma\left(\left[b\right]\right)$ then $\sigma\left(\left[a\right]\right)<\sigma\left(\left[b\right]\right)$ and thus our equivalent definition works.
\\
If $\sigma\left(\left[a\right]\right)=\sigma\left(\left[b\right]\right)=1\wedge\forall a\in\left[a\right],\exists b\in\left[b\right],\forall n\in\omega_+,a\left(n\right)\leq b\left(n\right)$\\
 then $\sigma\left(\left[a\right]\right)=\sigma\left(\left[b\right]\right)=1$. Additionally, as $\forall a\in\left[a\right],\exists b\in\left[b\right],\forall n\in\omega_+,a\left(n\right)\leq b\left(n\right)$, we have by the \nameref{PC Cor I} \ref{PC Cor I} 
\begin{align*}
\sum\limits_{k=1}^n\left(a_p\left(k\right)\prod\limits_{j=k}^{n-1}\left(\vartheta\left(j\right)\right)\right)&\leq\sum\limits_{k=1}^n\left(b\left(k\right)\prod\limits_{j=k}^{n-1}\left(\vartheta\left(j\right)\right)\right)\\
&\leq\sum\limits_{k=1}^n\left(h_b\left(k\right)\prod\limits_{j=k}^{n-1}\left(\vartheta\left(j\right)\right)\right)\leq\sum\limits_{k=1}^n\left(b_p\left(k\right)\prod\limits_{j=k}^{n-1}\left(\vartheta\left(j\right)\right)\right)
\end{align*}
We know that as $a_p\left(0\right)=b_p\left(0\right)$ (by definition) if for all $n\in\omega_+$ we have that $a_p\left(n\right)=b_p\left(n\right)$ then we would have $a_p=b_p$ and thus by the \nameref{Unicity of Real Numbers} \ref{Unicity of Real Numbers} we would have $\left[a\right]=\left[b\right]$ which is a contradiction to our presumption. Hence, we have that $\Psi\left(a_p,b_p\right)\neq0$. Therefore, as for all $k\in\left[0,\Psi\left(a_p,b_p\right)-1\right]_\omega$ we have $a_p\left(k\right)=b_p\left(k\right)$ we obtain that
\begin{align*}
\begin{array}{rll@{}}
\sum\limits_{k=1}^{\Psi\left(a_p,b_p\right)-1}\left(a_p\left(k\right)\hspace{-0.35cm}\prod\limits_{j=k}^{\Psi\left(a_p,b_p\right)-2}\hspace{-0.35cm}\left(\vartheta\left(j\right)\right)\right)&=\hspace{-0.35cm}\sum\limits_{k=1}^{\Psi\left(a_p,b_p\right)-1}\left(b_p\left(k\right)\hspace{-0.35cm}\prod\limits_{j=k}^{\Psi\left(a_p,b_p\right)-2}\hspace{-0.35cm}\left(\vartheta\left(j\right)\right)\right)&\text{and}\\
\sum\limits_{k=1}^{\Psi\left(a_p,b_p\right)}\left(a_p\left(k\right)\hspace{-0.35cm}\prod\limits_{j=k}^{\Psi\left(a_p,b_p\right)-1}\hspace{-0.35cm}\left(\vartheta\left(j\right)\right)\right)&\neq\hspace{-0.35cm}\sum\limits_{k=1}^{\Psi\left(a_p,b_p\right)}\left(b_p\left(k\right)\hspace{-0.35cm}\prod\limits_{j=k}^{\Psi\left(a_p,b_p\right)-1}\hspace{-0.35cm}\left(\vartheta\left(j\right)\right)\right)
\end{array}
\end{align*}
This means that $a_p\left(\Psi\left(a_p,b_p\right)\right)<b_p\left(\Psi\left(a_p,b_p\right)\right)$. Ergo,
\begin{align*}
\left(\sigma\left(\left[b\right]\right)=1\wedge\sigma\left(\left[a\right]\right)=1\wedge a_p\left(\Psi\left(a_p,b_p\right)\right)<b_p\left(\Psi\left(a_p,b_p\right)\right)\right)
\end{align*}
The proof that
\begin{align*}
&\left(\sigma\left(\left[a\right]\right)=\sigma\left(\left[b\right]\right)=0\wedge\forall b\in\left[b\right],\exists a\in\left[a\right],\forall n\in\omega_+,a\left(n\right)\geq b\left(n\right)\right)\\
\implies&\left(\sigma\left(\left[b\right]\right)=0\wedge\sigma\left(\left[a\right]\right)=0\wedge a_p\left(\Psi\left(a_p,b_p\right)\right)>b_p\left(\Psi\left(a_p,b_p\right)\right)\right)
\end{align*}
is the same as the proof above, but simply relabelled.\\\\
($\impliedby$):\\
If $\sigma\left(\left[a\right]\right)<\sigma\left(\left[b\right]\right)$ then $\sigma\left(\left[a\right]\right)<\sigma\left(\left[b\right]\right)$ and thus our equivalent definition works.\\
If $\left(\sigma\left(\left[a\right]\right)=\sigma\left(\left[b\right]\right)=1\wedge a_p\left(\Psi\left(a_p,b_p\right)\right)<b_p\left(\Psi\left(a_p,b_p\right)\right)\right)$ then $\sigma\left(\left[a\right]\right)=\sigma\left(\left[b\right]\right)=1$. We have also showed in the proof of the \nameref{Equivalent to the Definition of Order} \ref{Equivalent to the Definition of Order} that
\begin{align*}
&\left(\sigma\left(\left[a\right]\right)=\sigma\left(\left[b\right]\right)\wedge a_p\left(\Psi\left(a_p,b_p\right)\right)<b_p\left(\Psi\left(a_p,b_p\right)\right)\right)\\
\implies&\left(\forall a\in\left[a\right],\exists b\in\left[b\right],\forall n\in\omega_+,a\left(n\right)\leq b\left(n\right)\right)
\end{align*}
\begin{center}
thus
\end{center}
\begin{align*}
\left(\sigma\left(\left[a\right]\right)=\sigma\left(\left[b\right]\right)=1\right)\wedge\left(\forall a\in\left[a\right],\exists b\in\left[b\right],\forall n\in\omega_+,a\left(n\right)\leq b\left(n\right)\right)
\end{align*}
The proof that
\begin{align*}
&\left(\left(\sigma\left(\left[b\right]\right)=0\wedge\sigma\left(\left[a\right]\right)=0\right)\wedge\left(a_p\left(\Psi\left(a_p,b_p\right)\right)>b_p\left(\Psi\left(a_p,b_p\right)\right)\right)\right)\\
\implies&\left(\left(\sigma\left(\left[a\right]\right)=\sigma\left(\left[b\right]\right)=0\right)\wedge\left(\forall b\in\left[b\right],\exists a\in\left[a\right],\forall n\in\omega_+,a\left(n\right)\geq b\left(n\right)\right)\right)
\end{align*}
is the same as the proof above, but simply relabelled.
\\
Therefore, the Equivalent to the Definition of Strict Order for two non-zero real numbers $\left[a\right]<\left[b\right]$ implies $\left[a\right]\neq\left[b\right]$ and $\left[a\right]\leq\left[b\right]$, hence by the \nameref{Corollary of the Anti-Symmetry of Order} \ref{Corollary of the Anti-Symmetry of Order} $\neg\left(\left[a\right]\geq\left[b\right]\right)$.
\\
Now, $\left[a\right]<\left[0\right]$ implies $\left[a\right]\neq\left[0\right]$ and $\left[a\right]\leq\left[0\right]$, hence $\neg\left(\left[a\right]\geq\left[0\right]\right)$ and $\left[0\right]<\left[a\right]$ implies $\left[0\right]\neq\left[a\right]$ and $\left[0\right]\leq\left[a\right]$, hence $\neg\left(\left[0\right]\geq\left[a\right]\right)$ and finally $\left[0\right]<\left[0\right]$ is false and thus $\neg\left(\left[0\right]\geq\left[0\right]\right)$.
\end{proof}
\section{Analytic Tools}
The results of this section are not about order, and they are not yet about the arithmetic; they are the instruments with which every later inequality is proved. Two questions recur throughout the rest of the paper. Given two numbers, how do we establish an inequality between them without exhibiting a dominating representative, which is often impossible to write down? And given two numbers that appear to agree, how do we conclude that they do? The Boundedness Theorem answers the first by allowing us to weigh initial segments instead of comparing digits, and the Extension Theorems answer the second by showing that a difference smaller than every scale is no difference. Both are stated in terms of weighted truncations, which are the numerical shadow a name casts, and it is through them that facts about names become facts about numbers.
\subsection{Boundedness}
\begin{theorem}[Boundedness Theorem]\label{Boundedness Theorem}\noindent\\
Let $f,g\in \mathcal{N}_{\mathcal{F}in}^*$ where $\sigma\left(f\right)=\sigma\left(g\right)$. Then if for all $M\in\omega$ there exists $m\in\left[M,\infty\right]_\omega$ such that we have 
\begin{align*}
\sum\limits_{n=1}^{m}f\left(n\right)\prod\limits_{j=n}^{m-1}\left(\vartheta\left(j\right)\right)\geq\sum\limits_{n=1}^{m}g\left(n\right)\prod\limits_{j=n}^{m-1}\left(\vartheta\left(j\right)\right)
\end{align*}
then $\left[f\right]\geq\left[g\right]$ if $\sigma\left(\left[f\right]\right)=\sigma\left(\left[g\right]\right)=1$ and $\left[f\right]\leq\left[g\right]$ if $\sigma\left(\left[f\right]\right)=\sigma\left(\left[g\right]\right)=0$.
\end{theorem}

\begin{intuition}[Boundedness Theorem] Comparing two functions by producing a representative of one which dominates a representative of the other is often inconvenient, and for products it is close to impossible. This theorem offers a second route: instead of comparing functions position by position, \emph{weigh} them. What it asks for is remarkably little. We need not check that the running weighted total of $f$ stays above that of $g$ from some point onwards; it is enough that it catches up with it at arbitrarily deep levels of inspection, infinitely often, however sparsely. That weakness in the hypothesis is exactly what makes the theorem usable in practice, since to establish an inequality one need only produce, beyond any given depth, a single level at which the truncations fall the right way.
\end{intuition}

\begin{proof} Suppose $\sigma\left(\left[f\right]\right)=\sigma\left(\left[g\right]\right)=1$. Suppose for the sake of contrapositive that $\left[f\right]<\left[g\right]$. Then by the \nameref{Equivalent to the Definition of Strict Order} \ref{Equivalent to the Definition of Strict Order} $g_p\left(\Psi\left(f_p,g_p\right)\right)>f_p\left(\Psi\left(f_p,g_p\right)\right)$. Now, as $g$ is not in $\left[0\right]$ we know that $\exists P\in\omega_+,\forall p\in\left[P+1,\infty\right]_\omega$
\begin{align*}
\sum\limits_{k=1}^{p}h_g\left(k\right)\prod\limits_{j=k}^{p-1}\left(\vartheta\left(j\right)\right)\geq\sum\limits_{k=1}^{p}g_p\left(k\right)\prod\limits_{j=k}^{p-1}\left(\vartheta\left(j\right)\right)-1
\end{align*}
\footnote{This is due to the possibility of $h_g$ being a secondary auxiliary function. However, in that case $\forall n\in\left[\Psi\left(h_g,g_p\right),\infty\right]_\omega,\sum\limits_{k=1}^{n}g_p\left(k\right)\prod\limits_{j=k}^{n-1}\left(\vartheta\left(j\right)\right)-\sum\limits_{k=1}^{n}h_g\left(k\right)\prod\limits_{j=k}^{n-1}\left(\vartheta\left(j\right)\right)=1$. This takes care of both of the cases.} As $f_p$ is a primary auxiliary function, there exists $L\in\left[\Psi\left(f_p,g_p\right)+1,\infty\right]_\omega$ such that $f_p\left(L\right)<\vartheta\left(L-1\right)-1$. Denote minimal such value as $R$. We then see that
\begin{align*}
&\sum\limits_{n=1}^{R}\left(g_p\left(n\right)\prod\limits_{j=n}^{R-1}\left(\vartheta\left(j\right)\right)\right)-\sum\limits_{n=1}^{R}\left(f_p\left(n\right)\prod\limits_{j=n}^{R-1}\left(\vartheta\left(j\right)\right)\right)\\
\geq&\sum\limits_{n=1}^{\Psi\left(f_p,g_p\right)}\left(g_p\left(n\right)\prod\limits_{j=n}^{R-1}\left(\vartheta\left(j\right)\right)\right)-\sum\limits_{n=1}^{\Psi\left(f_p,g_p\right)}\left(f_p\left(n\right)\prod\limits_{j=n}^{R-1}\left(\vartheta\left(j\right)\right)\right)\\
&-\sum\limits_{n=\Psi\left(f_p,g_p\right)+1}^{R-1}\left(\left(\vartheta\left(n-1\right)-1\right)\prod\limits_{j=n}^{R-1}\left(\vartheta\left(j\right)\right)\right)-f_p\left(R\right)\\
\geq&\prod\limits_{j=\Psi\left(f_p,g_p\right)}^{R-1}\left(\vartheta\left(j\right)\right)-\vartheta\left(R-1\right)\left(\prod\limits_{j=\Psi\left(f_p,g_p\right)}^{R-2}\left(\vartheta\left(j\right)\right)-1\right)-f_p\left(R\right)\\
=&\vartheta\left(R-1\right)-f_p\left(R\right)\geq2
\end{align*}
which implies
\begin{align*}
\sum\limits_{n=1}^{R+1}\left(g_p\left(n\right)\prod\limits_{j=n}^{R}\left(\vartheta\left(j\right)\right)\right)-\sum\limits_{n=1}^{R+1}\left(f_p\left(n\right)\prod\limits_{j=n}^{R}\left(\vartheta\left(j\right)\right)\right)\geq2\vartheta\left(R\right)-\left(\vartheta\left(R\right)-1\right)=\vartheta\left(R\right)+1\geq3
\end{align*}
Thus, as $\forall m\in\left[R+2,\infty\right]_\omega$
\begin{align*}
\sum\limits_{n=R+2}^{m}\left(f_p\left(n\right)\prod\limits_{j=n}^{m-1}\left(\vartheta\left(j\right)\right)\right)\leq\sum\limits_{n=R+2}^{m}\left(\left(\vartheta\left(n-1\right)-1\right)\prod\limits_{j=n}^{m-1}\left(\vartheta\left(j\right)\right)\right)=\prod\limits_{j=R+1}^{m-1}\left(\vartheta\left(j\right)\right)-1
\end{align*}
and therefore
\begin{align*}
&\prod\limits_{j=R+1}^{m-1}\left(\vartheta\left(j\right)\right)\left(\sum\limits_{n=1}^{R+1}\left(g_p\left(n\right)\prod\limits_{j=n}^{R}\left(\vartheta\left(j\right)\right)\right)\right)-\sum\limits_{k=1}^{m}\left(f_p\left(k\right)\prod\limits_{j=k}^{m-1}\left(\vartheta\left(j\right)\right)\right)\\
\geq&3\prod\limits_{j=R+1}^{m-1}\left(\vartheta\left(j\right)\right)-\left(\prod\limits_{j=R+1}^{m-1}\left(\vartheta\left(j\right)\right)-1\right)=2\prod\limits_{j=R+1}^{m-1}\left(\vartheta\left(j\right)\right)+1
\end{align*}
Hence, $\forall m\in\left[\max\left(\left\{\Psi\left(h_g,g_p\right),R+2\right\}\right),\infty\right]_\omega$
\begin{align*}
&\prod\limits_{j=R+1}^{m-1}\left(\vartheta\left(j\right)\right)\left(\sum\limits_{k=1}^{R+1}\left(h_g\left(k\right)\prod\limits_{j=k}^{R}\left(\vartheta\left(j\right)\right)\right)\right)\geq\prod\limits_{j=R+1}^{m-1}\left(\vartheta\left(j\right)\right)\left(\sum\limits_{k=1}^{R+1}\left(g_p\left(k\right)\prod\limits_{j=k}^{R}\left(\vartheta\left(j\right)\right)\right)-1\right)\\
&=\prod\limits_{j=R+1}^{m-1}\left(\vartheta\left(j\right)\right)\left(\sum\limits_{k=1}^{R+1}\left(g_p\left(k\right)\prod\limits_{j=k}^{R}\left(\vartheta\left(j\right)\right)\right)\right)-\prod\limits_{j=R+1}^{m-1}\left(\vartheta\left(j\right)\right)\\
&\geq\sum\limits_{n=1}^{m}\left(f_p\left(n\right)\prod\limits_{j=n}^{m-1}\left(\vartheta\left(j\right)\right)\right)+\prod\limits_{j=R+1}^{m-1}\left(\vartheta\left(j\right)\right)+1
\end{align*}
Now, take the interval of inspection $\left[0,N\right]_\omega$ where $N=R+1$. By the definition of the finite function, there exists a sequence of contractions $\mathcal{C}_{k=q}^r\in\mathcal{FS}$ such that
\begin{align*}
\forall n\in\left[0,N\right]_\omega,\mathcal{C}_{k=q}^r\left(g\right)\left(n\right)=h_g\left(n\right)
\end{align*}
Define $\Xi:=\max\left(\text{con}\left(\mathcal{C}_{k=q}^r\right)\cup\left\{N\right\}\right)$. Let $m\in\left[\Xi+1,\infty\right]_\omega$ be given. By \eqref{eq:contraction-invariance} we have that
\begin{align*}
&\sum\limits_{n=1}^{m}\left(g\left(n\right)\prod\limits_{j=n}^{m-1}\left(\vartheta\left(j\right)\right)\right)=\sum\limits_{n=1}^{m}\left(\mathcal{C}_{k=q}^r\left(g\right)\left(n\right)\prod\limits_{j=n}^{m-1}\left(\vartheta\left(j\right)\right)\right)\\
\geq&\prod\limits_{j=R+1}^{m-1}\left(\vartheta\left(j\right)\right)\sum\limits_{n=1}^{R+1}\left(\mathcal{C}_{k=q}^r\left(g\right)\left(n\right)\prod\limits_{j=n}^{R}\left(\vartheta\left(j\right)\right)\right)=\prod\limits_{j=R+1}^{m-1}\left(\vartheta\left(j\right)\right)\sum\limits_{n=1}^{R+1}\left(h_g\left(n\right)\prod\limits_{j=n}^{R}\left(\vartheta\left(j\right)\right)\right)
\end{align*}
ergo $\forall m\in\left[\max\left(\left\{\Psi\left(h_g,g_p\right),\Xi+1\right\}\right),\infty\right]_\omega$
\begin{align*}
&\sum\limits_{n=1}^{m}\left(g\left(n\right)\prod\limits_{j=n}^{m-1}\left(\vartheta\left(j\right)\right)\right)\geq\prod\limits_{j=R+1}^{m-1}\left(\vartheta\left(j\right)\right)\sum\limits_{n=1}^{R+1}\left(h_g\left(n\right)\prod\limits_{j=n}^{R}\left(\vartheta\left(j\right)\right)\right)\\
>&\sum\limits_{n=1}^{m}\left(f_p\left(n\right)\prod\limits_{j=n}^{m-1}\left(\vartheta\left(j\right)\right)\right)+\prod\limits_{j=R+1}^{m-1}\left(\vartheta\left(j\right)\right)>\sum\limits_{n=1}^{m}\left(f_p\left(n\right)\prod\limits_{j=n}^{m-1}\left(\vartheta\left(j\right)\right)\right)
\end{align*}
This means that $\neg\left(\left[f\right]\geq\left[g\right]\right)$ and $\sigma\left(\left[f\right]\right)=\sigma\left(\left[g\right]\right)=1$ implies by the \nameref{Paradise City Lemma}\ref{Paradise City Lemma} that $\exists M\in\omega_+,\forall m\in\left[M,\infty\right]_\omega$
\begin{align*}
\sum\limits_{n=1}^{m}\left(f\left(n\right)\prod\limits_{j=n}^{m-1}\left(\vartheta\left(j\right)\right)\right)<\sum\limits_{n=1}^{m}\left(g\left(n\right)\prod\limits_{j=n}^{m-1}\left(\vartheta\left(j\right)\right)\right)
\end{align*}
Therefore, by contrapositive $\sigma\left(\left[f\right]\right)=\sigma\left(\left[g\right]\right)=1$ and $\forall M\in\omega_+,\exists m\in\left[M,\infty\right]_\omega$
\begin{align*}
\sum\limits_{n=1}^{m}\left(f\left(n\right)\prod\limits_{j=n}^{m-1}\left(\vartheta\left(j\right)\right)\right)\geq\sum\limits_{n=1}^{m}\left(g\left(n\right)\prod\limits_{j=n}^{m-1}\left(\vartheta\left(j\right)\right)\right)
\end{align*}
implies $\left[f\right]\geq\left[g\right]$.\\
The case for $\sigma\left(\left[f\right]\right)=\sigma\left(\left[g\right]\right)=0$ is the same.
\end{proof}
\subsection{Extension and Extended Extension Theorems}
\begin{theorem}[Extension Theorem]\label{Extension Theorem}\noindent\\
Define the difference function $\Upsilon:\omega\times\omega\rightarrow\omega$ as 
\begin{align*}\Upsilon\left(\begin{array}{l@{}}a\\b\end{array}\right)=\max\left(\left\{a,b\right\}\right)-\min\left(\left\{a,b\right\}\right)
\end{align*}
Let $f,g\in \mathcal{N}_{\mathcal{F}in}$ where $\sigma\left(f\right)=\sigma\left(g\right)$. Then if $\forall N\in\omega_+,\exists M\in\left[N,\infty\right]_\omega,\forall m\in\left[M,\infty\right]_\omega$
\begin{align*}
\Upsilon\left(\begin{array}{l@{}}\sum\limits_{n=1}^{m}\left(f\left(n\right)\prod\limits_{j=n}^{m-1}\left(\vartheta\left(j\right)\right)\right)\\\sum\limits_{n=1}^{m}\left(g\left(n\right)\prod\limits_{j=n}^{m-1}\left(\vartheta\left(j\right)\right)\right)\end{array}\right)\leq\prod\limits_{j=N}^{m-1}\left(\vartheta\left(j\right)\right)
\end{align*}
then $\left[f\right]=\left[g\right]$.
\end{theorem}

\begin{intuition}[Extension Theorem]If the running totals of two same-sign numbers eventually stay within $\nicefrac{1}{\prod_{j<N}\vartheta\left(j\right)}$ much of each other, and this holds for every choice of $N$, then no gap survives: a genuine difference between two distinct reals exceeds some fixed scale, and would therefore be detected at that one. A difference smaller than every scale is no difference at all.\\
The following variant (\nameref{Extension Theorem} \ref{Extended Extension Theorem}), needed for multiplication, bounds only one side, instead of bounding the gap symmetrically, we bound a rescaled coarse truncation of one number by a fine truncation of the other plus a single unit of scale. Combined with an already-known inequality in the opposite direction, this one-sided sandwich is still tight enough to force the two numbers to coincide.
\end{intuition}

\begin{proof} We shall prove the theorem in two cases. One where both $\left[f\right]$ and $\left[g\right]$ are non-zero and the second where one of them is.\\
Firstly let $\left[f\right],\left[g\right]\in\mathcal{S}_{\mathbb{R}}^*$, such that our presumption is satisfied, be given. Suppose that $\left[f\right]\neq\left[g\right]$ and $\sigma\left(\left[f\right]\right)=\sigma\left(\left[g\right]\right)=1$. Suppose without loss of generality $\left[f\right]<\left[g\right]$. By using the proof of the \nameref{Boundedness Theorem} \ref{Boundedness Theorem} above we have $\exists N\in\omega_+$\footnote{$N$ is $\max\left(\left\{\Psi\left(h_g,g_p\right),\Xi+1\right\}\right)$ from the proof above.},$\forall M\in\left[N,\infty\right]_\omega,\exists m\in\left[M,\infty\right]_\omega$
\begin{align*}
&\sum\limits_{n=1}^{m}\left(g\left(n\right)\prod\limits_{j=n}^{m-1}\left(\vartheta\left(j\right)\right)\right)>\sum\limits_{n=1}^{m}\left(f\left(n\right)\prod\limits_{j=n}^{m-1}\left(\vartheta\left(j\right)\right)\right)+\prod\limits_{j=N}^{m-1}\left(\vartheta\left(j\right)\right)\\
\implies&\sum\limits_{n=1}^{m}\left(g\left(n\right)\prod\limits_{j=n}^{m-1}\left(\vartheta\left(j\right)\right)\right)-\sum\limits_{n=1}^{m}\left(f\left(n\right)\prod\limits_{j=n}^{m-1}\left(\vartheta\left(j\right)\right)\right)>\prod\limits_{j=N}^{m-1}\left(\vartheta\left(j\right)\right)
\end{align*}
This means that $\exists N\in\omega_+,\forall M\in\left[N,\infty\right]_\omega,\exists m\in\left[M,\infty\right]_\omega$
\begin{align*}
\Upsilon\left(\begin{array}{l@{}}\sum\limits_{n=1}^{m}\left(f\left(n\right)\prod\limits_{j=n}^{m-1}\left(\vartheta\left(j\right)\right)\right)\\\sum\limits_{n=1}^{m}\left(g\left(n\right)\prod\limits_{j=n}^{m-1}\left(\vartheta\left(j\right)\right)\right)\end{array}\right)>\prod\limits_{j=N}^{m-1}\left(\vartheta\left(j\right)\right)
\end{align*}
Ergo, by contrapositive the theorem holds.\\
The case for $\sigma\left(\left[f\right]\right)=\sigma\left(\left[g\right]\right)=0$ is the same.\\
We see that if $\left[f\right]\in\left[0\right]$ then $\left[g\right]\in\left[0\right]$ as well, as if not, then $\exists k\in\omega_+,g\left(k\right)\neq0$. And thus, $\forall m\in\left[k+1,\infty\right]_\omega$
\begin{align*}
\Upsilon\left(\begin{array}{l@{}}\sum\limits_{n=1}^{m-1}\left(f\left(n\right)\prod\limits_{j=n}^{m-1}\left(\vartheta\left(j\right)\right)\right)\\\sum\limits_{n=1}^{m-1}\left(g\left(n\right)\prod\limits_{j=n}^{m-1}\left(\vartheta\left(j\right)\right)\right)\end{array}\right)=\sum\limits_{n=1}^{m-1}\left(g\left(n\right)\prod\limits_{j=n}^{m-1}\left(\vartheta\left(j\right)\right)\right)\geq g\left(k\right)\prod\limits_{j=k}^{m-1}\left(\vartheta\left(j\right)\right)\geq\prod\limits_{j=k}^{m-1}\left(\vartheta\left(j\right)\right)
\end{align*}
contradicting our presumption. Hence $\left[f\right]=\left[g\right]=\left[0\right]$.
\end{proof}

\begin{theorem}[Extended Extension Theorem]\label{Extended Extension Theorem}\noindent\\
Suppose that $f,g\in \mathcal{N}_{\mathcal{F}in}^*$.\\
If $\left[f\right]\leq\left[g\right]$ and $\forall N\in\omega_+,\forall\Xi\in\left[N+1,\infty\right]_\omega,\exists m\in\left[\Xi+1,\infty\right]_\omega$ such that
\begin{align*}
\prod\limits_{j=\Xi}^{m-1}\left(\vartheta\left(j\right)\right)\left(\sum\limits_{k=1}^\Xi\left(g\left(k\right)\prod\limits_{j=k}^{\Xi-1}\left(\vartheta\left(j\right)\right)\right)\right)\leq\sum\limits_{k=1}^m\left(f\left(k\right)\prod\limits_{j=k}^{m-1}\left(\vartheta\left(j\right)\right)\right)+\prod\limits_{j=N}^{m-1}\left(\vartheta\left(j\right)\right)
\end{align*}
and if $\sigma\left(f\right)=\sigma\left(g\right)=1$ then $\left[f\right]=\left[g\right]$.\\\\
If $\left[f\right]\geq\left[g\right]$ and $\forall N\in\omega_+,\forall\Xi\in\left[N+1,\infty\right]_\omega,\exists m\in\left[\Xi+1,\infty\right]_\omega$ such that
\begin{align*}
\prod\limits_{j=\Xi}^{m-1}\left(\vartheta\left(j\right)\right)\left(\sum\limits_{k=1}^\Xi\left(f\left(n\right)\prod\limits_{j=k}^{\Xi-1}\left(\vartheta\left(j\right)\right)\right)\right)\leq\sum\limits_{k=1}^m\left(g\left(k\right)\prod\limits_{j=k}^{m-1}\left(\vartheta\left(j\right)\right)\right)+\prod\limits_{j=N}^{m-1}\left(\vartheta\left(j\right)\right)
\end{align*}
and if $\sigma\left(f\right)=\sigma\left(g\right)=0$ then $\left[f\right]=\left[g\right]$.
\end{theorem}
\begin{proof} Suppose that $\sigma\left(f\right)=\sigma\left(g\right)=1$ and $\neg\left(\left[g\right]\leq\left[f\right]\right)$. By the proof of the \nameref{Boundedness Theorem} \ref{Boundedness Theorem} we have $\forall m\in\left[\max\left(\left\{\Psi\left(h_g,g_p\right),R+2\right\}\right),\infty\right]_\omega$
\begin{align*}
&\prod\limits_{j=R+1}^{m-1}\left(\vartheta\left(j\right)\right)\left(\sum\limits_{k=1}^{R+1}\left(h_g\left(k\right)\prod\limits_{j=k}^{R}\left(\vartheta\left(j\right)\right)\right)\right)>\sum\limits_{k=1}^{m}\left(f_p\left(k\right)\prod\limits_{j=k}^{m-1}\left(\vartheta\left(j\right)\right)\right)+\prod\limits_{j=R+1}^{m-1}\left(\vartheta\left(j\right)\right)
\end{align*}
Now take a sequence of contractions $\mathcal{C}_{k=q}^r\in\mathcal{FS}$ such that for all $n\in\left[0,R+1\right]_\omega$ we have $\mathcal{C}_{k=q}^r\left(g\right)\left(n\right)=h_g\left(n\right)$. Take $\Xi:=\max\left(\text{con}\left(\mathcal{C}_{k=q}^r\right)\cup\left\{\Psi\left(h_g,g_p\right),R+1\right\}\right)$. Then for all $m\in\left[\Xi+1,\infty\right]_\omega$ we have by \eqref{eq:contraction-invariance} and the \nameref{Paradise City Lemma} \ref{Paradise City Lemma}
\begin{align*}
&\prod\limits_{j=\Xi}^{m-1}\left(\vartheta\left(j\right)\right)\left(\sum\limits_{k=1}^{\Xi}\left(g\left(k\right)\prod\limits_{j=k}^{\Xi-1}\left(\vartheta\left(j\right)\right)\right)\right)=\prod\limits_{j=\Xi}^{m-1}\left(\vartheta\left(j\right)\right)\left(\sum\limits_{k=1}^{\Xi}\left(\mathcal{C}_{k=q}^r\left(g\right)\left(k\right)\prod\limits_{j=k}^{\Xi-1}\left(\vartheta\left(j\right)\right)\right)\right)\\
\geq&\prod\limits_{j=R+1}^{m-1}\left(\vartheta\left(j\right)\right)\left(\sum\limits_{k=1}^{R+1}\left(\mathcal{C}_{k=q}^r\left(g\right)\left(k\right)\prod\limits_{j=k}^{R}\left(\vartheta\left(j\right)\right)\right)\right)=\prod\limits_{j=R+1}^{m-1}\left(\vartheta\left(j\right)\right)\left(\sum\limits_{k=1}^{R+1}\left(h_g\left(k\right)\prod\limits_{j=k}^{R}\left(\vartheta\left(j\right)\right)\right)\right)\\
>&\sum\limits_{k=1}^{m}\left(f_p\left(k\right)\prod\limits_{j=k}^{m-1}\left(\vartheta\left(j\right)\right)\right)+\prod\limits_{j=R+1}^{m-1}\left(\vartheta\left(j\right)\right)\geq\sum\limits_{k=1}^{m}\left(f\left(k\right)\prod\limits_{j=k}^{m-1}\left(\vartheta\left(j\right)\right)\right)+\prod\limits_{j=R+1}^{m-1}\left(\vartheta\left(j\right)\right)
\end{align*}
Therefore, $\sigma\left(f\right)=\sigma\left(g\right)=1$ and $\neg\left(\left[g\right]\leq\left[f\right]\right)$ implies that there exists $N\in\omega_+$ ($N=R+1$) such that there exists $\Xi\in\left[N+1,\infty\right]_\omega$ such that for all $m\in\left[\Xi+1,\infty\right]_\omega$ we have
\begin{align*}
&\prod\limits_{j=\Xi}^{m-1}\left(\vartheta\left(j\right)\right)\left(\sum\limits_{k=1}^{\Xi}\left(g\left(k\right)\prod\limits_{j=k}^{\Xi-1}\left(\vartheta\left(j\right)\right)\right)\right)>\sum\limits_{k=1}^{m}\left(f\left(k\right)\prod\limits_{j=k}^{m-1}\left(\vartheta\left(j\right)\right)\right)+\prod\limits_{j=N}^{m-1}\left(\vartheta\left(j\right)\right)
\end{align*}
Hence, by contrapositive, we have that $\sigma\left(f\right)=\sigma\left(g\right)=1$ and the presumption of the theorem implies $\left[g\right]\leq\left[f\right]$. Thus as we also suppose $\left[f\right]\leq\left[g\right]$ then by the \nameref{Anti-Symmetry of Order} \ref{Anti-Symmetry of Order} we have $\left[g\right]=\left[f\right]$.
\\
The proof for $\sigma\left(f\right)=\sigma\left(g\right)=0$ is the same.
\end{proof}
\begin{corollary}[Extended Extension Theorem Corollary]\label{Extended Extension Theorem Corollary} Suppose that $f,g\in \mathcal{N}_{\mathcal{F}in}$, $\sigma\left(f\right)=\sigma\left(g\right)$. Then if $\forall N\in\omega_+,\forall\Xi\in\left[N+1,\infty\right]_\omega,\exists m\in\left[\Xi+1,\infty\right]_\omega$ such that
\begin{align*}
\prod\limits_{j=\Xi}^{m-1}\left(\vartheta\left(j\right)\right)\left(\sum\limits_{n=1}^\Xi\left(f\left(n\right)\prod\limits_{j=n}^{\Xi-1}\left(\vartheta\left(j\right)\right)\right)\right)\leq\sum\limits_{n=1}^m\left(g\left(n\right)\prod\limits_{j=n}^{m-1}\left(\vartheta\left(j\right)\right)\right)+\prod\limits_{j=N}^{m-1}\left(\vartheta\left(j\right)\right)
\end{align*}
and $\forall N\in\omega_+,\forall\Xi\in\left[N+1,\infty\right]_\omega,\exists m\in\left[\Xi+1,\infty\right]_\omega$ such that
\begin{align*}
\prod\limits_{j=\Xi}^{m-1}\left(\vartheta\left(j\right)\right)\left(\sum\limits_{n=1}^\Xi\left(g\left(n\right)\prod\limits_{j=n}^{\Xi-1}\left(\vartheta\left(j\right)\right)\right)\right)\leq\sum\limits_{n=1}^m\left(f\left(n\right)\prod\limits_{j=n}^{m-1}\left(\vartheta\left(j\right)\right)\right)+\prod\limits_{j=N}^{m-1}\left(\vartheta\left(j\right)\right)
\end{align*}
then $\left[f\right]=\left[g\right]$.
\end{corollary}
\begin{proof} By the \nameref{Totality of Order} \ref{Totality of Order} either $\left[f\right]\leq\left[g\right]$ or $\left[f\right]\geq\left[g\right]$. Then by the \nameref{Extended Extension Theorem} \ref{Extended Extension Theorem} $\left[f\right]=\left[g\right]$.
\end{proof}
\section{Completeness}\label{sec:completeness}
Completeness is what separates the reals from the rationals, and it is the one axiom that cannot be verified by a computation on representatives: the supremum of a bounded set is not assembled from the names of its members in any direct way, and must instead be constructed position by position and then shown to be a legitimate name in its own right. That second step is the delicate one, since the obvious construction can run into a maximal tail and so fail to be a primary auxiliary function. We prove the axiom from above and record its mirror image below.
\begin{axiom}[Completeness Axiom from above]\label{Completeness Axiom from above}\noindent\\
If $S$ is a non-empty set of real numbers and if $S$ has an upper bound, then $S$ has the least upper bound.
\end{axiom}
\begin{proofidea} The supremum is produced as the least element of the set $\Xi_S$ of upper bounds of $S$, and the argument splits according to whether $S$ contains a positive number, since that fixes the sign of every upper bound. The candidate is assembled position by position as a maximum over the digits of primary auxiliary functions, at the first position, the largest digit any member attains there, and that it is a legitimate primary auxiliary function. The second is the delicate one, the construction could in principle run into an all-$\left(\vartheta-1\right)$ tail, and excluding that requires a witness drawn from a later stage of the construction rather than from $S$ itself. That exclusion is why the proof is long.
\end{proofidea}
\begin{proof} Let a non-empty set of real numbers $S$ be given. Furthermore, suppose $S$ has an upper bound, i.e. $\exists\left[x\right]\in\mathcal{S}_{\mathbb{R}},\forall\left[a\right]\in S,\left[x\right]\geq\left[a\right]$. Define the set of all upper bounds of $S$ as $\Xi_S:=\left\{\left[x\right]\in\mathcal{S}_{\mathbb{R}}\middle|\forall\left[a\right]\in S,\left[x\right]\geq\left[a\right]\right\}$. We claim that for all such $S$ there exists $\left[\xi\right]\in\Xi_S$ such that $\forall\left[x\right]\in\Xi_S, \left[\xi\right]\leq\left[x\right]$.
\paragraph{Case 1: $\exists \left[a\right]\in S, \sigma\left(\left[a\right]\right)=1$:} We must have $\forall\left[x\right]\in\Xi_S,\sigma\left(\left[x\right]\right)=1$. Define \footnote{We will be using abusive notation where we combine the replacement and specification axioms at the same time.}
\begin{align*}
\Lambda_1:=\max\left(\left\{a_p\left(1\right)\middle|\begin{array}{rl@{}}&\left[a\right]\in S\\\wedge&\sigma\left(\left[a\right]\right)=1\end{array}\right\}\right)
\end{align*}
As $S$ is bounded from above we must have that $\Lambda_1$ exists by the \nameref{Equivalent to the Definition of Strict Order} \ref{Equivalent to the Definition of Strict Order}. This is seen as if it were not true, then for all $n\in\omega$ there exists $\left[a\right]\in S$ such that $\sigma\left(\left[a\right]\right)=1$ and $n<a_p\left(1\right)$. Therefore, $\forall\left[x\right]\in\mathcal{S}_{\mathbb{R}},\exists\left[a\right]\in S$ such that
\begin{align*}
\begin{array}{rl@{}}&\sigma\left(\left[a\right]\right)=1\\\wedge&\Psi\left(x_p,a_p\right)=1\\\wedge&x_p\left(\Psi\left(x_p,a_p\right)\right)<a_p\left(\Psi\left(x_p,a_p\right)\right)\end{array}
\end{align*}
and thus 
\begin{align*}
\forall\left[x\right]\in\mathcal{S}_{\mathbb{R}},\exists\left[a\right]\in S,\left[x\right]<\left[a\right]
\end{align*} contradicting $S$ being bounded from above. Therefore, $\Lambda_1$ is well-defined.\\
Now, define for all $n\in\left[2,\infty\right]_\omega$
\begin{align*}
\Lambda_n:=\max\left(\left\{a_p\left(n\right)\middle|\begin{array}{rl@{}}&\left[a\right]\in S\\\wedge&\sigma\left(\left[a\right]\right)=1\\\wedge&\forall k\in\left[1,n-1\right]_\omega,a_p\left(k\right)=\Lambda_k\end{array}\right\}\right)
\end{align*}
Notice that for all $n\in\left[2,\infty\right]_\omega$ we have $\Lambda_n\leq\vartheta\left(n-1\right)-1$, as 
\begin{align*}
\forall n\in\left[2,\infty\right]_\omega,\forall\left[a\right]\in S,0\leq a_p\left(n\right)\leq\vartheta\left(n-1\right)-1
\end{align*}
Take the function $\xi:\omega\rightarrow\omega$ defined as 
\begin{align*}
\xi\left(k\right)=\begin{cases}0&k=0\\\Lambda_n&n\in\omega_+
\end{cases}
\end{align*} As $\xi$ is an auxiliary function, it is finite. Now, we have two cases. Define 
\begin{align*}
\Phi:=\min\left(\left\{T\in\omega_+\middle|\forall t\in\left[T+1,\infty\right]_\omega,\xi\left(t\right)=\vartheta\left(t-1\right)-1\right\}\right)
\end{align*}
The first case is when $\Phi=\emptyset$, and the second case is when not.
\subparagraph{$\Phi=\emptyset$:} We see that $\xi=\xi_p$ as $\xi$ itself is an auxiliary function and it does not have a ``tail" of $\vartheta\left(t-1\right)-1$.\\
We have $\left[\xi\right]\in\Xi_S$ as $\forall\left[a\right]\in S,\left[a\right]\leq\left[\xi\right]$. Suppose not. Then 
\begin{align*}
\exists\left[a\right]\in S,\left[\xi\right]<\left[a\right]\text{ implies }\exists\left[a\right]\in S,\begin{array}{rl@{}}&\sigma\left(\left[a\right]\right)=1\\\wedge&\xi\left(\Psi\left(\xi,a_p\right)\right)<a_p\left(\Psi\left(\xi,a_p\right)\right)\end{array}
\end{align*} 
Now, we see that this is a contradiction, as we have by the definition of $\Psi$ that 
\begin{align*}
\begin{array}{ll@{}}
\forall k\in\left[1,\Psi\left(\xi,a_p\right)-1\right]_\omega,\Lambda_k=\xi\left(k\right)=a_p\left(k\right)&\text{and}\\
\Lambda_{\Psi\left(\xi,a_p\right)}=\xi\left(\Psi\left(\xi,a_p\right)\right)<a_p\left(\Psi\left(\xi,a_p\right)\right)
\end{array}
\end{align*} 
which is a contradiction against the definition of $\Lambda_{\Psi\left(\xi,a_p\right)}$.\\
Furthermore, $\forall\left[x\right]\in\Xi_S,\left[\xi\right]\leq\left[x\right]$. Suppose not. The presumption $\left[x\right]<\left[\xi\right]$ leads to $x_p\left(\Psi\left(\xi,x_p\right)\right)<\xi\left(\Psi\left(\xi,x_p\right)\right)$. If that were the case, 
\begin{align*}
\exists \left[a\right]\in S,\begin{array}{rl@{}}&\sigma\left(\left[a\right]\right)=1\\\wedge&\forall k\in\left[1,\Psi\left(\xi,x_p\right)-1\right]_\omega,a_p\left(k\right)=\xi\left(k\right)=x_p\left(k\right)\\\wedge&x_p\left(\Psi\left(\xi,x_p\right)\right)<\xi\left(\Psi\left(\xi,x_p\right)\right)=a_p\left(\Psi\left(\xi,x_p\right)\right)\end{array}
\end{align*} This implies $x_p\left(\Psi\left(a_p,x_p\right)\right)<a_p\left(\Psi\left(a_p,x_p\right)\right)\implies \left[x\right]<\left[a\right]$ and thus $\left[x\right]\not\in\Xi_S$.
\subparagraph{$\Phi\neq\emptyset$:} We see that $\xi$ is a secondary auxiliary function and by definition the primary auxiliary function $\xi_p$ is
\begin{align*}
\xi_p\left(k\right)=
\begin{cases}
\xi\left(k\right)&k\in\left[0,\Phi-1\right]_\omega\\
\xi\left(\Phi\right)+1&k=\Phi\\
0&k\in\left[\Phi+1,\infty\right]_\omega
\end{cases}
\end{align*}
We have that $\left[\xi\right]\in\Xi_S$ as $\forall\left[a\right]\in S,\left[a\right]\leq\left[\xi\right]$. Suppose not. Then 
\begin{align*}
\exists\left[a\right]\in S,\left[\xi\right]<\left[a\right]\text{ implies }\exists\left[a\right]\in S,\begin{array}{rl@{}}&\sigma\left(\left[a\right]\right)=1\\\wedge&\xi_p\left(\Psi\left(\xi_p,a_p\right)\right)<a_p\left(\Psi\left(\xi_p,a_p\right)\right)\end{array}
\end{align*}
Now, we see that this is a contradiction, which we will prove by cases.\\
If $\Psi\left(\xi_p,a_p\right)\leq\Phi$ then 
\begin{align*}
\begin{array}{ll@{}}
\forall k\in\left[1,\Psi\left(\xi_p,a_p\right)-1\right]_\omega,\Lambda_k=\xi\left(k\right)=\xi_p\left(k\right)=a_p\left(k\right)&\text{and}\\
\Lambda_{\Psi\left(\xi_p,a_p\right)}=\xi\left(\Psi\left(\xi_p,a_p\right)\right)\leq\xi_p\left(\Psi\left(\xi_p,a_p\right)\right)<a_p\left(\Psi\left(\xi_p,a_p\right)\right)
\end{array}
\end{align*}
which is a contradiction against the definition of $\Lambda_{\Psi\left(\xi,a_p\right)}$.\\
If $\Psi\left(\xi_p,a_p\right)>\Phi$ then
\begin{align*}
\begin{array}{ll@{}}
\forall k\in\left[1,\Phi-1\right]_\omega,\Lambda_k=\xi\left(k\right)=\xi_p\left(k\right)=a_p\left(k\right)&\text{and}\\
a_p\left(\Phi\right)=\xi_p\left(\Phi\right)=\xi\left(\Phi\right)+1=\Lambda_{\Phi}+1
\end{array}
\end{align*}
which is a contradiction to the definition of $\Lambda_{\Phi}$.\\
Furthermore, $\forall\left[x\right]\in\Xi_S,\left[\xi\right]\leq\left[x\right]$. Suppose not. The presumption $\left[x\right]<\left[\xi\right]$ leads to $x_p\left(\Psi\left(\xi_p,x_p\right)\right)<\xi_p\left(\Psi\left(\xi_p,x_p\right)\right)$. That is a contradiction, which we will prove by cases.\\
If $\Psi\left(\xi_p,x_p\right)<\Phi$, then 
\begin{align*}
\exists \left[a\right]\in S,
\begin{array}{rl@{}}
&\sigma\left(\left[a\right]\right)=1\\
\wedge&\forall k\in\left[1,\Psi\left(\xi_p,x_p\right)-1\right]_\omega,a_p\left(k\right)=\xi\left(k\right)=\xi_p\left(k\right)=x_p\left(k\right)\\
\wedge&x_p\left(\Psi\left(\xi_p,x_p\right)\right)<\xi_p\left(\Psi\left(\xi_p,x_p\right)\right)=\xi\left(\Psi\left(\xi_p,x_p\right)\right)=a_p\left(\Psi\left(\xi_p,x_p\right)\right)
\end{array}
\end{align*}
This implies $x_p\left(\Psi\left(x_p,a_p\right)\right)<a_p\left(\Psi\left(x_p,a_p\right)\right)\implies\left[x\right]<\left[a\right]$ and thus $\left[x\right]\not\in\Xi_S$.\\
If $\Psi\left(\xi_p,x_p\right)=\Phi$, then 
\begin{align*}
\begin{array}{rl@{}}
&\forall k\in\left[1,\Phi-1\right]_\omega,x_p\left(k\right)=\xi_p\left(k\right)\\
\wedge&x_p\left(\Phi\right)<\xi_p\left(\Phi\right)\\
\wedge&\exists T\in\left[\Phi+1,\infty\right]_\omega,\forall t\in\left[\Phi+1,T-1\right]_\omega\begin{array}{rl@{}}
&x_p\left(t\right)=\vartheta\left(t-1\right)-1\\\wedge&x_p\left(T\right)<\vartheta\left(T-1\right)-1\end{array}
\end{array}
\end{align*}
The last stipulation is added as were it not true, $x_p$ would not be a primary auxiliary function.\\
If $x_p\left(\Phi\right)<\xi_p\left(\Phi\right)-1=\xi\left(\Phi\right)$ then 
\begin{align*}
\exists \left[a\right]\in S,\begin{array}{ll@{}}&\sigma\left(\left[a\right]\right)=1\\\wedge&\forall k\in\left[1,\Phi-1\right]_\omega,a_p\left(k\right)=\xi\left(k\right)=\xi_p\left(k\right)=x_p\left(k\right)\\\wedge&x_p\left(\Phi\right)<\xi\left(\Phi\right)=a_p\left(\Phi\right)\end{array}
\end{align*}
This implies $x_p\left(\Psi\left(x_p,a_p\right)\right)<a_p\left(\Psi\left(x_p,a_p\right)\right)\implies \left[x\right]<\left[a\right]$ and thus $\left[x\right]\not\in\Xi_S$.\\
And if $x_p\left(\Phi\right)=\xi_p\left(\Phi\right)-1=\xi\left(\Phi\right)$ then 
\begin{align*}
\exists \left[a\right]\in S\begin{array}{ll@{}}&\sigma\left(\left[a\right]\right)=1\\\wedge&\forall k\in\left[1,T-1\right]_\omega,a_p\left(k\right)=\xi\left(k\right)=x_p\left(k\right)\\\wedge&x_p\left(T\right)<\xi\left(T\right)=a_p\left(T\right)=\vartheta\left(T-1\right)-1\end{array}
\end{align*}
This implies $x_p\left(\Psi\left(x_p,a_p\right)\right)<a_p\left(\Psi\left(x_p,a_p\right)\right)\implies \left[x\right]<\left[a\right]$ and thus $\left[x\right]\not\in\Xi_S$.\\
Finally, we see that $\Psi\left(\xi_p,x_p\right)$ cannot be possibly larger than $\Phi$ as $\forall k\in\left[\Phi+1,\infty\right]_\omega,\xi_p\left(k\right)=0$.
\paragraph{Case 2: $\forall\left[a\right]\in S,\left[a\right]\leq\left[0\right]$} If we have that $\forall\left[x\right]\in\Xi_S,\left[x\right]\geq\left[0\right]$ we are done. Suppose otherwise, thus, $\exists\left[x\right]\in\Xi_S,\left[x\right]<\left[0\right]$. Define 
\begin{align*}
\Lambda_1:=\min\left(\left\{a_p\left(1\right)\middle|\left[a\right]\in S\right\}\right)
\end{align*}
Now, define for all $n\in\left[2,\infty\right]_\omega$
\begin{align*}
\Lambda_n:=\min\left(\left\{a_p\left(n\right)\middle|\begin{array}{ll@{}}&\left[a\right]\in S\\\wedge&\forall k\in\left[1,n-1\right]_\omega,a_p\left(k\right)=\Lambda_k\end{array}\right\}\right)
\end{align*}
Notice that for all $n\in\left[2,\infty\right]_\omega$ we have that $\Lambda_n\leq\vartheta\left(n-1\right)-1$ as 
\begin{align*}
\forall n\in\left[2,\infty\right]_\omega,\forall\left[a\right]\in S,0\leq a_p\left(n\right)\leq\vartheta\left(n-1\right)-1
\end{align*}
Take the function $\xi:\omega\rightarrow\omega$ defined as 
\begin{align*}
\xi\left(k\right)=\begin{cases}1&k=0\\\Lambda_n&n\in\omega_+
\end{cases}
\end{align*} As $\xi$ is an auxiliary function, it is finite. Now, we have two cases. Define 
\begin{align*}
\Phi:=\min\left(\left\{T\in\omega_+\middle|\forall t\in\left[T+1,\infty\right]_\omega,\xi\left(t\right)=\vartheta\left(t-1\right)-1\right\}\right)
\end{align*}
The first case is when $\Phi=\emptyset$, and the second case is when not.
\subparagraph{$\Phi=\emptyset$:} We see that $\xi=\xi_p$. Therefore we have $\left[\xi\right]\in\Xi_S$ as $\forall\left[a\right]\in S,\left[a\right]\leq\left[\xi\right]$. Suppose not. Then 
\begin{align*}\exists\left[a\right]\in S,\left[\xi\right]<\left[a\right]\text{ implies }\exists\left[a\right]\in S,\xi\left(\Psi\left(\xi,a_p\right)\right)>a_p\left(\Psi\left(\xi,a_p\right)\right)
\end{align*}
Now, we see that this is a contradiction, as we have 
\begin{align*}
\begin{array}{ll@{}}
\forall k\in\left[1,\Psi\left(\xi,a_p\right)-1\right]_\omega,\xi\left(k\right)=a_p\left(k\right)&\text{and}\\
\Lambda_{\Psi\left(\xi,a_p\right)}=\xi\left(\Psi\left(\xi,a_p\right)\right)>a_p\left(\Psi\left(\xi,a_p\right)\right)
\end{array}
\end{align*} which is a contradiction against the definition of $\Lambda_{\Psi\left(\xi,a_p\right)}$.\\
Furthermore, $\forall\left[x\right]\in\Xi_S,\left[\xi\right]\leq\left[x\right]$. Suppose not. The presumption $\left[x\right]<\left[\xi\right]$ leads to $\sigma\left(\left[x\right]\right)=0$ and $x_p\left(\Psi\left(x_p,\xi\right)\right)>\xi\left(\Psi\left(x_p,\xi\right)\right)$. If that were the case, 
\begin{align*}
\exists \left[a\right]\in S,\begin{array}{ll@{}}&\forall k\in\left[1,\Psi\left(x_p,\xi\right)-1\right]_\omega,a_p\left(k\right)=\xi\left(k\right)=x_p\left(k\right)\\\wedge&x_p\left(\Psi\left(x_p,\xi\right)\right)>\xi\left(\Psi\left(x_p,\xi\right)\right)=a_p\left(\Psi\left(x_p,\xi\right)\right)\end{array}
\end{align*}
This implies $x_p\left(\Psi\left(x_p,a_p\right)\right)>a_p\left(\Psi\left(x_p,a_p\right)\right)\implies \left[x\right]<\left[a\right]$ and thus $\left[x\right]\not\in\Xi_S$.
\subparagraph{$\Phi\neq\emptyset$:} We see that this is impossible. If it were the case, then there exists $T\in\omega_+$ such that for all $t\in\left[T+1,\infty\right]_\omega$ we have $\Lambda_t=\vartheta\left(t-1\right)-1$. Therefore, for all $t\in\left[T+1,\infty\right]_\omega$ we have that for every $\left[a\right]\in S$ such that for all $n\in\left[1,t-1\right]_\omega,a_p\left(n\right)=\Lambda_n$ we have that $\vartheta(t-1)-1\leq a_p\left(t\right)$ and thus $a_p\left(t\right)=\vartheta\left(t-1\right)-1$. This however shows, that for any $a_p$ such that $\forall n\in\left[1,T\right]_\omega,a_p\left(n\right)=\Lambda_n$ we have $\forall t\in\left[T+1,\infty\right]_\omega,a_p\left(t\right)=\vartheta\left(t-1\right)-1$, which is a contradiction to being the primary auxiliary function. Therefore, 
\begin{align*}
\forall T\in\omega_+,\exists K\in\left[T+1,\infty\right]_\omega,\xi\left(K\right)=\Lambda_K<\vartheta\left(K-1\right)-1
\end{align*}
Ergo, $\Phi=\emptyset$.\\
Hence, we have shown that for any $S$ bounded from above there exists $\left[\xi\right]$, the lowest upper bound of $S$.
\end{proof}
\begin{axiom}[Completeness Axiom from below]\noindent\\
If $S$ is a non-empty set of real numbers and if $S$ has a lower bound, then $S$ has the greatest lower bound.
\end{axiom}
\begin{proof} Let a set of real numbers $S$ be given. Furthermore, suppose $S$ that has a lower bound. Define the set of lower bounds $\Xi_S:=\left\{\left[x\right]\in\mathbb{R}\middle|\forall\left[a\right]\in S,\left[x\right]\leq\left[a\right]\right\}$. We see that $\Xi_S$ is bounded from above (by any element of $S$). By the \nameref{Completeness Axiom from above} \ref{Completeness Axiom from above} there exists $\left[\xi\right]\in\mathcal{S}_{\mathbb{R}}$ such that
\begin{align*}
\begin{array}{ll@{}}
\forall\left[x\right]\in\Xi_S,\left[x\right]\leq\left[\xi\right]&\text{and}\\
\forall\left[y\right]\in\mathcal{S}_{\mathbb{R}},\left(\forall\left[x\right]\in\Xi_S,\left[x\right]\leq\left[y\right]\right)\implies\left[\xi\right]\leq\left[y\right]
\end{array}
\end{align*}
Thus as $\forall\left[y\right]\in S,\left(\forall\left[x\right]\in\Xi_S,\left[x\right]\leq\left[y\right]\right)$ we have $\forall\left[y\right]\in S,\left[\xi\right]\leq\left[y\right]$. Hence $\xi$ is a lower bound. And as $\forall\left[x\right]\in\Xi_S,\left[x\right]\leq\left[\xi\right]$ we have that $\xi$ is the greatest lower bound.
\end{proof}

\begin{remark} As always, we will denote the least upper bound as the $\sup\left(S\right)$ and the greatest lower bound as $\inf\left(S\right)$.
\end{remark}
\section{Addition}\label{sec:addition}
Unreduced digits earn their keep here. Because we never required a digit to stay below its base, two numbers of the same sign may be added position by position with no carrying at all, and the tidying up is deferred to a contraction afterwards, so addition is genuinely a digitwise operation, which it is not in any system that insists on reduced digits. The price is paid twice over elsewhere. Numbers of opposite sign cannot be combined this way, and we must first rewrite the larger so that it dominates the smaller everywhere, which is the role of the tracking functions; and having defined everything on representatives we must show that the result depends only on the classes, which is what the consistency theorems do and why they are the longest arguments in the section.
\subsection{Sub-Addition of Functions of the Same Sign}
In this part we are to define addition and multiplication of the real numbers. To do this we first define addition and multiplication of finite functions, which we shall call sub-addition and sub-multiplication. We will then show that the resulting sub-sums and sub-products do not depend on the choice of the functions of the summed or multiplied real numbers.
\begin{definition}[Sub-Addition of Functions of the same Sign]\noindent\\
Let $f,g\in \mathcal{N}_{\mathcal{F}in}^*$ and $\sigma\left(f\right)=\sigma\left(g\right)$. Define sub-addition as:
\begin{align*}\left(f+g\right)\left(k\right)=
\begin{cases}
f\left(0\right)g\left(0\right)&k=0\\
f\left(k\right)+g\left(k\right)&k\in\omega_+
\end{cases}
\end{align*}
\end{definition}

\begin{remark} Notice that by the commutativity of addition and multiplication of finite ordinals we have that sub-addition of functions is commutative.
\end{remark}
\begin{example} Continue in base ten and let $f=g=(0;0,5,0,0,\dots)$, the primary auxiliary function of $\nicefrac{1}{2}$ from the previous example. Both are positive, so sub-addition applies in its same-sign form. At position $0$ we multiply, giving $(f+g)(0)=0*0=0$, which is even and so records a positive result; at every other position we simply add, giving $(f+g)(2)=5+5=10$ and $(f+g)(n)=0$ elsewhere. Hence
\begin{align*}
f+g=(0;0,10,0,0,\dots).
\end{align*}
Note that this is a perfectly legitimate function even though the digit $10$ has reached its base: nothing in our definitions forbids it, and this is precisely what makes addition a digitwise operation.\\
To recover the primary auxiliary function we contract. Since $(f+g)(2)=10\geq\vartheta(1)$, the contraction $c_1$ applies, and it gives $c_1(f+g)=(0;1,0,0,\dots)$, a primary auxiliary function of $1$. So $\left[\nicefrac{1}{2}\right]+\left[\nicefrac{1}{2}\right]=\left[1\right]$, with the addition done digit by digit and the carrying done afterwards, separately.
\end{example}
\begin{lemma}[First Sub-Addition Finiteness Lemma]\label{Sub-Addition Finiteness Lemma}\noindent\\
For all non-zero finite functions $f$ and $g$ such that $\sigma\left(f\right)=\sigma\left(g\right)$ we have that $f+g$ is finite.
\end{lemma}

\begin{intuition}[Finiteness and Consistency of Sub-Addition] These lemmas discharge the two obligations any operation on real numbers must meet. Finiteness: the sum of two genuine numbers is again a genuine number, it does not run off to something with infinite material. Consistency: the sum depends only on the two real numbers, not on the particular functions we picked to represent them. The auxiliary lemmas are the bookkeeping steps that get us there, tracking how the primary and secondary auxiliary functions of the summands feed into the primary and secondary auxiliary functions of the sum.
\end{intuition}

\begin{proof} Suppose that $f,g\in \mathcal{N}_{\mathcal{F}in}^*$. Then by the \nameref{Tying Theorem} \ref{Tying Theorem} there exist constants $C_1,C_2\in\omega_+$ such that for all $K\in\omega_+$ we have 
\begin{align*}
\begin{array}{ll@{}}
\sum\limits_{n=1}^K\left(f\left(n\right)\prod\limits_{p=n}^{K-1}\left(\vartheta\left(p\right)\right)\right)<C_1\prod\limits_{p=1}^{K-1}\left(\vartheta\left(p\right)\right)&\text{and}\\
\sum\limits_{n=1}^K\left(g\left(n\right)\prod\limits_{p=n}^{K-1}\left(\vartheta\left(p\right)\right)\right)<C_2\prod\limits_{p=1}^{K-1}\left(\vartheta\left(p\right)\right)
\end{array}
\end{align*}
We see that for all $K\in\omega_+$
\begin{align*}
&\sum\limits_{n=1}^K\left(\left(f+g\right)\left(n\right)\right)\prod\limits_{p=n}^{K-1}\left(\vartheta\left(p\right)\right)=\sum\limits_{n=1}^K\left(f\left(n\right)+g\left(n\right)\right)\prod\limits_{p=n}^{K-1}\left(\vartheta\left(p\right)\right)\\
=&\sum\limits_{n=1}^K\left(f\left(n\right)\prod\limits_{p=n}^{K-1}\left(\vartheta\left(p\right)\right)\right)+\sum\limits_{n=1}^K\left(g\left(n\right)\prod\limits_{p=n}^{K-1}\left(\vartheta\left(p\right)\right)\right)\\
<&C_1\prod\limits_{p=1}^{K-1}\left(\vartheta\left(p\right)\right)+C_2\prod\limits_{p=1}^{K-1}\left(\vartheta\left(p\right)\right)=\left(C_1+C_2\right)\prod\limits_{p=1}^{K-1}\left(\vartheta\left(p\right)\right)
\end{align*}
Therefore, by the \nameref{Tying Theorem} \ref{Tying Theorem}, the sub-addition of two non-zero finite functions of the same sign is a finite function.
\end{proof}
\begin{lemma}[Sub-Addition Contraction Lemma]\label{Sub-Addition Contraction Lemma}\noindent\\
For all non-zero finite functions $f$ and $g$ such that $f$ is contractable about some $m$ and $\sigma\left(f\right)=\sigma\left(g\right)$ we have that 
\begin{align}\label{eq:subadd-carry}
    \left[f+g\right]=\left[c_m\left(f\right)+g\right]
\end{align}
\end{lemma}

\begin{intuition}[Sub-Addition Contraction Lemma] This lemma is the first and key step: carrying in one summand and then adding gives the same class as adding first and then carrying. Put differently, a carry commutes with addition, so nudging a representative to an equivalent one leaves the sum in the same real number.
\end{intuition}

\begin{proof} We shall prove the lemma by showing that if $f$ is contractable about some $m$ then $\left(f+g\right)$ is contractable about $m$ and $\left(c_m\left(f\right)+g\right)=c_m\left(f+g\right)$.\\
We see that if $f\left(m+1\right)\geq\vartheta\left(m\right)$ then $\left(f+g\right)\left(m+1\right)=f\left(m+1\right)+g\left(m+1\right)\geq\vartheta\left(m\right)$. Thus $\left(f+g\right)$ is contractable about $m$. Moreover,
\begin{align*}
\hspace{-0.5cm}\left(c_m\left(f\right)+g\right)\left(k\right)=
\begin{cases}
c_m\left(f\right)\left(0\right)g\left(0\right)&=f\left(0\right)g\left(0\right)\\&=\left(f+g\right)\left(0\right)\\&=c_m\left(f+g\right)\left(0\right)\\
c_m\left(f\right)\left(m\right)+g\left(m\right)&=f\left(m\right)+1+g\left(m\right)\\&=\left(f+g\right)\left(m\right)+1\\&=c_m\left(f+g\right)\left(m\right)\\
c_m\left(f\right)\left(m+1\right)+g\left(m+1\right)&=f\left(m+1\right)-\vartheta\left(m\right)+g\left(m+1\right)\\&=\left(f+g\right)\left(m+1\right)-\vartheta\left(m\right)\\&=c_m\left(f+g\right)\left(m+1\right)\\
c_m\left(f\right)\left(k\right)+g\left(k\right)&=f\left(k\right)+g\left(k\right)\\&=\left(f+g\right)\left(k\right)\\&=c_m\left(f+g\right)\left(k\right)
\end{cases}
\def\arraystretch{1.2}\begin{array}{@{}l}
\text{\indent if $k=0$}\\\\\\
\text{\indent if $k=m$}\\\\\\
\text{\indent if $k=m+1$}\\\\\\
\text{\indent otherwise}
\end{array}
\end{align*}
Therefore, $\left(c_m\left(f\right)+g\right)=c_m\left(f+g\right)$. And hence, by the \nameref{Shifting Theorem} \ref{Shifting Theorem} $\left[f+g\right]=\left[c_m\left(f+g\right)\right]=\left[c_m\left(f\right)+g\right]$.
\end{proof}

\begin{remark} Furthermore, by the commutativity of sub addition we have $\left[f+c_m\left(g\right)\right]=\left[c_m\left(g\right)+f\right]=\left[g+f\right]=\left[f+g\right]$.
\end{remark}

\begin{corollary} Suppose that $f$ and $g$ are non-zero finite functions of the same sign where $f$ is broadenable about some $m\in\omega_+$. Then by the \nameref{Cancellation Theorems} \ref{Cancellation Theorems} we have that $\left[b_m\left(f\right)+g\right]=\left[c_m\left(b_m\left(f\right)\right)+g\right]=\left[f+g\right]$.\\
Furthermore, suppose that $g$ is broadenable about $m$, then by the commutativity of sub-addition we have $\left[f+b_n\left(g\right)\right]=\left[b_n\left(g\right)+f\right]=\left[g+f\right]=\left[f+g\right]$.
\end{corollary}

\begin{lemma}[First Sub-Addition Auxiliary Lemma]\label{First Sub-Addition Auxiliary Lemma}\noindent\\
For all non-zero finite functions $f$ and $g$ such that $\sigma\left(f\right)=\sigma\left(g\right)$ we have that $\left[f+g\right]=\left[h_f+h_g\right]$.
\end{lemma}
\begin{proof} By induction on \eqref{eq:subadd-carry} (and the identical identity for broadenings which we have not shown outright) we see that for any two sequences of contractions and broadenings ${}^f\mathcal{D}_{k=q_1}^{r_1}$ and ${}^g\mathcal{D}_{k=q_2}^{r_2}$ applicable to $f$ and $g$ respectively we have:\footnote{Notice that the inverse does not necessarily hold.}
\begin{align*}
\left({}^f\mathcal{D}_{k=q_1}^{r_1}\left(f\right)+{}^g\mathcal{D}_{k=q_2}^{r_2}\left(g\right)\right)={}^f\mathcal{D}_{k=q_1}^{r_1}\left({}^g\mathcal{D}_{k=q_2}^{r_2}\left(f+g\right)\right)
\end{align*}
Moreover, we know, by the definition of a real numbers that for all $N\in\omega_+$ we have
\begin{align*}
\begin{array}{lll@{}}
\exists{}^f\mathcal{C}_{k=q_1}^{r_1}\in\mathcal{FS},\forall n\in\left[0,N\right]_\omega,&{}^f\mathcal{C}_{k=q_1}^{r_1}\left(f\right)\left(n\right)=h_f\left(n\right)&\text{and}\\
\exists{}^g\mathcal{C}_{k=q_2}^{r_2}\in\mathcal{FS},\forall n\in\left[0,N\right]_\omega,&{}^g\mathcal{C}_{k=q_2}^{r_2}\left(g\right)\left(n\right)=h_g\left(n\right)
\end{array}
\end{align*}
 We also know that 
\begin{align*}
\begin{array}{ll@{}}
\exists{}^\pi\mathcal{C}_{k=q_3}^{r_3}\in\mathcal{FS},\forall n\in\left[0,N\right]_\omega,&{}^\pi\mathcal{C}_{k=q_3}^{r_3}\left(h_f+h_g\right)\left(n\right)=\pi\left(n\right)
\end{array}
\end{align*}
where $\pi$ is the auxiliary function of $\left(h_f+h_g\right)$.\\
Thus, let $N\in\omega$ be given and take $\Xi:=\max\left(\text{con}\left({}^\pi\mathcal{C}_{k=q_3}^{r_3}\right)\cup\left\{N\right\}\right)$. Now take the number of inspection for $f$ and $g$ to be $\Xi$ and take their sequences of contractions as mentioned above. From this we see that $\forall n\in\left[0,\Xi\right]_\omega$:
\begin{align*}
&{}^f\mathcal{C}_{k=q_1}^{r_1}\left({}^g\mathcal{C}_{k=q_2}^{r_2}\left(f+g\right)\right)=\left({}^f\mathcal{C}_{k=q_1}^{r_1}\left(f\right)+{}^g\mathcal{C}_{k=q_2}^{r_2}\left(g\right)\right)=\left(h_f+h_g\right)\left(n\right)
\end{align*}
Furthermore, as they agree up to $\Xi$ by the \nameref{Submission Theorem} \ref{Submission Theorem} ${}^\pi\mathcal{C}_{k=q_3}^{r_3}$ is applicable to ${}^f\mathcal{C}_{k=q_1}^{r_1}\left({}^g\mathcal{C}_{k=q_2}^{r_2}\left(f+g\right)\right)$ and we have that for all $ n\in\left[0,N\right]_\omega$
\begin{align*}
&{}^\pi\mathcal{C}_{k=q_3}^{r_3}\left({}^f\mathcal{C}_{k=q_1}^{r_1}\left({}^g\mathcal{C}_{k=q_2}^{r_2}\left(f+g\right)\right)\right)\left(n\right)+\left(h_f+h_g\right)\left(n\right)\\
=&{}^f\mathcal{C}_{k=q_1}^{r_1}\left({}^g\mathcal{C}_{k=q_2}^{r_2}\left(f+g\right)\right)\left(n\right)+{}^\pi\mathcal{C}_{k=q_3}^{r_3}\left(h_f+h_g\right)\left(n\right)\\
=&\left(h_f+h_g\right)\left(n\right)+\pi\left(n\right)
\end{align*}
And therefore 
\begin{align*}
&{}^\pi\mathcal{C}_{k=q_3}^{r_3}\left({}^f\mathcal{C}_{k=q_1}^{r_1}\left({}^g\mathcal{C}_{k=q_2}^{r_2}\left(f+g\right)\right)\right)\left(n\right)=\pi\left(n\right)
\end{align*}
Therefore, by taking ${}^{f+g}\mathcal{C}\in\mathcal{CB}$ such that
\begin{align*}
\hspace{-0.5cm}
{}^{f+g}\mathcal{C}\left(k\right)=\begin{cases}
{}^g\mathcal{C}\left(k\right)&k\in\left[q_2,r_2\right]_\omega\\
{}^f\mathcal{C}\left(k+q_1-r_2-1\right)&k\in\left[r_2+1,r_1+r_2+1-q_1\right]_\omega\\
{}^\pi\mathcal{C}\left(k+q_3-r_1-r_2-2+q_1\right)&k\in\left[r_1+r_2+2-q_1,r_3+r_2+r_1+2-q_1-q_3\right]_\omega\\
c_1&otherwise
\end{cases}
\end{align*}
we have that for all $n\in\left[0,N\right]_\omega$ we have
\begin{align*}
{}^{f+g}\mathcal{C}_{k=q_2}^{r_3+r_2+r_1+2-q_1-q_3}\left(f+g\right)\left(n\right)=\pi\left(n\right)
\end{align*}
This shows by \nameref{Unicity of Real Numbers} \ref{Unicity of Real Numbers} that $\left[f+g\right]=\left[h_f+h_g\right]$.
\end{proof}

\begin{lemma}[Second Sub-Addition Auxiliary Lemma]\label{Second Sub-Addition Auxiliary Lemma}\noindent\\
For all non-zero finite functions $f$ and $g$ such that $\sigma\left(f\right)=\sigma\left(g\right)$ we have that $\left[f_p'+h_g\right]=\left[f_s'+h_g\right]$ where $f_p'$ and $f_s'$ are the primary and secondary auxiliary functions of $\left[f\right]$ respectively such that $f_p'\left(0\right)=f_s'\left(0\right)=f\left(0\right)$ (if there are any secondary auxiliary functions).
\end{lemma}
\begin{proof} By the \nameref{Sub-Addition Finiteness Lemma} \ref{Sub-Addition Finiteness Lemma}, we know that $\left(f_p'+h_g\right),\left(f_s'+h_g\right)$ are both finite. Furthermore, $\left(f_p'+h_g\right)\left(0\right)=\left(f_s'+h_g\right)\left(0\right)$ as $f_p'\left(0\right)=f_s'\left(0\right)$. Let $N\in\omega_+$ be given. Take $M=\max\left(\left\{N+1,\Psi\left(f_p',f_s'\right)\right\}\right)$. Now, we see that by \nameref{PC Cor I} \ref{PC Cor I} we have that for all $m\in\left[M,\infty\right]_\omega$:
\begin{align*}
&\sum\limits_{n=1}^{m}\left(\left(f_p'+h_g\right)\left(n\right)\prod\limits_{j=n}^{m-1}\left(\vartheta\left(j\right)\right)\right)=\sum\limits_{n=1}^{m}\left(f_p'\left(n\right)\prod\limits_{j=n}^{m-1}\left(\vartheta\left(j\right)\right)\right)+\sum\limits_{n=1}^{m}\left(h_g\left(n\right)\prod\limits_{j=n}^{m-1}\left(\vartheta\left(j\right)\right)\right)\\
=&\sum\limits_{n=1}^{m}\left(f_s'\left(n\right)\prod\limits_{j=n}^{m-1}\left(\vartheta\left(j\right)\right)\right)+1+\sum\limits_{n=1}^{m}\left(h_g\left(n\right)\prod\limits_{j=n}^{m-1}\left(\vartheta\left(j\right)\right)\right)=\sum\limits_{n=1}^{m}\left(\left(f_s'+h_g\right)\left(n\right)\prod\limits_{j=n}^{m-1}\left(\vartheta\left(j\right)\right)\right)+1
\end{align*}
Therefore,
\begin{align*}
\Upsilon\left(\begin{array}{l@{}}\sum\limits_{n=1}^{m}\left(\left(f_p'+h_g\right)\left(n\right)\prod\limits_{j=n}^{m-1}\left(\vartheta\left(j\right)\right)\right)\\\sum\limits_{n=1}^{m}\left(\left(f_s'+h_g\right)\left(n\right)\prod\limits_{j=n}^{m-1}\left(\vartheta\left(j\right)\right)\right)\end{array}\right)=1<\vartheta\left(N\right)\leq\prod\limits_{j=N}^{m-1}\left(\vartheta\left(j\right)\right)
\end{align*}
Ergo, by the \nameref{Extension Theorem} \ref{Extension Theorem}, $\left[f_p'+h_g\right]=\left[f_s'+h_g\right]$.
\end{proof}

\begin{lemma}[Third Sub-Addition Auxiliary Lemma]\label{Third Sub-Addition Auxiliary Lemma}\noindent\\
For all auxiliary functions $f,g$ of the same non-zero-real number and for all non-zero auxiliary functions $h$ of the same sign as $f$ and $g$ we have $\left[f+h\right]=\left[g+h\right]$.
\end{lemma}
\begin{proof} Suppose that $\forall n\in\omega_+,f\left(n\right)=g\left(n\right)$. Then $\sigma\left(f\right)=\sigma\left(g\right)$ and thus $f\left(0\right)+g\left(0\right)=2m$ for some $m\in\omega$. Furthermore, $\forall n\in\omega_+,\left(f+h\right)\left(n\right)=\left(g+h\right)\left(n\right)$ and thus both of the functions have the same auxiliary function up to a difference for $0$. We shall see what this difference actually is. 
\begin{align*}
\left(f+h\right)\left(0\right)+\left(g+h\right)\left(0\right)=f\left(0\right)h\left(0\right)+g\left(0\right)h\left(0\right)=\left(f\left(0\right)+g\left(0\right)\right)h\left(0\right)=2mh\left(0\right)
\end{align*} 
Therefore, by the \nameref{Equivalence Lemma} \ref{Equivalence Lemma} both the auxiliary functions are in the same real number and therefore $\left[f+h\right]=\left[g+h\right]$.\\
Suppose that $\exists n\in\omega_+,f\left(n\right)\neq g\left(n\right)$. This means that one of the functions is a primary auxiliary function and the other one is a secondary auxiliary function. Take without loss of generality $f$ to be the primary auxiliary function and $g$ to be the secondary auxiliary function. By the \nameref{Second Sub-Addition Auxiliary Lemma} \ref{Second Sub-Addition Auxiliary Lemma} $\left[g+h\right]=\left[g_p'+h\right]$ where $g_p'$ is the primary auxiliary function of $\left[g\right]=\left[f\right]$ with $g_p'\left(0\right)=g\left(0\right)$. By the case above $\left[g_p+h\right]=\left[f+h\right]$ and thus $\left[f+h\right]=\left[g+h\right]$.
\end{proof}
\begin{theorem}[Consistency Theorem]\label{Consistency Theorem}\noindent\\
\begin{subtheorems}
\subtheorem\label{1.I} For all non-zero finite functions $f$, $g$ and $w$ such that $\sigma\left(f\right)=\sigma\left(g\right)=\sigma\left(w\right)$ we have that
\begin{align*}
\left[f+g\right]=\left[f+w\right]\iff\left[g\right]=\left[w\right]
\end{align*}
\subtheorem\label{1.II} For all non-zero finite functions $f$, $g$, $w$ and $t$ such that $\sigma\left(f\right)=\sigma\left(w\right)$ we have that 
\begin{align*}
\left[f\right]=\left[g\right]\wedge\left[w\right]=\left[t\right]\implies\left[f+w\right]=\left[g+t\right]
\end{align*}
\subtheorem\label{1.III} For all non-zero finite functions $f$, $g$, $w$ and $t$ such that $\sigma\left(f\right)=\sigma\left(w\right)$ and $\sigma\left(g\right)=\sigma\left(t\right)$ we have that 
\begin{align*}
\left[f+w\right]=\left[g+t\right]\wedge\left[w\right]=\left[t\right]\implies\left[f\right]=\left[g\right]
\end{align*}
\end{subtheorems}
\end{theorem}
\begin{proof}[Proof of I]\noindent\\
$\left(\impliedby\right):$\\
Let $f,g,w\in\mathcal{N}_{\mathcal{F}in}^*$ such that $\sigma\left(f\right)=\sigma\left(g\right)=\sigma\left(w\right)$ be given. By the \nameref{First Sub-Addition Auxiliary Lemma} \ref{First Sub-Addition Auxiliary Lemma}, we have that $\left[f+g\right]=\left[h_f+h_g\right]$ and $\left[f+w\right]=\left[h_f+h_w\right]$. By commutativity of sub-addition and the \nameref{Third Sub-Addition Auxiliary Lemma} \ref{Third Sub-Addition Auxiliary Lemma} we have 
\begin{align*}
\left[f+g\right]=\left[h_f+h_g\right]=\left[h_g+h_f\right]=\left[h_w+h_f\right]=\left[h_f+h_w\right]=\left[f+w\right]
\end{align*}
$\left(\implies\right):$\\
Suppose $\left[w\right]\neq\left[g\right]$. By totality $\left[w\right]\leq\left[g\right]$ or $\left[w\right]\geq\left[g\right]$. Suppose without loss of generality that $\left[w\right]\leq\left[g\right]$.\\
If $\sigma\left(\left[w\right]\right)=\sigma\left(\left[g\right]\right)=0$ then by the definition of order 
\begin{align*}
\forall g\in\left[g\right],\exists a\in\left[w\right],\forall n\in\omega_+,a\left(n\right)\geq g\left(n\right)
\end{align*}
With the extra caveat, as $\sigma\left(\left[w\right]\right)=\sigma\left(\left[g\right]\right)$ if for all $n\in\omega_+$ we have $a\left(n\right)=g\left(n\right)$ then by the \nameref{Equivalence Lemma} \ref{Equivalence Lemma} that as there xists $m\in\omega_+$ such that $a\left(0\right)+g\left(0\right)=2m$ we would obtain $\left[g\right]=\left[a\right]=\left[w\right]$ which is a contradiction. Therefore, there has to exist $M\in\omega_+$ such that $a\left(M\right)>g\left(M\right)$.\\
This implies that for all $n\in\omega_+$ we have
\begin{align*}
\left(f+a\right)\left(n\right)=f\left(n\right)+a\left(n\right)\geq f\left(n\right)+g\left(n\right)=\left(f+g\right)\left(n\right)
\end{align*}
On top of that we have
\begin{align*}
\left(f+a\right)\left(M\right)=f\left(M\right)+a\left(M\right)<f\left(M\right)+g\left(M\right)=\left(f+g\right)\left(M\right)
\end{align*}
By the \nameref{Domination Theorem} \ref{Domination Theorem} we see that $\left[f+w\right]=\left[f+a\right]\neq\left[f+g\right]$.\\
The proof of $\sigma\left(\left[w\right]\right)=\sigma\left(\left[g\right]\right)=1$ is the same as the first case.\\
Ergo $\left[w\right]\neq\left[g\right]\implies\left[f+g\right]\neq\left[f+w\right]$. By contrapositive we have $\left[f+g\right]=\left[f+w\right]\implies\left[g\right]=\left[w\right]$
\end{proof}
\begin{proof}[Proof of II]\noindent\\
Let $f,g,w,t\in\mathcal{N}_{\mathcal{F}in}^*$ such that $\sigma\left(f\right)=\sigma\left(w\right)$ and $\left[f\right]=\left[g\right]$ and $\left[w\right]=\left[t\right]$ be given. By using the \nameref{Consistency Theorem} \ref{1.I} \ref{Consistency Theorem} twice we obtain 
\begin{align*}
\begin{array}{ll@{}}
\left[w\right]=\left[t\right]\implies\left[f+w\right]=\left[f+t\right]&\text{and}\\
\left[f\right]=\left[g\right]\implies\left[t+f\right]=\left[t+g\right]
\end{array}
\end{align*} As sub-addition is commutative $\left[t+f\right]=\left[f+t\right]$. Thus using transitivity of equality we obtain $\left[f+w\right]=\left[f+t\right]=\left[t+f\right]=\left[t+g\right]=\left[g+t\right]\implies\left[f+w\right]=\left[g+t\right]$.
\end{proof}
\begin{proof}[Proof of III]\noindent\\
Let $f,g,w,t\in\mathcal{N}_{\mathcal{F}in}^*$ such that $\sigma\left(f\right)=\sigma\left(w\right)$ and $\sigma\left(g\right)=\sigma\left(t\right)$ and $\left[f+w\right]=\left[g+t\right]$ and $\left[w\right]=\left[t\right]$ be given. By the \nameref{Consistency Theorem} \ref{1.I} \ref{Consistency Theorem} 
\begin{align*}
\left[w\right]=\left[t\right]\implies\left[f+w\right]=\left[f+t\right]
\end{align*} by the hypothesis $\left[g+t\right]=\left[f+w\right]$.\\
Therefore, we have $\left[g+t\right]=\left[f+t\right]$. By the commutativity of sub-addition and the \nameref{Consistency Theorem} \ref{1.I} \ref{Consistency Theorem} we have $\left[t+g\right]=\left[t+f\right]\implies\left[f\right]=\left[g\right]$.
\end{proof}
\subsection{Absolute Value, Opposites and Tracking Functions}
\begin{definition}[Absolute Value of a Function]\label{Absolute Value of Functions}\noindent\\
For any $f:\omega\rightarrow\omega$ we define the absolute value of $f$, written as $\left|f\right|$, as 
\begin{align*}
\left|f\right|\left(k\right):=\begin{cases}
0&k=0\\
f\left(k\right)&otherwise
\end{cases}
\end{align*}
\end{definition}
\begin{lemma}[Absolute Value Lemma]\label{Absolute Value Lemma}\noindent\\
For all finite functions $f$ and for all functions $g\in\left[f\right]$ we have $\left|g\right|\in\left[\left|f\right|\right]$.
\end{lemma}
\begin{proof} First we shall prove that $\left|g\right|\in\left[\left|h_g\right|\right]$. By the definition of real numbers, $\left|g\right|\left(0\right)=0=\left|h_g\right|\left(0\right)$. Furthermore, let a number of inspection $N\in\omega_+$ be given. We know that there exists a sequence of contractions $\mathcal{C}_{k=q}^r\in\mathcal{FS}$ such that for all $n\in\left[0,N\right]_\omega$ we have
\begin{align*}
\mathcal{C}_{k=q}^r\left(g\right)\left(n\right)=h_g\left(n\right)
\end{align*}
Therefore, as for all $n\in\omega_+$ we have $\left|g\right|\left(n\right)=g\left(n\right)$ and thus by the \nameref{Submission Theorem} \ref{Submission Theorem} we have that $\mathcal{C}_{k=q}^r$ is applicable to $\left|g\right|$ and moreover for all $n\in\left[1,N\right]_\omega$ we have
\begin{align*}
\mathcal{C}_{k=q}^r\left(\left|g\right|\right)\left(n\right)+g\left(n\right)=\left|g\right|\left(n\right)+\mathcal{C}_{k=q}^r\left(g\right)\left(n\right)=\left|g\right|\left(n\right)+h_g\left(n\right)=g\left(n\right)+\left|h_g\right|\left(n\right)
\end{align*}
Hence, for all $n\in\left[0,N\right]_\omega$ we have $\mathcal{C}_{k=q}^r\left(\left|g\right|\right)\left(n\right)=\left|h_g\right|\left(n\right)$ and thus $\left|g\right|\in\left[\left|h_g\right|\right]$.\\
Now we shall prove that for all auxiliary functions $h_1,h_2\in\left[f\right]$ where $\left[f\right]\neq\left[0\right]$ we have $\left[\left|h_1\right|\right]=\left[\left|h_2\right|\right]$. We have two cases.\\
\textbf{For all $n\in\omega_+,h_1\left(n\right)=h_2\left(n\right)$.} We have $\left|h_1\right|\left(0\right)+\left|h_2\right|\left(0\right)=0=2*0$. Furthermore, as for all $n\in\omega_+,h_1\left(n\right)=h_2\left(n\right)$, by the \nameref{Equivalence Lemma} \ref{Equivalence Lemma} we have that $\left[\left|h_1\right|\right]=\left[\left|h_2\right|\right]$.\\
\textbf{Exists $n\in\omega_+,h_1\left(n\right)\neq h_2\left(n\right)$}. This means that one of the functions is a primary auxiliary function and the second one is a secondary auxiliary function. Take without loss of generality that $h_1$ is the primary auxiliary function and $h_2$ is the secondary auxiliary function. We take the secondary auxiliary function $h'$ such that $\left(h_1,h'\right)\in W_1$ (as described in the \nameref{Base Set} \ref{Base Set}). Therefore, $\left(\left|h_1\right|,\left|h'\right|\right)\in W_1$. By the definition of real numbers we have that $\left[\left|h_1\right|\right]=\left[\left|h'\right|\right]$. Furthermore, by the same case as above $\left[\left|h'\right|\right]=\left[\left|h_2\right|\right]$. Therefore, $\left[\left|h_1\right|\right]=\left[\left|h_2\right|\right]$.\\
For all auxiliary functions $h_1,h_2\in\left[f\right]$ where $\left[f\right]=\left[0\right]$ we have $\left[\left|h_1\right|\right]=\left[\left|h_2\right|\right]$. This is because, by the \nameref{Null Lemma} \ref{Null Lemma} we have that for all $n\in\omega_+$ we have $h_1\left(n\right)=h_2\left(n\right)=0$. Therefore, we have that for all $n\in\omega_+$ we have $\left|h_1\right|\left(n\right)=\left|h_2\right|\left(n\right)=0$. Hence, by the \nameref{Null Lemma} \ref{Null Lemma} we have $\left[\left|h_1\right|\right]=\left[\left|h_2\right|\right]=\left[0\right]$.\\
These two parts together finish the proof.
\end{proof}

\begin{definition}[Absolute Value of a Real Number]\label{Absolute Value of a Real Number}\noindent\\
We define the absolute value of a real number $\left[x\right]$ as $\left|\left[x\right]\right|:=\left[\left|f\right|\right]$ for any function $f\in\left[x\right]$.
\end{definition}

\begin{remark} For all non-zero real numbers $\left[f\right]$ we have that $\sigma\left(\left|\left[f\right]\right|\right)=1$. Firstly suppose $\left[f\right]\neq\left[0\right]$. Then by the \nameref{Null Lemma} \ref{Null Lemma} there exists a function $f\in\left[f\right]$ such that there exists a value $n\in\omega_+$ such that $f\left(n\right)\neq0$. Then by the definition of the absolute value of a function $\left|f\right|\left(n\right)\neq0$ and thus by the \nameref{Null Lemma} \ref{Null Lemma} we have $\left|f\right|\not\in\left[0\right]$. Therefore, $\left|\left[f\right]\right|\neq\left[0\right]$. Furthermore, as $\sigma\left(\left|f\right|\right)=1$ we obtain the result.
\end{remark}

\begin{definition}[Opposite Function]\label{Opposite Function}
For all functions $g:\omega\rightarrow\omega$ define the opposite function ${}^og:\omega\rightarrow\omega$ as
\begin{align*}
{}^og\left(k\right):=\begin{cases}
g\left(0\right)+1&k=0\\
g\left(k\right)&otherwise
\end{cases}
\end{align*}
\end{definition}
\begin{remark} For all $g\in \mathcal{N}_{\mathcal{F}in}$ we have 
\begin{align*}
\begin{array}{ll@{}}
{}^o{}^og\left(0\right)=g\left(0\right)+2&\text{and}\\
\forall n\in\omega_+,{}^o{}^og\left(n\right)=g\left(n\right)
\end{array}
\end{align*}.
Thus as ${}^o{}^og\left(0\right)+g\left(0\right)=2\left(g\left(0\right)+1\right)$ we obtain by the \nameref{Equivalence Lemma} \ref{Equivalence Lemma} that $\left[g\right]=\left[{}^o{}^og\right]$.
\end{remark}
\begin{lemma}[Opposite Function Lemma]\label{Opposite Function Lemma}\noindent\\
For all real numbers $\left[x\right]\in\mathcal{S}_{\mathbb{R}}$ we have that for all functions $g,w\in\left[x\right]$ we have $\left[{}^og\right]=\left[{}^ow\right]$.
\end{lemma}
\begin{proof} First we shall prove that ${}^og\in\left[{}^oh_g\right]$\footnote{Notice that by the definition of auxiliary functions ${}^oh_g$ is an auxiliary function and thus finite.}. By the definition of real numbers, ${}^og\left(0\right)=g\left(0\right)+1=h_g\left(0\right)+1={}^oh_g\left(0\right)$. Furthermore, let a number of inspection $N\in\omega_+$ be given. We know that there exists a sequence of contractions $\mathcal{C}_{k=q}^r\in\mathcal{FS}$ such that for all $n\in\left[0,N\right]_\omega$ we have
\begin{align*}
\mathcal{C}_{k=q}^r\left(g\right)\left(n\right)=h_g\left(n\right)
\end{align*}
Therefore, as for all $n\in\omega_+$ we have ${}^og\left(n\right)=g\left(n\right)$ and thus by the \nameref{Submission Theorem} \ref{Submission Theorem} we have that $\mathcal{C}_{k=q}^r$ is applicable to ${}^og$ and moreover for all $n\in\left[1,N\right]_\omega$ we have
\begin{align*}
\mathcal{C}_{k=q}^r\left({}^og\right)\left(n\right)+g\left(n\right)={}^og\left(n\right)+\mathcal{C}_{k=q}^r\left(g\right)\left(n\right)={}^og\left(n\right)+h_g\left(n\right)=g\left(n\right)+{}^oh_g\left(n\right)
\end{align*}
Hence, for all $n\in\left[0,N\right]_\omega$ we have $\mathcal{C}_{k=q}^r\left({}^og\right)\left(n\right)={}^oh_g\left(n\right)$ and thus ${}^og\in\left[{}^oh_g\right]$.\\
Now we shall prove that for all auxiliary functions $h_1,h_2\in\left[x\right]$ where $\left[x\right]\neq\left[0\right]$ we have $\left[{}^oh_1\right]=\left[{}^oh_2\right]$. We have two cases.\\
\textbf{For all $n\in\omega_+,h_1\left(n\right)=h_2\left(n\right)$.} As $\sigma\left(h_1\right)=\sigma\left(h_2\right)$ we have $h_1\left(0\right)+h_2\left(0\right)=2m$ for some $m\in\omega$. Hence ${}^oh_1\left(0\right)+{}^oh_2\left(0\right)=h_1\left(0\right)+h_2\left(0\right)+2=2\left(m+1\right)$. Furthermore, as for all $n\in\omega_+,h_1\left(n\right)=h_2\left(n\right)$, by the \nameref{Equivalence Lemma} \ref{Equivalence Lemma} we have that $\left[{}^oh_1\right]=\left[{}^oh_2\right]$.\\
\textbf{Exists $n\in\omega_+,h_1\left(n\right)\neq h_2\left(n\right)$}. This means that one of the functions is a primary auxiliary function and the second one is a secondary auxiliary function. Take without loss of generality that $h_1$ is the primary auxiliary function and $h_2$ is the secondary auxiliary function. We take the secondary auxiliary function $h'$ such that $\left(h_1,h'\right)\in W_1$ (as described in the \nameref{Base Set} \ref{Base Set}). Therefore, $\left({}^oh_1,{}^oh'\right)\in W_1$. By the definition of real numbers we have that $\left[{}^oh_1\right]=\left[{}^oh'\right]$. Furthermore, by the same case as above $\left[{}^oh'\right]=\left[{}^oh_2\right]$. Therefore, $\left[{}^oh_1\right]=\left[{}^oh_2\right]$.\\
For all auxiliary functions $h_1,h_2\in\left[f\right]$ where $\left[f\right]=\left[0\right]$ we have $\left[{}^oh_1\right]=\left[{}^oh_2\right]$. This is because, by the \nameref{Null Lemma} \ref{Null Lemma} we have that for all $n\in\omega_+$ we have $h_1\left(n\right)=h_2\left(n\right)=0$. Therefore, we have that for all $n\in\omega_+$ we have ${}^oh_1\left(n\right)={}^oh_2\left(n\right)=0$. Hence, by the \nameref{Null Lemma} \ref{Null Lemma} we have $\left[{}^oh_1\right]=\left[{}^oh_2\right]=\left[0\right]$.\\
These two parts together finish the proof.
\end{proof}

\begin{definition}[Opposite of a Real Number]\label{Opposite of a Real Number}\noindent\\
We define the opposite of a real number $\left[x\right]$ as ${}^o\left[x\right]:=\left[{}^of\right]$ for any function $f\in\left[x\right]$.
\end{definition}

\begin{remark} We see that for all non-zero real numbers $\left[x\right]$ we have $\sigma\left(\left[x\right]\right)\neq\sigma\left({}^o\left[x\right]\right)$. This is seen as for any $f\in\left[x\right]$ we have by the \nameref{Null Lemma} \ref{Null Lemma} that there exists a value $n\in\omega_+$ such that $f\left(n\right)\neq0$ and thus by the definition of the opposite function ${}^of\left(n\right)\neq0$ and hence by the \nameref{Null Lemma} \ref{Null Lemma} we have that ${}^of$ is a non-zero finite function (the finiteness follows from the \nameref{Opposite Function Lemma} \ref{Opposite Function Lemma}). Ergo, as $\sigma\left(f\right)\neq\sigma\left({}^of\right)$ we obtain the result.
\end{remark}

\begin{lemma}[Absolute-Opposite Lemma]\label{Absolute-Opposite Lemma}\noindent\\
For all real numbers $\left[a\right]$ and $\left[b\right]$ such that $\left|\left[a\right]\right|=\left|\left[b\right]\right|$ either $\left[a\right]=\left[b\right]$ or $\left[a\right]={}^o\left[b\right]$.
\end{lemma}
\begin{proof} We prove this by contrapositive. Suppose that both $\left[a\right]=\left[b\right]$ and $\left[a\right]=\left[{}^ob\right]$ are false. Then we work with cases.\\
\textbf{Case 1. $\left[a\right]=\left[0\right]$ or $\left[b\right]=\left[0\right]$:} Suppose without loss of generality the former. Then $0_p\in\left[a\right]$ and thus $\left|0_p\right|\in\left|\left[a\right]\right|$ by the \nameref{Absolute Value Lemma} \ref{Absolute Value Lemma}. However, by the definition of an absolute value of a function $\left|0_p\right|=0_p$. Hence, by the \nameref{Unicity of Real Numbers} \ref{Unicity of Real Numbers} we have $\left|\left[a\right]\right|=\left[0\right]$.\\
Suppose $\left[b\right]\neq\left[a\right]=\left[0\right]$. Then by the \nameref{Null Lemma} \ref{Null Lemma} there exists a function $b\in\left[b\right]$ such that there exists a value $n\in\omega_+$ such that $b\left(n\right)\neq0$. Then by the definition of the absolute value of a function $\left|b\right|\left(n\right)\neq0$ and thus by the \nameref{Null Lemma} \ref{Null Lemma} we have $\left|b\right|\not\in\left[0\right]$. Therefore, $\left|\left[b\right]\right|\neq\left[0\right]=\left|\left[a\right]\right|$.\\
\textbf{Case 2. $\sigma\left(\left[a\right]\right)=\sigma\left(\left[b\right]\right)$:} Suppose $\left[a\right]\neq\left[b\right]$. Suppose $\sigma\left(\left[a\right]\right)=1$ Then $\left[a\right]\leq\left[b\right]$ or $\left[a\right]\geq\left[b\right]$. Suppose without loss of generality the former. Take any function $a\in\left[a\right]$. Then there exists a function $b\in\left[b\right]$ such that for all $n\in\omega_+$ we have $a\left(n\right)\leq b\left(n\right)$ and as $\left[a\right]\neq\left[b\right]$ there exists a value $m\in\omega_+$ such that $a\left(m\right)<b\left(m\right)$.\\
This means that for all $n\in\omega_+$ we have $\left|a\right|\left(n\right)\leq\left|b\right|\left(n\right)$ and $\left|a\right|\left(m\right)<\left|b\right|\left(m\right)$. By the \nameref{Domination Theorem} \ref{Domination Theorem} there does not exist any function in $\left|\left[a\right]\right|$ which would dominate $\left|b\right|$ as then it would strictly dominate $\left|a\right|$. Hence, as $\sigma\left(\left|\left[a\right]\right|\right)=\sigma\left(\left|\left[b\right]\right|\right)=1$ we have that $\neg\left(\left|\left[a\right]\right|\geq\left|\left[b\right]\right|\right)$. Suppose that $\left|\left[a\right]\right|=\left|\left[b\right]\right|$ then by the \nameref{Reflexivity of Order} \ref{Reflexivity of Order} $\left|\left[a\right]\right|\geq\left|\left[b\right]\right|$, a contradiction.\\
The proof for the case $\sigma\left(\left[a\right]\right)=0$ is the same.\\
\textbf{Case 3. $\sigma\left(\left[a\right]\right)\neq\sigma\left(\left[b\right]\right)$:} Suppose $\left[a\right]\neq{}^o\left[b\right]$. Suppose $\sigma\left(\left[a\right]\right)=1$. Then $\left[a\right]\leq{}^o\left[b\right]$ or $\left[a\right]\geq{}^o\left[b\right]$. Suppose firstly the former. Take any function $a\in\left[a\right]$. Then there exists a function $b'\in{}^o\left[b\right]$ such that for all $n\in\omega_+$ we have $a\left(n\right)\leq b'\left(n\right)$ and as $\left[a\right]\neq{}^o\left[b\right]$ there exists a value $m\in\omega_+$ such that $a\left(m\right)<b'\left(m\right)$.\\
This means that for all $n\in\omega_+$ we have $\left|a\right|\left(n\right)\leq\left|b'\right|\left(n\right)$ and $\left|a\right|\left(m\right)<\left|b'\right|\left(m\right)$. By the \nameref{Domination Theorem} \ref{Domination Theorem} there does not exist any function in $\left|\left[a\right]\right|$ which would dominate $\left|b'\right|$ as then it would strictly dominate $\left|a\right|$. Hence, as $\sigma\left(\left|\left[a\right]\right|\right)=\sigma\left(\left|\left[b'\right]\right|\right)=1$ we have that $\neg\left(\left|\left[a\right]\right|\geq\left|\left[b'\right]\right|\right)$. Suppose that $\left|\left[a\right]\right|=\left|\left[b'\right]\right|$ then by the \nameref{Reflexivity of Order} \ref{Reflexivity of Order} $\left|\left[a\right]\right|\geq\left|\left[b'\right]\right|$, a contradiction. Thus, $\left|\left[a\right]\right|\neq\left|\left[b'\right]\right|=\left|{}^o\left[b\right]\right|$.\\
Take any $b\in\left[b\right]$. Then for all $n\in\omega_+$ we have $b\left(n\right)={}^ob\left(n\right)$. Thus, for all $n\in\omega_+$ we have $\left|b\right|\left(n\right)=\left|{}^ob\right|\left(n\right)$. Hence, as $\left|b\right|\left(0\right)+\left|{}^ob\right|\left(0\right)=0=2*0$, we have by the \nameref{Equivalence Lemma} \ref{Equivalence Lemma} that \begin{align}\label{eq:abs-opposite}
\left|\left[f\right]\right|=\left|{}^o\left[f\right]\right|
\end{align}.
Therefore, $\left|\left[a\right]\right|\neq\left|\left[b\right]\right|$. \\
The case for $\left[a\right]\geq{}^o\left[b\right]$ is proved in the same way.
\end{proof}

\begin{definition}[Set of Tracking Functions]\label{Set of Tracking Functions}\noindent\\
Define for finite functions $f$ and $g$ where $\left|\left[f\right]\right|>\left|\left[g\right]\right|$ the set of tracking functions as
\begin{align*}
\mathcal{T}_{g}^{f}:=\left\{\varrho:\omega\rightarrow\omega\middle|\exists w\in\left[f\right]\begin{array}{ll@{}}&\forall n\in\omega_+,w\left(n\right)\geq g\left(n\right)\\\wedge&\varrho\left(0\right)=w\left(0\right)\\\wedge&\forall n\in\omega_+,\varrho\left(n\right)=w\left(n\right)-g\left(n\right)\end{array}\right\}
\end{align*}
\end{definition}

\begin{remark} For all finite functions $f$ and $g$ where $\left|\left[f\right]\right|>\left|\left[g\right]\right|$ we have $\mathcal{T}_{g}^{f}\neq\emptyset$. This is because, there exists $a\in\left|\left[f\right]\right|$ such that for all $n\in\omega_+$ we have $a\left(n\right)\geq\left|g\right|\left(n\right)=g\left(n\right)$. By the \nameref{Absolute-Opposite Lemma} \ref{Absolute-Opposite Lemma} we have $a\in\left[f\right]$ or $a\in{}^o\left[f\right]$. In the first case by taking $w=a$ we have shown that the set is non-empty. In the second case, we see that ${}^oa\in\left[f\right]$ and for all $n\in\omega_+$ we have ${}^oa\left(n\right)=a\left(n\right)\geq g\left(n\right)$ and thus by taking $w=a$ we have shown that the set is non-empty.
\end{remark}

\begin{intuition}[Tracking Functions] Subtraction is harder than addition because we cannot simply take differences position by position, a digit of the smaller number might exceed the corresponding digit of the larger. The remedy is to first rewrite the larger number, $\left[f\right]$, in a form $w$ whose digits sit above $g$ everywhere, and then read off the surplus $w-g$ position by position. That surplus is a "tracking function": what is left of $\left[f\right]$ once a copy of $g$ has been removed. There are many admissible rewritings $w$, hence many tracking functions, and the point of the next lemma is that they all describe the same real number.
\end{intuition}
\begin{example} In base ten let $f=(0;0,7,5,0,\dots)$, of value $\nicefrac{3}{4}$, and let $g=(1;0,5,0,\dots)$, of value $-\nicefrac{1}{2}$; the signs differ and $\left|\left[f\right]\right|>\left|\left[g\right]\right|$, so the tracking set $\mathcal{T}^{f}_{g}$ is defined.\\
Taking $w=f$ itself, we have $w(n)\geq g(n)$ at every position in $\omega_+$, so $\varrho(0)=0$ and $\varrho(n)=w(n)-g(n)$ give
\begin{align*}
\varrho=(0;0,2,5,0,\dots),
\end{align*}
of value $\nicefrac{1}{4}$. But $w$ was not forced on us. Broadening $f$ about position $2$ gives $w'=(0;0,6,15,0,\dots)$, still in $\left[f\right]$ and still dominating $g$ everywhere, and it yields the different function
\begin{align*}
\varrho'=(0;0,1,15,0,\dots),
\end{align*}
whose value is $\nicefrac{1}{10}+\nicefrac{15}{100}=\nicefrac{1}{4}$ as well. Indeed $c_2(\varrho')=\varrho$, so the two lie in the same real number, as the following lemma asserts in general. This is why we may speak of \emph{the} difference: the witnesses differ, but what they compute does not.
\end{example}
\begin{lemma}[Tracking Functions Lemma]\label{Tracking Functions Lemma}\noindent\\
For all finite functions $f$ and $g$ such that $\sigma\left(f\right)\neq\sigma\left(g\right)$ or $\left[g\right]=\left[0\right]$ and $\left|\left[f\right]\right|>\left|\left[g\right]\right|$ and all tracking functions $\varrho_1,\varrho_2\in\mathcal{T}_g^f$ we have that the tracking functions are finite and furthermore $\left[\varrho_1\right]=\left[\varrho_2\right]$.
\end{lemma}
\begin{proof} Firstly we have that $\mathcal{T}_{g}^{f}\subset \mathcal{N}_{\mathcal{F}in}$, by the \nameref{Tying Theorem} \ref{Tying Theorem}, as for all $w\in\left[f\right]$ there exists a constant $C\in\omega_+$ and for all $m\in\omega_+$ we have 
\begin{align*}
\sum\limits_{n=1}^m\left(w\left(n\right)\prod\limits_{p=n}^{m-1}\left(\vartheta\left(p\right)\right)\right)<C\prod\limits_{p=1}^{m-1}\left(\vartheta\left(p\right)\right)
\end{align*}
and thus for all $\varrho\in\mathcal{T}_{g}^{f}$ and for all $m\in\omega_+$ we have
\begin{align*}
\sum\limits_{n=1}^m\left(\varrho\left(n\right)\prod\limits_{p=n}^{m-1}\left(\vartheta\left(p\right)\right)\right)\leq\sum\limits_{n=1}^m\left(\left(\varrho\left(n\right)+g\left(n\right)\right)\prod\limits_{p=n}^{m-1}\left(\vartheta\left(p\right)\right)\right)<C\prod\limits_{p=1}^{m-1}\left(\vartheta\left(p\right)\right)
\end{align*}
Therefore, by the \nameref{Tying Theorem} \ref{Tying Theorem} we have that $\varrho$ is finite.\\
We see that all tracking functions $\varrho\in\mathcal{T}_g^f$ are non-zero. This is due to the fact that if there were a zero tracking function $\varrho$ then there would exist a function $w\in\left[f\right]$ such that for all $n\in\omega_+,w\left(n\right)=g\left(n\right)$. Hence for all $n\in\omega_+,\left|w\right|\left(n\right)=\left|g\right|\left(n\right)$. However, as $\left|w\right|\left(0\right)+\left|g\right|\left(0\right)=0=2*0$ we have by the \nameref{Equivalence Lemma} \ref{Equivalence Lemma} that $\left[\left|w\right|\right]=\left[\left|g\right|\right]$. However, by the \nameref{Absolute Value Lemma} \ref{Absolute Value Lemma} we have $\left[\left|f\right|\right]=\left[\left|w\right|\right]$. And thus $\left[\left|f\right|\right]=\left[\left|g\right|\right]$ which means that $\left|\left[f\right]\right|=\left|\left[g\right]\right|$ which is a contradiction to $\left|\left[f\right]\right|>\left|\left[g\right]\right|$.\\
If $g\not\in\left[0\right]$ then by definition for $i\in\left[1,2\right]$ there exists $w_i\in\left[f\right]$ such that we have for all $n\in\omega_+$ that $\left({}^og+\varrho_i\right)\left(n\right)=w_i\left(n\right)$. Therefore, by the \nameref{Equivalence Lemma} \ref{Equivalence Lemma} as $\sigma\left({}^og+\varrho_i\right)=\sigma\left(w_i\right)$ we have that there exists $m\in\omega$ such that $\left({}^og+\varrho_i\right)\left(0\right)+w_i\left(0\right)=2m$ and thus, $\left[{}^og+\varrho_i\right]=\left[w_i\right]=\left[f\right]$.\\
Thus, by the \nameref{Consistency Theorem} \ref{1.I} \ref{Consistency Theorem} as ${}^og,\varrho_1,\varrho_2\in \mathcal{N}_{\mathcal{F}in}^*$ and $\sigma\left({}^og\right)=\sigma\left(\varrho_1\right)=\sigma\left(\varrho_2\right)$ and $\left[{}^og+\varrho_1\right]=\left[{}^og+\varrho_2\right]=\left[f\right]$ we must have $\left[\varrho_1\right]=\left[\varrho_2\right]$.\\
If $g\in\left[0\right]$ then by the \nameref{Null Lemma} \ref{Null Lemma} we have that for all $n\in\omega_+,g\left(n\right)=0$. Hence, $\varrho_1=w_1$ and $\varrho_2=w_2$ for some $w_1,w_2\in\left[f\right]$ and thus $\left[\varrho_1\right]=\left[\varrho_2\right]$.
\end{proof}

\begin{definition}[Primary Tracking Function]\label{Primary Tracking Function}\noindent\\
By the \nameref{Tracking Functions Lemma} \ref{Tracking Functions Lemma} for all finite functions $f$ and $g$ such that $\sigma\left(f\right)\neq\sigma\left(g\right)$ and $\left|\left[f\right]\right|>\left|\left[g\right]\right|$ there exists a real number $\left[x\right]$ such that $\mathcal{T}_g^f\subset\left[x\right]$. We define the primary tracking function $\varrho_g^f:\omega\rightarrow\omega$ as
\begin{align*}
\varrho_g^f\left(k\right)=\begin{cases}f\left(0\right)&k=0\\x_p\left(k\right)&otherwise\end{cases}
\end{align*} 
\end{definition}

\begin{remark}[Primary Tracking Function] Among all the tracking functions, which agree as a real number but differ as functions, we single out a canonical one, built from the primary auxiliary function of their common class and wearing the sign of $\left[f\right]$. This gives us one definite function to name when we want the difference, rather than a whole set.
\end{remark}
\subsection{Complete Sub-Addition and Addition}
\begin{definition}[Sub-Addition of Functions of Opposite Sign or Zero Functions]\noindent\\
Now we define the sub-addition of functions with opposite signs or zero-functions as $f$ and $g$ as:
\begin{align*}
\left(f+g\right)=
\begin{cases}
0_p&\left|\left[f\right]\right|=\left|\left[g\right]\right|\\
\varrho_g^f&\left|\left[f\right]\right|>\left|\left[g\right]\right|\\
\varrho_f^g&\left|\left[f\right]\right|<\left|\left[g\right]\right|
\end{cases}
\end{align*}
where $0_p$ is the function such that for all $n\in\omega,0_p\left(n\right)=0$.
\end{definition}

\begin{intuition}[Sub-Addition of Opposite Sign] Adding a positive number to a negative one is a tug-of-war between their magnitudes. The larger magnitude wins and keeps its sign, and what survives is the surplus of the larger over the smaller, precisely a tracking function. If the two magnitudes are equal they cancel exactly, and the result is zero. Adding zero, of course, changes nothing. This is subtraction, wearing the clothes of addition.
\end{intuition}

\begin{lemma}[Second Sub-Addition Finiteness Lemma]\noindent\\
For all finite functions $f$ and $g$ such that $\sigma\left(f\right)\neq\sigma\left(g\right)$ or $f\in\left[0\right]$ or $g\in\left[0\right]$ we have that $f+g$ is finite.
\end{lemma}
\begin{proof} As sub-addition of any two functions of opposite signs or zero-functions is an auxiliary function, it is finite.
\end{proof}

\begin{lemma}[Included Exclusion Lemma]\label{Included Exclusion Lemma}\noindent\\
For all finite functions $f$ and $g$ such that $\sigma\left(f\right)\neq\sigma\left(g\right)$ or $g\in\left[0\right]$ and $\left|\left[f\right]\right|>\left|\left[g\right]\right|$ we have that $\left({}^og+\varrho_g^f\right)\in\left[f\right]$.
\end{lemma}
\begin{proof} This has been proven in the proof of \nameref{Tracking Functions Lemma} \ref{Tracking Functions Lemma}.
\end{proof}

\begin{theorem}[Consistency Theorem]\label{Consistency Theorem'}\noindent\\
\begin{subtheorems}
\subtheorem\label{2.I} For all finite functions $f$, $g$ and $w$ such that $\sigma\left(f\right)\neq\sigma\left(g\right)=\sigma\left(w\right)$ or $f\in\left[0\right]$ or $g\in\left[0\right]$ or $w\in\left[0\right]$ we have that 
\begin{align*}
\left[f+g\right]=\left[f+w\right]\iff\left[g\right]=\left[w\right]
\end{align*}
\subtheorem\label{2.II} For all finite functions $f$, $g$, $w$ and $t$ such that $\sigma\left(f\right)\neq\sigma\left(w\right)$ or $\sigma\left(g\right)\neq\sigma\left(t\right)$ or $f\in\left[0\right]$ or $w\in\left[0\right]$ we have that 
\begin{align*}
\left[f\right]=\left[g\right]\wedge\left[w\right]=\left[t\right]\implies\left[f+w\right]=\left[g+t\right]
\end{align*}
\subtheorem\label{2.III} For all finite functions $f$, $g$, $w$ and $t$ such that $\sigma\left(f\right)\neq\sigma\left(w\right)$ or $\sigma\left(g\right)\neq\sigma\left(t\right)$ or $\left[f\right]=0$ or $\left[w\right]=0$ we have that 
\begin{align*}
\left[f+w\right]=\left[g+t\right]\wedge\left[w\right]=\left[t\right]\implies\left[f\right]=\left[g\right]
\end{align*}
\end{subtheorems}
\end{theorem}
\begin{proof}[Proof of I]\noindent\\
$\left(\impliedby\right):$\\
Let $g,w\in\left[g\right]$ be given. Let $f\in \mathcal{N}_{\mathcal{F}in},\sigma\left(f\right)\neq\sigma\left(g\right)$ be given.\\
\textbf{Case 1 $\left|\left[f\right]\right|=\left|\left[g\right]\right|$:} In this case we see from the definition that $\left[f+g\right]=\left[f+w\right]=\left[0\right]$.\\
\textbf{Case 2 $\left|\left[f\right]\right|>\left|\left[g\right]\right|$:} Take $\left(f+g\right)$ and $\left(f+w\right)$. We know by definition that $\sigma\left(f+g\right)=\sigma\left(f+w\right)=\sigma\left(f\right)$. Take ${}^og,{}^ow\in\left[{}^ow\right]$ and sub-add them respectively to $\left(f+g\right)$ and $\left(f+w\right)$. Notice that $\sigma\left({}^og\right)=\sigma\left({}^ow\right)=\sigma\left(f\right)$.\\
We see that by the \nameref{Included Exclusion Lemma} \ref{Included Exclusion Lemma} we have 
\begin{align*}
\begin{array}{ll@{}}
\left({}^og+\left(f+g\right)\right)=\left({}^og+\varrho_g^f\right)\in\left[f\right]&\text{and}\\
\left({}^ow+\left(f+w\right)\right)=\left({}^ow+\varrho_w^f\right)\in\left[f\right]
\end{array}
\end{align*} 
Thus, $\left[{}^og+\left(f+g\right)\right]=\left[{}^ow+\left(f+w\right)\right]$. By the \nameref{Consistency Theorem} \ref{1.III} \ref{Consistency Theorem} as $\left[{}^og\right]=\left[{}^ow\right]$, by the \nameref{Opposite Function Lemma} \ref{Opposite Function Lemma}, we must have $\left[f+g\right]=\left[f+w\right]$.\\
\textbf{Case 3 $\left|\left[f\right]\right|<\left|\left[g\right]\right|$:} Take $\left(f+g\right)$ and $\left(f+w\right)$. Then sub-add to each one of them ${}^of$. By the \nameref{Included Exclusion Lemma} \ref{Included Exclusion Lemma} we obtain 
\begin{align*}
\begin{array}{ll@{}}
\left({}^of+\left(f+g\right)\right)=\left({}^of+\varrho_f^g\right)\in\left[g\right]=\left[w\right]&\text{and}\\
\left({}^of+\left(f+w\right)\right)=\left({}^of+\varrho_f^w\right)\in\left[w\right]=\left[g\right]
\end{array}
\end{align*}
By the \nameref{Consistency Theorem} \ref{1.I} \ref{Consistency Theorem} as $\left[{}^of+\left(f+g\right)\right]=\left[{}^of+\left(f+w\right)\right]$ we have $\left[f+g\right]=\left[f+w\right]$.\\
$\left(\implies\right):$\\
\textbf{Case 1 $\left|\left[f\right]\right|=\left|\left[g\right]\right|$:} We have $\left[f+g\right]=\left[0\right]=\left[f+w\right]$. The only way $\left[f+g\right]=0$ is if $\left|\left[f\right]\right|=\left|\left[g\right]\right|$ as in the other cases it was shown in the proof of the \nameref{Tracking Functions Lemma} \ref{Tracking Functions Lemma} that $\varrho_g^f$ and $\varrho_f^g$ are non-zero if $\left|\left[f\right]\right|>\left|\left[g\right]\right|$ and $\left|\left[f\right]\right|<\left|\left[g\right]\right|$ respectively. The same holds for $w$.\\
Hence, by the \nameref{Absolute-Opposite Lemma} \ref{Absolute-Opposite Lemma} we have $\left[g\right]=\left[{}^of\right]=\left[w\right]$ if $f\not\in\left[0\right]$ as $\sigma\left(f\right)\neq\sigma\left(g\right)=\sigma\left(w\right)$. And if $f\in\left[0\right]$ we have $\left[g\right]=\left[0\right]=\left[w\right]$, as $\left[0\right]={}^o\left[0\right]$, which is seen by the \nameref{Null Lemma} \ref{Null Lemma}.\\ 
\textbf{Case 2 $\left|\left[f\right]\right|>\left|\left[g\right]\right|$:} We have $\left[f+g\right]=\left[f+w\right]$. Additionally, by the \nameref{Included Exclusion Lemma} \ref{Included Exclusion Lemma} we have
\begin{align*}
\begin{array}{ll@{}}
\left({}^og+\left(f+g\right)\right)=\left({}^og+\varrho_g^f\right)\in\left[f\right]&\text{and}\\
\left({}^ow+\left(f+w\right)\right)=\left({}^ow+\varrho_w^f\right)\in\left[f\right]
\end{array}
\end{align*} Thus $\left[{}^og+\left(f+g\right)\right]=\left[{}^ow+\left(f+w\right)\right]$. By the \nameref{Consistency Theorem} \ref{1.III} \ref{Consistency Theorem} as $\left[f+g\right]=\left[f+w\right]$ and $\left[{}^og+\left(f+g\right)\right]=\left[{}^ow+\left(f+w\right)\right]$ we must have $\left[{}^og\right]=\left[{}^ow\right]$. Thus, by the \nameref{Opposite Function Lemma} \ref{Opposite Function Lemma} we have $\left[g\right]=\left[{}^o{}^og\right]=\left[{}^o{}^ow\right]=\left[w\right]$.\\
\textbf{Case 3 $\left|\left[f\right]\right|<\left|\left[g\right]\right|$:} We have $\left[f+g\right]=\left[f+w\right]$. Additionally, by the \nameref{Included Exclusion Lemma} \ref{Included Exclusion Lemma} we have
\begin{align*}
\begin{array}{ll@{}}
\left({}^of+\left(f+g\right)\right)=\left({}^of+\varrho_f^g\right)\in\left[g\right]&\text{and}\\
\left({}^of+\left(f+w\right)\right)=\left({}^of+\varrho_f^w\right)\in\left[w\right]
\end{array}
\end{align*} By the \nameref{Consistency Theorem} \ref{1.I} \ref{Consistency Theorem} we must have $\left[g\right]=\left[w\right]$.\footnote{In this proof we have actually supposed that $\left|\left[f\right]\right|>\left|\left[g\right]\right|\iff\left|\left[f\right]\right|>\left|\left[w\right]\right|$. This is true as if we suppose not and take $\left|\left[f\right]\right|>\left|\left[g\right]\right|\wedge\left|\left[f\right]\right|<\left|\left[w\right]\right|$ without loss of generality, then $\left(f+w\right)\left(0\right)=f\left(0\right)$ and $\left(f+g\right)\left(0\right)=g\left(0\right)$. Also by definition $\left(f+g\right)\not\in\left[0\right]$ and $\left(f+w\right)\not\in\left[0\right]$. Thus, we have $\sigma\left(f\right)\neq\sigma\left(g\right)\implies\left[f+w\right]\neq\left[f+g\right]$. This is a contradiction.}
\end{proof}
\begin{proof}[Proof of II]\noindent\\
Let finite functions $f$, $g$, $w$ and $t$ such that $\sigma\left(f\right)\neq\sigma\left(w\right)$ or $\sigma\left(g\right)\neq\sigma\left(t\right)$ or $f\in\left[0\right]$ or $g\in\left[0\right]$ or $w\in\left[0\right]$ be given. By using the \nameref{Consistency Theorem'} \ref{2.I} \ref{Consistency Theorem'} twice we obtain 
\begin{align*}
\begin{array}{ll@{}}
\left[w\right]=\left[t\right]\implies\left[f+w\right]=\left[f+t\right]&\text{and}\\
\left[f\right]=\left[g\right]\implies\left[t+f\right]=\left[t+g\right]
\end{array}
\end{align*} As sub-addition is commutative $\left[t+f\right]=\left[f+t\right]$. Thus using transitivity of equality we obtain $\left[f+w\right]=\left[f+t\right]=\left[t+f\right]=\left[t+g\right]=\left[g+t\right]\implies\left[f+w\right]=\left[g+t\right]$.
\end{proof}
\begin{proof}[Proof of III]\noindent\\
Let finite functions $f$, $g$, $w$ and $t$ such that $\sigma\left(f\right)\neq\sigma\left(w\right)$ or $\sigma\left(g\right)\neq\sigma\left(t\right)$ and $\left[f+w\right]=\left[g+t\right]$ and $\left[w\right]=\left[t\right]$ be given. By the \nameref{Consistency Theorem'} \ref{2.I} \ref{Consistency Theorem'} 
\begin{align*}
\left[w\right]=\left[t\right]\implies\left[f+w\right]=\left[f+t\right]
\end{align*} by the hypothesis $\left[g+t\right]=\left[f+w\right]$.\\
Therefore, we have $\left[g+t\right]=\left[f+t\right]$. By the commutativity of sub-addition and the \nameref{Consistency Theorem} \ref{1.I} \ref{Consistency Theorem} we have $\left[t+g\right]=\left[t+f\right]\implies\left[f\right]=\left[g\right]$.
\end{proof}

\begin{definition}[Complete Sub-Addition of Functions]\label{Complete Sub-Addition of Functions}\noindent\\
For all finite functions $f$ and $g$ we define $\left(f+g\right)$ as:
\begin{align*}
\left(f+g\right)\left(k\right)=
\begin{cases}
\begin{cases}
f\left(0\right)g\left(0\right)&k=0\\
f\left(k\right)+g\left(k\right)&k\in\omega_+
\end{cases}&\sigma\left(f\right)=\sigma\left(g\right)\wedge f\not\in\left[0\right]\wedge g\not\in\left[0\right]\\
\begin{cases}
\begin{cases}
0&\left|\left[f\right]\right|=\left|\left[g\right]\right|\\
f\left(0\right)&\left|\left[f\right]\right|>\left|\left[g\right]\right|\\
g\left(0\right)&\left|\left[f\right]\right|<\left|\left[g\right]\right|
\end{cases}&k=0\\
\begin{cases}
0&\left|\left[f\right]\right|=\left|\left[g\right]\right|\\
\varrho_g^f\left(k\right)&\left|\left[f\right]\right|>\left|\left[g\right]\right|\\
\varrho_f^g\left(k\right)&\left|\left[f\right]\right|<\left|\left[g\right]\right|
\end{cases}&k\in\omega_+
\end{cases}&\sigma\left(f\right)\neq\sigma\left(g\right)\vee f\in\left[0\right]\vee g\in\left[0\right]
\end{cases}
\end{align*}
\end{definition}

\begin{remark} Notice that sub-addition of functions is commutative. Furthermore, notice that $\sigma\left(f\right)=\sigma\left(g\right)$ where at least one of the functions is non-zero then $\sigma\left(f+g\right)=\sigma\left(f\right)$. This is because, if neither of them were zero functions then $\left(f+g\right)\left(0\right)=f\left(0\right)g\left(0\right)$. Thus if both $f\left(0\right)$ and $g\left(0\right)$ were odd then $\left(f+g\right)\left(0\right)$ is odd and the same with evens. Hence $\sigma\left(f+g\right)=\sigma\left(f\right)$.\\
If one of them is a zero-function then we suppose without loss of generality that $f$ is not a zero-function. In that case $\left(f+g\right)\left(0\right)=f\left(0\right)$ and thus $\sigma\left(f+g\right)=\sigma\left(f\right)$.
\end{remark}
\begin{theorem}[Consistency Theorem]\label{Consistency Theorem''}\noindent\\
Combining the \nameref{Consistency Theorem} \ref{Consistency Theorem} and \nameref{Consistency Theorem'} \ref{Consistency Theorem'} we obtain.
\begin{subtheorems}
\subtheorem\label{3.I} For all finite functions $f$, $g$ and $w$ we have that
\begin{align*}
\left[f+g\right]=\left[f+w\right]\iff\left[g\right]=\left[w\right]
\end{align*}
\subtheorem\label{3.II} For all finite functions $f$, $g$, $w$ and $t$ we have that 
\begin{align*}
\left[f\right]=\left[g\right]\wedge\left[w\right]=\left[t\right]\implies\left[f+w\right]=\left[g+t\right]
\end{align*}
\subtheorem\label{3.III} For all finite functions $f$, $g$, $w$ and $t$ we have that 
\begin{align*}
\left[f+w\right]=\left[g+t\right]\wedge\left[w\right]=\left[t\right]\implies\left[f\right]=\left[g\right]
\end{align*}
\end{subtheorems}
\end{theorem}

\begin{definition}[Addition]\noindent\\
We define addition of two real numbers $\left[f\right]$ and $\left[g\right]$ as:
\begin{align*}
\left[f\right]+\left[g\right]:=\left[f+g\right]
\end{align*}
for any functions $f\in\left[f\right]$ and $g\in\left[g\right]$.
\end{definition}

\begin{remark}The \nameref{Consistency Theorem''} \ref{3.II} \ref{Consistency Theorem''} ensures that the sum of any two real numbers does not depend on the functions chosen to represent those real numbers. Therefore, addition is well-defined.
\end{remark}

\begin{theorem}[Consistency Theorem]\label{Consistency Theorem'''}\noindent\\
\begin{subtheorems}
\subtheorem\label{4.I} For all real numbers $\left[f\right]$, $\left[g\right]$ and $\left[w\right]$ we have that
\begin{align*}
\left[f\right]+\left[g\right]=\left[f\right]+\left[w\right]\iff\left[g\right]=\left[w\right]
\end{align*}
\subtheorem\label{4.II} For all real numbers $\left[f\right]$, $\left[g\right]$, $\left[w\right]$ and $\left[t\right]$ we have that 
\begin{align*}
\left[f\right]=\left[g\right]\wedge\left[w\right]=\left[t\right]\implies\left[f\right]+\left[w\right]=\left[g\right]+\left[t\right]
\end{align*}
\subtheorem\label{4.III} For all real numbers $\left[f\right]$, $\left[g\right]$, $\left[w\right]$ and $\left[t\right]$ we have that 
\begin{align*}
\left[f\right]+\left[w\right]=\left[g\right]+\left[t\right]\wedge\left[w\right]=\left[t\right]\implies\left[f\right]=\left[g\right]
\end{align*}
\end{subtheorems}
\end{theorem}
\begin{proof} By the definition of addition and the \nameref{Consistency Theorem''} \ref{Consistency Theorem''} these theorems follow.
\end{proof}

\begin{lemma}[Rearrangement Lemma]\label{Rearrangement Lemma}\noindent\\
For all finite functions $f$ and $g$ such that $\sigma\left(f\right)\neq\sigma\left(g\right)$ and $\left|\left[f\right]\right|>\left|\left[g\right]\right|$ we have that for all $a\in\mathcal{T}_g^f$ we have $\left[f\right]+\left[g\right]=\left[a\right]$.
\end{lemma}

\begin{intuition}[Rearrangement Lemma] This lemma records the practical consequence: \emph{any} tracking function of the larger over the smaller computes the sum, so when proving things we may reach for whichever witness is easiest to handle.
\end{intuition}

\begin{proof} By definition $\left[f\right]+\left[g\right]=\left[f+g\right]$ where $\left(f+g\right)=\varrho_g^f$. Now, we see that by taking any $a\in\mathcal{T}_g^f$ we have by the \nameref{Tracking Functions Lemma} \ref{Tracking Functions Lemma} that $a\in\left[\varrho_g^f\right]=\left[f\right]+\left[g\right]$. Therefore, by the \nameref{Unicity of Real Numbers} \ref{Unicity of Real Numbers} we have that $\left[a\right]=\left[f\right]+\left[g\right]$.
\end{proof}
\subsection{Additive Axioms}
\begin{axiom}[Additive Identity]\noindent\\
There exists a real number $\left[x\right]$ such that for all real numbers $\left[a\right]$ we have $\left[x\right]+\left[a\right]=\left[a\right]$.
\end{axiom}
\begin{proof} We claim that $\left[0\right]=\left[x\right]$. Let a real number $\left[a\right]$ be given.  Take any function $a\in\left[a\right]$ and the primary auxiliary function $0_p\in\left[0\right]$.\\
If $\left[a\right]=\left[0\right]$ then $\left|\left[a\right]\right|=\left|\left[0\right]\right|$. Thus, by the definition of sub-addition we have $0_p+a=0_p$ and thus $\left[0\right]+\left[a\right]=\left[0\right]=\left[a\right]$.\\
If $\left[a\right]\neq\left[0\right]$ then $\left[a\right]\neq{}^o\left[0\right]$ as, by the \nameref{Null Lemma} \ref{Null Lemma} we have that ${}^o0_p\in\left[0\right]$ and thus $\left[0\right]={}^o\left[0\right]$. By the contrapositive of the \nameref{Absolute-Opposite Lemma} \ref{Absolute-Opposite Lemma} we have that $\left|\left[a\right]\right|\neq\left|\left[0\right]\right|$.\\
Furthermore, as any function $\omega\rightarrow\omega$ dominates $0_p$ we have by the \nameref{One-Case Lemma} \ref{One-Case Lemma}, as $\sigma\left(\left|a\right|\right)=\sigma\left(\left|0_p\right|\right)=1$, that $\left|\left|a\right|\right|\geq\left|\left[0\right|\right]$. Moreover, as $\left|\left[a\right]\right|\neq\left|\left[0\right]\right|$ we have that $\left|\left[a\right]\right|>\left|\left[0\right]\right|$.\\
We see that $a\in\mathcal{T}_{0_p}^a$ as for all $n\in\omega_+$ we have $a\left(n\right)\geq0_p\left(n\right)=0$ and
\begin{align*}
a\left(k\right)=\begin{cases}
a\left(0\right)&k=0\\
a\left(k\right)-0_p\left(0\right)=a\left(k\right)-0&k\in\omega_+
\end{cases}
\end{align*}
Hence, by the \nameref{Rearrangement Lemma} \ref{Rearrangement Lemma} we have that $\left[0\right]+\left[a\right]=\left[a\right]$.
\end{proof}

\begin{axiom}[Additive Inverse]\noindent\\
For any real number $\left[f\right]$ there exists a real number $\left[g\right]$ such that $\left[f\right]+\left[g\right]=\left[0\right]$.
\end{axiom}
\begin{proof} We claim that $\left[g\right]={}^o\left[f\right]$.\\
If $\left[f\right]=\left[0\right]$ then ${}^o\left[f\right]=\left[0\right]$. Hence, as $\left[0\right]$ is the additive identity we have that
\begin{align*}
\left[f\right]+{}^o\left[f\right]=\left[0\right]+\left[0\right]=\left[0\right]
\end{align*}
If $\left[f\right]\neq\left[0\right]$ then by \eqref{eq:abs-opposite} we have $\left|\left[f\right]\right|=\left|{}^o\left[f\right]\right|$. Let a function $f\in\left[f\right]$ be given. Then $\sigma\left(f\right)\neq\sigma\left({}^of\right)$ and thus $\left(f+{}^of\right)=0_p$. Hence, $\left[f\right]+{}^o\left[f\right]=\left[0\right]$.
\end{proof}

\begin{axiom}[Commutativity of Addition]\noindent\\
For all real numbers $\left[f\right]$ and $\left[g\right]$ we have $\left[f\right]+\left[g\right]=\left[g\right]+\left[f\right]$.
\end{axiom}
\begin{proof} As sub-addition of any two finite functions $f$ and $g$ is commutative so is addition of any two real numbers.
\end{proof}

\begin{lemma}[Middle Lemma]\label{Middle Lemma}\noindent\\
For all non-zero finite functions $f$ and $g$ and all finite functions $w$ such that $\sigma\left(g\right)=\sigma\left(w\right)$ we have that $\left[\left(f+{}^og\right)+\left(g+w\right)\right]=\left[f+w\right]$.
\end{lemma}
\begin{proofidea} The statement is that a number and its opposite, placed in the middle of a three-fold sum, annihilate and leave the outer two combined. Its proof is not a manipulation but an exhaustion, organised first by whether $f$ and $g$ share a sign and then by which magnitude dominates. Two tools recur in every branch: the \nameref{Absolute-Opposite Lemma} \ref{Absolute-Opposite Lemma}, which identifies the class once magnitudes are known to agree, and the \nameref{Rearrangement Lemma} \ref{Rearrangement Lemma}, which lets any convenient tracking function stand in for the difference. Only the branches in which a positive and a negative meet in the middle carry real content; the others reduce to same-sign sub-addition, where the material simply adds.
\end{proofidea}
\begin{proof}\noindent
\paragraph{Firstly,} we suppose $f$, $g$ and $w$ are all non-zero finite functions. There are two cases $\sigma\left(f\right)=\sigma\left(g\right)$ and $\sigma\left(f\right)=\sigma\left({}^og\right)$:
\subparagraph{$\sigma\left(f\right)=\sigma\left(g\right)$:}\noindent\\
\textbf{If $\left|\left[f\right]\right|=\left|\left[{}^og\right]\right|$:} we have by the \nameref{Absolute-Opposite Lemma} \ref{Absolute-Opposite Lemma} we have $\left[f\right]=\left[g\right]$ or $\left[f\right]={}^o\left[g\right]$. However, as $\sigma\left(f\right)\neq\sigma\left({}^og\right)$ we must have $\sigma\left(\left[f\right]\right)\neq\sigma\left({}^o\left[g\right]\right)$ and thus $\left[f\right]\neq{}^o\left[g\right]$. Therefore, $\left[f\right]=\left[g\right]$. Thus 
\begin{align*}
\left[\left(f+{}^og\right)+\left(g+w\right)\right]=\left(\left[f\right]+{}^o\left[g\right]\right)+\left(\left[g\right]+\left[w\right]\right)=\left[0\right]+\left(\left[f\right]+\left[w\right]\right)=\left[f+w\right]
\end{align*}
\textbf{If $\left|\left[f\right]\right|>\left|\left[{}^og\right]\right|$:} then $\left(\left(f+{}^og\right)+\left(g+w\right)\right)\left(0\right)=f\left(0\right)w\left(0\right)g\left(0\right)$. As $\sigma\left(f\right)=\sigma\left(g\right)$ we have that 
\begin{align*}
\text{$f\left(0\right)w\left(0\right)g\left(0\right)$ is odd}\iff\text{$f\left(0\right)$ and $w\left(0\right)$ are odd}\iff\text{$f\left(0\right)w\left(0\right)$ is odd}
\end{align*}
Hence, $\sigma\left(\left(f+{}^og\right)+\left(g+w\right)\right)=\sigma\left(f+w\right)$.\\
Also, by the the definition of $\mathcal{T}_{{}^og}^f$ we have for all $a\in\mathcal{T}_{{}^og}^f$ there exists $b\in\left[f\right]$ such that $\forall n\in\omega_+$
\begin{align*}
\left(a+\left(g+w\right)\right)\left(n\right)&=a\left(n\right)+g\left(n\right)+w\left(n\right)=b\left(n\right)-{}^og\left(n\right)+g\left(n\right)+w\left(n\right)\\
&=b\left(n\right)+w\left(n\right)=\left(b+w\right)\left(n\right)
\end{align*}
Furthermore, as $f\not\in\left[0\right]$ we have that $\sigma\left(f\right)=\sigma\left(b\right)$. Moreover, $\sigma\left(g\right)=\sigma\left(w\right)$ which means that $\sigma\left(g+w\right)=\sigma\left(g\right)=\sigma\left(f\right)$. Thus, as $\sigma\left(a\right)=\sigma\left(f\right)=\sigma\left(g+w\right)$ we have that $\sigma\left(a+\left(g+w\right)\right)=\sigma\left(f\right)$. Also, as $\sigma\left(b\right)=\sigma\left(f\right)=\sigma\left(w\right)$ we have that $\sigma\left(b+w\right)=\sigma\left(f\right)$. Therefore, $\sigma\left(a+\left(g+w\right)\right)=\sigma\left(b+w\right)$.\\
This means that $\left(a+\left(g+w\right)\right)\left(0\right)+\left(b+w\right)\left(0\right)=2m$ for some $m\in\omega$. Therefore, by the \nameref{Equivalence Lemma} \ref{Equivalence Lemma} and the \nameref{Rearrangement Lemma} \ref{Rearrangement Lemma} we have that  
\begin{align*}
\left[\left(f+{}^og\right)+\left(g+w\right)\right]=\left[a\right]+\left[g+w\right]=\left[a+\left(g+w\right)\right]=\left[b+w\right]=\left[f+w\right]
\end{align*}
\textbf{If $\left|\left[f\right]\right|<\left|\left[{}^og\right]\right|$:} then $\left(f+{}^og\right)=\varrho_f^{{}^og}$. Notice that $\left|\left[f+{}^og\right]\right|\leq\left|\left[g+w\right]\right|$. Suppose not, then $\left|\left[f+{}^og\right]\right|\geq\left|\left[g+w\right]\right|\wedge\left|\left[f+{}^og\right]\right|\neq\left|\left[g+w\right]\right|$. By the \nameref{Included Exclusion Lemma} \ref{Included Exclusion Lemma} we have that $\left({}^of+\varrho_f^{{}^og}\right)\in\left[{}^og\right]$, hence ${}^o\left({}^of+\varrho_f^{{}^og}\right)\in\left[{}^o{}^og\right]=\left[g\right]$.\\
Take $\left({}^o\left({}^of+\varrho_f^{{}^og}\right)+w\right)\in\left[g+w\right]$. We see that if there was a function in $\left[f+{}^og\right]$ strictly dominating $\left({}^o\left({}^of+\varrho_f^{{}^og}\right)+w\right)$ then it would be strictly dominating $\left(\varrho_f^{{}^og}\right)$ as for all $n\in\omega_+$ we have 
\begin{align*}
\left({}^o\left({}^of+\varrho_f^{{}^og}\right)+w\right)\left(n\right)&={}^o\left({}^of+\varrho_f^{{}^og}\right)\left(n\right)+w\left(n\right)=\left({}^of+\varrho_f^{{}^og}\right)\left(n\right)+w\left(n\right)\\
&={}^of\left(n\right)+\varrho_f^{{}^og}\left(n\right)+w\left(n\right)=f\left(n\right)+w\left(n\right)+\varrho_f^{{}^og}\left(n\right)\geq\varrho_f^{{}^og}\left(n\right)
\end{align*}
We see that the function in question cannot be in the same equivalence class as $\left(f+{}^og\right)$ as, by the \nameref{Domination Theorem} \ref{Domination Theorem} no function is strictly dominated by any other function in the same real number. A contradiction, therefore $\left|\left[f+{}^og\right]\right|\leq\left|\left[g+w\right]\right|$.\\
Notice that by what we have shown above $\left|\left[f+{}^og\right]\right|=\left|\left[g+w\right]\right|\iff\forall n\in\omega_+,f\left(n\right)=w\left(n\right)=0$, as otherwise there exists $m\in\omega_+$ such that $f\left(m\right)+w\left(m\right)+\varrho_f^{{}^og}\left(m\right)>\varrho_f^{{}^og}\left(m\right)$ and thus $\left|\left({}^o\left({}^of+\varrho_f^{{}^og}\right)+w\right)\right|\in\left|\left[g+w\right]\right|$ would strictly dominate $\left|\varrho_f^{{}^og}\right|\in\left|\left[f+{}^og\right]\right|$.\\
$\forall n\in\omega_+,f\left(n\right)=w\left(n\right)=0$ is in turn equivalent, by the \nameref{Null Lemma} \ref{Null Lemma} to $f,w\in\left[0\right]$. Therefore, $\left|\left[f+{}^og\right]\right|\neq\left|\left[g+w\right]\right|$.\\
Thus $\left|\left[f+{}^og\right]\right|<\left|\left[g+w\right]\right|$ then as addition is commutative we have $\left[\left(f+{}^og\right)+\left(g+w\right)\right]=\left[\left(g+w\right)+\left(f+{}^og\right)\right]$. As seen above, $\left({}^o\left({}^of+\varrho_f^{{}^og}\right)+w\right)\in\left[g+w\right]$ and for all $n\in\omega_+$ we have 
\begin{align*}
\left({}^o\left({}^of+\varrho_f^{{}^og}\right)+w\right)\left(n\right)-\varrho_f^{{}^og}\left(n\right)=f\left(n\right)+w\left(n\right)+\varrho_f^{{}^og}\left(n\right)-\varrho_f^{{}^og}\left(n\right)=\left(f+w\right)\left(n\right)
\end{align*}
Thus define the function $\left(f+w\right)':\omega\rightarrow\omega$ as
\begin{align*}
\left(f+w\right)'\left(k\right)=\begin{cases}
\left({}^o\left({}^of+\varrho_f^{{}^og}\right)+w\right)\left(0\right)&k=0\\
\left({}^o\left({}^of+\varrho_f^{{}^og}\right)+w\right)\left(k\right)-\varrho_f^{{}^og}\left(k\right)=\left(f+w\right)\left(k\right)&k\in\omega_+
\end{cases}
\end{align*}
We see that $\left(f+w\right)'\in\mathcal{T}_{\varrho_f^{{}^og}}^{g+w}$.
\begin{align*}
&\left({}^o\left({}^of+\varrho_f^{{}^og}\right)+w\right)\left(0\right)={}^o\left({}^of+\varrho_f^{{}^og}\right)\left(0\right)w\left(0\right)=\left(\left({}^of+\varrho_f^{{}^og}\right)\left(0\right)+1\right)w\left(0\right)\\
=&\left({}^of\left(0\right)\varrho_f^{{}^og}\left(0\right)+1\right)w\left(0\right)=\left(\left(f\left(0\right)+1\right){}^og\left(0\right)+1\right)w\left(0\right)=\left(\left(f\left(0\right)+1\right)\left(g\left(0\right)+1\right)+1\right)w\left(0\right)
\end{align*}
We see that as $\sigma\left(f\right)=\sigma\left(g\right)=\sigma\left(w\right)$
\begin{align*}
\text{$\left(\left(f\left(0\right)+1\right)\left(g\left(0\right)+1\right)+1\right)w\left(0\right)$ is odd}\iff\text{$w\left(0\right)$ is odd}\iff\text{$f\left(0\right)w\left(0\right)$ is odd}
\end{align*}
Therefore, $\sigma\left(\left(f+w\right)'\right)=\sigma\left(\left(f+w\right)\right)$. Hence, $\left(f+w\right)'\left(0\right)+\left(f+w\right)\left(0\right)=2m$ for some $m\in\omega$ and thus by the \nameref{Equivalence Lemma} \ref{Equivalence Lemma} we have $\left[\left(f+w\right)'\right]=\left[\left(f+w\right)\right]$.\\
Therefore, by the \nameref{Rearrangement Lemma} \ref{Rearrangement Lemma} 
\begin{align*}
\left[f+{}^og\right]+\left[g+w\right]=\left[\varrho_f^{{}^og}\right]+\left[g+w\right]=\left[\left(f+w\right)'\right]=\left[f+w\right]
\end{align*}
\subparagraph{$\sigma\left(f\right)=\sigma\left({}^og\right)$:}\noindent\\
\textbf{If $\left|\left[f+{}^og\right]\right|\geq\left|\left[g+w\right]\right|$:} then there exists a function $a\in\left[f+{}^og\right]$ such that for all $n\in\omega_+$ we have $a\left(n\right)\geq g\left(n\right)+w\left(n\right)$. Thus, for all $n\in\omega_+$ we have $a\left(n\right)\geq g\left(n\right)={}^og\left(n\right)$. Define $b:\omega\rightarrow\omega$ as 
\begin{align*}
b\left(k\right)=\begin{cases}
f\left(0\right)&k=0\\
a\left(k\right)-{}^og\left(k\right)&k\in\omega_+
\end{cases}
\end{align*}
By the \nameref{Tying Theorem} \ref{Tying Theorem} there exists a constant $C\in\omega_+$ such that for all $n\in\omega_+$ we have that 
\begin{align*}
\sum\limits_{k=1}^{n}b\left(n\right)\prod\limits_{j=k}^{n-1}\left(\vartheta\left(j\right)\right)\leq\sum\limits_{k=1}^{n}a\left(n\right)\prod\limits_{j=k}^{n-1}\left(\vartheta\left(j\right)\right)<C\prod\limits_{j=1}^{n-1}\left(\vartheta\left(j\right)\right)
\end{align*}
Therefore, $b\in \mathcal{N}_{\mathcal{F}in}$. Moreover, $b\not\in\left[0\right]$. This is because if it were, then for all $n\in\omega_+$ we would have $a\left(n\right)={}^og\left(n\right)$. However, then as $f\not\in\left[0\right]$ we have by the \nameref{Null Lemma} \ref{Null Lemma} that there exists a value $m\in\omega_+$ such that $f\left(m\right)>0$. Therefore, for all  we have 
\begin{align*}
\begin{array}{ll@{}}
\left(f+{}^og\right)\left(n\right)=f\left(n\right)+{}^og\left(n\right)=f\left(n\right)+a\left(n\right)\geq a\left(n\right)&\forall n\in\omega_+\text{ and}\\
\left(f+{}^og\right)\left(m\right)=f\left(m\right)+{}^og\left(m\right)=f\left(m\right)+a\left(m\right)> a\left(n\right)
\end{array}
\end{align*}
Therefore, $\left(f+{}^og\right)$ would strictly dominate $a$ which is a contradiction to $a\in\left[f+{}^og\right]$ by the \nameref{Domination Theorem} \ref{Domination Theorem}. Ergo, $b\not\in\left[0\right]$.\\
Furthermore, as $\sigma\left(b\right)=\sigma\left({}^og\right)$ we have that $\left[b+{}^og\right]=\left[a\right]=\left[f+{}^og\right]$ and thus, by the first consistency theorem, $\left[b\right]=\left[f\right]$. We have that for all $n\in\omega_+$ 
\begin{align*}
a\left(n\right)-\left(g+w\right)\left(n\right)=b\left(n\right)+g\left(n\right)-g\left(n\right)-w\left(n\right)=b\left(n\right)-w\left(n\right)
\end{align*}
This shows us that for all $n\in\omega_+$ we have $b\left(n\right)\geq w\left(n\right)$ and thus for all $n\in\omega$ $\left|b\right|\left(n\right)\geq\left|w\right|\left(n\right)$. Hence, by the \nameref{One-Case Lemma} \ref{One-Case Lemma}, as $\sigma\left(\left|b\right|\right)=\sigma\left(\left|w\right|\right)=1$ we have that $\left|\left[f\right]\right|=\left|\left[b\right]\right|\geq\left|\left[w\right]\right|$.\\
If $\left|\left[f\right]\right|=\left|\left[w\right]\right|$ then by the \nameref{Absolute-Opposite Lemma} \ref{Absolute-Opposite Lemma} we have that $\left[f\right]=\left[w\right]$ or $\left[f\right]={}^o\left[w\right]$. However, as $f$ and $w$ are non-zero, $\sigma\left(\left[f\right]\right)=\sigma\left(f\right)\neq\sigma\left(w\right)=\sigma\left(\left[w\right]\right)$ and therefore, the former option is not possible. Thus, $\left[f\right]={}^o\left[w\right]$.\\
Hence, $\left[\left(f+{}^og\right)+\left(g+w\right)\right]=\left[{}^ow+{}^og\right]+\left[w+g\right]$. We see that 
\begin{equation*}\label{eq:middle}
    \left[{}^ow+{}^og\right]=\left[{}^o\left(w+g\right)\right]
\end{equation*}
as $\sigma\left(w\right)=\sigma\left(g\right)$ and thus $w\left(0\right)+g\left(0\right)=2m$ for some $m\in\omega$ and hence
\begin{align*}
&\left({}^ow+{}^og\right)\left(0\right)+{}^o\left(w+g\right)\left(0\right)=\left(w\left(0\right)+1\right)\left(g\left(0\right)+1\right)+w\left(0\right)g\left(0\right)+1=2\left(w\left(0\right)g\left(0\right)+m+1\right)\\
&\left({}^ow+{}^og\right)\left(n\right)={}^ow\left(n\right)+{}^og\left(n\right)=w\left(n\right)+g\left(n\right)=\left(w+g\right)\left(n\right)={}^o\left(w+g\right)\left(n\right)\hspace{0.5cm}\text{for all $n\in\omega_+$}
\end{align*}
Therefore, by the \nameref{Equivalence Lemma} \ref{Equivalence Lemma} we have that $\left[{}^ow+{}^og\right]=\left[{}^o\left(w+g\right)\right]$. This shows that 
\begin{align*}
\left[\left(f+{}^og\right)+\left(g+w\right)\right]={}^o\left[\left(w+g\right)\right]+\left[w+g\right]=\left[0\right]={}^o\left[w\right]+\left[w\right]=\left[f+w\right]
\end{align*}
 Also, if $\left|\left[f\right]\right|>\left|\left[w\right]\right|$ then $b$ strictly dominates $w$. This is because if it did not then for all $n\in\omega_+$ we have $b\left(n\right)=w\left(n\right)$ as $b$ dominates $w$. Therefore, for all $n\in\omega_+$ we have $\left|b\right|\left(n\right)=\left|w\right|\left(n\right)$ and $\left|b\right|\left(0\right)+\left|w\right|\left(0\right)=0=2*0$ and thus by the \nameref{Equivalence Lemma} \ref{Equivalence Lemma} we have $\left|\left[f\right]\right|=\left|\left[b\right]\right|=\left|\left[w\right]\right|$, which is a contradiction to $\left|\left[f\right]\right|>\left|\left[w\right]\right|$. Hence, $b$ strictly dominates $w$.\\
Therefore $\left(b+{}^og\right)$ strictly dominates $\left(g+w\right)$ as we have
\begin{align*}
\begin{array}{l@{}}
\forall n\in\omega_+,b\left(n\right)\geq w\left(n\right)\\
\exists m\in\omega_+,b\left(m\right)>w\left(m\right)
\end{array}
&\implies
\begin{array}{l@{}}
\forall n\in\omega_+,b\left(n\right)+{}^og\left(n\right)\geq w\left(n\right)+g\left(n\right)\\
\exists m\in\omega_+,b\left(m\right)+{}^og\left(m\right)>w\left(m\right)+g\left(m\right)
\end{array}\\
&\implies
\begin{array}{l@{}}
\forall n\in\omega_+,\left(b+{}^og\right)\left(n\right)\geq\left(g+w\right)\left(n\right)\\
\exists m\in\omega_+,\left(b+{}^og\right)\left(m\right)>\left(g+w\right)\left(m\right)
\end{array}
\end{align*}
Hence, we have by the \nameref{One-Case Lemma} \ref{One-Case Lemma} and the \nameref{Domination Theorem} \ref{Domination Theorem} that $\left|\left[f+{}^og\right]\right|=\left|\left[b+{}^og\right]\right|>\left|\left[g+w\right]\right|$. This means that by the \nameref{Rearrangement Lemma} \ref{Rearrangement Lemma} we have $\left[f+{}^og\right]+\left[g+w\right]=\left[\mu\right]$ where 
\begin{align*}
\mu\left(k\right)=\begin{cases}
f\left(0\right){}^og\left(0\right)&k=0\\
a\left(k\right)-\left(g+w\right)\left(k\right)=b\left(k\right)-w\left(k\right)&k\in\omega_+
\end{cases}
\end{align*} 
as $\mu\in\mathcal{T}_{g+w}^{f+{}^og}$.\\
Also, as $b\in\left[f\right]$ and $w\in\left[w\right]$ and $\left|\left[b\right]\right|>\left|\left[w\right]\right|$ we take $\mu'\in\mathcal{T}_{w}^{b}$, where
\begin{align*}
\mu'\left(k\right)=\begin{cases}
f\left(0\right)&k=0\\
b\left(k\right)-w\left(k\right)&k\in\omega_+
\end{cases}
\end{align*} 
Thus, for all $n\in\omega_+$ we have $\mu\left(n\right)=\mu'\left(n\right)$ and furthermore $\mu\left(0\right)+\mu'\left(0\right)=f\left(0\right)\left({}^og\left(0\right)+1\right)$. This is odd if and only if both $f\left(0\right)$ is odd and ${}^og\left(0\right)$ is even. However, that is impossible, as $\sigma\left(f\right)=\sigma\left({}^og\right)$. Hence, $\mu\left(0\right)+\mu'\left(0\right)=2m$ for some $m\in\omega$. Therefore, by the \nameref{Equivalence Lemma} \ref{Equivalence Lemma}, $\left[\mu\right]=\left[\mu'\right]$. Hence, 
\begin{align*}
\left[f+{}^og\right]+\left[g+w\right]=\left[\mu\right]=\left[\mu'\right]=\left[b+w\right]=\left[f+w\right]
\end{align*}
\textbf{If $\left|\left[f+{}^og\right]\right|<\left|\left[g+w\right]\right|$:} then $\left|\left[g+w\right]\right|\geq\left|\left[f+{}^og\right]\right|$ and as 
\begin{align*}
\left[\left(f+{}^og\right)+\left(g+w\right)\right]&=\left(\left[w\right]+\left[g\right]\right)+\left[{}^og+f\right]\\
&=\left(\left[w\right]+\left[{}^o{}^og\right]\right)+\left[{}^og+f\right]=\left[\left(w+{}^o{}^og\right)+\left({}^og+f\right)\right]
\end{align*}
and $\sigma\left({}^og\right)=\sigma\left(f\right)$ we have met the conditions for which we have proven above that 
\begin{align*}
\left[\left(w+{}^o{}^og\right)+\left({}^og+f\right)\right]=\left[w+f\right]=\left[f+w\right]
\end{align*} and thus $\left[\left(f+{}^og\right)+\left(g+w\right)\right]=\left[f+w\right]$.\footnote{This case was just commuting and relabeling to obtain the same result as was proven before.}
\paragraph{Secondly,} we suppose that $w$ is a zero function. Then
\begin{align*}
\left[\left(f+{}^og\right)+\left(g+w\right)\right]&=\left[f+{}^og\right]+\left(\left[g\right]+\left[w\right]\right)=\left[f+{}^og\right]+\left(\left[g\right]+\left[0\right]\right)\\
&=\left[f+{}^og\right]+\left[g\right]=\left[\left(f+{}^og\right)+g\right]
\end{align*}
We see that if $f\in\left[0\right]$ we obtain that
\begin{align*}
\left[\left(f+{}^og\right)+g\right]&=\left(\left[f\right]+\left[{}^og\right]\right)+\left[g\right]=\left(\left[0\right]+\left[{}^og\right]\right)+\left[g\right]=\left[{}^og\right]+\left[g\right]
\\&=\left[0\right]=\left[0\right]+\left[0\right]=\left[f\right]+\left[w\right]=\left[f+w\right]
\end{align*}
Going forward, we shall suppose that $f$ is not a zero function.\\
\textbf{If $\sigma\left(f\right)=\sigma\left({}^og\right)$:} then we have that $\left|\left[f+{}^og\right]\right|>\left|\left[g\right]\right|$. This is by the \nameref{One-Case Lemma} \ref{One-Case Lemma} as $\sigma\left(\left|f+{}^og\right|\right)=\sigma\left(\left|g\right|\right)=1$ and for all $n\in\omega_+$ we have that
\begin{align*}
\left|f+{}^og\right|\left(n\right)=\left(f+{}^og\right)\left(n\right)=f\left(n\right)+{}^og\left(n\right)=f\left(n\right)+g\left(n\right)\geq g\left(n\right)=\left|g\right|\left(n\right)
\end{align*}
Moreover, by the \nameref{Null Lemma} \ref{Null Lemma} if $f$ is a non-zero function, then there exists $m\in\omega_+$ such that 
\begin{align*}
\left|f+{}^og\right|\left(m\right)=\left(f+{}^og\right)\left(m\right)=f\left(m\right)+{}^og\left(m\right)=f\left(m\right)+g\left(m\right)>g\left(m\right)=\left|g\right|\left(m\right)
\end{align*}
Therefore, by the \nameref{Domination Theorem} \ref{Domination Theorem} we have that $\left|\left[f+{}^og\right]\right|\neq\left|\left[g\right]\right|$. Thus, $\left|\left[f+{}^og\right]\right|>\left|\left[g\right]\right|$.\\
As $\sigma\left(f\right)=\sigma\left(f+{}^og\right)\neq\sigma\left(g\right)$ we have that by the \nameref{Rearrangement Lemma} \ref{Rearrangement Lemma} that $\left[\left(f+{}^og\right)+g\right]=\left[\mu\right]$ where
\begin{align*}
\mu\left(k\right)=\begin{cases}\left(f+{}^og\right)\left(0\right)=f\left(0\right)\left(g\left(0\right)+1\right)&k=0\\
\left(f+{}^og\right)\left(k\right)-g\left(k\right)=f\left(k\right)+{}^og\left(k\right)-g\left(k\right)=f\left(k\right)&k\in\omega_+
\end{cases}
\end{align*}
as $\mu\in\mathcal{T}_g^{f+{}^og}$. Furthermore, as $\sigma\left(f\right)\neq\sigma\left(g\right)$ we have that $f\left(0\right)+\mu\left(0\right)=f\left(0\right)\left(g\left(0\right)+2\right)=2m$ for some $m\in\omega$ as either $f\left(0\right)$ or $g\left(0\right)$ is even. Then we have by the \nameref{Equivalence Lemma} \ref{Equivalence Lemma} that $\left[\mu\right]=\left[f\right]$. Hence,
\begin{align*}
\left[\left(f+{}^og\right)+\left(g+w\right)\right]=\left[\left(f+{}^og\right)+g\right]=\left[\mu\right]=\left[f\right]=\left[f\right]+\left[0\right]=\left[f\right]+\left[w\right]=\left[f+w\right]
\end{align*}
\textbf{If $\sigma\left(f\right)=\sigma\left(g\right)$:} we have three cases. If $\left|\left[f\right]\right|=\left|\left[{}^og\right]\right|$ then by the \nameref{Absolute-Opposite Lemma} \ref{Absolute-Opposite Lemma} we have that $\left[f\right]={}^o\left[{} ^og\right]=\left[g\right]$. Furthermore,
\begin{align*}
\left[\left(f+{}^og\right)+\left(g+w\right)\right]&=\left[\left(f+{}^og\right)+g\right]=\left[f+{}^og\right]+\left[g\right]=\left[0\right]+\left[g\right]\\
&=\left[g\right]=\left[f\right]=\left[f\right]+\left[0\right]=\left[f\right]+\left[w\right]=\left[f+w\right]
\end{align*} 
If $\left|\left[f\right]\right|>\left|\left[{}^og\right]\right|$ then $\left(\left(f+{}^og\right)+g\right)\left(0\right)=f\left(0\right)g\left(0\right)$. As $\sigma\left(f\right)=\sigma\left(g\right)$ we have that 
\begin{align*}
\text{$f\left(0\right)g\left(0\right)$ is odd}\iff\text{$f\left(0\right)$ is odd}
\end{align*}
Hence, $\sigma\left(\left(f+{}^og\right)+g\right)=\sigma\left(f\right)$.\\
Also, by the the definition of $\mathcal{T}_{{}^og}^f$ we have for all $a\in\mathcal{T}_{{}^og}^f$ there exists $b\in\left[f\right]$ such that $\forall n\in\omega_+$
\begin{align*}
\left(a+g\right)\left(n\right)&=a\left(n\right)+g\left(n\right)=b\left(n\right)-{}^og\left(n\right)+g\left(n\right)=b\left(n\right)
\end{align*}
Furthermore, as $f\not\in\left[0\right]$ we have that $\sigma\left(f\right)=\sigma\left(b\right)$. Thus, as $\sigma\left(a\right)=\sigma\left(f\right)=\sigma\left(g\right)$ we have that $\sigma\left(a+g\right)=\sigma\left(f\right)=\sigma\left(b\right)$.\\
This means that $\left(a+g\right)\left(0\right)+\left(b\right)\left(0\right)=2m$ for some $m\in\omega$. Therefore, by the \nameref{Equivalence Lemma} \ref{Equivalence Lemma} and the \nameref{Rearrangement Lemma} \ref{Rearrangement Lemma} we have that  
\begin{align*}
\left[\left(f+{}^og\right)+\left(g+w\right)\right]=\left[f+{}^og\right]+\left[g\right]=\left[a\right]+\left[g\right]=\left[a+g\right]=\left[b\right]=\left[f+w\right]
\end{align*}
If $\left|\left[f\right]\right|<\left|\left[{}^og\right]\right|$ then $\left(f+{}^og\right)=\varrho_f^{{}^og}$. Notice that $\left|\left[f+{}^og\right]\right|\leq\left|\left[g\right]\right|$. Suppose not, then $\left|\left[f+{}^og\right]\right|\geq\left|\left[g\right]\right|\wedge\left|\left[f+{}^og\right]\right|\neq\left|\left[g\right]\right|$. By the \nameref{Included Exclusion Lemma} \ref{Included Exclusion Lemma} we have that $\left({}^of+\varrho_f^{{}^og}\right)\in\left[{}^og\right]$, hence ${}^o\left({}^of+\varrho_f^{{}^og}\right)\in\left[{}^o{}^og\right]=\left[g\right]$.\\
We see that if there was a function in $\left[f+{}^og\right]$ strictly dominating ${}^o\left({}^of+\varrho_f^{{}^og}\right)$ then it would be strictly dominating $\left(\varrho_f^{{}^og}\right)$ as for all $n\in\omega_+$ we have 
\begin{align*}
{}^o\left({}^of+\varrho_f^{{}^og}\right)\left(n\right)&={}^o\left({}^of+\varrho_f^{{}^og}\right)\left(n\right)=\left({}^of+\varrho_f^{{}^og}\right)\left(n\right)\\
&={}^of\left(n\right)+\varrho_f^{{}^og}\left(n\right)=f\left(n\right)+\varrho_f^{{}^og}\left(n\right)\geq\varrho_f^{{}^og}\left(n\right)
\end{align*}
We see that the function in question cannot be in the same equivalence class as $\left(f+{}^og\right)$ as, by the \nameref{Domination Theorem} \ref{Domination Theorem} no function is strictly dominated by any other function in the same real number. A contradiction, therefore $\left|\left[f+{}^og\right]\right|\leq\left|\left[g+w\right]\right|$.\\
Notice that by what we have shown above $\left|\left[f+{}^og\right]\right|=\left|\left[g\right]\right|\iff\forall n\in\omega_+,f\left(n\right)=0$ which is in turn equivalent, by the \nameref{Null Lemma} \ref{Null Lemma} to $f\in\left[0\right]$. Therefore, $\left|\left[f+{}^og\right]\right|\neq\left|\left[g\right]\right|$.\\
Thus $\left|\left[f+{}^og\right]\right|<\left|\left[g\right]\right|$ then as addition is commutative we have $\left[\left(f+{}^og\right)+g\right]=\left[g+\left(f+{}^og\right)\right]$. As seen above, ${}^o\left({}^of+\varrho_f^{{}^og}\right)\in\left[g\right]$ and for all $n\in\omega_+$ we have 
\begin{align*}
{}^o\left({}^of+\varrho_f^{{}^og}\right)\left(n\right)-\varrho_f^{{}^og}\left(n\right)=f\left(n\right)+\varrho_f^{{}^og}\left(n\right)-\varrho_f^{{}^og}\left(n\right)=f\left(n\right)
\end{align*}
Thus define the function $f':\omega\rightarrow\omega$ as
\begin{align*}
f'\left(k\right)=\begin{cases}
{}^o\left({}^of+\varrho_f^{{}^og}\right)\left(0\right)&k=0\\
{}^o\left({}^of+\varrho_f^{{}^og}\right)\left(k\right)-\varrho_f^{{}^og}\left(k\right)=f\left(k\right)&k\in\omega_+
\end{cases}
\end{align*}
We see that $f'\in\mathcal{T}_{\varrho_f^{{}^og}}^{g}$.
\begin{align*}
{}^o\left({}^of+\varrho_f^{{}^og}\right)\left(0\right)&=\left({}^of+\varrho_f^{{}^og}\right)\left(0\right)+1={}^of\left(0\right)\varrho_f^{{}^og}\left(0\right)+1\\
&=\left(f\left(0\right)+1\right){}^og\left(0\right)+1=\left(f\left(0\right)+1\right)\left(g\left(0\right)+1\right)+1
\end{align*}
We see that as $\sigma\left(f\right)=\sigma\left(g\right)$
\begin{align*}
\text{$\left(f\left(0\right)+1\right)\left(g\left(0\right)+1\right)+1$ is odd}\iff\text{$f\left(0\right)$ is odd}
\end{align*}
Therefore, $\sigma\left(f'\right)=\sigma\left(f\right)$. Hence, $f'\left(0\right)+f\left(0\right)=2m$ for some $m\in\omega$ and thus by the \nameref{Equivalence Lemma} \ref{Equivalence Lemma} we have $\left[f'\right]=\left[f\right]$.\\
Therefore, by the \nameref{Rearrangement Lemma} \ref{Rearrangement Lemma} 
\begin{align*}
\left[f+{}^og\right]+\left[g+w\right]&=\left[\left(f+{}^og\right)+g\right]=\left[\varrho_f^{{}^og}\right]+\left[g\right]=\left[f'\right]\\
&=\left[f\right]=\left[f\right]+\left[0\right]=\left[f\right]+\left[w\right]=\left[f+w\right]
\end{align*}
\paragraph{Finally} if $f\in\left[0\right]$ and $w\not\in\left[0\right]$ then we have that
\begin{align*}
\left[\left(f+{}^og\right)+\left(g+w\right)\right]&=\left(\left[f\right]+\left[{}^og\right]\right)+\left[g+w\right]=\left(\left[0\right]+\left[{}^og\right]\right)+\left[g+w\right]\\
&=\left[{}^og\right]+\left[g+w\right]=\left[{}^og+\left(g+w\right)\right]
\end{align*}
However, by relabeling we obtain the same expression as above (only one of the cases however, as we demand $\sigma\left(g\right)=\sigma\left(w\right)$) we obtain that 
\begin{align*}
\left[\left(f+{}^og\right)+\left(g+w\right)\right]=\left[{}^og+\left(g+w\right)\right]=\left[w\right]=\left[0\right]+\left[w\right]=\left[f\right]+\left[w\right]=\left[f+w\right]
\end{align*}
\end{proof}

\begin{corollary}[Middle Lemma Corollary]\label{Middle Lemma Corollary}\noindent\\
 For all non-zero finite real numbers $\left[g\right]$ and all real numbers $\left[f\right]$ and $\left[w\right]$ such that $\sigma\left(\left[g\right]\right)=\sigma\left(\left[w\right]\right)$ or $\left[w\right]=\left[0\right]$ we have that $\left(\left[f\right]+{}^o\left[g\right]\right)+\left(\left[g\right]+\left[w\right]\right)=\left[f\right]+\left[w\right]$.
\end{corollary}
\begin{proofidea} If any one of the three numbers is zero the additive identity settles the matter immediately, and those three cases are disposed of first. Everything else is inherited: the \nameref{Middle Lemma} \ref{Middle Lemma} and its corollary already say that a three-fold combination may be regrouped without disturbing the total, so what remains is to see the two bracketings of $\left[f\right]+\left[g\right]+\left[w\right]$ as instances of that statement and to invoke consistency of addition to pass between representatives.
\end{proofidea}
\begin{axiom}[Associativity of Addition]\noindent\\
For all real numbers $\left[f\right]$, $\left[g\right]$ and $\left[w\right]$ we have that $\left(\left[f\right]+\left[g\right]\right)+\left[w\right]=\left[f\right]+\left(\left[g\right]+\left[w\right]\right)$.
\end{axiom}
\begin{proof}\noindent\\
\textbf{If $\left[f\right]=\left[0\right]$:} then as $\left[0\right]$ is the additive identity we have that
\begin{align*}
\left(\left[f\right]+\left[g\right]\right)+\left[w\right]=\left(\left[0\right]+\left[g\right]\right)+\left[w\right]=\left[g\right]+\left[w\right]=\left[0\right]+\left(\left[g\right]+\left[w\right]\right)=\left[f\right]+\left(\left[g\right]+\left[w\right]\right)
\end{align*}
\textbf{If $\left[g\right]=\left[0\right]$:} then as $\left[0\right]$ is the additive identity we have that
\begin{align*}
\left(\left[f\right]+\left[g\right]\right)+\left[w\right]=\left(\left[f\right]+\left[0\right]\right)+\left[w\right]=\left[f\right]+\left[w\right]=\left[f\right]+\left(\left[0\right]+\left[w\right]\right)=\left[f\right]+\left(\left[g\right]+\left[w\right]\right)
\end{align*}
\textbf{If $\left[w\right]=\left[0\right]$:} then as $\left[0\right]$ is the additive identity we have that
\begin{align*}
\left(\left[f\right]+\left[g\right]\right)+\left[w\right]=\left(\left[f\right]+\left[g\right]\right)+\left[0\right]=\left[f\right]+\left[g\right]=\left[f\right]+\left(\left[g\right]+\left[0\right]\right)=\left[f\right]+\left(\left[g\right]+\left[w\right]\right)
\end{align*}
\textbf{If all $\left[f\right]$, $\left[g\right]$ and $\left[w\right]$ are not zero} then there are four cases, 
\begin{align*}
\begin{array}{rll@{}}
\text{either}&\sigma\left(\left[f\right]\right)=\sigma\left(\left[g\right]\right)=\sigma\left(\left[w\right]\right)&\text{or}\\
&\sigma\left(\left[g\right]\right)=\sigma\left(\left[w\right]\right)\neq\sigma\left(\left[f\right]\right)&\text{or}\\
&\sigma\left(\left[f\right]\right)=\sigma\left(\left[w\right]\right)\neq\sigma\left(\left[g\right]\right)&\text{or}\\
&\sigma\left(\left[f\right]\right)=\sigma\left(\left[g\right]\right)\neq\sigma\left(\left[w\right]\right)
\end{array}
\end{align*}
$\sigma\left(f\right)=\sigma\left(g\right)=\sigma\left(w\right)$: We see straight by applying the definition of addition twice that 
\begin{align*}
\begin{array}{rll@{}}
&\left(\left(f+g\right)+w\right)\left(0\right)=\left(f+\left(g+w\right)\right)\left(0\right)=f\left(0\right)g\left(0\right)w\left(0\right)&\text{and}\\
\forall n\in\omega_+,&\left(\left(f+g\right)+w\right)\left(n\right)=\left(f+\left(g+w\right)\right)\left(n\right)=f\left(n\right)+g\left(n\right)+w\left(n\right)
\end{array}
\end{align*} And therefore $\left(\left[f\right]+\left[g\right]\right)+\left[w\right]=\left[f\right]+\left(\left[g\right]+\left[w\right]\right)$.\\
$\sigma\left(g\right)=\sigma\left(w\right)\neq\sigma\left(f\right)$: By the \nameref{Middle Lemma Corollary} \ref{Middle Lemma Corollary} and \ref{eq:middle} we have 
\begin{align*}
\begin{array}{lll@{}}
\left(\left(\left[f\right]+\left[g\right]\right)+\left[w\right]\right)+\left(\left[{}^ow\right]+\left[{}^og\right]\right)&=\left(\left[f\right]+\left[g\right]\right)+\left[{}^og\right]=\left(\left[f\right]+\left[g\right]\right)+\left(\left[{}^og\right]+\left[0\right]\right)\\
&=\left[f\right]+\left[0\right]=\left[f\right]&\text{and}\\
\left(\left[f\right]+\left(\left[g\right]+\left[w\right]\right)\right)+\left(\left[{}^ow\right]+\left[{}^og\right]\right)&=\left(\left[f\right]+\left(\left[g+w\right]\right)\right)+\left[{}^ow+{}^og\right]\\
&=\left(\left[f\right]+\left(\left[g+w\right]\right)\right)+\left({}^o\left[w+g\right]+\left[0\right]\right)\\
&=\left[f\right]+\left[0\right]=\left[f\right]
\end{array}
\end{align*} Thus, by the \nameref{Consistency Theorem'''} \ref{4.I} $\left(\left[f\right]+\left[g\right]\right)+\left[w\right]=\left[f\right]+\left(\left[g\right]+\left[w\right]\right)$.\\
$\sigma\left(f\right)=\sigma\left(w\right)\neq\sigma\left(g\right)$: By the \nameref{Middle Lemma Corollary} \ref{Middle Lemma Corollary} applied twice
\begin{align*}
\begin{array}{lll@{}}
\left(\left(\left[f\right]+\left[g\right]\right)+\left[w\right]\right)+\left(\left[{}^ow\right]+\left[{}^of\right]\right)&=\left(\left[f\right]+\left[g\right]\right)+\left[{}^of\right]=\left(\left[g\right]+\left[f\right]\right)+\left(\left[{}^of\right]+\left[0\right]\right)\\
&=\left[g\right]+\left[0\right]=\left[g\right]&\text{and}\\
\left(\left[f\right]+\left(\left[g\right]+\left[w\right]\right)\right)+\left(\left[{}^ow\right]+\left[{}^of\right]\right)&=\left(\left(\left[g\right]+\left[w\right]\right)+\left[f\right]\right)+\left(\left[{}^of\right]+\left[{}^ow\right]\right)\\
&=\left(\left[g\right]+\left[w\right]\right)+\left[{}^ow\right]=\left(\left[g\right]+\left[w\right]\right)+\left(\left[{}^ow\right]+\left[0\right]\right)\\
&=\left[g\right]+\left[0\right]=\left[g\right]
\end{array}
\end{align*} Thus, by the \nameref{Consistency Theorem'''} \ref{4.I} $\left(\left[f\right]+\left[g\right]\right)+\left[w\right]=\left[f\right]+\left(\left[g\right]+\left[w\right]\right)$.\\
$\sigma\left(f\right)=\sigma\left(g\right)\neq\sigma\left(w\right)$: By the \nameref{Middle Lemma Corollary} \ref{Middle Lemma Corollary} and \ref{eq:middle} we have 
\begin{align*}
\begin{array}{lll@{}}
\left(\left(\left[f\right]+\left[g\right]\right)+\left[w\right]\right)+\left(\left[{}^of\right]+\left[{}^og\right]\right)&=\left(\left[f+g\right]+\left[w\right]\right)+\left[{}^ow+{}^og\right]\\
&=\left(\left[w\right]+\left[f+g\right]\right)+\left({}^o\left[w+g\right]+\left[0\right]\right)\\
&=\left[w\right]+\left[0\right]=\left[w\right]&\text{and}\\
\left(\left[f\right]+\left(\left[g\right]+\left[w\right]\right)\right)+\left(\left[{}^of\right]+\left[{}^og\right]\right)&=\left(\left(\left[g\right]+\left[w\right]\right)+\left[f\right]\right)+\left(\left[{}^of\right]+\left[{}^og\right]\right)\\
&=\left(\left[g\right]+\left[w\right]\right)+\left[{}^og\right]=\left(\left[w\right]+\left[g\right]\right)+\left(\left[{}^og\right]+\left[0\right]\right)\\
&=\left[w\right]+\left[0\right]=\left[w\right]
\end{array}
\end{align*} Thus, by the \nameref{Consistency Theorem'''} \ref{4.I} $\left(\left[f\right]+\left[g\right]\right)+\left[w\right]=\left[f\right]+\left(\left[g\right]+\left[w\right]\right)$.\\
\end{proof}
\subsection{Compatibility with Order}
\begin{lemma}[Strict Additive Relation of Order Lemma]\label{Strict Additive Relation of Order Lemma}\noindent\\
For all real numbers $\left[a\right]$, $\left[b\right]$, and $\left[c\right]$ such that $\left[a\right]<\left[b\right]$ we have $\left[a\right]+\left[c\right]<\left[b\right]+\left[c\right]$.
\end{lemma}
\begin{proofidea} One mechanism is used over and over. To compare two sums we do not compare the classes directly: we choose representatives, observe that adding a common number leaves one pointwise above the other at every position, and note that the \nameref{Null Lemma} \ref{Null Lemma} supplies a position at which the inequality is strict. The \nameref{One-Case Lemma} \ref{One-Case Lemma} then converts that single pointwise comparison into a comparison of real numbers, and the \nameref{Domination Theorem} \ref{Domination Theorem} upgrades it to a strict one. The length of the proof is the case matrix, sub-addition branches on the signs of the three numbers, and on which magnitude is larger when signs differ, but every branch runs on that same pairing.
\end{proofidea}
\begin{proof} Suppose $\left[a\right]<\left[b\right]$. We see by the axiom of additive identity that if $\left[c\right]=\left[0\right]$ then the implication holds. Moving forward, we suppose $\left[c\right]\neq\left[0\right]$. We shall prove this by cases:\\
\textbf{Case 1 $\left[a\right]=\left[0\right]$:} Hence, $\sigma\left(\left[b\right]\right)=1$. If $\sigma\left(\left[c\right]\right)=1$ then $\sigma\left(\left[c\right]\right)=\sigma\left(\left[b\right]+\left[c\right]\right)=1$ and by taking any $b\in\left[b\right]$ and $f\in\left[c\right]$ we see that $\left(b+f\right)\in\left(\left[b\right]+\left[c\right]\right)$ where for all $n\in\omega_+$ we have that $\left(b+f\right)\left(n\right)=b\left(n\right)+f\left(n\right)\geq f\left(n\right)$. thus, as, by the \nameref{Null Lemma} \ref{Null Lemma} we have that there exists $m\in\omega_+$ such that $b\left(m\right)\geq1$ and thus $\left(b+f\right)$ strictly dominates $f$. Hence by the \nameref{One-Case Lemma} \ref{One-Case Lemma} and the \nameref{Domination Theorem} \ref{Domination Theorem} we have that 
\begin{align*}
\left[a\right]+\left[c\right]=\left[0\right]+\left[c\right]=\left[c\right]<\left[b\right]+\left[c\right]
\end{align*}
If $\sigma\left(\left[c\right]\right)=0$ then either, $\left|\left[c\right]\right|=\left|\left[b\right]\right|$ or $\left|\left[c\right]\right|<\left|\left[b\right]\right|$ or $\left|\left[c\right]\right|>\left|\left[b\right]\right|$. In the first case $\left[b\right]+\left[c\right]=\left[0\right]$ and $\sigma\left(\left[c\right]\right)=0$ and thus 
\begin{align*}
\left[a\right]+\left[c\right]=\left[0\right]+\left[c\right]=\left[c\right]<\left[b\right]+\left[c\right]
\end{align*}
In the second case, $\sigma\left(\left[c\right]\right)=0$ and $\sigma\left(\left[b\right]+\left[c\right]\right)=\sigma\left(\left[b\right]\right)=1$ and thus 
\begin{align*}
\left[a\right]+\left[c\right]=\left[0\right]+\left[c\right]=\left[c\right]<\left[b\right]+\left[c\right]
\end{align*}
In the last case, $\sigma\left(\left[b\right]+\left[c\right]\right)=\sigma\left(\left[c\right]\right)=0$. Take any $b\in\left[b\right]$. By the \nameref{Included Exclusion Lemma} \ref{Included Exclusion Lemma} we have that $\left({}^ob+\varrho_b^c\right)\in\left[c\right]$ and thus for all $n\in\omega_+$ we have that $\left({}^ob+\varrho_b^c\right)\left(n\right)={}^ob\left(n\right)+\varrho_b^c\left(n\right)\geq\varrho_b^c\left(n\right)$. By the \nameref{Null Lemma} \ref{Null Lemma} there exists $m\in\omega_+$ such that ${}^ob\left(m\right)=b\left(m\right)\geq1$ and thus $\left({}^ob+\varrho_b^c\right)$ strictly dominates $\varrho_b^c\left(n\right)$. Hence by the \nameref{One-Case Lemma} \ref{One-Case Lemma} and the \nameref{Domination Theorem} \ref{Domination Theorem} we have
\begin{align*}
\left[a\right]+\left[c\right]=\left[0\right]+\left[c\right]=\left[c\right]<\left[b\right]+\left[c\right]
\end{align*}
\textbf{Case 2 $\left[a\right]\neq\left[0\right]\wedge\left[b\right]=\left[0\right]$:} We must have $\sigma\left(\left[a\right]\right)=0$. The proof is the same as above with inverted signs.\\\\
\textbf{Case 3 $\sigma\left(\left[a\right]\right)=\sigma\left(\left[b\right]\right)=1$:} If $\sigma\left(\left[c\right]\right)=1$ then we take any $a\in\left[a\right]$ and we know that there exists $b\in\left[b\right]$ such that for all $n\in\omega_+$ we have $a\left(n\right)\leq b\left(n\right)$ and there exists $m\in\omega_+$ such that $a\left(m\right)<b\left(m\right)$. Let any $c\in\left[c\right]$ be given. We see that for all $n\in\omega_+$ we have $\left(a+c\right)\left(n\right)=a\left(n\right)+c\left(n\right)\leq b\left(n\right)+c\left(n\right)=\left(b+c\right)\left(n\right)$ and there exists $m\in\omega_+$ such that $\left(a+c\right)\left(m\right)=a\left(m\right)+c\left(m\right)<b\left(m\right)+c\left(m\right)=\left(b+c\right)\left(m\right)$. Hence by the \nameref{One-Case Lemma} \ref{One-Case Lemma} and the \nameref{Domination Theorem} \ref{Domination Theorem} we have $\left[a\right]+\left[c\right]<\left[b\right]+\left[c\right]$.\\
If $\sigma\left(\left[c\right]\right)=0$ then we have five cases 
\begin{align*}
\begin{array}{l@{}}
\left|\left[c\right]\right|=\left|\left[a\right]\right|<\left|\left[b\right]\right|\\
\left|\left[a\right]\right|<\left|\left[c\right]\right|=\left|\left[b\right]\right|\\
\left|\left[c\right]\right|<\left|\left[a\right]\right|<\left|\left[b\right]\right|\\
\left|\left[a\right]\right|<\left|\left[b\right]\right|<\left|\left[c\right]\right|\\
\left|\left[a\right]\right|<\left|\left[c\right]\right|<\left|\left[b\right]\right|
\end{array}
\end{align*}
If $\left|\left[c\right]\right|=\left|\left[a\right]\right|<\left|\left[b\right]\right|$ then $\left[c\right]=\left[{}^oa\right]$ and $\sigma\left(\left[b\right]+\left[c\right]\right)=\sigma\left(\left[b\right]\right)=1$ and thus 
\begin{align*}
\left[a\right]+\left[c\right]=\left[0\right]<\left[b\right]+\left[c\right]
\end{align*}
If $\left|\left[a\right]\right|<\left|\left[c\right]\right|=\left|\left[b\right]\right|$ then $\left[c\right]=\left[{}^ob\right]$ and $\sigma\left(\left[a\right]+\left[c\right]\right)=\sigma\left(\left[c\right]\right)=0$ and thus 
\begin{align*}
\left[a\right]+\left[c\right]<\left[0\right]=\left[b\right]+\left[c\right]
\end{align*}
If $\left|\left[c\right]\right|<\left|\left[a\right]\right|<\left|\left[b\right]\right|$ then $\sigma\left(\left[a\right]+\left[c\right]\right)=\sigma\left(\left[b\right]+\left[c\right]\right)=1$. Furthermore, by the \nameref{Included Exclusion Lemma} \ref{Included Exclusion Lemma} we have $\left({}^oc+\varrho_c^a\right)\in\left[a\right]$. Moreover, we know that there exists $b\in\left[b\right]$ such that for all $n\in\omega_+$ we have $b\left(n\right)\geq\left({}^oc+\varrho_c^a\right)\left(n\right)={}^oc\left(n\right)+\varrho_c^a\left(n\right)=\varrho_c^a\left(n\right)+c\left(n\right)$ and there exists $m\in\omega_+$ such that we have $b\left(m\right)>\left({}^oc+\varrho_c^a\right)\left(m\right)={}^oc\left(m\right)+\varrho_c^a\left(m\right)=\varrho_c^a\left(m\right)+c\left(m\right)$.\\
Take $\mu\in\mathcal{T}_c^b$ where 
\begin{align*}
\mu\left(k\right)=\begin{cases}
b\left(0\right)&k=0\\
b\left(k\right)-c\left(k\right)&k\in\omega_+
\end{cases}
\end{align*}
Hence, by the \nameref{One-Case Lemma} \ref{One-Case Lemma} and the \nameref{Domination Theorem} \ref{Domination Theorem} as for all $n\in\omega_+$ we have that $\mu\left(n\right)\geq\varrho_c^a\left(n\right)$ and there exists $m\in\omega_+$ such that we have that $\mu\left(m\right)>\varrho_c^a\left(m\right)$ we have 
\begin{align*}
\left[a\right]+\left[c\right]=\left[\varrho_c^a\right]<\left[\mu\right]=\left[b\right]+\left[c\right]
\end{align*}
If $\left|\left[a\right]\right|<\left|\left[b\right]\right|<\left|\left[c\right]\right|$ then $\sigma\left(\left[a\right]+\left[c\right]\right)=\sigma\left(\left[b\right]+\left[c\right]\right)=0$. Furthermore, take any $a\in\left[a\right]$. By the definition of order there exists $b\in\left[b\right]$ such that for all $n\in\omega_+$ we have $a\left(n\right)\leq b\left(n\right)$ and there exists $m\in\omega_+$ we have $a\left(m\right)>b\left(m\right)$. By the definition of order and absolute value we have that $\exists g\in\left|\left[c\right]\right|,\forall n\in\omega_+,a\left(n\right)\leq b\left(n\right)=\left|b\right|\left(n\right)\leq g\left(n\right)$. However, $\left|\left[{}^og\right]\right|=\left|\left[c\right]\right|$ and $\sigma\left({}^og\right)=\sigma\left(c\right)=0$ and thus by the \nameref{Absolute-Opposite Lemma} \ref{Absolute-Opposite Lemma} $\left[{}^og\right]=\left[c\right]$ and thus ${}^og\in\left[c\right]$.\\
Hence, we have $\mu_1\in\mathcal{T}_a^{{}^og}$ and $\mu_2\in\mathcal{T}_b^{{}^og}$ where 
\begin{align*}
\begin{array}{ll@{}}
\mu_1\left(k\right)=
\begin{cases}
{}^og\left(0\right)&k=0\\
{}^og\left(k\right)-a\left(k\right)&k\in\omega_+
\end{cases}&
\mu_2\left(k\right)=
\begin{cases}
{}^og\left(0\right)&k=0\\
{}^og\left(k\right)-b\left(k\right)&k\in\omega_+
\end{cases}
\end{array}
\end{align*} 
Therefore, for all $n\in\omega_+$ we have $\mu_1\left(n\right)\geq\mu_2\left(n\right)$ and there exists $m\in\omega_+$ such that we have $\mu_1\left(m\right)>\mu_2\left(m\right)$ and thus by the \nameref{One-Case Lemma} \ref{One-Case Lemma}, the \nameref{Domination Theorem} \ref{Domination Theorem} and the \nameref{Rearrangement Lemma} \ref{Rearrangement Lemma} we have that 
\begin{align*}
\left[a\right]+\left[c\right]=\left[\mu_1\right]<\left[\mu_2\right]=\left[b\right]+\left[c\right]
\end{align*}
If $\left|\left[a\right]\right|<\left|\left[c\right]\right|<\left|\left[b\right]\right|$ then $\sigma\left(\left[a\right]+\left[c\right]\right)=\sigma\left(\left[c\right]\right)=0$ and $\sigma\left(\left[b\right]+\left[c\right]\right)=\sigma\left(\left[b\right]\right)=1$ and thus
\begin{align*}
\left[a\right]+\left[c\right]<\left[0\right]<\left[b\right]+\left[c\right]
\end{align*}
\textbf{Case 4 $\sigma\left(\left[a\right]\right)=\sigma\left(\left[b\right]\right)=0$:} The proof is the same as above with inverted signs.\\\\
\textbf{Case 5 $\sigma\left(\left[a\right]\right)=0\wedge\sigma\left(\left[b\right]\right)=1$:} We have that
\begin{align*}
\left[a\right]<\left[0\right]<\left[b\right]
\end{align*}
And thus by the cases 1 and 2 we have that for any real number $\left[c\right]$
\begin{align*}
\begin{array}{ll@{}}
\left[a\right]+\left[c\right]<\left[0\right]+\left[c\right]=\left[c\right]&\text{and}\\
\left[c\right]=\left[0\right]+\left[c\right]<\left[b\right]+\left[c\right]
\end{array}
\end{align*}
Hence,
\begin{align*}
\left[a\right]+\left[c\right]<\left[c\right]<\left[b\right]+\left[c\right]
\end{align*}
\end{proof}

\begin{axiom}[Additive Relation of Order]\noindent\\
For all real numbers $\left[a\right]$, $\left[b\right]$, and $\left[c\right]$ such that $\left[a\right]\leq\left[b\right]$ then we have $\left[a\right]+\left[c\right]\leq\left[b\right]+\left[c\right]$.
\end{axiom}
\begin{proof} $\left[a\right]\leq\left[b\right]$ implies $\left[a\right]<\left[b\right]$ or $\left[a\right]=\left[b\right]$. By the \nameref{Strict Additive Relation of Order Lemma} \ref{Strict Additive Relation of Order Lemma} the former case implies  $\left[a\right]+\left[c\right]<\left[b\right]+\left[c\right]$ which implies $\left[a\right]+\left[c\right]\leq\left[b\right]+\left[c\right]$. The latter case implies that $\left[a\right]+\left[c\right]=\left[b\right]+\left[c\right]$ by the \nameref{Consistency Theorem'''} \ref{4.I} \ref{Consistency Theorem'''}, which implies $\left[a\right]+\left[c\right]\leq\left[b\right]+\left[c\right]$.
\end{proof}
\section{Multiplication}\label{sec:mult}
Multiplication is where a base that flows makes itself felt. In a constant base the product of a digit at position $n$ with a digit at position $k$ lands squarely at position $n+k-1$, which is the rule of long multiplication; when the base varies, the combined scale of the two positions need not be the scale of any position at all, and the contribution has to be distributed across several. The rest function is the bookkeeping for that distribution, and it is the reason this section is longer than the last. The architecture is otherwise the same as for addition: define the operation on representatives, show a carry inside a factor is invisible to the product, prove that the definition descends to the classes, and then verify the axioms.
\subsection{The Rest Function and Sub-Multiplication}
\begin{definition}[Rest Function]\label{Rest Function}\noindent\\
For all $k$, $n$ and $M$ in $\omega_+$ we define the rest function as:
\begin{align*}
    r\left(n,k,M\right)=\max\left(\left\{t\in\omega\middle|t\prod\limits_{j=1}^{n-1}\left(\vartheta\left(j\right)\right)\prod\limits_{j=1}^{k-1}\left(\vartheta\left(j\right)\right)\leq\prod\limits_{j=1}^{M-1}\left(\vartheta\left(j\right)\right)\right\}\right)
\end{align*}
if $M=1$ and
\begin{align*}
r\left(n,k,M\right):=\Upsilon\left(\begin{array}{r@{}}
\max\left(\left\{t\in\omega\middle|t\prod\limits_{j=1}^{n-1}\left(\vartheta\left(j\right)\right)\prod\limits_{j=1}^{k-1}\left(\vartheta\left(j\right)\right)\leq\prod\limits_{j=1}^{M-1}\left(\vartheta\left(j\right)\right)\right\}\right)\\
\vartheta\left(M-1\right)\max\left(\left\{t\in\omega\middle|t\prod\limits_{j=1}^{n-1}\left(\vartheta\left(j\right)\right)\prod\limits_{j=1}^{k-1}\left(\vartheta\left(j\right)\right)\leq\prod\limits_{j=1}^{M-2}\left(\vartheta\left(j\right)\right)\right\}\right)
\end{array}\right)
\end{align*}
if $M\geq2$
\end{definition}

\begin{remark} Notice that the second argument in the $\Upsilon$ function is always smaller, as for all $a$, $b$ and $c$ in $\omega_+$ we have that
\begin{align*}
a\max\left(\left\{t\in\omega|tb\leq c\right\}\right)\leq\max\left(\left\{t\in\omega|tb\leq ac\right\}\right)
\end{align*}
Also notice that for all $k$, $n$, $M$ in $\omega_+$ we have $r\left(n,k,M\right)=r\left(k,n,M\right)$.
\end{remark}

\begin{intuition}[Rest Function] Multiplication is where a mixed-radix system earns its keep. When we multiply a digit at position $n$ by a digit at position $k$, their product lives at a combined, finer scale; to write the answer as a proper number we must redistribute that value into position-$M$ digits. The rest function $r\left(n,k,M\right)$ is exactly that redistribution rule, it counts how many whole units of the scale at position $M$ fit inside the combined scale of $n$ and $k$. It is the grade-school rule for carrying a partial product into its column, written out honestly for a base that changes from position to position. (For the first position this is a single count; past it, it is the difference of two counts, which is what isolates the $M$-th digit.)
\end{intuition}

\begin{definition}[Sub-Multiplication]\label{Sub-Multiplication}\noindent\\
We define the sub-multiplication of two finite functions $f$ and $g$ as a function $\left(f*g\right):\omega\rightarrow\omega$ where
\begin{align*}
\left(f*g\right)\left(M\right)=
\begin{cases}
f\left(0\right)+g\left(0\right)&M=0\\
\sum\limits_{n=1}^{M}\left(\sum\limits_{k=1}^{M}f\left(n\right)g\left(k\right)\left(r\left(n,k,M\right)\right)\right)&M\in\omega_+
\end{cases}
\end{align*}
\end{definition}

\begin{remark} Notice that as $r\left(n,k,M\right)=r\left(k,n,M\right)$ we have that sub-multiplication is commutative.
\end{remark}

\begin{intuition}[Sub-Multiplication] To multiply two numbers we convolve their digit strings: every pair of positions $\left(n,k\right)$ contributes the product of its digits, weighted by the rest function that says how much of that contribution lands at the output position $M$. The sign slot at $0$ is combined by \emph{adding} the two values there, which, again reading positive as even and negative as odd, multiplies the signs correctly, so like signs give a positive product and unlike signs a negative one.
\end{intuition}
\begin{example}\label{ex:submult} Take the base of Example 2, $\vartheta(n)=n+1$, so that position $n$ carries weight $\nicefrac{1}{n!}$: the weights are $1,\nicefrac{1}{2},\nicefrac{1}{6},\nicefrac{1}{24},\dots$ In this system let $f=g=(0;0,1,0,\dots)$, both of value $\nicefrac{1}{2}$.\\
The only non-zero digits are $f(2)=g(2)=1$, so the convolution collapses to the single term $r(2,2,M)$ at each output position, and everything depends on the rest function. Since $\prod_{j=1}^{1}(\vartheta(j))=2$, the two denominators multiply to $4$, and the running products $\prod_{j=1}^{M-1}(\vartheta(j))$ are $1,2,6,24,120,\dots$ Reading off the maxima in the definition,
\begin{align*}
r(2,2,3)=\Upsilon(1,\vartheta(2)\cdot0)=1,\qquad
r(2,2,4)=\Upsilon(6,\vartheta(3)\cdot1)=\Upsilon(6,4)=2,
\end{align*} 
while $r(2,2,1)=r(2,2,2)=0$, and from $M=5$ onwards the two arguments coincide, at $M=5$ they are $30$ and $\vartheta(4)\cdot6=30$, so every later value is $0$ as well. Hence
\begin{align*}
f*g=(0;0,0,1,2,0,\dots),
\end{align*} 
of value $\nicefrac{1}{6}+\nicefrac{2}{24}=\nicefrac{1}{4}$, and no digit reaches its base, since $1<\vartheta(2)=3$ and $2<\vartheta(3)=4$, so this is already the primary auxiliary function of $\nicefrac{1}{4}$.\\
This is the behaviour the rest function exists to describe. A single pair of positions here contributes to \emph{two} output positions at once, because the combined weight $\nicefrac{1}{2}\cdot\nicefrac{1}{2}=\nicefrac{1}{4}$ is not a weight of any position in this system: it must be expressed as $1\cdot\nicefrac{1}{6}+2\cdot\nicefrac{1}{24}$. In a constant base this never happens, there $r(n,k,M)$ is $1$ when $M=n+k-1$ and $0$ otherwise, which is just the rule of long multiplication, and it is only when the base flows that the general definition earns its complexity. Even though some strings may be infinite, by considering at each point only strings from lower positions we keep the calculations at every point finite. Figure~\ref{fig:convolution} shows which pairs $(n,k)$ contribute to which output position, and how a single partial product can straddle several when the base flows.
\end{example}

\begin{figure}[tbp]
  \centering
\begin{tikzpicture}[x=1.05cm, y=0.72cm,
  cell/.style={draw=black!35, minimum width=1.02cm, minimum height=0.7cm,
               font=\scriptsize, inner sep=0pt},
  d1/.style={cell, fill=blue!13}, d2/.style={cell, fill=orange!20},
  d3/.style={cell, fill=green!16!black!8}, d4/.style={cell, fill=black!7},
  ax/.style={font=\scriptsize, text=black!70},
]
 \node[ax] at (-1.15,1) {$k=$};
 \foreach \k in {1,2,3,4}{\node[ax] at (\k,1) {\k};}
 \node[ax] at (-1.15,0) {$n=$};
 \foreach \n in {1,2,3,4}{\node[ax] at (-1.15,-\n+0) {\n};}
 
 \foreach \n in {1,2,3,4}{
  \foreach \k in {1,2,3,4}{
    \pgfmathtruncatemacro{\M}{\n+\k-1}
    \pgfmathtruncatemacro{\C}{min(\M,4)}
    \node[d\C] at (\k,-\n) {$M=\M$};
  }}
 
 \draw[decorate,decoration={brace,amplitude=5pt},black!55]
   (4.55,-0.65) -- (4.55,-4.35)
   node[midway,right=6pt,ax,align=left,text width=4.4cm]
   {cell $(n,k)$ holds $f(n)\,g(k)$, and
    $r(n,k,M)$ says how much of it lands at output position $M$};
 
 \node[ax,align=left,text width=9.2cm] at (2.4,-5.9)
  {\textbf{Constant base.} Each partial product lands in exactly one column:
   $r(n,k,M)=1$ when $M=n+k-1$ and $0$ otherwise, so the shaded
   anti-diagonals are the columns of ordinary long multiplication.\\[2pt]
   \textbf{Flowing base.} The diagonals blur. With
   $\vartheta=(2,3,5,\dots)$ the single cell $n=k=2$ contributes
   $r(2,2,3)=1$ at position $3$ \emph{and} $r(2,2,4)=2$ at position $4$:
   one partial product spread across two columns. The rest function is
   exactly the bookkeeping for that spreading.};
\end{tikzpicture}
  \caption{Sub-multiplication as a convolution over pairs of
positions. In a constant base the contribution of the pair $(n,k)$ falls
entirely at position $n+k-1$, recovering schoolbook long multiplication. When
the base flows, a single partial product may straddle several output
positions, and $r(n,k,M)$ records how much of it lands at each.}
  \label{fig:convolution}
\end{figure}
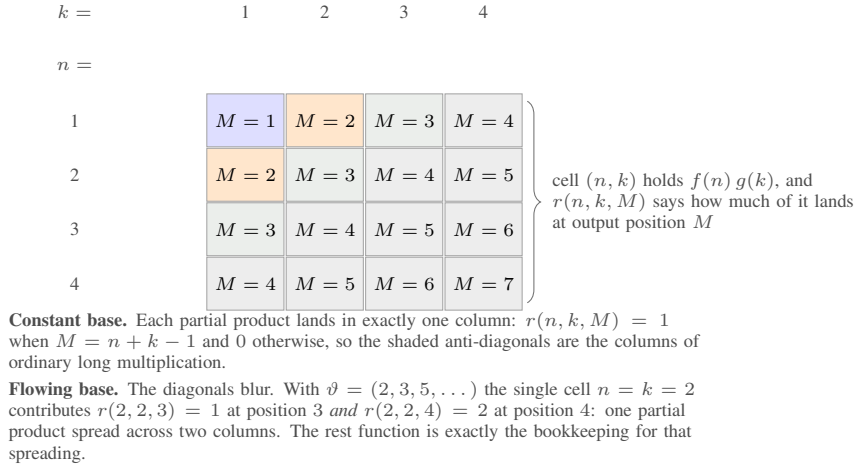
\begin{remark} For the sake of clarity and not cluttering the pages we will introduce new notation for the identity showed above. For all $n$, $k$ and $K$ in $\omega_+$ we denote
\begin{align*}
R\left(n,k,K\right):=\max\left(\left\{t\in\omega\middle|t\prod\limits_{j=1}^{n-1}\left(\vartheta\left(j\right)\right)\prod\limits_{j=1}^{k-1}\left(\vartheta\left(j\right)\right)\leq\prod\limits_{j=1}^{K-1}\left(\vartheta\left(j\right)\right)\right\}\right)
\end{align*}
\begin{lemma}[Rest Identity]\label{Rest Identity}\noindent\\
For all $n$, $k$, $K$ in $\omega_+$ we have
\begin{equation}\label{eq:rest-identity}
\sum\limits_{M=\max\left(\left\{n,k\right\}\right)}^K\left(r\left(n,k,M\right)\prod\limits_{p=M}^{K-1}\left(\vartheta\left(p\right)\right)\right)=R\left(n,k,K\right)
\end{equation}
\end{lemma}

\begin{intuition}[Rest Identity] This identity is the guarantee that sub-multiplication computes the true product and not merely something that looks like it. The rest coefficients are, position by position, the mixed-radix digits of the combined scale; summing them back with their scales must reconstruct the number of whole units of the scale at level $K$ that fit inside the combined scale of $n$ and $k$.
\end{intuition}

\begin{proof} We shall prove this by induction. We see directly that for $K=\max\left(\left\{n,k\right\}\right)$ that is simply the definition of the ``rest" function. Now suppose that the identity holds for $K-1$. We see that:
\begin{align*}
&\sum\limits_{M=\max\left(\left\{n,k\right\}\right)}^K\left(r\left(n,k,M\right)\prod\limits_{p=M}^{K-1}\left(\vartheta\left(p\right)\right)\right)\\
=&\vartheta\left(K-1\right)\sum\limits_{M=\max\left(\left\{n,k\right\}\right)}^{K-1}\left(r\left(n,k,M\right)\prod\limits_{p=M}^{K-2}\left(\vartheta\left(p\right)\right)\right)+r\left(n,k,K\right)\\
=&\vartheta\left(K-1\right)\max\left(\left\{t\in\omega\middle|t\prod\limits_{j=1}^{n-1}\left(\vartheta\left(j\right)\right)\prod\limits_{j=1}^{k-1}\left(\vartheta\left(j\right)\right)\leq\prod\limits_{j=1}^{K-2}\left(\vartheta\left(j\right)\right)\right\}\right)\\
&+\Upsilon\left(\begin{array}{r@{}}
\max\left(\left\{t\in\omega\middle|t\prod\limits_{j=1}^{n-1}\left(\vartheta\left(j\right)\right)\prod\limits_{j=1}^{k-1}\left(\vartheta\left(j\right)\right)\leq\prod\limits_{j=1}^{K-1}\left(\vartheta\left(j\right)\right)\right\}\right)\\
\vartheta\left(K-1\right)\max\left(\left\{t\in\omega\middle|t\prod\limits_{j=1}^{n-1}\left(\vartheta\left(j\right)\right)\prod\limits_{j=1}^{k-1}\left(\vartheta\left(j\right)\right)\leq\prod\limits_{j=1}^{K-2}\left(\vartheta\left(j\right)\right)\right\}\right)
\end{array}\right)\\
=&\max\left(\left\{t\in\omega\middle|t\prod\limits_{j=1}^{n-1}\left(\vartheta\left(j\right)\right)\prod\limits_{j=1}^{k-1}\left(\vartheta\left(j\right)\right)\leq\prod\limits_{j=1}^{K-1}\left(\vartheta\left(j\right)\right)\right\}\right)
\end{align*}
\end{proof}
\end{remark}
\subsection{Consistency}
\begin{lemma}[Sub-Multiplication Finiteness Lemma]\label{Sub-Multiplication Finiteness Lemma}\noindent\\
For all finite functions $f$ and $g$ we have that $\left(f*g\right)$ is finite.
\end{lemma}

\begin{intuition}[Finiteness and Carrying for Sub-Multiplication] As with addition, before multiplication can descend to real numbers we must know two things. Finiteness: the product of two finite numbers has finite material and so is again a finite number. Carrying: a carry inside a factor commutes with the product, so multiplying is insensitive to how each factor is written. These are the multiplicative echoes of the sub-addition lemmas, and they set up the consistency theorems the same way.
\end{intuition}

\begin{proof} Take two finite functions $f$ and $g$. By the \nameref{Tying Theorem} \ref{Tying Theorem} we know that there exist constants $C_1$ and $C_2$ in $\omega_+$ such that for all $K\in\omega_+$ we have 
\begin{align*}
\begin{array}{ll@{}}
\sum\limits_{k=1}^K\left(f\left(k\right)\prod\limits_{p=k}^{K-1}\left(\vartheta\left(p\right)\right)\right)<C_1\prod\limits_{p=1}^{K-1}\left(\vartheta\left(p\right)\right)&\text{and}\\
\sum\limits_{k=1}^K\left(g\left(k\right)\prod\limits_{p=k}^{K-1}\left(\vartheta\left(p\right)\right)\right)<C_2\prod\limits_{p=1}^{K-1}\left(\vartheta\left(p\right)\right)
\end{array}
\end{align*}
Hence, $\forall K\in\omega_+$ we have
\begin{align*}
&\sum\limits_{M=1}^K\left(\left(f*g\right)\left(M\right)\prod\limits_{p=M}^{K-1}\left(\vartheta\left(p\right)\right)\right)=\sum\limits_{M=1}^K\left(\prod\limits_{p=M}^{K-1}\left(\vartheta\left(p\right)\right)\left(\sum\limits_{n=1}^{M}\left(\sum\limits_{k=1}^{M}f\left(n\right)g\left(k\right)\left(r\left(n,k,M\right)\right)\right)\right)\right)\\
=&\sum\limits_{n=1}^K\left(f\left(n\right)\sum\limits_{k=1}^K\left(g\left(k\right)\sum\limits_{M=\max\left(\left\{n,k\right\}\right)}^K\left(r\left(n,k,M\right)\prod\limits_{p=M}^{K-1}\left(\vartheta\left(p\right)\right)\right)\right)\right)\\
=&\sum\limits_{n=1}^K\left(f\left(n\right)\sum\limits_{k=1}^K\left(g\left(k\right)\max\left(\left\{t\in\omega\middle|t\prod\limits_{j=1}^{k-1}\left(\vartheta\left(j\right)\right)\prod\limits_{j=1}^{n-1}\left(\vartheta\left(j\right)\right)\leq\prod\limits_{j=1}^{K-1}\left(\vartheta\left(j\right)\right)\right\}\right)\right)\right)\\
\leq&\sum\limits_{n=1}^K\left(f\left(n\right)\left(\max\left(\left\{t\in\omega\middle|t\prod\limits_{j=1}^{n-1}\left(\vartheta\left(j\right)\right)\leq\sum\limits_{k=1}^K\left(g\left(k\right)\prod\limits_{j=k}^{K-1}\left(\vartheta\left(j\right)\right)\right)\right\}\right)\right)\right)\\
\leq&\sum\limits_{n=1}^K\left(f\left(n\right)\max\left(\left\{t\in\omega\middle|t\prod\limits_{j=1}^{n-1}\left(\vartheta\left(j\right)\right)<C_2\prod\limits_{p=1}^{K-1}\left(\vartheta\left(p\right)\right)\right\}\right)\right)\\
\leq&\max\left(\left\{t\in\omega\middle|t<C_2\sum\limits_{n=1}^K\left(f\left(n\right)\prod\limits_{p=n}^{K-1}\left(\vartheta\left(p\right)\right)\right)\right\}\right)\\
\leq&\max\left(\left\{t\in\omega\middle|t<C_1C_2\prod\limits_{p=1}^{K-1}\left(\vartheta\left(p\right)\right)\right\}\right)<C_1C_2\prod\limits_{p=1}^{K-1}\left(\vartheta\left(p\right)\right)
\end{align*}
And therefore, by the \nameref{Tying Theorem} \ref{Tying Theorem}, $f*g$ is a finite function for all finite functions $f$ and $g$.
\end{proof}

\begin{lemma}[Sub-Multiplication Contraction Lemma]\label{Sub-Multiplication Contraction Lemma}\noindent\\
For all finite functions $f$ and $g$ where $f$ is contractable about some $m\in\omega_+$ we have $\left[f*g\right]=\left[c_m\left(f\right)*g\right]$.
\end{lemma}
\begin{proofidea} We never match the two products digit by digit. Instead we show that their level-$K$ weighted truncations stay within one unit of the scale at $K$, for every $K$, and let the \nameref{Extension Theorem} \ref{Extension Theorem} conclude that the two products are the same real number, the sign at position $0$ being settled separately and at once, since a contraction does not touch it. Almost all of the work is that estimate, and what keeps it under control is the step identity for the rest function, which says precisely that replacing a digit at $m+1$ by a unit at $m$ redistributes the partial products up to a remainder smaller than $\vartheta\left(m\right)$.
\end{proofidea}
\begin{proof} We will want to use the extension theorem. Let a finite function $f$ contractable about some $m$ be given. First of all, we see that as $f\left(0\right)=c_m\left(f\right)\left(0\right)$ we have $\left(f*g\right)\left(0\right)=\left(c_m\left(f\right)*g\right)\left(0\right)=f\left(0\right)+g\left(0\right)$ and thus $\sigma\left(f*g\right)=\sigma\left(c_m\left(f\right)*g\right)$.\\
Now, take a $K\in\omega_+$ such that $K>m+1$, we see that 
\begin{align*}
&\Upsilon\left(\begin{array}{l@{}}\sum\limits_{M=1}^{K}\left(\left(f*g\right)\left(M\right)\prod\limits_{j=M}^{K-1}\left(\vartheta\left(j\right)\right)\right)\\\sum\limits_{M=1}^{K}\left(\left(c_m\left(f\right)*g\right)\left(M\right)\prod\limits_{j=M}^{K-1}\left(\vartheta\left(j\right)\right)\right)\end{array}\right)\\
=&\Upsilon\left(\begin{array}{l@{}}\sum\limits_{M=1}^{K}\left(\sum\limits_{n=1}^{M}\left(f\left(n\right)\sum\limits_{k=1}^{M}g\left(k\right)r\left(n,k,M\right)\prod\limits_{j=M}^{K-1}\left(\vartheta\left(j\right)\right)\right)\right)\\\sum\limits_{M=1}^{K}\left(\sum\limits_{n=1}^{M}\left(c_m\left(f\right)\left(n\right)\sum\limits_{k=1}^{M}g\left(k\right)r\left(n,k,M\right)\prod\limits_{j=M}^{K-1}\left(\vartheta\left(j\right)\right)\right)\right)\end{array}\right)\\
=&\Upsilon\left(\begin{array}{l@{}}\sum\limits_{n=1}^{K}\left(f\left(n\right)\sum\limits_{k=1}^{K}\left(g\left(k\right)\sum\limits_{M=\max\left(\left\{n,k\right\}\right)}^{K}\left(r\left(n,k,M\right)\prod\limits_{j=M}^{K-1}\left(\vartheta\left(j\right)\right)\right)\right)\right)\\\sum\limits_{n=1}^{K}\left(c_m\left(f\right)\left(n\right)\sum\limits_{k=1}^{K}\left(g\left(k\right)\sum\limits_{M=\max\left(\left\{n,k\right\}\right)}^{K}\left(r\left(n,k,M\right)\prod\limits_{j=M}^{K-1}\left(\vartheta\left(j\right)\right)\right)\right)\right)\end{array}\right)\\
=&\Upsilon\left(\begin{array}{l@{}}\sum\limits_{n=1}^{K}\left(f\left(n\right)\sum\limits_{k=1}^{K}\left(g\left(k\right)R\left(n,k,K\right)\right)\right)\\\sum\limits_{n=1}^{K}\left(c_m\left(f\right)\left(n\right)\sum\limits_{k=1}^{K}\left(g\left(k\right)R\left(n,k,K\right)\right)\right)\end{array}\right)\\
=&\Upsilon\left(\begin{array}{l@{}}\sum\limits_{n=m}^{m+1}\left(f\left(n\right)\sum\limits_{k=1}^{K}\left(g\left(k\right)R\left(n,k,K\right)\right)\right)\\\sum\limits_{n=m}^{m+1}\left(c_m\left(f\right)\left(n\right)\sum\limits_{k=1}^{K}\left(g\left(k\right)R\left(n,k,K\right)\right)\right)\end{array}\right)
\end{align*}
Where we used the facts that for all $n\in\omega-\left\{m,m+1\right\}$ we have that $f\left(n\right)=c_m\left(f\right)\left(n\right)$ and $\Upsilon\left(a+c,b+c\right)=\Upsilon\left(a,b\right)$. Now notice that:
\normalsize
\begin{align*}
&\sum\limits_{n=m}^{m+1}\left(c_m\left(f\right)\left(n\right)\sum\limits_{k=1}^{K}\left(g\left(k\right)R\left(n,k,K\right)\right)\right)\\
=&\sum\limits_{n=m}^{m+1}\left(f\left(n\right)\sum\limits_{k=1}^{K}\left(g\left(k\right)R\left(n,k,K\right)\right)\right)+\sum\limits_{k=1}^{K}\left(g\left(k\right)R\left(m,k,K\right)\right)-\vartheta\left(m\right)\sum\limits_{k=1}^{K}\left(g\left(k\right)R\left(m+1,k,K\right)\right)
\end{align*}
Observe that
\begin{align*}
0\leq R\left(m,k,K\right)-\vartheta\left(m\right)R\left(m+1,k,K\right)<\vartheta\left(m\right)
\end{align*}
Therefore, we see that
\begin{align*}
\sum\limits_{n=m}^{m+1}\left(c_m\left(f\right)\left(n\right)\sum\limits_{k=1}^{K}\left(g\left(k\right)R\left(n,k,K\right)\right)\right)\geq\sum\limits_{n=m}^{m+1}\left(f\left(n\right)\sum\limits_{k=1}^{K}\left(g\left(k\right)R\left(n,k,K\right)\right)\right)
\end{align*}
And furthermore
\begin{align*}
&\sum\limits_{n=m}^{m+1}\left(c_m\left(f\right)\left(n\right)\sum\limits_{k=1}^{K}\left(g\left(k\right)R\left(n,k,K\right)\right)\right)-\sum\limits_{n=m}^{m+1}\left(f\left(n\right)\sum\limits_{k=1}^{K}\left(g\left(k\right)R\left(n,k,K\right)\right)\right)<\vartheta\left(m\right)\sum\limits_{k=1}^K\left(g\left(k\right)\right)
\end{align*}
Now, by the contrapositive of the extension theorem, if these two functions are not in the same real number then there exists $N\in\omega_+$ such that for all $M\in\left[N,\infty\right]_\omega$ there exists $K\in\left[M,\infty\right]_\omega$ such that 
\begin{align*}
\Upsilon\left(\begin{array}{l@{}}\sum\limits_{M=1}^{K}\left(\left(f*g\right)\left(M\right)\prod\limits_{j=M}^{K-1}\left(\vartheta\left(j\right)\right)\right)\\\sum\limits_{M=1}^{K}\left(\left(c_m\left(f\right)*g\right)\left(M\right)\prod\limits_{j=M}^{K-1}\left(\vartheta\left(j\right)\right)\right)\end{array}\right)>\prod\limits_{j=N}^{K-1}\left(\vartheta\left(j\right)\right)
\end{align*}
This means that $\vartheta\left(m\right)\sum\limits_{k=1}^{K}\left(g\left(k\right)\right)>\prod\limits_{p=N}^{K-1}\vartheta\left(p\right)$. Denote 
\begin{align*}
L:=\min\left(\left\{t\in\left[N+2,\infty\right]_\omega\middle|\prod\limits_{p=N+1}^{t-1}\left(\vartheta\left(p\right)\right)\geq\vartheta\left(m\right)\right\}\right)
\end{align*}
Therefore, for any $M>L$ we know that there exist $K_3>K_2>K_1>M$ defined as
\begin{align*}
\begin{array}{ll@{}}
K_1:=\min\left(\left\{t\in\left[M+1,\infty\right]_\omega\middle|\vartheta\left(m\right)\sum\limits_{k=1}^{t}\left(g\left(k\right)\right)>\prod\limits_{p=N}^{t-1}\vartheta\left(p\right)\right\}\right)\\
K_2:=\min\left(\left\{t\in\left[K_1+1,\infty\right]_\omega\middle|\begin{array}{ll@{}}&\vartheta\left(m\right)\sum\limits_{k=1}^{t}\left(g\left(k\right)\right)>\prod\limits_{p=N}^{t-1}\vartheta\left(p\right)\\\wedge&\prod\limits_{p=N}^{t-1}\vartheta\left(p\right)\geq\vartheta\left(m\right)\sum\limits_{k=1}^{K_1}\left(g\left(k\right)\right)\end{array}\right\}\right)\\
K_3:=\min\left(\left\{t\in\left[K_2+1,\infty\right]_\omega\middle|\begin{array}{ll@{}}&\vartheta\left(m\right)\sum\limits_{k=1}^{t}\left(g\left(k\right)\right)>\prod\limits_{p=N}^{t-1}\vartheta\left(p\right)\\\wedge&\prod\limits_{p=N}^{t-1}\vartheta\left(p\right)\geq\vartheta\left(m\right)\sum\limits_{k=1}^{K_2}\left(g\left(k\right)\right)\end{array}\right\}\right)
\end{array}
\end{align*}
Therefore, we have that
\begin{align*}
\vartheta\left(m\right)\sum\limits_{k=1}^{K_3}\left(g\left(k\right)\right)\geq\prod\limits_{p=N}^{K_3-1}\vartheta\left(p\right)\geq\vartheta\left(m\right)\sum\limits_{k=1}^{K_2}\left(g\left(k\right)\right)\geq\prod\limits_{p=N}^{K_2-1}\vartheta\left(p\right)\geq\vartheta\left(m\right)\sum\limits_{k=1}^{K_1}\left(g\left(k\right)\right)
\end{align*}
Thus, 
\begin{align*}
\vartheta\left(m\right)\sum\limits_{k=K_1+1}^{K_3}\left(g\left(k\right)\right)\geq\left(\prod\limits_{p=K_2}^{K_3-1}\left(\vartheta\left(p\right)\right)-1\right)\prod\limits_{p=N}^{K_2-1}\left(\vartheta\left(p\right)\right)\geq\prod\limits_{p=N+1}^{K_3-1}\left(\vartheta\left(p\right)\right)
\end{align*}
Hence, $\sum\limits_{k=K_1+1}^{K_3}\left(g\left(k\right)\right)\geq\prod\limits_{p=L}^{K_3-1}\left(\vartheta\left(p\right)\right)$. Notice that we may take $K_1$ to be arbitrarily large and this construction would still hold. This means that $g$ is not finite, as there are arbitrarily large intervals with enough ``material" to contract into the interval of inspection $\left[1,L\right]_\omega$.\\
We may prove this rigorously. For any finite ordinal $H$ define a finite sequence of pairs of finite ordinals on $\left[1,H\right]_\omega$ where 
\begin{align*}
\begin{array}{ll@{}}
K_{\left(1,1\right)}:=K_1&\text{and}
K_{\left(1,2\right)}:=K_3
\end{array}
\end{align*}
and for any $n\in\left[2,H\right]$ we have
\begin{align*}
\begin{array}{ll@{}}
K_{\left(n,1\right)}&:=\min\left(\left\{t\in\left[K_{\left(n-1,2\right)}+1,\infty\right]_\omega\middle|\vartheta\left(m\right)\sum\limits_{k=1}^{t}\left(g\left(k\right)\right)>\prod\limits_{p=N}^{t-1}\vartheta\left(p\right)\right\}\right)\\
K_n^*&:=\min\left(\left\{t\in\left[K_{\left(n,1\right)}+1,\infty\right]_\omega\middle|\begin{array}{ll@{}}&\vartheta\left(m\right)\sum\limits_{k=1}^{t}\left(g\left(k\right)\right)>\prod\limits_{p=N}^{t-1}\vartheta\left(p\right)\\\wedge&\prod\limits_{p=N}^{t-1}\vartheta\left(p\right)\geq\vartheta\left(m\right)\sum\limits_{k=1}^{K_{\left(n,1\right)}}\left(g\left(k\right)\right)\end{array}\right\}\right)\\
K_{\left(n,2\right)}&:=\min\left(\left\{t\in\left[K_n^*+1,\infty\right]_\omega\middle|\begin{array}{ll@{}}&\vartheta\left(m\right)\sum\limits_{k=1}^{t}\left(g\left(k\right)\right)>\prod\limits_{p=N}^{t-1}\vartheta\left(p\right)\\\wedge&\prod\limits_{p=N}^{t-1}\vartheta\left(p\right)\geq\vartheta\left(m\right)\sum\limits_{k=1}^{K_n^*}\left(g\left(k\right)\right)\end{array}\right\}\right)
\end{array}
\end{align*}
Notice that for any $n\in\left[1,H\right]_\omega$ we have
\begin{align*}
T\left(n\right):=\sum\limits_{k=K_{\left(n,1\right)}}^{K_{\left(n,2\right)}}\left(g\left(k\right)\right)\geq\prod\limits_{p=L}^{K_{\left(n,2\right)}-1}\left(\vartheta\left(p\right)\right)
\end{align*} 
Thus, for any $C\in\omega_+$ take $H=C\prod\limits_{p=1}^{L-1}\left(\vartheta\left(p\right)\right)$. Then 
\begin{align*}
&\sum\limits_{k=1}^{K_{\left(H,2\right)}}g\left(k\right)\prod\limits_{p=k}^{K_{\left(H,2\right)}-1}\left(\vartheta\left(p\right)\right)\geq\sum\limits_{n=1}^{H}\left(T\left(n\right)\prod\limits_{p=K_{\left(n,2\right)}}^{K_{\left(H,2\right)}-1}\left(\vartheta\left(p\right)\right)\right)\\
\geq&\sum\limits_{n=1}^{H}\left(\prod\limits_{p=L}^{K_{\left(H,2\right)}-1}\left(\vartheta\left(p\right)\right)\right)=H\prod\limits_{p=L}^{K_{\left(H,2\right)}-1}\left(\vartheta\left(p\right)\right)\geq C\prod\limits_{p=1}^{K_{\left(H,2\right)}-1}\left(\vartheta\left(p\right)\right)
\end{align*}
That is a contradiction to $g$ being a finite function by the \nameref{Tying Theorem} \ref{Tying Theorem}, and thus $\left[f*g\right]=\left[c_m\left(f\right)*g\right]$.
\end{proof}

\begin{corollary}[Sub-Multiplication Contraction Lemma Corollary]\label{Sub-Multiplication Contraction Lemma Corollary} By commutativity of sub-multiplication we also have that if $f$ and $g$ are finite functions where $g$ is contractable about some $m\in\omega_+$ then $\left[f*c_m\left(g\right)\right]=\left[c_m\left(g\right)*f\right]=\left[g*f\right]=\left[f*g\right]$.\\
Suppose that $f$ and $g$ are finite functions where $f$ is broadenable about some $m\in\omega_+$. Then by the \nameref{Cancellation Theorems} \ref{Cancellation Theorems} we have that $\left[b_m\left(f\right)*g\right]=\left[c_m\left(b_m\left(f\right)\right)*g\right]=\left[f*g\right]$.\\
Furthermore, suppose that $g$ is broadenable about some $n\in\omega_+$, then by the commutativity of sub-multiplication we have $\left[f*b_n\left(g\right)\right]=\left[b_n\left(g\right)*f\right]=\left[g*f\right]=\left[f*g\right]$.
\end{corollary}
\begin{lemma}[First Sub-Multiplication Auxiliary Lemma]\label{First Sub-Multiplication Auxiliary Lemma}\noindent\\
For all finite functions $f$ and $g$ we have that $\left[f*g\right]=\left[h_f*h_g\right]$ where $h_f$ and $h_g$ are the auxiliary functions of $f$ and $g$ respectively.
\end{lemma}
\begin{proofidea} Two things are needed, and the \nameref{Null Lemma} \ref{Null Lemma} is the tool for both ends. First, that a product of non-zero functions is itself non-zero: choose positions where each factor has a positive digit, and the level at which their combined scale first fits gives an output position whose rest coefficient is positive, so the product cannot vanish there. Second, and this is the substance, that replacing each factor by its auxiliary function does not move the product. For that we take the auxiliary function of $h_f*h_g$, available because the product of finite functions is finite, and transport the sequences that match $f$ and $g$ against their auxiliary functions through the multiplication, which is legitimate precisely because a carry inside a factor commutes with the product. The \nameref{Submission Theorem} \ref{Submission Theorem} certifies applicability at each step and the \nameref{Paradise City Lemma} \ref{Paradise City Lemma} controls the truncations. The reverse inequality runs the same comparison in the opposite direction; we broaden each auxiliary function back towards its own representative, well-definedness of the broadening coming from the \nameref{Paradise City Lemma} \ref{Paradise City Lemma}, and the two inequalities together identify the classes.
\end{proofidea}
\begin{proof} Notice, that if $f\in\left[0\right]$ or $g\in\left[0\right]$ then we are done as, by the \nameref{Null Lemma} \ref{Null Lemma} and the definition of sub-multiplication we have that both sides are equal in $\left[0\right]$ by the \nameref{Null Lemma} \ref{Null Lemma} once again. Furthermore, by the \nameref{Null Lemma} if $f$ and $g$ were both non-zero then there exist constants $m_1$ and $m_2$ in $\omega_+$ such that $f\left(m_1\right)>0$ and $g\left(m_2\right)>0$. Take
\begin{align*}
M:=\min\left(\left\{t\in\omega_+\middle|\prod\limits_{j=1}^{m_1-1}\left(\vartheta\left(j\right)\right)\prod\limits_{j=1}^{m_2-1}\left(\vartheta\left(j\right)\right)\leq\prod\limits_{j=1}^{t-1}\left(\vartheta\left(j\right)\right)\right\}\right)
\end{align*}
then $r\left(m_1,m_2,M\right)>0$ and thus $\left(f*g\right)\left(M\right)>0$ which means that by the \nameref{Null Lemma} \ref{Null Lemma} we have that $\left(f*g\right)\not\in\left[0\right]$. Therefore, $\left(f*g\right)\in\left[0\right]\iff\left(f\in\left[0\right]\vee g\in\left[0\right]\right)$, therefore, going forward we suppose both $f$ and $g$ non-zero.\\
Take the function $h_f*h_g$. By the \nameref{Sub-Multiplication Finiteness Lemma} \ref{Sub-Multiplication Finiteness Lemma} we know that this function is finite and thus there exists an auxiliary function $\pi$ such that for all $N\in\omega_+$ there is a sequence of contractions ${}^\pi\mathcal{C}_{k=q_1}^{r_1}$ such that for all $n\in\left[0,N\right]_\omega$ we have ${}^\pi\mathcal{C}_{k=q_1}^{r_1}\left(h_f*h_g\right)\left(n\right)=\pi\left(n\right)$. Take $\Xi:=\max\left(\text{con}\left({}^\pi\mathcal{C}_{k=q_1}^{r_1}\right)\cup\left\{N\right\}\right)$. Moreover, as $\left(h_f*h_g\right)\left(0\right)=h_f\left(0\right)+h_g\left(0\right)=f\left(0\right)+g\left(0\right)=\left(f*g\right)\left(0\right)$ we have $\sigma\left(\left[h_f*h_g\right]\right)=\sigma\left(\left[f*g\right]\right)$.\\
Suppose $\sigma\left(\left[h_f*h_g\right]\right)=\sigma\left(\left[f*g\right]\right)=1$. Take the sequences of contractions ${}^f\mathcal{C}_{k=q_2}^{r_2}$ and ${}^g\mathcal{C}_{k=q_3}^{r_3}$ where for all $n\in\left[0,\Xi\right]_\omega$ we have ${}^f\mathcal{C}_{k=q_2}^{r_2}\left(f\right)\left(n\right)=h_f\left(n\right)$ and ${}^g\mathcal{C}_{k=q_3}^{r_3}\left(g\right)\left(n\right)=h_g\left(n\right)$. Denote $w:={}^f\mathcal{C}_{k=q_2}^{r_2}\left(f\right)$ and $s:={}^g\mathcal{C}_{k=q_3}^{r_3}\left(g\right)$. By induction on the \nameref{Sub-Multiplication Contraction Lemma} \ref{Sub-Multiplication Contraction Lemma} we have that $\left(w*s\right)\in\left[f*g\right]$\footnote{We also used the \nameref{Sub-Multiplication Finiteness Lemma} \ref{Sub-Multiplication Finiteness Lemma} to ensure that each in-between step is well-defined.}.\\
We see that for all $n\in\left[0,\Xi\right]_\omega$ we have that $\left(w*s\right)\left(n\right)=\left(h_f*h_g\right)\left(n\right)$. Thus by the \nameref{Submission Theorem} \ref{Submission Theorem} ${}^\pi\mathcal{C}_{k=q_1}^{r_1}$ is applicable to $w*s$ and furthermore for all $n\in\left[0,N\right]_\omega$ we have 
\begin{align*}
{}^\pi\mathcal{C}_{k=q_1}^{r_1}\left(w*s\right)\left(n\right)+\left(h_f*h_g\right)\left(n\right)&=\left(w*s\right)\left(n\right)+{}^\pi\mathcal{C}_{k=q_1}^{r_1}\left(h_f*h_g\right)\left(n\right)\\
&=\left(h_f*h_g\right)\left(n\right)+\pi\left(n\right)
\end{align*}
Thus, for all $n\in\left[0,N\right]_\omega$ we have ${}^\pi\mathcal{C}_{k=q_1}^{r_1}\left(w*s\right)\left(n\right)=\pi\left(n\right)$.\\
Therefore, for all $N\in\omega_+$ there exists $\delta\in\left[f*g\right]$ such that for all $n\in\left[0,N\right]_\omega$ we have $\delta\left(n\right)=\pi\left(n\right)$.\\
This means that $\left[h_f*h_g\right]\leq\left[f*g\right]$, as if not then $\left[h_f*h_g\right]+{}^o\left[\left(f*g\right)\right]>\left[0\right]$. Hence by the \nameref{Null Lemma} \ref{Null Lemma}, there exists $m\in\omega_+$ such that $\varrho_{{}^o\left(f*g\right)}^{h_f*h_g}\left(m\right)\geq1$, denote the minimal such value as $M$. Thus, take $\delta\in\left[f*g\right]$ such that for all $n\in\left[0,M+1\right]_\omega$ we have $\delta\left(n\right)=\pi\left(n\right)$. Hence, $\left[\varrho_{{}^o\left(f*g\right)}^{h_f*h_g}+\delta\right]=\left[h_f*h_g\right]$. However, 
\begin{align*}
\left(\varrho_{{}^o\left(f*g\right)}^{h_f*h_g}+\delta\right)\left(\Psi\left(\varrho_{{}^o\left(f*g\right)}^{h_f*h_g}+\delta,\pi\right)\right)=\left(\varrho_{{}^o\left(f*g\right)}^{h_f*h_g}+\delta\right)\left(M\right)>\pi\left(M\right)=\pi\left(\Psi\left(\varrho_{{}^o\left(f*g\right)}^{h_f*h_g}+\delta,\pi\right)\right)
\end{align*}
 and 
\begin{align*}
\left(\varrho_{{}^o\left(f*g\right)}^{h_f*h_g}+\delta\right)\left(M+1\right)\geq\pi\left(M+1\right)
\end{align*}
Denote $a$ the auxiliary function of $\left(\varrho_{{}^o\left(f*g\right)}^{h_f*h_g}+\delta\right)$. Therefore, by the \nameref{Paradise City Lemma} \ref{Paradise City Lemma} we have
\begin{align*}
&\sum\limits_{n=1}^{M+1}\left(a\left(n\right)\prod\limits_{j=n}^M\left(\vartheta\left(j\right)\right)\right)\geq\sum\limits_{n=1}^{M+1}\left(\left(\varrho_{{}^o\left(f*g\right)}^{h_f*h_g}+\delta\right)\left(n\right)\prod\limits_{j=n}^M\left(\vartheta\left(j\right)\right)\right)\\
\geq&\sum\limits_{n=1}^{M+1}\left(\pi\left(n\right)\prod\limits_{j=n}^M\left(\vartheta\left(j\right)\right)\right)+\vartheta\left(M\right)\geq\sum\limits_{n=1}^{M+1}\left(\pi\left(n\right)\prod\limits_{j=n}^M\left(\vartheta\left(j\right)\right)\right)+2
\end{align*}
which contradicts $\pi$ being an auxiliary function of $\left[h_f*h_g\right]$ by the \nameref{PC Cor I} \ref{PC Cor I} and \nameref{PC Cor II} \ref{PC Cor II}. Therefore $\left[h_f*h_g\right]\leq\left[f*g\right]$.\\
We know that $f*g$ is finite and thus there exists an auxiliary function $\gamma$ such that $\forall N\in\omega_+$ there is a sequence of contractions ${}^\gamma\mathcal{C}_{k=q}^{r}$ such that for all $n\in\left[0,N\right]_\omega$ we have ${}^\gamma\mathcal{C}_{k=q}^{r}\left(f*g\right)\left(n\right)=\gamma\left(n\right)$. Take $\Xi:=\max\left(\text{con}\left({}^\gamma\mathcal{C}_{k=q_1}^{r_1}\right)\cup\left\{N\right\}\right)$.\\
For any $K\in\omega$ define inductively the functions $f_K$ as
\begin{align*}
\begin{array}{ll@{}}
f_0=h_f&\text{and}\\
f_{K+1}={}^f\mathcal{B}_{k=\nu\left(K\right)+1}^{\nu\left(K+1\right)}\left(f_K\right)
\end{array}
\end{align*}
where we define ${}^f\mathcal{B}\in\mathcal{CB}$
\begin{align*}
{}^f\mathcal{B}\left(k\right)=
\begin{cases}
b_{K+1}&k\in\left[\nu_{K}+1,\nu_{K+1}\right]_\omega\\
b_1&otherwise
\end{cases}
\end{align*}
where for all $K\in\omega$ we define
\begin{align*}
\nu\left(K\right)=\sum\limits_{n=1}^{K}\left(\kappa\left(n\right)\right)
\end{align*}
where we define for all $K\in\omega_+$ we define
\begin{align*}
\kappa\left(K+1\right)=f_{K}\left(K+1\right)-f\left(K+1\right)
\end{align*}
Firstly we prove that for all $K\in\omega$ $f_K$ is well-defined. Suppose that it is not for some $K$. Then there exist a minimal such $K$, which we shall denote as $L$ (notice that $L>1$), and this means that $f_{L-1}\left(L\right)<f\left(L\right)$. Moreover, we know that for all $K\in\left[0,L-1\right]_\omega$ we have $f_{K-1}\left(K\right)\geq f\left(K\right)$.\\
By our presumption $f_{L-1}$ is well-defined and furthermore, notice that for all $n\in\left[0,L-1\right]_\omega$ we have $f_{L-1}\left(n\right)=f\left(n\right)$. This is shown inductively as for all $K\in\left[1,L-1\right]_\omega$ we have that for all $n\in\left[0,K-1\right]$ we have $f_{K-1}\left(n\right)=f_{K}\left(n\right)$. This is because, ${}^f\mathcal{B}_{k=\nu\left(K\right)+1}^{\nu\left(K+1\right)}$ is a sequence of broadenings only about $K$ (or is the trivial sequence when we do not broaden anything in the case of $f_{K}\left(K+1\right)=f\left(K+1\right)$) and thus leaves all the values less than $K$ unchanged.\\
Remark that for all $K\in\left[0,L-1\right]_\omega$ we have $f_K\left(K\right)=f\left(K\right)$. The case $K=0$ follows from the definition of auxiliary functions and the other cases follow as, by definition 
\begin{align*}
f_{K}\left(K\right)&={}^f\mathcal{B}_{k=\nu\left(K-1\right)+1}^{\nu\left(K\right)}\left(f_{K-1}\right)\left(K\right)=f_{K-1}\left(K\right)-\left(\nu\left(K\right)-\nu\left(K-1\right)\right)\\
&=f_{K-1}\left(K\right)-\kappa\left(K\right)=f_{K-1}\left(K\right)-\left(f_{K-1}\left(K\right)-f\left(K\right)\right)=f\left(K\right)
\end{align*}
Therefore, all put together we have that for all $n\in\left[0,L-1\right]$ we have $f_{L-1}\left(n\right)=f\left(n\right)$. Thus,
\begin{align*}
\sum\limits_{n=1}^{L}\left(f_{L-1}\left(n\right)\prod\limits_{j=n}^{L-1}\left(\vartheta\left(j\right)\right)\right)<\sum\limits_{n=1}^{L}\left(f\left(n\right)\prod\limits_{j=n}^{L-1}\left(\vartheta\left(j\right)\right)\right)
\end{align*}
However, by induction we see that for all $K\in\left[0,L-1\right]_\omega$ we have 
\begin{align*}
\sum\limits_{n=1}^{L}\left(f_{K}\left(n\right)\prod\limits_{j=n}^{L-1}\left(\vartheta\left(j\right)\right)\right)=\sum\limits_{n=1}^{L}\left(h_f\left(n\right)\prod\limits_{j=n}^{L-1}\left(\vartheta\left(j\right)\right)\right)
\end{align*}
as this holds for $K=0$ and if it holds for some $K\in\left[0,L-2\right]_\omega$ then 
\begin{align*}
&\sum\limits_{n=1}^{L}\left(f_{K+1}\left(n\right)\prod\limits_{j=n}^{L-1}\left(\vartheta\left(j\right)\right)\right)\\
=&\sum\limits_{n=1}^{L}\left(f_{K}\left(n\right)\prod\limits_{j=n}^{L-1}\left(\vartheta\left(j\right)\right)\right)-\kappa\left(K+1\right)\prod\limits_{j=K+1}^{L-1}\left(\vartheta\left(j\right)\right)+\vartheta\left(K+1\right)\kappa\left(K+1\right)\prod\limits_{j=K+2}^{L-1}\left(\vartheta\left(j\right)\right)\\
=&\sum\limits_{n=1}^{L}\left(f_{K}\left(n\right)\prod\limits_{j=n}^{L-1}\left(\vartheta\left(j\right)\right)\right)=\sum\limits_{n=1}^{L}\left(h_f\left(n\right)\prod\limits_{j=n}^{L-1}\left(\vartheta\left(j\right)\right)\right)
\end{align*}
Therefore, we have that
\begin{align*}
\sum\limits_{n=1}^{L}\left(h_f\left(n\right)\prod\limits_{j=n}^{L-1}\left(\vartheta\left(j\right)\right)\right)<\sum\limits_{n=1}^{L}\left(f\left(n\right)\prod\limits_{j=n}^{L-1}\left(\vartheta\left(j\right)\right)\right)
\end{align*}
However, this is a contradiction to the \nameref{Paradise City Lemma} \ref{Paradise City Lemma}. Hence, $f_K$ is well-defined for any $K\in\omega$.\\
Define in the same matter the function $g_K$. Notice that by construction for any $K\in\omega$
\begin{align*}
\begin{array}{ll@{}}
f_K={}^f\mathcal{B}_{k=1}^{\nu\left(K+1\right)}\left(h_f\right)&\text{and}\\
g_K={}^g\mathcal{B}_{k=1}^{\nu\left(K+1\right)}\left(h_g\right)
\end{array}
\end{align*}
Thus, by induction on the \nameref{Sub-Multiplication Contraction Lemma Corollary} \ref{Sub-Multiplication Contraction Lemma Corollary} we have that $\left(f_K*g_K\right)\in\left[h_f*h_g\right]$ for any $K\in\omega$. Furthermore, as for all $n\in\left[0,\Xi\right]_\omega$ we have $\left(f_\Xi*g_\Xi\right)\left(n\right)=\left(f*g\right)\left(n\right)$. Thus by the \nameref{Submission Theorem} \ref{Submission Theorem} ${}^\gamma\mathcal{C}_{k=q}^r$ is applicable to $\left(f_\Xi*g_\Xi\right)$ and furthermore for all $n\in\left[0,N\right]_\omega$ we have 
\begin{align*}
{}^\gamma\mathcal{C}_{k=q}^r\left(f_\Xi*g_\Xi\right)\left(n\right)+\left(f*g\right)\left(n\right)&=\left(f_\Xi*g_\Xi\right)\left(n\right)+{}^\gamma\mathcal{C}_{k=q}^r\left(f*g\right)\left(n\right)\\
&=\left(f*g\right)\left(n\right)+\gamma\left(n\right)
\end{align*}
Hence, for all $n\in\left[0,N\right]_\omega$ we have ${}^\gamma\mathcal{C}_{k=q}^r\left(f_\Xi*g_\Xi\right)\left(n\right)=\gamma\left(n\right)$. Therefore, for all $N\in\omega_+$ there exists $\kappa\in\left[h_f*h_g\right]$ such that for all $n\in\left[0,N\right]_\omega$ we have $\kappa\left(n\right)=\gamma\left(n\right)$. This means that $\left[f*g\right]\leq\left[h_f*h_g\right]$ by the same argument as above.\\
Put together, we obtain $\left[f*g\right]=\left[h_f*h_g\right]$.\\
The proof for $\sigma\left(\left[h_f*h_g\right]\right)=\sigma\left(\left[f*g\right]\right)=0$ is the same as above, only with inverted inequalities.
\end{proof}

\begin{lemma}[Second Sub-Multiplication Auxiliary Lemma]\label{Second Sub-Multiplication Auxiliary Lemma}\noindent\\
For all finite functions $f$ and all auxiliary functions $h_g$ we have $\left[f_p'*h_g\right]=\left[f_s'*h_g\right]$ where $f_p'$ and $f_s'$ are the primary and secondary auxiliary functions of $\left[f\right]$ respectively such that $f_p'\left(0\right)=f_s'\left(0\right)=f\left(0\right)$ (if there are any).
\end{lemma}
\begin{proofidea} The two products are not compared directly. We truncate the secondary auxiliary function at the position where it parts from the primary, obtaining a function $\xi$ which is finite by the \nameref{Tying Theorem} \ref{Tying Theorem}, and check, through the \nameref{One-Case Lemma} \ref{One-Case Lemma} and the \nameref{Domination Theorem} \ref{Domination Theorem}, that $\xi*h_g$ is strictly smaller in magnitude than either product. Subtracting it from both is then legitimate, and the \nameref{Rearrangement Lemma} \ref{Rearrangement Lemma} lets us represent each difference by whichever tracking function is convenient. The digits of those two differences are explicit sums of rest coefficients, small enough to compare truncation by truncation: the \nameref{Boundedness Theorem} \ref{Boundedness Theorem} yields an inequality between them, and the \nameref{Extended Extension Theorem} \ref{Extended Extension Theorem} upgrades that to the equality of classes the lemma asserts. The degenerate cases where $h_g$ is zero, or where $\left[f\right]$ has no secondary auxiliary function, are cleared away at the start.
\end{proofidea}
\begin{proof} If $h_g\in\left[0\right]$ we are done as, by the \nameref{Null Lemma} \ref{Null Lemma} and the definition of sub-multiplication we have that both sides are equal in $\left[0\right]$ by the \nameref{Null Lemma} \ref{Null Lemma} once again. If $f_p'\in\left[0\right]$ then there is no $f_s'$. Hence, going forward we suppose that $h_g\not\in\left[0\right]$ and $\left[f\right]$ also has a secondary auxiliary function (and therefore is non-zero). Notice that in that case $\sigma\left(\left[f_p'*h_g\right]\right)=\sigma\left(\left[f_s'*h_g\right]\right)$.\\
Suppose $\sigma\left(\left[f_p'*h_g\right]\right)=\sigma\left(\left[f_s'*h_g\right]\right)=1$. Take the function $\xi:\omega\rightarrow\omega$ such that 
\begin{align*}
\xi\left(k\right)=\begin{cases}
f_s'\left(k\right)&k\in\left[0,\Psi\left(f_p',f_s'\right)\right]_\omega\\
0&otherwise
\end{cases}
\end{align*}
This function is finite by the \nameref{Tying Theorem} \ref{Tying Theorem}, as for all $n\in\omega_+$ we have $\xi\left(n\right)\leq f_s'\left(n\right)$. By the \nameref{Consistency Theorem'''} \ref{4.I} \ref{Consistency Theorem'''} we have 
\begin{align*}
\left[f_p'*h_g\right]=\left[f_s'*h_g\right]&\iff\left[f_p'*h_g\right]+\left[{}^o\left(\xi*h_g\right)\right]=\left[f_s'*h_g\right]+\left[{}^o\left(\xi*h_g\right)\right]\\
&\iff\left[f_p'*h_g+{}^o\left(\xi*h_g\right)\right]=\left[f_s'*h_g+{}^o\left(\xi*h_g\right)\right]
\end{align*}
We have $\sigma\left({}^o\left(\xi*h_g\right)\right)\neq\sigma\left(f_p'*h_g\right)=\sigma\left(f_s'*h_g\right)$. Furthermore, as for all $n\in\omega_+$ we have 
\begin{align*}
\begin{array}{ll@{}}
f_p'\left(n\right)\geq\xi\left(n\right)&\text{and}\\
f_s'\left(n\right)\geq\xi\left(n\right)
\end{array}
\end{align*} we have for all $n\in\omega_+$
\begin{align*}
\begin{array}{ll@{}}
\left|\left(f_p'*h_g\right)\right|\left(n\right)=\left(f_p'*h_g\right)\left(n\right)\geq\left(\xi*h_g\right)\left(n\right)={}^o\left(\xi*h_g\right)\left(n\right)=\left|{}^o\left(\xi*h_g\right)\right|\left(n\right)&\text{and}\\
\left|\left(f_s'*h_g\right)\right|\left(n\right)=\left(f_s'*h_g\right)\left(n\right)\geq\left(\xi*h_g\right)\left(n\right)={}^o\left(\xi*h_g\right)\left(n\right)=\left|{}^o\left(\xi*h_g\right)\right|\left(n\right)
\end{array}
\end{align*}
Thus, by the \nameref{One-Case Lemma} \ref{One-Case Lemma} as $\sigma\left(\left|{}^o\left(\xi*h_g\right)\right|\right)=\sigma\left(\left|\left(f_p'*h_g\right)\right|\right)=\sigma\left(\left|\left(f_s'*h_g\right)\right|\right)=1$\footnote{Notice that even if $\xi\in\left[0\right]$ and thus $\left|{}^o\left(\xi*h_g\right)\right|\in\left[0\right]$ we still can use the \nameref{One-Case Lemma} \ref{One-Case Lemma}} we have that $\left|\left[f_p'*h_g\right]\right|\geq\left|\left[{}^o\left(\xi*h_g\right)\right]\right|$ and $\left|\left[f_s'*h_g\right]\right|\geq\left|\left[{}^o\left(\xi*h_g\right)\right]\right|$. These inequalities are furthermore strict, as for the minimal $m_1\in\omega_+$ such that $h_g\left(m_1\right)\geq1$ we define
\begin{align*}
m_2&:=\min\left(\left\{M\in\omega_+\middle|\prod\limits_{j=1}^{\Psi\left(f_p',f_s'\right)-1}\left(\vartheta\left(j\right)\right)\prod\limits_{j=1}^{m_1-1}\left(\vartheta\left(j\right)\right)\leq\prod\limits_{j=1}^{M-1}\left(\vartheta\left(j\right)\right)\right\}\right)\\
m_3&:=\min\left(\left\{M\in\omega_+\middle|\prod\limits_{j=1}^{\Psi\left(f_p',f_s'\right)}\left(\vartheta\left(j\right)\right)\prod\limits_{j=1}^{m_1-1}\left(\vartheta\left(j\right)\right)\leq\prod\limits_{j=1}^{M-1}\left(\vartheta\left(j\right)\right)\right\}\right)
\end{align*}
Then, as
\begin{align*}
\begin{array}{ll@{}}
\left|\left(f_p'*h_g\right)\right|\left(m_2\right)=\left(f_p'*h_g\right)\left(m_2\right)>\left(\xi*h_g\right)\left(m_2\right)={}^o\left(\xi*h_g\right)\left(m_2\right)=\left|{}^o\left(\xi*h_g\right)\right|\left(m_2\right)&\text{and}\\
\left|\left(f_s'*h_g\right)\right|\left(m_3\right)=\left(f_s'*h_g\right)\left(m_3\right)>\left(\xi*h_g\right)\left(m_3\right)={}^o\left(\xi*h_g\right)\left(m_3\right)=\left|{}^o\left(\xi*h_g\right)\right|\left(m_3\right)
\end{array}
\end{align*}
we have by the \nameref{Domination Theorem} \ref{Domination Theorem} that $\left|\left[f_p'*h_g\right]\right|>\left|\left[{}^o\left(\xi*h_g\right)\right]\right|$ and $\left|\left[f_s'*h_g\right]\right|>\left|\left[{}^o\left(\xi*h_g\right)\right]\right|$.\\
Hence, by the \nameref{Rearrangement Lemma} \ref{Rearrangement Lemma} we have $\left[f_p'*h_g+{}^o\left(\xi*h_g\right)\right]=\left[w\right]$ where $w\in\mathcal{T}_{{}^o\left(\xi*h_g\right)}^{f_p'*h_g}$ such that 
\begin{align*}
w\left(M\right)=\begin{cases}
\left(f_p'*h_g\right)\left(0\right)&M=0\\
\left(f_p'*h_g\right)\left(M\right)-{}^o\left(\xi*h_g\right)\left(M\right)&M\in\omega_+
\end{cases}
\end{align*}
Notice that for all $M\in\omega_+$ we have
\begin{align*}
w\left(M\right)&=\left(f_p'*h_g\right)\left(M\right)-{}^o\left(\xi*h_g\right)\left(M\right)\\
&=\sum\limits_{n=1}^{M}\left(\sum\limits_{k=1}^{M}\left(f_p'\left(n\right)h_g\left(k\right)\left(r\left(n,k,M\right)\right)\right)\right)-\sum\limits_{n=1}^{M}\left(\sum\limits_{k=1}^{M}\left(\xi\left(n\right)h_g\left(k\right)\left(r\left(n,k,M\right)\right)\right)\right)\\
&=\sum\limits_{n=1}^{M}\left(\left(f_p'\left(n\right)-\xi\left(n\right)\right)\sum\limits_{k=1}^{M}\left(h_g\left(k\right)\left(r\left(n,k,M\right)\right)\right)\right)=\sum\limits_{k=1}^{M}\left(h_g\left(k\right)\left(r\left(\Psi\left(f_p',f_s'\right),k,M\right)\right)\right)
\end{align*}
Notice, that when $M<\Psi\left(f_p',f_s'\right)$ this should be explicitly $0$ as all the terms of $f_p'$ and $\xi$ are equal. We however write it this way as when $M<\Psi\left(f_p',f_s'\right)$ we have $r\left(\Psi\left(f_p',f_s'\right),k,M\right)=0$.\\
Similarly, by the \nameref{Rearrangement Lemma} \ref{Rearrangement Lemma} we have $\left[f_s'*h_g+{}^o\left(\xi*h_g\right)\right]=\left[s\right]$ where $s\in\mathcal{T}_{{}^o\left(\xi*h_g\right)}^{f_s'*h_g}$ such that 
\begin{align*}
s\left(M\right)=
\begin{cases}
\left(f_s'*h_g\right)\left(0\right)&M=0\\
\left(f_s'*h_g\right)\left(M\right)-{}^o\left(\xi*h_g\right)\left(M\right)&M\in\omega_+
\end{cases}
\end{align*}
Notice that for all $M\in\omega_+$ we have
\begin{align*}
s\left(M\right)&=\left(f_s'*h_g\right)\left(M\right)-{}^o\left(\xi*h_g\right)\left(M\right)\\
&=\sum\limits_{n=1}^{M}\left(\sum\limits_{k=1}^{M}\left(f_s'\left(n\right)h_g\left(k\right)\left(r\left(n,k,M\right)\right)\right)\right)-\sum\limits_{n=1}^{M}\left(\sum\limits_{k=1}^{M}\left(\xi\left(n\right)h_g\left(k\right)\left(r\left(n,k,M\right)\right)\right)\right)\\
&=\sum\limits_{n=1}^{M}\left(\left(f_s'\left(n\right)-\xi\left(n\right)\right)\sum\limits_{k=1}^{M}\left(h_g\left(k\right)\left(r\left(n,k,M\right)\right)\right)\right)\\
&=\sum\limits_{k=1}^{M}\left(h_g\left(k\right)\sum\limits_{n=\Psi\left(f_p',f_s'\right)+1}^{M}\left(\left(\vartheta\left(n-1\right)-1\right)\left(r\left(n,k,M\right)\right)\right)\right)
\end{align*}
Now, for any $K\in\omega_+$ we inspect the value:
\begin{align*}
&\sum\limits_{M=1}^{K}\left(s\left(M\right)\prod\limits_{j=M}^{K-1}\left(\vartheta\left(j\right)\right)\right)\\
=&\sum\limits_{M=1}^{K}\left(\sum\limits_{k=1}^{M}\left(h_g\left(k\right)\sum\limits_{n=\Psi\left(f_p',f_s'\right)+1}^{M}\left(\left(\vartheta\left(n-1\right)-1\right)r\left(n,k,M\right)\prod\limits_{j=M}^{K-1}\left(\vartheta\left(j\right)\right)\right)\right)\right)\\
=&\sum\limits_{k=1}^{K}\left(h_g\left(k\right)\sum\limits_{n=\Psi\left(f_p',f_s'\right)+1}^{K}\left(\left(\vartheta\left(n-1\right)-1\right)\sum\limits_{M=\max\left(\left\{n,k\right\}\right)}^{K}\left(r\left(n,k,M\right)\prod\limits_{j=M}^{K-1}\left(\vartheta\left(j\right)\right)\right)\right)\right)\\
=&\sum\limits_{k=1}^{K}\left(h_g\left(k\right)\sum\limits_{n=\Psi\left(f_p',f_s'\right)+1}^{K}\left(\left(\vartheta\left(n-1\right)-1\right)R\left(n,k,K\right)\right)\right)
\end{align*}
We see from the definition of $R\left(n,k,K\right)$ that this is less or equal to
\begin{align*}
&\sum\limits_{k=1}^{K}\left(h_g\left(k\right)\max\left(\left\{t\in\omega\middle|t\prod\limits_{j=1}^{k-1}\left(\vartheta\left(j\right)\right)\leq\sum\limits_{n=\Psi\left(f_p',f_s'\right)+1}^{K}\left(\left(\vartheta\left(n-1\right)-1\right)\prod\limits_{j=n}^{K-1}\left(\vartheta\left(j\right)\right)\right)\right\}\right)\right)\\
\leq&\sum\limits_{k=1}^{K}\left(h_g\left(k\right)\max\left(\left\{t\in\omega\middle|t\prod\limits_{j=1}^{k-1}\left(\vartheta\left(j\right)\right)\leq\prod\limits_{j=\Psi\left(f_p',f_s'\right)}^{K-1}\left(\vartheta\left(j\right)\right)\right\}\right)\right)
\end{align*}
We also inspect the value (for the same $K$):
\begin{align*}
&\sum\limits_{M=1}^{K}\left(w\left(M\right)\prod\limits_{j=M}^{K-1}\left(\vartheta\left(j\right)\right)\right)=\sum\limits_{M=1}^{K}\left(\sum\limits_{k=1}^{M}\left(h_g\left(k\right)r\left(\Psi\left(f_p',f_s'\right),k,M\right)\prod\limits_{j=M}^{K-1}\left(\vartheta\left(j\right)\right)\right)\right)\\
=&\sum\limits_{k=1}^{K}\left(h_g\left(k\right)\left(\sum\limits_{M=k}^{K}\left(r\left(\Psi\left(f_p',f_s'\right),k,M\right)\prod\limits_{j=M}^{K-1}\left(\vartheta\left(j\right)\right)\right)\right)\right)\\
=&\sum\limits_{k=1}^{K}\left(h_g\left(k\right)\left(\sum\limits_{M=\max\left(\left\{\Psi\left(f_p',f_s'\right),k\right\}\right)}^{K}\left(r\left(\Psi\left(f_p',f_s'\right),k,M\right)\prod\limits_{j=M}^{K-1}\left(\vartheta\left(j\right)\right)\right)\right)\right)\\
=&\sum\limits_{k=1}^{K}\left(h_g\left(k\right)\max\left(\left\{t\in\omega\middle|t\prod\limits_{j=1}^{k-1}\left(\vartheta\left(j\right)\right)\leq\prod\limits_{j=\Psi\left(f_p',f_s'\right)}^{K-1}\left(\vartheta\left(j\right)\right)\right\}\right)\right)
\end{align*}
Therefore, we have shown that for all $K\in\omega_+$ we have 
\begin{align*}
\sum\limits_{M=1}^{K}\left(s\left(M\right)\prod\limits_{j=M}^{K-1}\left(\vartheta\left(j\right)\right)\right)\leq\sum\limits_{M=1}^{K}\left(w\left(M\right)\prod\limits_{j=M}^{K-1}\left(\vartheta\left(j\right)\right)\right)
\end{align*}
and thus by the \nameref{Boundedness Theorem} \ref{Boundedness Theorem} we have $\left[w\right]\geq\left[s\right]$. Now we wish to apply the \nameref{Extended Extension Theorem} \ref{Extended Extension Theorem}. Let $N\in\omega_+$ and $\Xi\in\left[N+1,\infty\right]_\omega$ be given. Firstly, we see that for all $m\in\left[\Xi+1,\infty\right]_\omega$ we have
\begin{align*}
&\prod\limits_{j=\Xi}^{m-1}\left(\vartheta\left(j\right)\right)\sum\limits_{n=1}^{\Xi}\left(w\left(n\right)\prod\limits_{j=n}^{\Xi-1}\left(\vartheta\left(j\right)\right)\right)\\
=&\sum\limits_{k=1}^{\Xi}\left(h_g\left(k\right)\prod\limits_{j=\Xi}^{m-1}\left(\vartheta\left(j\right)\right)\max\left(\left\{t\in\omega\middle|t\prod\limits_{j=1}^{k-1}\left(\vartheta\left(j\right)\right)\leq\prod\limits_{j=\Psi\left(f_p',f_s'\right)}^{\Xi-1}\left(\vartheta\left(j\right)\right)\right\}\right)\right)
\end{align*}
We shall first show that there exists $m\in\left[\Xi+1,\infty\right]_\omega$ such that 
\begin{align*}
\left(\sum\limits_{k=1}^{\Xi}h_g\left(k\right)\right)\left(1+\hspace{-0.5cm}\sum\limits_{n=\Psi\left(f_p',f_s'\right)+1}^{m}\hspace{-0.5cm}\left(\vartheta\left(n-1\right)-1\right)\right)<\prod\limits_{j=N}^{m-1}\left(\vartheta\left(j\right)\right)
\end{align*}
We show this by taking $L\in\omega_+$ to be the minimal value such that 
\begin{align*}
\prod\limits_{j=N}^{L-1}\left(\vartheta\left(j\right)\right)\geq\sum\limits_{k=1}^{\Xi}h_g\left(k\right)
\end{align*}
Therefore, we only need to show that there exists $m\in\left[L+1,\infty\right]_\omega$ such that
\begin{align*}
u\left(m\right):=1+\hspace{-0.5cm}\sum\limits_{n=\Psi\left(f_p',f_s'\right)}^{m-1}\hspace{-0.5cm}\left(\vartheta\left(n\right)-1\right)<\prod\limits_{j=L}^{m-1}\left(\vartheta\left(j\right)\right)=:v\left(m\right)
\end{align*}
We see that for all $m\in\left[\max\left(\left\{L+1,\Psi\left(f_p',f_s'\right)+1\right\}\right),\infty\right]_\omega$ we have
\begin{align*}
\begin{array}{ll@{}}
U\left(m\right):=u\left(m+1\right)-u\left(m\right)=\vartheta\left(m\right)-1&\text{and}\\
V\left(m\right):=v\left(m+1\right)-v\left(m\right)=\left(\vartheta\left(m\right)-1\right)\prod\limits_{j=L}^{m-1}\left(\vartheta\left(j\right)\right)
\end{array}
\end{align*}
We see that by taking $m\geq L+1$ we have $\prod\limits_{j=L}^{m-1}\left(\vartheta\left(j\right)\right)\geq2$ and therefore 
\begin{align*}
V\left(m\right)\geq2\left(\vartheta\left(m\right)-1\right)>\vartheta\left(m\right)-1=U\left(m\right)
\end{align*}
as $\vartheta\left(m\right)\geq2$. Hence, $v$ is a strictly faster growing sequence of natural numbers and thus will eventually be larger than $u$. Take $m$ (from now on) to be the smallest such value. Then we have
\begin{align*}
&\prod\limits_{j=\Xi}^{m-1}\left(\vartheta\left(j\right)\right)\sum\limits_{n=1}^{\Xi}\left(w\left(n\right)\prod\limits_{j=n}^{\Xi-1}\left(\vartheta\left(j\right)\right)\right)\leq\sum\limits_{k=1}^{\Xi}\left(h_g\left(k\right)\max\left(\left\{t\in\omega\middle|t\prod\limits_{j=1}^{k-1}\left(\vartheta\left(j\right)\right)\leq\hspace{-0.5cm}\prod\limits_{j=\Psi\left(f_p',f_s'\right)}^{m-1}\hspace{-0.5cm}\left(\vartheta\left(j\right)\right)\right\}\right)\right)\\
<&\sum\limits_{k=1}^{\Xi}\left(h_g\left(k\right)\max\left(\left\{t\in\omega\middle|t\prod\limits_{j=1}^{k-1}\left(\vartheta\left(j\right)\right)\leq\hspace{-0.5cm}\prod\limits_{j=\Psi\left(f_p',f_s'\right)}^{m-1}\hspace{-0.5cm}\left(\vartheta\left(j\right)\right)\right\}\right)\right)+\prod\limits_{j=N}^{m-1}\left(\vartheta\left(j\right)\right)\\
-&\sum\limits_{k=1}^{\Xi}\left(h_g\left(k\right)\left(1+\hspace{-0.5cm}\sum\limits_{n=\Psi\left(f_p',f_s'\right)+1}^{m}\hspace{-0.5cm}\left(\vartheta\left(n-1\right)-1\right)\right)\right)
\end{align*}
where we used the properties of $m$. Furthermore, this means that
\begin{align*}
&&&\hspace{-1.5cm}\prod\limits_{j=\Xi}^{m-1}\left(\vartheta\left(j\right)\right)\sum\limits_{n=1}^{\Xi}\left(w\left(n\right)\prod\limits_{j=n}^{\Xi-1}\left(\vartheta\left(j\right)\right)\right)\\
&\hspace{-1.5cm}<&&\hspace{-1.5cm}\sum\limits_{k=1}^{\Xi}\left(h_g\left(k\right)\left(\max\left(\left\{t\in\omega\middle|t\prod\limits_{j=1}^{k-1}\left(\vartheta\left(j\right)\right)\leq\hspace{-0.5cm}\prod\limits_{j=\Psi\left(f_p',f_s'\right)}^{m-1}\hspace{-0.5cm}\left(\vartheta\left(j\right)\right)\right\}\right)-1-\hspace{-0.5cm}\sum\limits_{n=\Psi\left(f_p',f_s'\right)+1}^{m}\hspace{-0.5cm}\left(\vartheta\left(n-1\right)-1\right)\right)\right)+\prod\limits_{j=N}^{m-1}\left(\vartheta\left(j\right)\right)\\
&\hspace{-1.5cm}\leq&&\hspace{-1.5cm}\sum\limits_{k=1}^{\Xi}\left(h_g\left(k\right)\left(\max\left(\left\{t\in\omega\middle|t\prod\limits_{j=1}^{k-1}\left(\vartheta\left(j\right)\right)\leq\hspace{-0.5cm}\prod\limits_{j=\Psi\left(f_p',f_s'\right)}^{m-1}\hspace{-0.5cm}\left(\vartheta\left(j\right)\right)-1\right\}\right)-\hspace{-0.5cm}\sum\limits_{n=\Psi\left(f_p',f_s'\right)+1}^{m}\hspace{-0.5cm}\left(\vartheta\left(n-1\right)-1\right)\right)\right)+\prod\limits_{j=N}^{m-1}\left(\vartheta\left(j\right)\right)\\
\end{align*}
Finally, by using the identity $\prod\limits_{j=\Psi\left(f_p',f_s'\right)}^{m-1}\left(\vartheta\left(j\right)\right)-1=\hspace{-0.5cm}\sum\limits_{n=\Psi\left(f_p',f_s'\right)+1}^{m}\left(\left(\vartheta\left(n-1\right)-1\right)\prod\limits_{j=n}^{m-1}\left(\vartheta\left(j\right)\right)\right)$.
\begin{align*}
&&&\hspace{-1.5cm}\prod\limits_{j=\Xi}^{m-1}\left(\vartheta\left(j\right)\right)\sum\limits_{n=1}^{\Xi}\left(w\left(n\right)\prod\limits_{j=n}^{\Xi-1}\left(\vartheta\left(j\right)\right)\right)\\
&\hspace{-1.5cm}<&&\hspace{-1.5cm}\sum\limits_{k=1}^{\Xi}\left(h_g\left(k\right)\left(\max\left(\left\{t\in\omega\middle|t\prod\limits_{j=1}^{k-1}\left(\vartheta\left(j\right)\right)\leq\hspace{-0.5cm}\sum\limits_{n=\Psi\left(f_p',f_s'\right)+1}^{m}\hspace{-0.5cm}\left(\vartheta\left(n-1\right)-1\right)\prod\limits_{j=n}^{m-1}\left(\vartheta\left(j\right)\right)\right\}\right)-\hspace{-0.5cm}\sum\limits_{n=\Psi\left(f_p',f_s'\right)+1}^{m}\hspace{-0.5cm}\left(\vartheta\left(n-1\right)-1\right)\right)\right)\\
&&&\hspace{-1.5cm}+\prod\limits_{j=N}^{m-1}\left(\vartheta\left(j\right)\right)\\
&\hspace{-1.5cm}\leq&&\hspace{-1.5cm}\sum\limits_{k=1}^{\Xi}\left(h_g\left(k\right)\left(\sum\limits_{n=\Psi\left(f_p',f_s'\right)+1}^{m}\hspace{-0.5cm}\left(\vartheta\left(n-1\right)-1\right)\max\left(\left\{t\in\omega\middle|t\prod\limits_{j=1}^{k-1}\left(\vartheta\left(j\right)\right)\leq\prod\limits_{j=n}^{m-1}\left(\vartheta\left(j\right)\right)\right\}\right)\right)\right)+\prod\limits_{j=N}^{m-1}\left(\vartheta\left(j\right)\right)\\
&\hspace{-1.5cm}=&&\hspace{-1.5cm}\sum\limits_{k=1}^{\Xi}\left(h_g\left(k\right)\hspace{-0.5cm}\sum\limits_{n=\Psi\left(f_p',f_s'\right)+1}^{m}\hspace{-0.5cm}\left(\vartheta\left(n-1\right)-1\right)R\left(n,k,m\right)\right)+\prod\limits_{j=N}^{m-1}\left(\vartheta\left(j\right)\right)\\
&\hspace{-1.5cm}\leq&&\hspace{-1.5cm}\sum\limits_{k=1}^{m}\left(h_g\left(k\right)\hspace{-0.5cm}\sum\limits_{n=\Psi\left(f_p',f_s'\right)+1}^{m}\hspace{-0.5cm}\left(\vartheta\left(n-1\right)-1\right)R\left(n,k,m\right)\right)+\prod\limits_{j=N}^{m-1}\left(\vartheta\left(j\right)\right)=\sum\limits_{M=1}^{m}\left(s\left(M\right)\prod\limits_{j=M}^{m-1}\left(\vartheta\left(j\right)\right)\right)+\prod\limits_{j=N}^{m-1}\left(\vartheta\left(j\right)\right)
\end{align*}
Therefore, by the \nameref{Extended Extension Theorem} \ref{Extended Extension Theorem} we have that $\left[w\right]=\left[s\right]$ and thus $\left[f_p'*h_g\right]=\left[f_s'*h_g\right]$.\\
The proof for the case $\sigma\left(\left[f_p'*h_g\right]\right)=\sigma\left(\left[f_s'*h_g\right]\right)=0$ is the same as above.
\end{proof}

\begin{lemma}[Third Sub-Multiplication Auxiliary Lemma]\label{Third Sub-Multiplication Auxiliary Lemma}\noindent\\
For all auxiliary functions $f,g$ of the same non-zero-real number and for all non-zero auxiliary functions $h$ we have $\left[f*h\right]=\left[g*h\right]$.
\end{lemma}
\begin{proof} Suppose that $\forall n\in\omega_+,f\left(n\right)=g\left(n\right)$. Then $\sigma\left(f\right)=\sigma\left(g\right)$ and thus $f\left(0\right)+g\left(0\right)=2m$ for some $m\in\omega$. Furthermore, for all $n\in\omega_+$ we have $\left(f*h\right)\left(n\right)=\left(g*h\right)\left(n\right)$ and thus both of the functions have the same auxiliary function up to a difference for $0$. We shall see what this difference actually is. 
\begin{align*}
\left(f*h\right)\left(0\right)+\left(g*h\right)\left(0\right)=f\left(0\right)+h\left(0\right)+g\left(0\right)+h\left(0\right)=2m+2h\left(0\right)
\end{align*} 
Therefore, by the \nameref{Equivalence Lemma} \ref{Equivalence Lemma} both the auxiliary functions are in the same real number and therefore $\left[f*h\right]=\left[g*h\right]$.\\
Suppose that $\exists n\in\omega_+,f\left(n\right)\neq g\left(n\right)$. This means that one of the functions is a primary auxiliary function and the other one is a secondary auxiliary function. Take without loss of generality $f$ to be the primary auxiliary function and $g$ to be the secondary auxiliary function. By the \nameref{Second Sub-Multiplication Auxiliary Lemma} \ref{Second Sub-Multiplication Auxiliary Lemma} $\left[g*h\right]=\left[g_p'*h\right]$ where $g_p'$ is the primary auxiliary function of $\left[g\right]=\left[f\right]$ with $g_p'\left(0\right)=g\left(0\right)$. By the case above $\left[g_p'*h\right]=\left[f*h\right]$ and thus $\left[f*h\right]=\left[g*h\right]$.
\end{proof}

\begin{theorem}[Consistency Theorem]\label{Consistency Theorem''''}\noindent\\
\begin{subtheorems}
\subtheorem\label{5.I} For all non-zero finite functions $f$ and all finite functions $g$ and $w$ we have that
\begin{align*}
\left[f*g\right]=\left[f*w\right]\iff\left[g\right]=\left[w\right]
\end{align*}
\subtheorem\label{5.II} For all finite functions $f$, $g$, $w$ and $t$ we have that 
\begin{align*}
\left[f\right]=\left[g\right]\wedge\left[w\right]=\left[t\right]\implies\left[f*w\right]=\left[g*t\right]
\end{align*}
\subtheorem\label{5.III} For all finite functions $f$ and $g$ and non-zero finite functions $w$ and $t$ we have that
\begin{align*}
\left[f*w\right]=\left[g*t\right]\wedge\left[w\right]=\left[t\right]\implies\left[f\right]=\left[g\right]
\end{align*}
\end{subtheorems}
\end{theorem}
\begin{proof}[Proof of I]\noindent\\
$\left(\impliedby\right):$\\
By the \nameref{First Sub-Multiplication Auxiliary Lemma} \ref{First Sub-Multiplication Auxiliary Lemma}, we have that $\left[f*g\right]=\left[h_f*h_g\right]$ and $\left[f*w\right]=\left[h_f*h_w\right]$. By commutativity of sub-multiplication and the \nameref{Third Sub-Multiplication Auxiliary Lemma} \ref{Third Sub-Multiplication Auxiliary Lemma} we have 
\begin{align*}
\left[f*g\right]=\left[h_f*h_g\right]=\left[h_g*h_f\right]=\left[h_w*h_f\right]=\left[h_f*h_w\right]=\left[f*w\right]
\end{align*}
$\left(\implies\right):$\\
Suppose $\left[w\right]\neq\left[g\right]$. By totality $\left[w\right]\leq\left[g\right]$ or $\left[w\right]\geq\left[g\right]$. Suppose without loss of generality that $\left[w\right]\leq\left[g\right]$.\\
If $\sigma\left(\left[w\right]\right)=\sigma\left(\left[g\right]\right)=0$ then by the definition of order 
\begin{align*}
\forall g\in\left[g\right],\exists a\in\left[w\right],\forall n\in\omega_+,a\left(n\right)\geq g\left(n\right)
\end{align*}
This implies that for all $n\in\omega_+$ we have
\begin{align*}
\left(f*a\right)\left(n\right)\geq\left(f*g\right)\left(n\right)
\end{align*}
On top of that, as $\left[g\right]\neq\left[w\right]$ there has to exist $M\in\omega_+$ such that $a\left(M\right)>g\left(M\right)$ as if not, then as $\sigma\left(a\right)=\sigma\left(g\right)$ we would have that $\left[w\right]=\left[a\right]=\left[g\right]$ by the \nameref{Equivalence Lemma} \ref{Equivalence Lemma} as $a\left(0\right)+g\left(0\right)=2m$ for some $m\in\omega$ and for all $n\in\omega_+$ we have $a\left(n\right)=g\left(n\right)$. Hence, there exists $M\in\omega_+$ such that $a\left(M\right)>g\left(M\right)$.\\
As $f$ is non-zero, by the \nameref{Null Lemma} \ref{Null Lemma} there exists a value $N\in\omega_+$ such that $f\left(N\right)\geq1$. Let $M'\in\omega_+$ denote the minimal $t\in\omega_+$ such that
\begin{align*}
    \prod\limits_{j=1}^{N-1}\left(\vartheta\left(j\right)\right)\prod\limits_{j=1}^{M-1}\left(\vartheta\left(j\right)\right)\leq\prod\limits_{j=1}^{t-1}\left(\vartheta\left(j\right)\right)
\end{align*}
Thus, by the minimality of $M'$ we have $r\left(N,M,M'\right)\geq1$ and hence
\begin{align*}
\left(f*a\right)\left(M'\right)>\left(f*g\right)\left(M'\right)
\end{align*}
By the \nameref{Domination Theorem} \ref{Domination Theorem} we see that $\left[f*w\right]=\left[f*a\right]\neq\left[f*g\right]$.\\
The proof of $\sigma\left(\left[w\right]\right)=\sigma\left(\left[g\right]\right)=1$ is the same as the first case.\\
Furthermore, if $\sigma\left(\left[w\right]\right)\neq\sigma\left(\left[g\right]\right)$, then we have that $\sigma\left(\left[f*g\right]\right)\neq\sigma\left(\left[f*w\right]\right)$ (where we have shown in the proof of the \nameref{First Sub-Multiplication Auxiliary Lemma} \ref{First Sub-Multiplication Auxiliary Lemma} that these real numbers are not $\left[0\right]$).\\
Finally, if $w$ or $g$ is in $\left[0\right]$ then so is their respective product with $f$. However, as $\left[w\right]\neq\left[g\right]$ and $\left[f\right]\neq\left[0\right]$ we see that the other product is not in $\left[0\right]$ as shown in the proof of the \nameref{First Sub-Multiplication Auxiliary Lemma} \ref{First Sub-Multiplication Auxiliary Lemma}.\\
Ergo $\left[w\right]\neq\left[g\right]\implies\left[f*g\right]\neq\left[f*w\right]$. By contrapositive we have $\left[f*g\right]=\left[f*w\right]\implies\left[g\right]=\left[w\right]$
\end{proof}
\begin{proof}[Proof of II]\noindent\\
Let $f,g,w,t\in\mathcal{N}_{\mathcal{F}in}$ such that $\left[f\right]=\left[g\right]$ and $\left[w\right]=\left[t\right]$ be given. By using the \nameref{Consistency Theorem''''} \ref{5.I} \ref{Consistency Theorem''''} twice we obtain 
\begin{align*}
\begin{array}{ll@{}}
\left[w\right]=\left[t\right]\implies\left[f*w\right]=\left[f*t\right]&\text{and}\\
\left[f\right]=\left[g\right]\implies\left[t*f\right]=\left[t*g\right]
\end{array}
\end{align*}
(Notice that we used only the direction which did not require $f\not\in\left[0\right]$.) As sub-multiplication is commutative $\left[t*f\right]=\left[f*t\right]$. Thus using transitivity of equality we obtain $\left[f*w\right]=\left[f*t\right]=\left[t*f\right]=\left[t*g\right]=\left[g*t\right]\implies\left[f*w\right]=\left[g*t\right]$.
\end{proof}
\begin{proof}[Proof of III]\noindent\\
Let $f,g\in\mathcal{N}_{\mathcal{F}in}$ and $w,t\in\mathcal{N}_{\mathcal{F}in}^*$ such that $\left[f*w\right]=\left[g*t\right]$ and $\left[w\right]=\left[t\right]$ be given. By the \nameref{Consistency Theorem''''} \ref{5.I} \ref{Consistency Theorem''''} 
\begin{align*}
\left[w\right]=\left[t\right]\implies\left[f*w\right]=\left[f*t\right]
\end{align*} by the hypothesis $\left[g*t\right]=\left[f*w\right]$.\\
Therefore, we have $\left[g*t\right]=\left[f*t\right]$. By the commutativity of sub-multiplication and the \nameref{Consistency Theorem''''} \ref{5.I} \ref{Consistency Theorem''''} we have $\left[t*g\right]=\left[t*f\right]\implies\left[f\right]=\left[g\right]$.
\end{proof}

\begin{definition}[Multiplication]\noindent\\
We define multiplication of two real numbers $\left[f\right]$ and $\left[g\right]$ as:
\begin{align*}
\left[f\right]*\left[g\right]:=\left[f*g\right]
\end{align*}
for any functions $f\in\left[f\right]$ and $g\in\left[g\right]$.
\end{definition}

\begin{remark}The \nameref{Consistency Theorem''''} \ref{5.II} \ref{Consistency Theorem''''} ensures that the product of any two real numbers does not depend on the functions chosen to represent those real numbers. Therefore, multiplication is well-defined.
\end{remark}

\begin{theorem}[Consistency Theorem]\label{Consistency Theorem'''''}\noindent\\
\begin{subtheorems}
\subtheorem\label{6.I} For all non-zero real numbers $\left[f\right]$ and all real numbers $\left[g\right]$ and $\left[w\right]$ we have that
\begin{align*}
\left[f\right]*\left[g\right]=\left[f\right]*\left[w\right]\iff\left[g\right]=\left[w\right]
\end{align*}
\subtheorem\label{6.II} For all real numbers $\left[f\right]$, $\left[g\right]$, $\left[w\right]$ and $\left[t\right]$ we have that 
\begin{align*}
\left[f\right]=\left[g\right]\wedge\left[w\right]=\left[t\right]\implies\left[f\right]*\left[w\right]=\left[g\right]*\left[t\right]
\end{align*}
\subtheorem\label{6.III} For all real numbers $\left[f\right]$ and $\left[g\right]$ and non-zero real numbers $\left[w\right]$ and $\left[t\right]$ we have that 
\begin{align*}
\left[f\right]*\left[w\right]=\left[g\right]*\left[t\right]\wedge\left[w\right]=\left[t\right]\implies\left[f\right]=\left[g\right]
\end{align*}
\end{subtheorems}
\end{theorem}
\begin{proof} By the definition of multiplication and the \nameref{Consistency Theorem''''} \ref{Consistency Theorem''''} these theorems follow.
\end{proof}
\subsection{Multiplicative Axioms}
\begin{axiom}[Multiplicative Identity]\noindent\\
There exists a real number $\left[x\right]$ such that for all real numbers $\left[a\right]$ we have $\left[x\right]*\left[a\right]=\left[a\right]$.
\end{axiom}
\begin{proof} We claim that $\left[1\right]:=\left[1_p\right]$ where $1_p:\omega\rightarrow\omega$ where
\begin{align*}
1_p\left(k\right)=\begin{cases}1&k=1\\
0&otherwise
\end{cases}
\end{align*}
We see that for any finite function $g$ we have
\begin{align*}
\left(1_p*g\right)\left(M\right)=\begin{cases}1_p\left(0\right)+g\left(0\right)=g\left(0\right)&M=0\\
\sum\limits_{n=1}^M\left(\sum\limits_{k=1}^M\left(1_p\left(n\right)g\left(k\right)r\left(n,k,M\right)\right)\right)=\sum\limits_{k=1}^M\left(g\left(k\right)r\left(1,k,M\right)\right)=g\left(M\right)&M\in\omega_+
\end{cases}
\end{align*}
Hence, $1_p*g=g$. Thus $\left[1\right]$ is the multiplicative identity.
\end{proof}

\begin{axiom}[Commutativity of Multiplication]\noindent\\
For all real numbers $\left[f\right]$ and $\left[g\right]$ we have $\left[f\right]*\left[g\right]=\left[g\right]*\left[f\right]$.
\end{axiom}
\begin{proof} As sub-multiplication of any two finite functions $f$ and $g$ is commutative so is multiplication of any two real numbers.
\end{proof}

\begin{axiom}[Associativity of Multiplication]\noindent\\
For all real numbers $\left[f\right]$, $\left[g\right]$ and $\left[w\right]$ we have that $\left(\left[f\right]*\left[g\right]\right)*\left[w\right]=\left[f\right]*\left(\left[g\right]*\left[w\right]\right)$.
\end{axiom}
\begin{proofidea} Expand each of the two associations into a triple sum over positions. Both collapse, by the \nameref{Rest Identity} \ref{Rest Identity} and the composition rule for rational floors, to the same symmetric expression, the number of whole units of the scale at the output position that fit inside the combined scale of the three input positions. Symmetry of that expression in its three arguments is the whole proof; the surrounding work is checking that the floor slack stays below one unit of scale, so that the \nameref{Extended Extension Theorem} \ref{Extended Extension Theorem} applies.
\end{proofidea}
\begin{proof} Let $f,g,w\in \mathcal{N}_{\mathcal{F}in}$ be given. Notice that if any of the functions is in $\left[0\right]$ then the resulting real numbers will both be equal to zero, by the \nameref{Null Lemma} \ref{Null Lemma} applied twice, thus associativity holds.\\\
Going forward we suppose that $f,g,w\in \mathcal{N}_{\mathcal{F}in}^*$. Because of this $\left(\left(f*g\right)*w\right)\not\in\left[0\right]$ and $\left(f*\left(g*w\right)\right)\not\in\left[0\right]$ by repeating the argument used in the proof of \nameref{First Sub-Multiplication Auxiliary Lemma} \ref{First Sub-Multiplication Auxiliary Lemma}.\\
We see that $\left(\left(f*g\right)*w\right)\left(0\right)=\left(f*\left(g*w\right)\right)\left(0\right)=f\left(0\right)+g\left(0\right)+w\left(0\right)$. Therefore, $\sigma\left(\left(f*g\right)*w\right)=\sigma\left(f*\left(g*w\right)\right)$. Also, for all $K\in\omega_+$ we have
\begin{align*}
\begin{array}{ll@{}}
\left(\left(f*g\right)*w\right)\left(K\right)&=\sum\limits_{M=1}^{K}\left(\sum\limits_{p=1}^K\left(\sum\limits_{n=1}^{M}\left(\sum\limits_{k=1}^{M}\left(f\left(n\right)g\left(k\right)r\left(n,k,M\right)\right)\right)w\left(p\right)\left(r\left(M,p,K\right)\right)\right)\right)\\
&=\sum\limits_{M=1}^{K}\left(\sum\limits_{p=1}^K\left(\sum\limits_{n=1}^{M}\left(\sum\limits_{k=1}^{M}\left(f\left(n\right)g\left(k\right)w\left(p\right)r\left(n,k,M\right)r\left(M,p,K\right)\right)\right)\right)\right)\\
&=\sum\limits_{n=1}^{K}\left(f\left(n\right)\sum\limits_{k=1}^{K}\left(g\left(k\right)\sum\limits_{p=1}^K\left(w\left(p\right)\hspace{-0.5cm}\sum\limits_{M=\max\left(\left\{n,k\right\}\right)}^K\hspace{-0.5cm}\left(r\left(n,k,M\right)r\left(M,p,K\right)\right)\right)\right)\right)\\
\left(f*\left(g*w\right)\right)\left(K\right)&=\sum\limits_{n=1}^{K}\left(\sum\limits_{M=1}^{K}\left(f\left(n\right)\sum\limits_{k=1}^{M}\left(\sum\limits_{p=1}^{M}\left(g\left(k\right)w\left(p\right)r\left(k,p,M\right)\right)\right)\right)r\left(n,M,K\right)\right)\\
&=\sum\limits_{n=1}^{K}\left(\sum\limits_{M=1}^{K}\left(\sum\limits_{k=1}^{M}\left(\sum\limits_{p=1}^{M}\left(f\left(n\right)g\left(k\right)w\left(p\right)r\left(k,p,M\right)r\left(n,M,K\right)\right)\right)\right)\right)\\
&=\sum\limits_{n=1}^{K}\left(f\left(n\right)\sum\limits_{k=1}^{K}\left(g\left(k\right)\sum\limits_{p=1}^K\left(w\left(p\right)\hspace{-0.5cm}\sum\limits_{M=\max\left(\left\{k,p\right\}\right)}^K\hspace{-0.5cm}\left(r\left(k,p,M\right)r\left(n,M,K\right)\right)\right)\right)\right)
\end{array}
\end{align*}
Now, in order to apply the \nameref{Extended Extension Theorem Corollary} \ref{Extended Extension Theorem Corollary} we inspect for all $\Xi\in\omega_+$ the value
\begin{align*}
&\sum\limits_{K=1}^{\Xi}\left(\left(\left(f*g\right)*w\right)\left(K\right)\prod\limits_{j=K}^{\Xi-1}\left(\vartheta\left(j\right)\right)\right)\\
=&\sum\limits_{n=1}^{\Xi}\left(f\left(n\right)\sum\limits_{k=1}^{\Xi}\left(g\left(k\right)\sum\limits_{p=1}^{\Xi}\left(w\left(p\right)\hspace{-0.5cm}\sum\limits_{K=\max\left(\left\{n,k,p\right\}\right)}^{\Xi}\left(\sum\limits_{M=\max\left(\left\{n,k\right\}\right)}^K\left(r\left(n,k,M\right)r\left(M,p,K\right)\prod\limits_{j=K}^{\Xi-1}\left(\vartheta\left(j\right)\right)\right)\right)\right)\right)\right)\\
=&\sum\limits_{n=1}^{\Xi}\left(f\left(n\right)\sum\limits_{k=1}^{\Xi}\left(g\left(k\right)\sum\limits_{p=1}^{\Xi}\left(w\left(p\right)\hspace{-0.5cm}\sum\limits_{M=\max\left(\left\{n,k\right\}\right)}^\Xi\left(\sum\limits_{K=\max\left(\left\{n,k,p,M\right\}\right)}^{\Xi}\left(r\left(n,k,M\right)r\left(M,p,K\right)\prod\limits_{j=K}^{\Xi-1}\left(\vartheta\left(j\right)\right)\right)\right)\right)\right)\right)\\
=&\sum\limits_{n=1}^{\Xi}\left(f\left(n\right)\sum\limits_{k=1}^{\Xi}\left(g\left(k\right)\sum\limits_{p=1}^{\Xi}\left(w\left(p\right)\hspace{-0.5cm}\sum\limits_{M=\max\left(\left\{n,k\right\}\right)}^\Xi\left(r\left(n,k,M\right)\hspace{-0.5cm}\sum\limits_{K=\max\left(\left\{M,p\right\}\right)}^{\Xi}\left(r\left(M,p,K\right)\prod\limits_{j=K}^{\Xi-1}\left(\vartheta\left(j\right)\right)\right)\right)\right)\right)\right)\\
=&\sum\limits_{n=1}^{\Xi}\left(f\left(n\right)\sum\limits_{k=1}^{\Xi}\left(g\left(k\right)\sum\limits_{p=1}^{\Xi}\left(w\left(p\right)\hspace{-0.5cm}\sum\limits_{M=\max\left(\left\{n,k\right\}\right)}^\Xi\left(r\left(n,k,M\right)R\left(M,p,\Xi\right)\right)\right)\right)\right)\\
\end{align*}
And similarly we have:
\begin{align*}
&\sum\limits_{K=1}^{\Xi}\left(\left(f*\left(g*w\right)\right)\left(K\right)\prod\limits_{j=K}^{\Xi-1}\left(\vartheta\left(j\right)\right)\right)\\
=&\sum\limits_{n=1}^{\Xi}\left(f\left(n\right)\sum\limits_{k=1}^{\Xi}\left(g\left(k\right)\sum\limits_{p=1}^{\Xi}\left(w\left(p\right)\sum\limits_{M=\max\left(\left\{k,p\right\}\right)}^\Xi\left(r\left(k,p,M\right)R\left(n,M,\Xi\right)\right)\right)\right)\right)
\end{align*}
Let $N\in\omega_+$ and $\Xi\in\left[N+1,\infty\right]_\omega$ be given. For any $m\in\left[\Xi+1,\infty\right]_\omega$ we see that
\begin{align*}
&\prod\limits_{j=\Xi}^{m-1}\left(\vartheta\left(j\right)\right)\sum\limits_{K=1}^{\Xi}\left(\left(f*g\right)*w\right)\left(K\right)\prod\limits_{j=K}^{\Xi-1}\left(\vartheta\left(j\right)\right)\\
=&\prod\limits_{j=\Xi}^{m-1}\left(\vartheta\left(j\right)\right)\sum\limits_{n=1}^{\Xi}\left(f\left(n\right)\sum\limits_{k=1}^{\Xi}\left(g\left(k\right)\sum\limits_{p=1}^{\Xi}\left(w\left(p\right)\right.\right.\right.\\
&\left.\left.\left.\sum\limits_{M=\max\left(\left\{n,k\right\}\right)}^\Xi\left(r\left(n,k,M\right)\max\left(\left\{t\in\omega\middle|t\prod\limits_{j=1}^{p-1}\left(\vartheta\left(j\right)\right)\leq\prod\limits_{j=M}^{\Xi-1}\left(\vartheta\left(j\right)\right)\right\}\right)\right)\right)\right)\right)\\
\leq&\prod\limits_{j=\Xi}^{m-1}\left(\vartheta\left(j\right)\right)\sum\limits_{n=1}^{\Xi}\left(f\left(n\right)\sum\limits_{k=1}^{\Xi}\left(g\left(k\right)\sum\limits_{p=1}^{\Xi}\left(w\left(p\right)\right.\right.\right.\\
&\left.\left.\left.\max\left(\left\{t\in\omega\middle|t\prod\limits_{j=1}^{p-1}\left(\vartheta\left(j\right)\right)\leq\hspace{-0.5cm}\sum\limits_{M=\max\left(\left\{n,k\right\}\right)}^\Xi\left(r\left(n,k,M\right)\prod\limits_{j=M}^{\Xi-1}\left(\vartheta\left(j\right)\right)\right)\right\}\right)\right)\right)\right)\\
=&\prod\limits_{j=\Xi}^{m-1}\left(\vartheta\left(j\right)\right)\sum\limits_{n=1}^{\Xi}\left(f\left(n\right)\sum\limits_{k=1}^{\Xi}\left(g\left(k\right)\sum\limits_{p=1}^{\Xi}\left(w\left(p\right)\right.\right.\right.\\
&\left.\left.\left.\max\left(\left\{t\in\omega\middle|t\prod\limits_{j=1}^{p-1}\left(\vartheta\left(j\right)\right)\leq\max\left(\left\{s\in\omega\middle|s\prod\limits_{j=1}^{n-1}\left(\vartheta\left(j\right)\right)\prod\limits_{j=1}^{k-1}\left(\vartheta\left(j\right)\right)\leq\prod\limits_{j=1}^{\Xi-1}\left(\vartheta\left(j\right)\right)\right\}\right)\right\}\right)\right)\right)\right)\\
=&\prod\limits_{j=\Xi}^{m-1}\left(\vartheta\left(j\right)\right)\sum\limits_{n=1}^{\Xi}\left(f\left(n\right)\sum\limits_{k=1}^{\Xi}\left(g\left(k\right)\sum\limits_{p=1}^{\Xi}\left(w\left(p\right)\right.\right.\right.\\
&\left.\left.\left.\max\left(\left\{t\in\omega\middle|t\prod\limits_{j=1}^{p-1}\left(\vartheta\left(j\right)\right)\prod\limits_{j=1}^{n-1}\left(\vartheta\left(j\right)\right)\prod\limits_{j=1}^{k-1}\left(\vartheta\left(j\right)\right)\leq\prod\limits_{j=1}^{\Xi-1}\left(\vartheta\left(j\right)\right)\right\}\right)\right)\right)\right)\\
\leq&\sum\limits_{n=1}^{\Xi}\left(f\left(n\right)\sum\limits_{k=1}^{\Xi}\left(g\left(k\right)\sum\limits_{p=1}^{\Xi}\left(w\left(p\right)\right.\right.\right.\\
&\left.\left.\left.\max\left(\left\{t\in\omega\middle|t\prod\limits_{j=1}^{p-1}\left(\vartheta\left(j\right)\right)\prod\limits_{j=1}^{n-1}\left(\vartheta\left(j\right)\right)\prod\limits_{j=1}^{k-1}\left(\vartheta\left(j\right)\right)\leq\prod\limits_{j=1}^{m-1}\left(\vartheta\left(j\right)\right)\right\}\right)\right)\right)\right)\\
\end{align*}
We know by the same reasoning as in the proof \nameref{Second Sub-Multiplication Auxiliary Lemma} \ref{Second Sub-Multiplication Auxiliary Lemma} that there exists $m\in\left[\Xi+1,\infty\right]_\omega$ such that 
\begin{align*}
&\sum\limits_{n=1}^{\Xi}f\left(n\right)\sum\limits_{k=1}^{\Xi}g\left(k\right)\sum\limits_{p=1}^{\Xi}w\left(p\right)\sum\limits_{M=\max\left(\left\{k,p\right\}\right)}^mr\left(k,p,M\right)\\
\leq&\left(\sum\limits_{n=1}^{\Xi}f\left(n\right)\right)\left(\sum\limits_{k=1}^{\Xi}g\left(k\right)\right)\left(\sum\limits_{p=1}^{\Xi}w\left(p\right)\right)\sum\limits_{M=\max\left(\left\{k,p\right\}\right)}^m\left(\vartheta\left(M-1\right)-1\right)<\prod\limits_{j=N}^{m-1}\left(\vartheta\left(j\right)\right)
\end{align*}
Furthermore, we see that
\begin{align*}
&\sum\limits_{K=1}^{m}\left(f*\left(g*w\right)\right)\left(K\right)\prod\limits_{j=K}^{m-1}\left(\vartheta\left(j\right)\right)+\prod\limits_{j=N}^{m-1}\left(\vartheta\left(j\right)\right)\\
=&\sum\limits_{n=1}^{m}\left(f\left(n\right)\sum\limits_{k=1}^{m}\left(g\left(k\right)\sum\limits_{p=1}^{m}\left(w\left(p\right)\sum\limits_{M=\max\left(\left\{k,p\right\}\right)}^m\left(r\left(k,p,M\right)\right.\right.\right.\right.\\
&\left.\left.\left.\left.\max\left(\left\{t\in\omega\middle|t\prod\limits_{j=1}^{n-1}\left(\vartheta\left(j\right)\right)\prod\limits_{j=1}^{M-1}\left(\vartheta\left(j\right)\right)\leq\prod\limits_{j=1}^{m-1}\left(\vartheta\left(j\right)\right)\right\}\right)\right)\right)\right)\right)+\prod\limits_{j=N}^{m-1}\left(\vartheta\left(j\right)\right)\\
\geq&\sum\limits_{n=1}^{\Xi}\left(f\left(n\right)\sum\limits_{k=1}^{\Xi}\left(g\left(k\right)\sum\limits_{p=1}^{\Xi}\left(w\left(p\right)\hspace{-0.5cm}\sum\limits_{M=\max\left(\left\{k,p\right\}\right)}^m\hspace{-0.5cm}\left(r\left(k,p,M\right)\right.\right.\right.\right.\\
&\left.\left.\left.\left.\max\left(\left\{t\in\omega\middle|t\prod\limits_{j=1}^{n-1}\left(\vartheta\left(j\right)\right)\prod\limits_{j=1}^{M-1}\left(\vartheta\left(j\right)\right)\leq\prod\limits_{j=1}^{m-1}\left(\vartheta\left(j\right)\right)\right\}\right)\right)\right)\right)\right)+\prod\limits_{j=N}^{m-1}\left(\vartheta\left(j\right)\right)\\
\geq&\sum\limits_{n=1}^{\Xi}\left(f\left(n\right)\sum\limits_{k=1}^{\Xi}\left(g\left(k\right)\sum\limits_{p=1}^{\Xi}\left(w\left(p\right)\right.\right.\right.\\
&\left.\left.\left.\left(\max\left(\left\{t\in\omega\middle|t\prod\limits_{j=1}^{n-1}\left(\vartheta\left(j\right)\right)\leq\hspace{-0.5cm}\sum\limits_{M=\max\left(\left\{k,p\right\}\right)}^m\hspace{-0.5cm}r\left(k,p,M\right)\prod\limits_{j=M}^{m-1}\left(\vartheta\left(j\right)\right)\right\}\right)\right)\right)\right)\right)\\
&+\prod\limits_{j=N}^{m-1}\left(\vartheta\left(j\right)\right)-\sum\limits_{n=1}^{\Xi}\left(f\left(n\right)\sum\limits_{k=1}^{\Xi}\left(g\left(k\right)\sum\limits_{p=1}^{\Xi}\left(w\left(p\right)\hspace{-0.5cm}\sum\limits_{M=\max\left(\left\{k,p\right\}\right)}^m\hspace{-0.5cm}r\left(k,p,M\right)\right)\right)\right)\\
=&\sum\limits_{n=1}^{\Xi}\left(f\left(n\right)\sum\limits_{k=1}^{\Xi}\left(g\left(k\right)\sum\limits_{p=1}^{\Xi}\left(w\left(p\right)\right.\right.\right.\\
&\left.\left.\left.\left(\max\left(\left\{t\in\omega\middle|t\prod\limits_{j=1}^{n-1}\left(\vartheta\left(j\right)\right)\prod\limits_{j=1}^{k-1}\left(\vartheta\left(j\right)\right)\prod\limits_{j=1}^{p-1}\left(\vartheta\left(j\right)\right)\leq\prod\limits_{j=1}^{m-1}\left(\vartheta\left(j\right)\right)\right\}\right)\right)\right)\right)\right)\\
&+\prod\limits_{j=N}^{m-1}\left(\vartheta\left(j\right)\right)-\sum\limits_{n=1}^{\Xi}\left(f\left(n\right)\sum\limits_{k=1}^{\Xi}\left(g\left(k\right)\sum\limits_{p=1}^{\Xi}\left(w\left(p\right)\hspace{-0.5cm}\sum\limits_{M=\max\left(\left\{k,p\right\}\right)}^m\hspace{-0.5cm}r\left(k,p,M\right)\right)\right)\right)\\
\geq&\sum\limits_{n=1}^{\Xi}\left(f\left(n\right)\sum\limits_{k=1}^{\Xi}\left(g\left(k\right)\sum\limits_{p=1}^{\Xi}\left(w\left(p\right)\right.\right.\right.\\
&\left.\left.\left.\max\left(\left\{t\in\omega\middle|t\prod\limits_{j=1}^{p-1}\left(\vartheta\left(j\right)\right)\prod\limits_{j=1}^{n-1}\left(\vartheta\left(j\right)\right)\prod\limits_{j=1}^{k-1}\left(\vartheta\left(j\right)\right)\leq\prod\limits_{j=1}^{m-1}\left(\vartheta\left(j\right)\right)\right\}\right)\right)\right)\right)\\
\geq&\prod\limits_{j=\Xi}^{m-1}\left(\vartheta\left(j\right)\right)\sum\limits_{K=1}^{\Xi}\left(\left(f*g\right)*w\right)\left(K\right)\prod\limits_{j=K}^{\Xi-1}\left(\vartheta\left(j\right)\right)
\end{align*}
Hence we may conclude that $\forall N\in\omega_+,\forall \Xi\in\left[N+1,\infty\right]_\omega,\exists m\in\left[\Xi+1,\infty\right]_\omega$ such that
\begin{align*}
\prod\limits_{j=\Xi}^{m-1}\left(\vartheta\left(j\right)\right)\sum\limits_{K=1}^{\Xi}\left(\left(f*g\right)*w\right)\left(K\right)\prod\limits_{j=K}^{\Xi-1}\left(\vartheta\left(j\right)\right)\leq\sum\limits_{K=1}^{m}\left(f*\left(g*w\right)\right)\left(K\right)\prod\limits_{j=K}^{m-1}\left(\vartheta\left(j\right)\right)+\prod\limits_{j=N}^{m-1}\left(\vartheta\left(j\right)\right)
\end{align*}
By the same proof, only relabelled, we have $\forall N\in\omega_+,\forall \Xi\in\left[N+1,\infty\right]_\omega,\exists m\in\left[\Xi+1,\infty\right]_\omega$ such that
\begin{align*}
\prod\limits_{j=\Xi}^{m-1}\left(\vartheta\left(j\right)\right)\sum\limits_{K=1}^{\Xi}\left(f*\left(g*w\right)\right)\left(K\right)\prod\limits_{j=K}^{\Xi-1}\left(\vartheta\left(j\right)\right)\leq\sum\limits_{K=1}^{m}\left(\left(f*g\right)*w\right)\left(K\right)\prod\limits_{j=K}^{m-1}\left(\vartheta\left(j\right)\right)+\prod\limits_{j=N}^{m-1}\left(\vartheta\left(j\right)\right)
\end{align*}
By the \nameref{Extended Extension Theorem Corollary} \ref{Extended Extension Theorem Corollary} we have $\left[\left(f*g\right)*w\right]=\left[f*\left(g*w\right)\right]$ which means that $\left(\left[f\right]*\left[g\right]\right)*\left[w\right]=\left[f\right]*\left(\left[g\right]*\left[w\right]\right)$
\end{proof}

\begin{axiom}[Distributivity]\noindent\\
For all real numbers $\left[f\right]$, $\left[g\right]$ and $\left[w\right]$ we have $\left(\left[f\right]+\left[g\right]\right)*\left[w\right]=\left(\left[f\right]*\left[w\right]\right)+\left(\left[g\right]*\left[w\right]\right)$.
\end{axiom}
\begin{proofidea} After the zero cases, everything turns on whether the two summands share a sign, the sign of the multiplier never matters, since the parity computation at position $0$ comes out the same either way. When they do share a sign the identity is a direct computation on the convolution, closed by the \nameref{Equivalence Lemma} \ref{Equivalence Lemma}. When they do not, their sum is the surplus of the larger over the smaller, represented by a tracking function, and distributing over that surplus returns us to the same-sign computation. The delicate step is choosing the representative: it must dominate pointwise and carry the sign of the larger summand, so the argument splits on which summand that is, with the \nameref{Absolute-Opposite Lemma} \ref{Absolute-Opposite Lemma} handling the case of equal magnitudes.
\end{proofidea}
\begin{proof}Firstly, if $\left[w\right]=\left[0\right]$ then we have seen in the proof of the \nameref{First Sub-Multiplication Auxiliary Lemma} \ref{First Sub-Multiplication Auxiliary Lemma} that $\left[0\right]*\left[t\right]=\left[0\right]$ for all real numbers $\left[t\right]$. Therefore,
\begin{align*}
\left(\left[f\right]+\left[g\right]\right)*\left[w\right]&=\left(\left[f\right]+\left[g\right]\right)*\left[0\right]=\left[0\right]=\left[0\right]+\left[0\right]\\
&=\left(\left[f\right]*\left[0\right]\right)+\left(\left[g\right]*\left[0\right]\right)=\left(\left[f\right]*\left[w\right]\right)+\left(\left[g\right]*\left[w\right]\right)
\end{align*}
going forward we suppose $\left[w\right]\neq\left[0\right]$.
\textbf{Case 1. $\left[f\right]=\left[0\right]\vee\left[g\right]=\left[0\right]$:} Take without loss of generality $\left[g\right]=\left[0\right]$. We have seen in the proof of the \nameref{First Sub-Multiplication Auxiliary Lemma} \ref{First Sub-Multiplication Auxiliary Lemma} that $\left[0\right]*\left[w\right]=\left[0\right]$ for all real numbers $\left[w\right]$. Thus, 
\begin{align*}
\left(\left[f\right]+\left[g\right]\right)*\left[w\right]&=\left(\left[f\right]+\left[0\right]\right)*\left[w\right]=\left[f\right]*\left[w\right]=\left(\left[f\right]*\left[w\right]\right)+\left[0\right]\\
&=\left(\left[f\right]*\left[w\right]\right)+\left(\left[0\right]*\left[w\right]\right)=\left(\left[f\right]*\left[w\right]\right)+\left(\left[g\right]*\left[w\right]\right)
\end{align*}
\textbf{Case 2. $\sigma\left(\left[f\right]\right)=\sigma\left(\left[g\right]\right)$:} Thus $f\left(0\right)+g\left(0\right)=2m$ for some $m\in\omega$. Then 
\begin{align*}
\begin{array}{lll@{}}
\left(\left(f+g\right)*w\right)\left(0\right)&=f\left(0\right)g\left(0\right)+w\left(0\right)&\text{and}\\
\left(\left(f*w\right)+\left(g*w\right)\right)\left(0\right)&=\left(f\left(0\right)+w\left(0\right)\right)\left(g\left(0\right)+w\left(0\right)\right)\\
&=f\left(0\right)g\left(0\right)+w\left(0\right)\left(f\left(0\right)+g\left(0\right)+w\left(0\right)\right)
\end{array}
\end{align*}
Thus, 
\begin{align*}
&\left(\left(f+g\right)*w\right)\left(0\right)+\left(f*w+g*w\right)\left(0\right)\\
=&f\left(0\right)g\left(0\right)+w\left(0\right)+f\left(0\right)g\left(0\right)+w\left(0\right)\left(f\left(0\right)+g\left(0\right)+w\left(0\right)\right)\\
=&2f\left(0\right)g\left(0\right)+w\left(0\right)\left(2m+w\left(0\right)+1\right)
\end{align*}
We see that this is even for any $w\left(0\right)$. Furthermore, as for all $M\in\omega_+$ we have
\begin{align*}
\left(\left(f+g\right)*w\right)\left(M\right)&=\sum\limits_{n=1}^{M}\left(\sum\limits_{k=1}^{M}\left(\left(f\left(n\right)+g\left(n\right)\right)w\left(k\right)r\left(n,k,M\right)\right)\right)\\
&=\sum\limits_{n=1}^{M}\left(\sum\limits_{k=1}^{M}\left(f\left(n\right)w\left(k\right)r\left(n,k,M\right)\right)\right)+\sum\limits_{n=1}^{M}\left(\sum\limits_{k=1}^{M}\left(g\left(n\right)w\left(k\right)r\left(n,k,M\right)\right)\right)\\
&=\left(\left(f*w\right)+\left(g*w\right)\right)\left(M\right)
\end{align*}
Therefore, by the \nameref{Equivalence Lemma} \ref{Equivalence Lemma} we have that $\left[\left(f+g\right)*w\right]=\left[\left(f*w\right)+\left(g*w\right)\right]$.\\
\textbf{Case 3. $\sigma\left(\left[f\right]\right)\neq\sigma\left(\left[g\right]\right)$:} Take without loss of generality $\left|\left[f\right]\right|\geq\left|\left[g\right]\right|$.\\
If $\left|\left[f\right]\right|=\left|\left[g\right]\right|$ then by the \nameref{Absolute-Opposite Lemma} \ref{Absolute-Opposite Lemma} we have that $\left[f\right]={}^o\left[g\right]$. Hence 
\begin{align*}
\left(\left[f\right]+\left[g\right]\right)*\left[w\right]=\left[0\right]*\left[w\right]=\left[0\right]
\end{align*}
Furthermore, by taking $f_p\in\left[f\right]$ and ${}^of_p\in\left[g\right]$ we have that
\begin{align*}
\left(f_p*w\right)\left(0\right)+\left({}^of_p*w\right)\left(0\right)=2\left(f_p\left(0\right)+w\left(0\right)\right)+1
\end{align*}
which is always odd and thus $\sigma\left(f_p*w\right)\neq\sigma\left({}^of_p*w\right)$. Hence, as for all $n\in\omega_+$ we have $f_p\left(n\right)={}^of_p\left(n\right)$ we have for all $K\in\omega_+$
\begin{align*}
\left|\left(f_p*w\right)\right|\left(K\right)=\left(f_p*w\right)\left(K\right)=\left({}^of_p*w\right)\left(K\right)=\left|\left({}^of_p*w\right)\right|\left(K\right)
\end{align*}
Thus, by the \nameref{Equivalence Lemma} \ref{Equivalence Lemma} as $\left|\left(f_p*w\right)\right|\left(0\right)+\left|\left({}^of_p*w\right)\right|\left(0\right)=0+0=2*0$ we have that $\left|\left(f_p*w\right)\right|=\left|\left({}^of_p*w\right)\right|$ and thus 
\begin{align*}
\left(\left[f\right]*\left[w\right]\right)+\left(\left[g\right]*\left[w\right]\right)=\left[0\right]=\left(\left[f\right]+\left[g\right]\right)*\left[w\right]
\end{align*}
If $\left|\left[f\right]\right|>\left|\left[g\right]\right|$ then by the definition of order we have a function $a\in\left|\left[f\right]\right|$ such that for all $n\in\omega_+$ we have 
\begin{align*}
a\left(n\right)\geq\left|g\right|\left(n\right)=g\left(n\right)
\end{align*}
By the \nameref{Absolute-Opposite Lemma} \ref{Absolute-Opposite Lemma} we have that $a\in\left[f\right]$ or $a\in{}^o\left[f\right]$. Define $b\in\left[f\right]$ as
\begin{align*}
b=\begin{cases}a&a\in\left[f\right]\\
{}^oa&a\in{}^o\left[f\right]
\end{cases}
\end{align*}
Notice that as for all $n\in\omega_+$ we have
\begin{align*}
{}^oa\left(n\right)=a\left(n\right)\geq g\left(n\right)
\end{align*}
And thus for all $n\in\omega_+$ we have $b\left(n\right)\geq g\left(n\right)$. We define $\xi\in\mathcal{T}_g^f$ where 
\begin{align*}
\xi\left(k\right)=\begin{cases}b\left(0\right)&k=0\\
b\left(k\right)-g\left(k\right)&k\in\omega_+
\end{cases}
\end{align*}
Hence, by the \nameref{Rearrangement Lemma} \ref{Rearrangement Lemma} we have $\left[\xi\right]=\left[f\right]+\left[g\right]$. Therefore,
\begin{align*}
\left(\left[f\right]+\left[g\right]\right)*\left[w\right]=\left[\xi\right]*\left[w\right]=\left[\xi*w\right]
\end{align*}
As for all $n\in\omega_+$ we have $b\left(n\right)\geq g\left(n\right)$ we also have that for all $M\in\omega_+$
\begin{align*}
\left(b*w\right)\left(M\right)\geq\left(g*w\right)\left(M\right)
\end{align*}
Therefore, as $b\left(0\right)+g\left(0\right)=2m+1$ we have 
\begin{align*}
\left(b*w\right)\left(0\right)+\left(g*w\right)\left(0\right)=b\left(0\right)+w\left(0\right)+g\left(0\right)+w\left(0\right)=2\left(m+w\left(0\right)\right)+1
\end{align*}
and thus $\sigma\left(b*w\right)\neq\sigma\left(g*w\right)$. Hence, by the \nameref{Rearrangement Lemma} \ref{Rearrangement Lemma} we have that $\left[\zeta\right]=\left[\left(b*w\right)+\left(g*w\right)\right]$ where $\zeta\in\mathcal{T}_{g*w}^{b*w}$ such that
\begin{align*}
\zeta\left(k\right)=\begin{cases}
\left(b*w\right)\left(0\right)=b\left(0\right)+w\left(0\right)&k=0\\
\left(b*w\right)\left(M\right)-\left(g*w\right)\left(M\right)&M\in\omega_+
\end{cases}
\end{align*}
Observe that we have for all $M\in\omega_+$ we have
\begin{align*}
\zeta\left(M\right)&=\sum\limits_{n=1}^{M}\left(\sum\limits_{k=1}^{M}\left(b\left(n\right)w\left(k\right)r\left(n,k,M\right)\right)\right)-\sum\limits_{n=1}^{M}\left(\sum\limits_{k=1}^{M}\left(g\left(n\right)w\left(k\right)r\left(n,k,M\right)\right)\right)\\
&=\sum\limits_{n=1}^{M}\left(\sum\limits_{k=1}^{M}\left(\left(b\left(n\right)-g\left(n\right)\right)w\left(k\right)r\left(n,k,M\right)\right)\right)=\sum\limits_{n=1}^{M}\left(\sum\limits_{k=1}^{M}\left(\xi\left(n\right)w\left(k\right)r\left(n,k,M\right)\right)\right)\\
&=\left(\xi*w\right)\left(M\right)
\end{align*}
And thus, as $\left(\xi*w\right)\left(0\right)=\xi\left(0\right)+w\left(0\right)=b\left(0\right)+w\left(0\right)$ we have 
\begin{align*}
\left(\left[f\right]*\left[w\right]\right)+\left(\left[g\right]*\left[w\right]\right)&=\left[\left(b*w\right)+\left(g*w\right)\right]=\left[\zeta\right]=\left[\xi*w\right]\\
&=\left[\left(f+g\right)*w\right]=\left(\left[f\right]+\left[g\right]\right)*\left[w\right]
\end{align*}
\end{proof}
\begin{lemma}[Strict Multiplicative Relation of Order Lemma]\label{Strict Multiplicative Relation of Order Lemma}\noindent\\
For any real numbers $\left[a\right]$ and $\left[b\right]$ such that $\left[a\right]<\left[b\right]$ and any real number $\left[c\right]$ such that $\sigma\left(\left[c\right]\right)=1$ we have $\left[a\right]*\left[c\right]<\left[b\right]*\left[c\right]$.
\end{lemma}
\begin{proof} By the \nameref{Strict Additive Relation of Order Lemma} \ref{Strict Additive Relation of Order Lemma} we have 
\begin{align*}
\left[0\right]=\left[a\right]+{}^o\left[a\right]<\left[b\right]+{}^o\left[a\right]
\end{align*}
Hence, we have that $\sigma\left(\left[b\right]+{}^o\left[a\right]\right)=1$ and thus we have that $\left(b+{}^oa\right)\left(0\right)=2m$ for some $m\in\omega$. Furthermore, as $\sigma\left(c\right)=1$ we have $c\left(0\right)=2m'$ for some $m'\in\omega$. Therefore, we have that
\begin{align*}
\left(\left(b+{}^oa\right)*c\right)\left(0\right)=\left(b+{}^oa\right)\left(0\right)+c\left(0\right)=2\left(m+m'\right)
\end{align*}
Moreover, as neither $\left(b+{}^oa\right)$ nor $c$ is a zero-function we have by the argument from the proof of \nameref{First Sub-Multiplication Auxiliary Lemma} \ref{First Sub-Multiplication Auxiliary Lemma} that $\left(\left(b+{}^oa\right)*c\right)$ is not a zero function as well. Hence, we have that $\sigma\left(\left[\left(b+{}^oa\right)*c\right]\right)=1$ and thus
\begin{align*}
&\left[0\right]<\left(\left[b\right]+{}^o\left[a\right]\right)*\left[c\right]=\left[b\right]*\left[c\right]+{}^o\left[a\right]*\left[c\right]\\
\end{align*}
Which means that
\begin{align*}
\left[a\right]*\left[c\right]&<\left[b\right]*\left[c\right]+{}^o\left[a\right]*\left[c\right]+\left[a\right]*\left[c\right]=\left[b\right]*\left[c\right]+\left({}^o\left[a\right]+\left[a\right]\right)*\left[c\right]\\
&=\left[b\right]*\left[c\right]+\left(\left[0\right]\right)*\left[c\right]=\left[b\right]*\left[c\right]+\left[0\right]=\left[b\right]*\left[c\right]
\end{align*}
\end{proof}

\begin{axiom}[Multiplicative Relation of Order]\noindent\\
For all real numbers $\left[a\right]$ and $\left[b\right]$ such that $\left[a\right]\leq\left[b\right]$, and all real numbers $\left[c\right]$ such that $\sigma\left(\left[c\right]\right)=1$ we have $\left[a\right]*\left[c\right]\leq\left[b\right]*\left[c\right]$.
\end{axiom}
\begin{proof} $\left[a\right]\leq\left[b\right]$ implies $\left[a\right]<\left[b\right]$ or $\left[a\right]=\left[b\right]$. By the \nameref{Strict Multiplicative Relation of Order Lemma} \ref{Strict Multiplicative Relation of Order Lemma} the former case implies  $\left[a\right]*\left[c\right]<\left[b\right]*\left[c\right]$ which implies $\left[a\right]*\left[c\right]\leq\left[b\right]*\left[c\right]$. The latter case implies that $\left[a\right]*\left[c\right]=\left[b\right]*\left[c\right]$ by the \nameref{Consistency Theorem'''''} \ref{6.I} \ref{Consistency Theorem'''''}, which implies $\left[a\right]*\left[c\right]\leq\left[b\right]*\left[c\right]$.
\end{proof}
\subsection{Positive Denseness and Multiplicative Inverse}
\begin{lemma}[Positive Denseness Lemma]\label{Positive Denseness Lemma}\noindent\\
For all real numbers $\left[g\right]$ and $\left[r\right]$ such that $\left[0\right]\leq\left[g\right]<\left[r\right]$ we have $\left[\left[g\right],\left[r\right]\right]_{\left(\mathcal{S}_{\mathbb{R}},<\right)}\neq\emptyset$.
\end{lemma}
\begin{proof} As $\left[0\right]\leq\left[g\right]<\left[r\right]$ we have by the \nameref{Absolute-Opposite Lemma} \ref{Absolute-Opposite Lemma} that $\left|{}^o\left[g\right]\right|=\left[g\right]<\left[r\right]=\left|\left[r\right]\right|$.\\
Therefore, $\left(r+{}^og\right)=\varrho_g^r$. Furthermore, as $\varrho_g^r\not\in\left[0\right]$ by the \nameref{Null Lemma} \ref{Null Lemma} we have that there exist $k\in\omega_+$ such that $\varrho_g^r\left(k\right)\geq1$. Take minimal such $k$ and denote it as $K$. Define $\varrho':\omega\rightarrow\omega$ as
\begin{align*}
\varrho'\left(k\right)=\begin{cases}b_K\left(\varrho_g^r\right)\left(k\right)&k\neq K+1\\
b_K\left(\varrho_g^r\right)\left(K+1\right)-1&k=K+1
\end{cases}
\end{align*}
where $b_K$ is broadening about $K$. By the \nameref{Tying Theorem} \ref{Tying Theorem} we have that $\varrho'$ is finite as $b_K\left(\varrho_g^r\right)$ is finite and for all $n\in\omega_+$ we have $\varrho'\left(n\right)\leq b_K\left(\varrho_g^r\right)\left(n\right)$.\\
 Furthermore, as $\sigma\left(\varrho'\right)=\sigma\left(b_K\left(\varrho_g^r\right)\right)=\sigma\left(g\right)=\sigma\left(r\right)$ we have that for all $n\in\omega_+$
\begin{align*}
g\left(n\right)\leq\left(\varrho'+g\right)\left(n\right)=\varrho'\left(n\right)+g\left(n\right)\leq b_K\left(\varrho_g^r\right)\left(n\right)+g\left(n\right)=\left(b_K\left(\varrho_g^r\right)+g\right)\left(n\right)
\end{align*}
where $\left(b_K\left(\varrho_g^r\right)+g\right)\in\left[\varrho_g^r+g\right]=\left[r+{}^og+g\right]=\left[r\right]$. However, this shows by the \nameref{One-Case Lemma} that $\left[g\right]\leq\left[\varrho'+g\right]\leq\left[r\right]$.\\
Moreover, as 
\begin{align*}
g\left(K+1\right)<\varrho'\left(K+1\right)+g\left(K+1\right)<b_K\left(\varrho_g^r\right)\left(K+1\right)+g\left(K+1\right)
\end{align*}
we have $\left[\varrho'+g\right]\neq\left[g\right]$ and $\left[\varrho'+g\right]\neq\left[r\right]$ by the \nameref{Domination Theorem} \ref{Domination Theorem}. Therefore, $\left[\varrho'+g\right]\in\left[\left[g\right],\left[r\right]\right]_{\left(\mathcal{S}_{\mathbb{R}},<\right)}$ and thus the set is non-empty. This shows that the sub-set of elements greater or equal to zero in $\mathcal{S}_{\mathbb{R}}$ is dense.\\\\
\end{proof}

\begin{lemma}[Edging Lemma]\label{Edging Lemma}\noindent\\
For all real numbers $\left[a\right]$ and $\left[c\right]$ such that $\sigma\left(\left[a\right]\right)=\sigma\left(\left[c\right]\right)=1$ there exists a real number $\left[b\right]$ such that $\sigma\left(\left[b\right]\right)=1$ and $\left[a\right]*\left[b\right]<\left[c\right]$.
\end{lemma}

\begin{intuition}[Edging Lemma] This lemma supplies wiggle room. Given a positive number and any positive threshold, we can always find a positive multiplier small enough that the product slips just under the threshold. It is the multiplicative counterpart of being able to take a small enough step, and it is exactly what we need to squeeze a reciprocal from below without overshooting.
\end{intuition}

 \begin{proof} As $\left[c\right]\neq\left[0\right]$ we have by the \nameref{Null Lemma} \ref{Null Lemma} that $\delta:=\min\left(\left\{j\in\omega_+\middle|c_p\left(j\right)\neq0\right\}\right)\neq\emptyset$. Take $a_p\in\left[a\right]$. By the \nameref{Tying Theorem} \ref{Tying Theorem} we have that there exists a constant $C\in\omega_+$ such that for all $K\in\omega_+$ we have $\sum\limits_{n=1}^{K}\left(a_p\left(n\right)\prod\limits_{j=n}^{K-1}\left(\vartheta\left(j\right)\right)\right)\leq C\prod\limits_{j=1}^{K-1}\left(\vartheta\left(j\right)\right)$. Take $\kappa:=\min\left(\left\{k\in\omega_+\middle|\prod\limits_{n=\delta}^{k-1}\left(\vartheta\left(n\right)\right)>C\right\}\right)$. Define $b':\omega\rightarrow\omega$ 
\begin{align*}
b'\left(k\right)=\begin{cases}
1&k=\kappa\\
0&otherwise
\end{cases}
\end{align*}
As $b'$ is auxiliary it is finite. By definition we have that for all $K\geq\kappa+1:$
\begin{align*}
&\sum\limits_{M=1}^{K}\left(\left(a_p*b'\right)\left(M\right)\prod\limits_{j=M}^{K-1}\left(\vartheta\left(j\right)\right)\right)=\sum\limits_{M=1}^{K}\left(\sum\limits_{n=1}^{M}\left(\sum\limits_{k=1}^{M}\left(a_p\left(n\right)b'\left(k\right)r\left(n,k,M\right)\right)\right)\prod\limits_{j=M}^{K-1}\left(\vartheta\left(j\right)\right)\right)\\
=&\sum\limits_{n=1}^{K}\left(\sum\limits_{k=1}^{K}\left(a_p\left(n\right)b'\left(k\right)\sum\limits_{M=\max\left(\left\{n,k\right\}\right)}^{K}\left(r\left(n,k,M\right)\prod\limits_{j=M}^{K-1}\left(\vartheta\left(j\right)\right)\right)\right)\right)\\
=&\sum\limits_{n=1}^{K}\left(a_p\left(n\right)\max\left(\left\{t\in\omega\middle|t\prod\limits_{j=1}^{\kappa-1}\left(\vartheta\left(j\right)\right)\leq\prod\limits_{j=n}^{K-1}\left(\vartheta\left(j\right)\right)\right\}\right)\right)\\
\leq&\max\left(\left\{t\in\omega\middle|t\prod\limits_{j=1}^{\kappa-1}\left(\vartheta\left(j\right)\right)\leq\sum\limits_{n=1}^{K}\left(a_p\left(n\right)\prod\limits_{j=n}^{K-1}\left(\vartheta\left(j\right)\right)\right)\right\}\right)\\
\leq&\max\left(\left\{t\in\omega\middle|t\prod\limits_{j=1}^{\kappa-1}\left(\vartheta\left(j\right)\right)\leq C\prod\limits_{j=1}^{K-1}\left(\vartheta\left(j\right)\right)\right\}\right)=C\prod\limits_{j=\kappa}^{K-1}\left(\vartheta\left(j\right)\right)\\
<&\prod\limits_{j=\delta}^{K-1}\left(\vartheta\left(j\right)\right)\leq c_p\left(\delta\right)\prod\limits_{j=\delta}^{K-1}\left(\vartheta\left(j\right)\right)\leq\sum\limits_{n=1}^{K}c_p\left(n\right)\prod\limits_{j=n}^{K-1}\left(\vartheta\left(j\right)\right)
\end{align*}
Hence, as $\sigma\left(\left[a\right]*\left[b'\right]\right)=\sigma\left(\left[c\right]\right)=1$ we have by the \nameref{Boundedness Theorem} \ref{Boundedness Theorem} (as for any $T\in\omega_+$ we may find a $K\geq\kappa+1$), that $\left[a_p*b'\right]\leq\left[c_p\right]$ and thus $\left[a\right]*\left[b'\right]\leq\left[c\right]$. We define $b:\omega\rightarrow\omega$ as 
\begin{align*}
b\left(k\right)=\begin{cases}
1&k=\kappa+1\\
0&otherwise
\end{cases}
\end{align*}
Again, as $b$ is an auxiliary function it is finite. Moreover, we see that $\left[b\right]<\left[b'\right]$ by the \nameref{Equivalent to the Definition of Strict Order} \ref{Equivalent to the Definition of Strict Order}. Hence, by the \nameref{Strict Multiplicative Relation of Order Lemma} \ref{Strict Multiplicative Relation of Order Lemma}, as $\sigma\left(\left[a\right]\right)=1$ we have $\left[a\right]*\left[b\right]<\left[a\right]*\left[b'\right]\leq\left[c\right]$.
\end{proof}

\begin{axiom}[Multiplicative Inverse]\noindent\\
For all non-zero real numbers $\left[a\right]$ there exists a real number $\left[\xi\right]$ such that $\left[a\right]*\left[\xi\right]=\left[1\right]$.
\end{axiom}
\begin{proofidea} The reciprocal is located as a boundary rather than constructed digit by digit. Split the positive numbers into those which multiply $\left[a\right]$ up to at least $\left[1\right]$ and those which do not; completeness supplies the number sitting exactly at the divide. It remains to exclude both strict inequalities for that number, and each exclusion uses one of the two lemmas just proved: denseness produces a number strictly between, and the \nameref{Edging Lemma} \ref{Edging Lemma} produces a multiplier small enough to cross the gap without overshooting.
\end{proofidea}
\begin{proof} Let a non-zero real number $\left[a\right]$ be given. Define the set $\Xi:=\left\{\left[b\right]\in\mathcal{S}_{\mathbb{R}}\middle|\left[a\right]*\left[b\right]\geq\left[1\right]\right\}$. Notice that this set is not empty, as, by taking the auxiliary function $f:\omega\rightarrow\omega$ defined as
\begin{align*}
f\left(k\right)=\begin{cases}
a_p\left(0\right)&k=0\\
\prod\limits_{j=1}^{\delta-1}\left(\vartheta\left(j\right)\right)&k=1\\
0&otherwise
\end{cases}
\end{align*}
where $\delta:=\min\left(\left\{t\in\omega_+|a_p\left(t\right)\neq0\right\}\right)$. We see that for all $T\in\omega$ there exists $K\in\left[\max\left(\left\{\delta,T\right\}\right),\infty\right]_\omega\subset\left[T,\infty\right]_\omega$ we have
\begin{align*}
&\sum\limits_{M=1}^K\left(\left(a_p*f\right)\left(M\right)\prod\limits_{p=M}^{K-1}\left(\vartheta\left(p\right)\right)\right)\geq a_p\left(\delta\right)f\left(1\right)r\left(\delta,1,\delta\right)\prod\limits_{p=\delta}^{K-1}\left(\vartheta\left(p\right)\right)\\
=&a_p\left(\delta\right)\prod\limits_{j=1}^{\delta-1}\left(\vartheta\left(j\right)\right)1\prod\limits_{p=\delta}^{K-1}\left(\vartheta\left(p\right)\right)=a_p\left(\delta\right)\prod\limits_{j=1}^{K-1}\left(\vartheta\left(j\right)\right)\geq\prod\limits_{j=1}^{K-1}\left(\vartheta\left(j\right)\right)=\sum\limits_{M=1}^K\left(1_p\left(M\right)\prod\limits_{p=M}^{K-1}\left(\vartheta\left(p\right)\right)\right)
\end{align*}
Therefore, as $\sigma\left(a_p*f\right)=\sigma\left(1_p\right)=1$ we have by the \nameref{Boundedness Theorem} \ref{Boundedness Theorem} that $\left[a\right]*\left[f\right]=\left[a_p\right]*\left[f\right]\geq\left[1\right]$. Hence, $\Xi$ is not empty.\\
We see that $\Xi$ is bounded from below if $\sigma\left(\left[a\right]\right)=1$ as $\left[b\right]\leq\left[0\right]$ implies $\left[a\right]*\left[b\right]\leq\left[0\right]<\left[1\right]$ and from above if $\sigma\left(\left[a\right]\right)=0$ as then $\left[a\right]\leq\left[0\right]$ and thus $\left[b\right]>\left[0\right]$ implies $\sigma\left(\left[b\right]\right)=1$ which means that $\sigma\left(\left[a\right]*\left[b\right]\right)=0$ and hence $\left[a\right]*\left[b\right]\leq\left[0\right]<\left[1\right]$.\\
Suppose $\sigma\left(\left[a\right]\right)=1$ take $\left[\xi\right]:=\inf\left(\Xi\right)$. Define also 
\begin{align*}
\Phi:=\mathcal{S}_\mathbb{R}-\Xi=\left\{\left[b\right]\in\mathcal{S}_\mathbb{R}\middle|\neg\left(\left[a\right]*\left[b\right]\geq\left[1\right]\right)\right\}=\left\{\left[b\right]\in\mathcal{S}_\mathbb{R}\middle|\left[a\right]*\left[b\right]<\left[1\right]\right\}
\end{align*}
Thus, by definition, for all real numbers $\left[x\right]$ we have that 
\begin{align*}
\left(\left[x\right]\in\Xi\wedge\left[x\right]\not\in\Phi\right)\vee\left(\left[x\right]\not\in\Xi\wedge\left[x\right]\in\Phi\right)
\end{align*}
Furthermore, we see that for all $\left[b\right]\in\Xi$ and for all $\left[c\right]\in\Phi$ we have $\left[b\right]>\left[c\right]$. This is because, $\left[b\right]\neq\left[c\right]$ and $\left[b\right]\geq\left[c\right]$ as if $\left[b\right]\leq\left[c\right]$ then as $\sigma\left(\left[a\right]\right)=1$ we have $\left[a\right]*\left[b\right]\leq\left[a\right]*\left[c\right]<\left[1\right]$, a contradiction to $\left[b\right]\in\Xi$. Hence, $\Phi$ is bounded from above and thus it has a supremum. Moreover, any element of $\Xi$ is an upper bound of $\Phi$ and any element of $\Phi$ is a lower bound of $\Xi$.\\
Suppose $\sup\left(\Phi\right)<\left[\xi\right]$. Thus, as $\left[0\right]\in\Phi$ we have $\left[0\right]\leq\sup\left(\Phi\right)<\left[\xi\right]$. Thus, by the \nameref{Positive Denseness Lemma} \ref{Positive Denseness Lemma} there exists $\left[y\right]\in\left[\sup\left(\Phi\right),\left[\xi\right]\right]_{\left(\mathcal{S}_{\mathbb{R}},<\right)}$. But by definition $\left[y\right]>\sup\left(\Phi\right)$ means that $\left[a\right]*\left[y\right]\geq\left[1\right]$ and $\left[y\right]<\left[\xi\right]$ means that $\left[a\right]*\left[y\right]<\left[1\right]$ which is a contradiction. Therefore, $\sup\left(\Phi\right)\geq\left[\xi\right]$.\\
However, if $\sup\left(\Phi\right)>\left[\xi\right]$ then as $\left[0\right]\in\Phi$ we have that $\left[0\right]$ is a lower bound of $\Xi$ and thus by definition $\left[0\right]\leq\left[\xi\right]$. Hence, by the \nameref{Positive Denseness Lemma} \ref{Positive Denseness Lemma} there exists $\left[y\right]\in\left[\left[\xi\right],\sup\left(\Phi\right)\right]_{\left(\mathcal{S}_{\mathbb{R}},<\right)}$. But by definition either $\left[y\right]\in\Xi$ or $\left[y\right]\in\Phi$. In the first case $\left[y\right]$ is an upper bound of $\Phi$ which is a contradiction to $\left[y\right]<\sup\left(\Phi\right)$. In the second case $\left[y\right]$ is a lower bound of $\Xi$ which is a contradiction to $\left[y\right]>\inf\left(\Xi\right)=\left[\xi\right]$.\\
Therefore, $\sup\left(\Phi\right)=\left[\xi\right]$. We claim that $\left[a\right]*\left[\xi\right]=\left[1\right]$. Suppose not, then by totality we have two cases $\left[a\right]*\left[\xi\right]<\left[1\right]$ or $\left[a\right]*\left[\xi\right]>\left[1\right]$.\\
If $\left[a\right]*\left[\xi\right]<\left[1\right]$ then $\left[1\right]+{}^o\left(\left[a\right]*\left[\xi\right]\right)>0$ and thus by the \nameref{Edging Lemma} \ref{Edging Lemma} there exists a real number $\left[b\right]$ where $\sigma\left(\left[b\right]\right)=1$ and $\left[a\right]*\left[b\right]<\left[1\right]+{}^o\left(\left[a\right]*\left[\xi\right]\right)$ and thus $\left[a\right]*\left[b\right]+\left[a\right]*\left[\xi\right]<\left[1\right]$, which means that $\left[a\right]*\left(\left[b\right]+\left[\xi\right]\right)<\left[1\right]$. Thus $\left[b\right]+\left[\xi\right]\in\Phi$. However, as $\sigma\left(\left[b\right]\right)=1$ we have $\left[b\right]>\left[0\right]$ and thus $\left[b\right]+\left[\xi\right]>\left[\xi\right]$. This contradicts $\left[\xi\right]=\sup\left(\Phi\right)$.\\
If $\left[a\right]*\left[\xi\right]>\left[1\right]$ then $\left[a\right]*\left[\xi\right]+{}^o\left[1\right]>0$ and thus by the \nameref{Edging Lemma} \ref{Edging Lemma} there exists a real number $\left[b\right]$ where $\sigma\left(\left[b\right]\right)=1$ and $\left[a\right]*\left[b\right]<\left[a\right]*\left[\xi\right]+{}^o\left[1\right]$. This means that $\left[a\right]*\left[b\right]+\left[1\right]<\left[a\right]*\left[\xi\right]$ and therefore
\begin{align*}
&\left[a\right]*\left[b\right]+\left[a\right]*{}^o\left[b\right]+\left[1\right]=\left[a\right]*\left(\left[b\right]+{}^o\left[b\right]\right)+\left[1\right]=\left[a\right]*\left[0\right]+\left[1\right]=\left[0\right]+\left[1\right]=\left[1\right]\\
<&\left[a\right]*\left[\xi\right]+\left[a\right]*{}^o\left[b\right]=\left[a\right]*\left(\left[\xi\right]+{}^o\left[b\right]\right)
\end{align*}
Hence, $\left[1\right]<\left[a\right]*\left(\left[\xi\right]+{}^o\left[b\right]\right)$. Thus $\left[\xi\right]+{}^o\left[b\right]\in\Xi$. However, as $\left[b\right]>\left[0\right]$ we have ${}^o\left[b\right]<\left[0\right]$ and thus $\left[\xi\right]+{}^o\left[b\right]<\left[\xi\right]$. This contradicts $\left[\xi\right]=\inf\left(\Xi\right)$.\\
Therefore $\left[a\right]*\left[\xi\right]=\left[1\right]$, which means that $\exists\left[\xi\right]\in\mathcal{S}_{\mathbb{R}},\left[a\right]*\left[\xi\right]=1$.\\\\
If $\sigma\left(\left[a\right]\right)=0$ then $\left[a\right]<\left[0\right]$ and thus $\left[0\right]<{}^o\left[a\right]$ which means that $\sigma\left({}^o\left[a\right]\right)=1$. Hence, there exists a real number $\left[\xi\right]$ such that ${}^o\left[a\right]*\left[\xi\right]=\left[1\right]$. Hence, as, by direct calculation ${}^o\left[a\right]*\left[\xi\right]=\left[a\right]*{}^o\left[\xi\right]$, by taking $\left[\xi'\right]={}^o\left[\xi\right]$, we have shown that there exists a real number $\left[\xi'\right]$ such that $\left[a\right]*\left[\xi'\right]=1$.
\end{proof}
\begin{remark}
    It follows that for any non-zero real number $\left[a\right]\in\mathcal{S}_\mathbb{R}$ the inverse of $\left[a\right]$ is never $\left[0\right]$ as then the product would be too. On top of that, by the \nameref{Consistency Theorem'''''} \ref{6.I} \ref{Consistency Theorem'''''} the inverse is unique. We shall denote it as $\left[a\right]^{-1}$.
\end{remark}
\section{Models Fitted to a Question: Rationality and Cantor Series}\label{sec:cantor}
In this section we wish to present some results which arise from our construction. The main result has been shown by Cantor \cite{Cantor1869} and concerns itself with irrationality of a certain type of series. Here we present an alternative proof which is brought up to show that there are results which arise from working with particular models of real arithmetic. By this we do not mean to say that the result presented is extremely important by itself, but it serves as a proof of concept, as we obtain a classical result whose proof requires normally deeper knowledge of number theory through model theoretic means.\\
We shall call a model with a \nameref{System Function} \ref{System Function} $\vartheta:\omega_+\rightarrow\left[2,\infty\right]_\omega$ such that for all $N\in\omega_+$ there exists $m\in\omega$ such that
\begin{align*}
        N\mid\prod\limits_{n=1}^{m-1}\left(\vartheta\left(n\right)\right)
    \end{align*}
    a model with Cantor's divisibility property.\\
Also, going forward we adopt the notation $\left[a\right]-\left[b\right]$ for $\left[a\right]+^o\left[b\right]$ and $a-b$ for $a+^ob$.
\begin{definition}[Integer]
    A real number $\left[a\right]\in\mathcal{S}_\mathbb{R}$ is an integer if it contains a secondary auxiliary function and has $\Psi\left(a_p,a_s\right)=1$ or if $\left[a\right]=\left[0\right]$.\\
    The set of all integers shall be denoted as $\mathcal{S}_\mathbb{Z}$.
\end{definition}

\begin{definition}[Rational Number]
    A real number $\left[a\right]\in\mathcal{S}_\mathbb{R}$ is a rational number if there exist integers $\left[b\right],\left[c\right]\in\mathcal{S}_\mathbb{Z}$, $\left[c\right]\neq\left[0\right]$ such that $\left[a\right]=\left[b\right]*\left[c\right]^{-1}$.\\
    Any real number which is not rational is irrational.
\end{definition}
\begin{theorem}[Rational Number Representation]\label{Rational Number Representation}
    In a model with Cantor's divisibility property any rational number $\left[a\right]=\left[b\right]*\left[c\right]^{-1}$ where $\left[b\right],\left[c\right]\in\mathcal{S}_\mathbb{Z}$, $\left[c\right]\neq\left[0\right]$ contains the function $w:\omega\rightarrow\omega$ defined as
    \begin{align*}
        w\left(k\right)=\begin{cases}
            b_p\left(0\right)+c_p\left(0\right)&k=0\\
            b_p\left(1\right)*R&k=m\\
            0&otherwise
        \end{cases}
    \end{align*}
    where $m\in\omega_+$ is minimal such that
    \begin{align*}
        c_p\left(1\right)\mid\prod\limits_{n=1}^{m-1}\left(\vartheta(n)\right)
    \end{align*}
    and $R$ is a finite ordinal value
    \begin{align*}
        R:=\frac{\prod\limits_{n=1}^{m-1}\left(\vartheta(n)\right)}{c_p(1)}
    \end{align*}
\end{theorem}

\begin{intuition}[Rational Number Representation] In a base rich enough that every integer eventually divides one of the running products, a rational number has a representation that stops, all of its material sitting at a single position. Which position depends on the denominator: it is the first at which the running product absorbs it. This is the functionary form of the schoolbook fact that a fraction terminates precisely once the base accommodates its denominator, and it is what connects rationality to the terminating shape the next theorem exploits.
\end{intuition}

\begin{proof}
    As $a$ is an auxiliary function it is finite. Then by direct calculation we obtain that
    \begin{align*}
        \left(w*c_p\right)\left(n\right)=\begin{cases}
            b_p\left(0\right)+2c_p\left(0\right)&n=0\\
            b_p\left(1\right)*\prod\limits_{n=1}^{m-1}\left(\vartheta(n)\right)&n=m\\
            0&otherwise
            \end{cases}
    \end{align*}
    Then by inductive contractions we obtain that $\left[w\right]*\left[c\right]=\left[w*c_p\right]=\left[h\right]$ where
    \begin{align*}
        h(n)=\begin{cases}
            b_p\left(0\right)+2c_p\left(0\right)&n=0\\
            b_p\left(1\right)&n=1\\
            0&otherwise
            \end{cases}
    \end{align*}
    Then by the \nameref{Equivalence Lemma} \ref{Equivalence Lemma} we have that $\left[h\right]=\left[b\right]$ as $b_p\left(0\right)+h\left(0\right)=2\left(b_p\left(0\right)+c_p\left(0\right)\right)$. Therefore, $\left[a\right]*\left[c\right]=\left[w\right]*\left[c\right]$ and thus by the \nameref{Consistency Theorem'''''} \ref{6.I} \ref{Consistency Theorem'''''} we have $\left[a\right]=\left[w\right]$, hence $w\in\left[a\right]$. 
\end{proof}

\begin{theorem}[Irrational Models Theorem]\label{Irrational Models Theorem}
    In a model with Cantor's divisibility property a real number $\left[a\right]$ is rational if and only if $\left[a\right]=\left[0\right]$ or $\left[a\right]$ has a secondary auxiliary function.
\end{theorem}

\begin{intuition}[Irrational Models Theorem] This is the punchline of the demonstration: in a divisibility-rich base, a real number is rational exactly when it has a second name or it is zero. Rationality is precisely the $0.999\ldots=1$ phenomenon, having both a terminating form and a $\left(\vartheta-1\right)$-tailed twin, and irrationality is simply the absence of that second representation. The whole question of irrationality has been translated into a question about how a number can be written.
\end{intuition}

\begin{proof}
    By definition, if $\left[a\right]=\left[0\right]$ then $\left[a\right]$ is rational. Suppose now that $\left[a\right]$ has a secondary auxiliary function. Then the well-defined primary auxiliary function $a_p$ has the property that for all $n\in\left[\Psi\left(a_p,a_s\right)+1,\infty\right]_\omega$ we have $a_p\left(n\right)=0$. As shown in the proof of the \nameref{Paradise City Lemma} \ref{Paradise City Lemma} we have that the function defined as
    \begin{align*}
        f(n)=\begin{cases}
            a_p\left(0\right)&n=0\\
            \sum\limits_{n=1}^{\Psi\left(a_p,a_s\right)}\left(a_p\left(n\right)\prod\limits_{j=n}^{\Psi\left(a_p,a_s\right)-1}\left(\vartheta\left(j\right)\right)\right)&n=\Psi\left(a_p,a_s\right)\\
            0&otherwise
        \end{cases}
    \end{align*}
    is in $\left[a\right]$ (this is merely a complete broadening up to $\Psi\left(a_p,a_s\right)$). Then by taking the function defined as
    \begin{align*}
        g(n)=\begin{cases}
            \prod\limits_{j=1}^{\Psi\left(a_p,a_s\right)-1}\left(\vartheta\left(j\right)\right)&n=1\\
            0&otherwise
        \end{cases}
    \end{align*}
    which is a primary auxiliary function of an integer. We obtain by direct calculation that
    \begin{align*}
        \left(f*g\right)\left(n\right)=\begin{cases}
            a_p\left(0\right)&n=0\\
            \prod\limits_{j=1}^{\Psi\left(a_p,a_s\right)-1}\left(\vartheta\left(j\right)\right)*\sum\limits_{n=1}^{\Psi\left(a_p,a_s\right)}\left(a_p\left(n\right)\prod\limits_{j=n}^{\Psi\left(a_p,a_s\right)-1}\left(\vartheta\left(j\right)\right)\right)&n=\Psi\left(a_p,a_s\right)\\
            0&otherwise
        \end{cases}
    \end{align*}
    This can be then inductively contracted to become
    \begin{align*}
        h\left(n\right)=\begin{cases}
            a_p\left(0\right)&n=0\\
            \sum\limits_{n=1}^{\Psi\left(a_p,a_s\right)}\left(a_p\left(n\right)\prod\limits_{j=n}^{\Psi\left(a_p,a_s\right)-1}\left(\vartheta\left(j\right)\right)\right)&n=1\\
            0&otherwise
        \end{cases}
    \end{align*}
    which is a primary auxiliary function for an integer. Thus, $\left[f*g\right]=\left[h\right]$ by the \nameref{Shifting Theorem} \ref{Shifting Theorem}. This then shows that $\left[a\right]$ is rational, as $\left[f\right]=\left[f\right]*\left[g\right]*\left[g\right]^{-1}=\left[h\right]*\left[g\right]^{-1}$.\\
    Now suppose that $\left[a\right]$ is rational. If $\left[a\right]=\left[0\right]$ we are done. Suppose not, then $\left[a\right]=\left[b\right]*\left[c\right]^{-1}$ where $\left[b\right]\neq\left[0\right]$. By \nameref{Rational Number Representation} \ref{Rational Number Representation} $w\in\left[a\right]$ as constructed. By our construction of the function $w$ it has a tail of $0$s and thus its auxiliary function will have to have a tail of $0$s. Furthermore, as $\left[w\right]=\left[a\right]\neq\left[0\right]$ $w$ cannot be identically zero. Thus $\left[w\right]$ and hence $\left[a\right]$ has a secondary auxiliary function.
\end{proof}
\begin{example} The divisibility hypothesis is a real restriction, and it is
easy to see both sides of it.\\
Take the base of Example 2, $\vartheta(n)=n+1$. Its running products are $\prod_{j=1}^{m-1}(\vartheta(j))=2*3\cdots m=m!$, and every positive integer $k$ divides $k!$, so the hypothesis holds. Accordingly every rational terminates here: for instance $\nicefrac{1}{3}$ is $(0;0,0,2,0,\dots)$, since position $3$ carries weight $\nicefrac{1}{6}$ and $2*\nicefrac16=\nicefrac13$, and the digit $2$ is admissible because $\vartheta(2)=3$.\\
Now take the constant base ten. Its running products are the powers of ten, and $3$ divides none of them, so the hypothesis fails, and indeed $\nicefrac{1}{3}=(0;0,3,3,3,\dots)$ does not terminate, has no second name, and is therefore not detected as rational by the criterion above. The theorem is not asserting that irrationality depends on the base; it is asserting that in a base rich enough to absorb every denominator, and only in such a base, rationality becomes visible as a property of how a number is written.
\end{example}
\begin{theorem}[Model Series Theorem]\label{Model Series Theorem}
    Let a convergent infinite series of the form
    \begin{align*}
        \sum_{n=1}^{\infty}\frac{a_n}{\prod\limits_{k=1}^{n-1}b_k}
    \end{align*}
    where $\left(a_n\right),\left(b_n\right)$ are sequences of non-negative integers with $b_n\geq 2$ for all $n\in\omega$ be given. Then by taking $\vartheta:\omega_+\rightarrow\left[2,\infty\right]_\omega$ defined as $\vartheta\left(n\right)=b_n$ to be the \nameref{System Function} \ref{System Function} we have that the function $a:\omega\rightarrow\omega$ defined as
        \begin{align*}
            a(n)=\begin{cases}
                0&n=0\\
                a_n&otherwise
            \end{cases}
        \end{align*}
        is a finite function in the model of $\vartheta$ and the series converges and equals $\left[a\right]$.
\end{theorem}

\begin{intuition}[Model Series Theorem] This theorem is the bridge from ordinary infinite series to the model. Given a series whose terms are laid out over the running products of a base, we simply read its numerators off as the digits of a function; the theorem verifies that this function is a genuine number and that the series converges to exactly the real number it names. Choosing the base to match the series turns an analytic object into a single point of our construction, which is what makes the model-theoretic proof of irrationality possible.
\end{intuition}

\begin{proof}
    By \nameref{Rational Number Representation} we see that each $\frac{a_n}{\prod_{k=1}^{n-1}b_k}$ is represented by $f_n:\omega\rightarrow\omega$ defined as
    \begin{align*}
        f_n\left(k\right)=\begin{cases}
            a_n&k=n\\
            0&otherwise
        \end{cases}
    \end{align*}
    By the definition of addition of functions $\sum_{n=1}^{N}\frac{a_n}{\prod_{k=1}^{n-1}b_k}$ equals a number represented by $F_N:\omega\rightarrow\omega$ defined as
    \begin{align*}
        F_N\left(k\right)=\begin{cases}
            a_k&k\in\left[1,N\right]_\omega\\
            0&otherwise
        \end{cases}
    \end{align*}
    As the series converges by the Archimedean property we take an integer $\left[C\right]\in\mathcal{S}_\mathbb{Z}$ such that $\left[C\right]>\sum_{n=1}^{\infty}\frac{a_n}{\prod_{k=1}^{n-1}b_k}\geq\sum_{n=1}^{N}\frac{a_n}{\prod_{k=1}^{n-1}b_k}$. Then $\left[C\right]\prod_{k=1}^{N-1}\left(b_k\right)>\sum_{n=1}^N\left(a_n*\prod_{k=n}^{N-1}\left(b_k\right)\right)$. Therefore, for all $N\in\omega_+$ we have
    \begin{align*}
        C_p\left(1\right)\prod_{k=1}^{N-1}\left(\vartheta\left(k\right)\right)>\sum_{n=1}^N\left(a\left(n\right)*\prod_{k=n}^{N-1}\left(\vartheta\left(k\right)\right)\right)
    \end{align*}
    and thus by the \nameref{Tying Theorem} \ref{Tying Theorem} $a$ is a finite function.\\
    Let $\epsilon>0$ be given. Take $m\in\omega$ such that $S_m:=\frac{1}{\prod_{k=1}^{m-1}b_k}<\epsilon$. Then by the \nameref{Rational Number Representation} \ref{Rational Number Representation} $S_m$ is represented by a function $s_m:\omega\rightarrow\omega$ defined as
    \begin{align*}
        s_m\left(n\right)=\begin{cases}
            1&n=m\\
            0&otherwise
        \end{cases}
    \end{align*}
    As $\sum\limits_{n=1}^{\infty}\frac{a_n}{\prod_{k=1}^{n-1}b_k}$ converges there exists $N\in\left[m,\infty\right]_\omega$ such that $\sum\limits_{n=N+1}^{\infty}\frac{a_n}{\prod_{k=1}^{n-1}b_k}<S_m$ and this means that for all $M\in\left[N+1,\infty\right]_\omega$ we have $\sum\limits_{n=N+1}^{M}\frac{a_n}{\prod_{k=1}^{n-1}b_k}<S_m=\frac{1}{\prod_{k=1}^{m-1}b_k}$ and thus
    \begin{align*}
        \sum_{n=N+1}^{M}\left(a_n\prod_{k=n}^{M-1}\left(b_k\right)\right)<\prod_{k=m}^{M-1}\left(b_k\right)
    \end{align*}
    Therefore
    \begin{align*}
        \sum_{n=1}^{M}\left((a-F_N)(n)\prod_{k=n}^{M-1}\left(b_k\right)\right)=\sum_{n=N+1}^{M}\left(a_n\prod_{k=n}^{M-1}\left(b_k\right)\right)<\prod_{k=m}^{M-1}\left(b_k\right)=\sum_{n=1}^{M}\left(s_m(n)\prod_{k=n}^{M-1}\left(b_k\right)\right)
    \end{align*}
    And thus, for all $T\in\omega_+$ there exists $M\in\left[\max\left(\left\{N+1,T\right\}\right),\infty\right]_\omega$ such that
    \begin{align*}
        \sum_{n=1}^{M}\left((a-F_N)(n)\prod_{k=n}^{M-1}\left(b_k\right)\right)<\sum_{n=1}^{M}\left(s_m(n)\prod_{k=n}^{M-1}\left(b_k\right)\right)
    \end{align*}
    and hence by the \nameref{Boundedness Theorem} \ref{Boundedness Theorem} $0\leq\left[a\right]-\sum_{n=1}^{N}\frac{a_n}{\prod_{k=1}^{n-1}b_k}<S_m<\epsilon$. Ergo $\sum\limits_{n=1}^{N}\frac{a_n}{\prod_{k=1}^{n-1}\left(b_k\right)}\rightarrow\left[a\right]$.
\end{proof}
\begin{theorem}[Cantor's Irrationality Theorem]
    A series of the form
    \begin{align*}
        \sum_{n=1}^\infty\frac{a_n}{\prod_{k=1}^{n-1}(b_k)}
    \end{align*}
    where
    \begin{align*}
        &\forall N\in[1,\infty]_\omega,\exists m\in\omega,&&N\mid\prod_{n=1}^{m-1}\left(b_n\right)\\
        &\exists N\in\omega,\forall n\in[N+1,\infty]_\omega,&&0\leq a_n\leq b_{n-1}-1\\
        &\forall N\in\omega,\exists n\in[N+1,\infty]_\omega,&&0<a_n\\
        &\forall N\in\omega,\exists n\in[N+1,\infty]_\omega,&&a_n<b_{n-1}-1\\
    \end{align*}
    is always convergent and equals an irrational real number.
\end{theorem}

\begin{intuition}[Cantor's Irrationality Theorem] Everything now converges. Reading such a series as a digit function places it inside a model whose base has the divisibility property; the growth conditions on the terms guarantee the function never terminates and never settles into an all-$\left(\vartheta-1\right)$ tail; so by the Irrational Models Theorem it has no second name and is therefore irrational. A classical result whose usual proof leans on number theory falls out of the way our numbers are represented, which is the entire point of choosing models to fit the question at hand.
\end{intuition}

\begin{proof} The convergence follows by the comparison test (on the tail of the series) to the series $\sum\frac{b_{n-1}}{\prod_{k=1}^{n-1}\left(b_k\right)}=\sum\frac{1}{\prod_{k=1}^{n-2}\left(b_k\right)}\leq\sum\frac{1}{2^{n-2}}$. Then by the \nameref{Model Series Theorem} \ref{Model Series Theorem} where we invoke the third condition to conclude that the series is not $\left[0\right]$ and the \nameref{Irrational Models Theorem} \ref{Irrational Models Theorem} the theorem follows.
\end{proof}
\begin{definition}[$b$-adically Terminating Rationals]
    A rational number $\left[a\right]$ is $b$-adically terminating for an integer $\left[b\right]\geq\left[2\right]:=\left[1\right]+\left[1\right]$ if there exist an integer $\left[c\right]\in\mathcal{S}_\mathbb{Z}$ and $N\in\omega$ such that
    \begin{align*}
        \left[a\right]=\left[c\right]*\left[b^N\right]^{-1}
    \end{align*}
    here $\left[b^N\right]$ of course means $\left[b\right]*\left[b\right]*\dots*\left[b\right]$ $N$ times.
\end{definition}
\begin{theorem}[$b$-adic Model Theorem]\label{b-adic}
    In a model with a constant system function $\vartheta\left(n\right)=b\geq2$ for all $n\in\omega_+$ we have that a non-zero number is $b$-adically terminating\footnote{With respect to the obvious integer with a primary auxiliary function $b_p:\omega\rightarrow\omega$ where $b_p\left(1\right)=b$ and $b_p\left(n\right)=0$ otherwise.} rational if and only if it has a secondary auxiliary function.
\end{theorem}
\begin{proof}
    As $\left[a\right]$ is $b$-adically terminating and non-zero, there exist a non-zero integer $\left[c\right]$ and $N\in\omega$ such that
    \begin{align*}
        \left[a\right]=\left[c\right]*\left[b^N\right]^{-1}
    \end{align*}
    Therefore, by defining the function $f:\omega\rightarrow\omega$ as
    \begin{align*}
        f\left(n\right)=\begin{cases}
            c_p\left(0\right)&n=0\\
            c_p\left(1\right)&n=N+1\\
            0&otherwise
        \end{cases}
    \end{align*}
    Then by inductive contractions we obtain that $\left[f\right]*\left[b^N\right]=\left[f*b^N_p\right]=\left[c_p\right]$\footnote{This is because $\left(f*b^N_p\right)\left(0\right)=f\left(0\right)+b^N_p\left(0\right)=f\left(0\right)+0=c_p\left(0\right)$} Therefore, $\left[a\right]*\left[b^N\right]=\left[f\right]*\left[b^N\right]$ and thus by the \nameref{Consistency Theorem'''''} \ref{6.I} \ref{Consistency Theorem'''''} we have $\left[a\right]=\left[f\right]$, hence $f\in\left[a\right]$. $f$ has a tail of $0$s and thus its auxiliary function will have to have a tail of $0$s. Furthermore, as $\left[f\right]=\left[a\right]\neq\left[0\right]$ $f$ cannot be identically zero. Thus $\left[f\right]$ and hence $\left[a\right]$ has a secondary auxiliary function.\\
    The other implication is done the same way as in the \nameref{Irrational Models Theorem} \ref{Irrational Models Theorem}. Notice that even though the theorem demanded Cantor's divisibility property, the part we invoke didn't and is independent of it.
\end{proof}
\section{Related Work}\label{sec:related}
Constructions of $\mathbb{R}$ are numerous enough to have been surveyed, and we do not claim novelty in having produced another; the survey of Weiss \cite{weiss2015survey} catalogues nineteen. What we claim is that a construction which carries the base as a parameter can be asked questions that a construction without one cannot, and this section is where that claim is made precise against the alternatives. We locate the construction among its relatives, compare it in detail with its nearest neighbour, distinguish our parametrisation from earlier parametrised families, and place the rationality theorem of Section \ref{sec:cantor} in the tradition it belongs to.
\subsection{Constructions of the real numbers}
Constructions of $\mathbb{R}$ fall into a few families: order-theoretic completions, of which Dedekind's cuts \cite{Dedekind1901} are the archetype; metric completions, beginning with Cantor \cite{Cantor1872} and including the nested-interval and uniform-space variants; constructions from approximate endomorphisms of $\mathbb{Z}$, the Eudoxus reals \cite{arthan2004eudoxus}; constructions from series and product expansions of rationals \cite{shiu1974,pintilie1988,knopfmacher1987,knopfmacher1988}; and constructions from digit strings, where the reals are expansions and the work lies in carrying. It is in this last family that the present construction belongs, alongside Faltin, Metropolis, Ross and Rota \cite{FALTIN1976271} and de Bruijn \cite{de1976defining}.

\subsection{Comparison with the wreath construction}
The closest relative is \cite{FALTIN1976271}, and the resemblance is not superficial. Faltin, Metropolis, Ross and Rota also abandon the requirement that digits be reduced below the base, for the reason we gave in Section~1: it is the reduction, not the carrying, that makes digitwise arithmetic awkward. They also restrict to strings satisfying a boundedness condition which, in base two, is the condition our Tying Theorem shows to characterise the finite functions. Two constructions meeting at this condition is no accident; it is what makes unreduced strings represent reals at all.\\
The constructions part company in how the identification is generated. In \cite{FALTIN1976271} the strings form a ring of formal Laurent series over $\mathbb{Z}$, a carry constant $K$ is singled out, and two strings are identified when they differ by a multiple of $K$ by a carry string; the arithmetic is then inherited from the ring at no cost, and the price is paid in recovering the order, which requires a normal form (``clear strings'') and a separate argument. We instead generate the identification by local, individually invertible moves and test it by agreement on finite initial segments. The order is then almost immediate, it is domination of representatives, and the price is paid in showing that the arithmetic descends to the classes, which is the content of our consistency theorems. Neither resolution dominates the other; they distribute the same difficulty differently. Two further differences are of presentation rather than substance: our digits are non-negative with the sign carried in a distinguished position, where \cite{FALTIN1976271} allows integer digits indexed by $\mathbb{Z}$; and our base is permitted to vary with position, where \cite{FALTIN1976271} is stated for a fixed base.

\subsection{Families of constructions}
That last difference is what yields a family rather than a single model, and it should be stated carefully, since parametrised constructions are not new. Pintilie \cite{pintilie1988} constructs $\mathbb{R}$ from an arbitrary sequence of positive rationals tending to zero with divergent sum, and Shiu \cite{shiu1974} observes that the harmonic series in his construction may be replaced by any such sequence; each therefore already yields continuum-many constructions. What we claim is narrower and, we think, more interesting than cardinality. Because the parameter here is the base, it acts directly on the representation theory: the set of numbers admitting a second representative is the $b$-adically terminating rationals when $\vartheta$ is constantly $b$ (\ref{b-adic}), is all of $\mathbb{Q}$ when every integer divides some partial product (\ref{Irrational Models Theorem}), and lies strictly between for intermediate choices. The models are indistinguishable as ordered fields, by categoricity, and differ precisely in how their elements are named. It is that variation, not the cardinality of the family, which makes the base worth choosing to suit a problem.
\section{Conclusion}\label{sec:conclusion}
The question this paper began with was where the multiplicity of names for a real number comes from, and the answer it arrives at is that the multiplicity is entirely a function of the base: fix $\vartheta$ and the set of numbers with two names is determined; vary $\vartheta$ and that set moves, over a range that includes both the $b$-adic rationals and all of $\mathbb{Q}$. Establishing this required building the models rather than assuming them, and it is worth recording where the labour actually fell. Not in the arithmetic, which unreduced digits make easy, and not in the order, which domination makes almost immediate; but in showing that operations defined on representatives descend to the classes at all. The consistency theorems are the longest arguments here for that reason, and we hope that having carried them out in full generality spares others the same work, leaving any particular model of the family ready to use.
 
The freedom this buys is a freedom of description rather than of substance, and it is worth being exact about which. Nothing about the real numbers changes when one moves through the family; what changes is which facts about them are visible on the surface of a name. Section \ref{sec:cantor} is our demonstration that the distinction has consequences: in a base whose partial products absorb every denominator, being rational and having a second name are the same property, and a criterion whose usual proof draws on number theory becomes a remark about how numbers are written. We expect this to be the general pattern, a well-chosen base does not make a hard problem easy, but it can move a problem from one register into another, and the second register may be the one in which the tools are sharper.
 
Two directions follow. The first is internal. Sequences of reals, functions of a real variable, and sequences of functions and much more can all be redefined through a functionary model. This redefinition (which has to be shown equivalent to the original definition) carries with itself new notions of convergence which we intend to develop. The second is the question Section \ref{sec:cantor} leaves open. Oppenheim removed the divisibility hypothesis from Cantor's criterion, and beyond it the rationality of general Cantor series is only partially understood. A base tuned to a given series is precisely the instrument that regime appears to call for; finding out how far it reaches is the work we intend to take up next.

\appendix
\section{Notation}\label{app:notation}
The table below collects the notation introduced throughout the paper, with the place of first definition. Symbols are grouped by role rather than by order of appearance, since a reader consulting the table generally knows what kind of object they are looking for.
\begin{longtable}{@{}>{$}l<{$}p{7.2cm}l@{}}
\caption{Notation used throughout.}\label{tab:notation}\\
\toprule
\textnormal{\textbf{Symbol}} & \textbf{Meaning} & \textbf{Introduced}\\
\midrule
\endfirsthead
\multicolumn{3}{@{}l}{\footnotesize\itshape Table \ref{tab:notation}, continued}\\
\toprule
\textnormal{\textbf{Symbol}} & \textbf{Meaning} & \textbf{Introduced}\\
\midrule
\endhead
\midrule
\multicolumn{3}{r@{}}{\footnotesize\itshape}\\
\endfoot
\bottomrule
\endlastfoot
 
\multicolumn{3}{@{}l}{\textit{Ordinals and intervals}}\\
\omega,\;\omega_+ & the finite ordinals; the non-zero finite ordinals &\ref{Notation}\\
\left[a,b\right]_\omega & the ordinals from $a$ to $b$ inclusive &\ref{Notation}\\
\max\left(\left\{t\mid t*a\leq b\right\}\right) & rational floor of $b$ by $a$ &\ref{Methods}\\
\addlinespace
\multicolumn{3}{@{}l}{\textit{Systems and moves}}\\
\vartheta & system function; $\vartheta\left(n\right)\geq2$ is the base at position $n$ &\ref{System Function}\\
c_n,\;b_n & contraction and broadening about position $n$ &\ref{Contraction} \& \ref{Broadening}\\
\mathcal{CB} & the set of CB-functions &\ref{Sequences of Contractions and Broadenings}\\
\mathcal{D}_{k=q}^{r} & a sequence of contractions and broadenings &\ref{Sequences of Contractions and Broadenings}\\
\mathcal{C}_{k=q}^{r} & a sequence of contractions only &\ref{Sequences of Contractions and Broadenings}\\
\mathcal{B}_{k=q}^{r} & a sequence of broadenings only &\ref{Sequences of Contractions and Broadenings}\\
\mathcal{FS} & the set of finite sequences &\ref{Sequences of Contractions and Broadenings}\\
\text{con}\left(\mathcal{D}_{k=q}^{r}\right) & the conditional of $\mathcal{D}_{k=q}^{r}$ &\ref{Sequences of Contractions and Broadenings}\\
\addlinespace
\multicolumn{3}{@{}l}{\textit{Functions and their auxiliaries}}\\
\sigma\left(f\right) & sign of a function: $0$ if $f\left(0\right)$ is odd, $1$ if even &\ref{Sign of a Function}\\
A_p,\;A_s & the primary and secondary auxiliary functions &\ref{Primary Auxiliary Function} \& \ref{Secondary Auxiliary Function}\\
h_f & an auxiliary function of $f$ &\ref{Paradise City Lemma}\\
f_p,\;f_s & the primary and secondary auxiliary functions of $\left[f\right]$ &\ref{PC Cor I}\\
\Psi\left(f,g\right) & least position at which $f$ and $g$ disagree &\ref{Unicity of Real Numbers}\\
N^{h}_{\min} & least position past which $h$ is constantly $\vartheta-1$ &\ref{Secondary Auxiliary Function}\\
\addlinespace
\multicolumn{3}{@{}l}{\textit{Real numbers}}\\
r_b & the real number for $b\in B$ &\ref{Real Number}\\
\mathcal{S}_\mathbb{R} & the set of real numbers &\ref{Set of Reals}\\
\mathcal{S}_\mathbb{R}^{*} & the set of non-zero real numbers &\ref{Set of Reals}\\
\left[a\right] & the real number containing $a$ &\ref{Set of Reals}\\
\sigma\left(\left[a\right]\right) & sign of a real number (undefined on $\left[0\right]$) &\ref{Sign of a Real Number}\\
\mathcal{N}_{\mathcal{F}in} & the set of finite functions &\ref{Set of Reals}\\
\mathcal{N}_{\mathcal{F}in}^{*} & the set of non-zero finite functions &\ref{Set of Reals}\\
\addlinespace
\multicolumn{3}{@{}l}{\textit{Operations}}\\
{}^{o}f,\;{}^{o}\left[a\right] & the opposite of a function, of a real number &\ref{Opposite Function} \& \ref{Opposite of a Real Number}\\
\left|f\right|,\;\left|\left[a\right]\right| & absolute value of a function, of a real number &\ref{Absolute Value of Functions} \& \ref{Absolute Value of a Real Number}\\
f+g & sub-addition of functions &\ref{Complete Sub-Addition of Functions}\\
\mathcal{T}^{f}_{g} & the set of tracking functions of $f$ over $g$ &\ref{Set of Tracking Functions}\\
\varrho^{f}_{g} & the primary tracking function &\ref{Primary Tracking Function}\\
\Upsilon\left(a,b\right) & the absolute difference of $a$ and $b$ &\ref{Extension Theorem}\\
f*g & sub-multiplication of functions &\ref{Sub-Multiplication}\\
r\left(n,k,M\right) & rest function: weight of $\left(n,k\right)$ landing at $M$ &\ref{Rest Function}\\
R\left(n,k,K\right) & the running total the rest function differences &\ref{Rest Identity}\\
\addlinespace
\multicolumn{3}{@{}l}{\textit{Ad hoc, local to their proofs}}\\
\Xi & a maximum over a condition set; also the set of upper bounds & as used\\
\Lambda,\;\Phi & the stage minima in the Completeness proof & \S5\\
\end{longtable}
 
\thanks{This work is dedicated to all the hurting people of the world. Those who wish to be seen, heard and acknowledged, this is the world seeing, hearing and acknowledging you.\\
With recognition of Royal Free, University College and Middlesex Hospitals Medical School Boat Club and its members.\\
I thank with love my parents.\\
I thank The Kellner Family Foundation for their support of my studies.\\
I would also like to thank all my professors of mathematics who have helped me all throughout my life, namely Sam Coskey and Michel Lenczner.\\
Last special thanks has to go to Sára Michalská, thank you for everything.}
\bibliography{refs}
\end{document}